\documentclass[UTF-8,reqno]{amsart}
\usepackage{enumerate}
\usepackage{mhequ}
\usepackage{amssymb,url,color, booktabs,nccmath,amsmath}
\usepackage{amsthm}
\usepackage[left=2.7cm,right=2.7cm,top=3.4cm,bottom=3.4cm]{geometry}
\usepackage{mathrsfs}
\usepackage{dsfont}
\usepackage{amsfonts}
\usepackage{esint}
\usepackage{amssymb}
\usepackage{mathtools}
\usepackage{mathrsfs}
\usepackage{setspace}
\usepackage{xcolor}
\usepackage{textgreek}
\usepackage{todonotes}
\usepackage{scalerel}
\usepackage{xfrac}
\usepackage{bm}
\usepackage{pifont}

\usepackage[utf8]{inputenc} 

\usepackage{color}
\usepackage[colorlinks=true]{hyperref}

\hypersetup{
	linkcolor=red,          
	citecolor=blue,        
	filecolor=blue,      
	urlcolor=cyan
}

\definecolor{darkergreen}{rgb}{0.0, 0.5, 0.0}

\numberwithin{equation}{section}
\def\theequation{\arabic{section}.\arabic{equation}}
\newcommand{\be}{\begin{eqnarray}}
	\newcommand{\ee}{\end{eqnarray}}
\newcommand{\ce}{\begin{eqnarray*}}
	\newcommand{\de}{\end{eqnarray*}}

\newtheorem{theorem}{Theorem}[section]

\newtheorem{lemma}[theorem]{Lemma}
\newtheorem{proposition}[theorem]{Proposition}

\newtheorem{conjecture}[theorem]{Conjecture}
\newtheorem{corollary}[theorem]{Corollary}

\newtheorem{definition}[theorem]{Definition}
\theoremstyle{definition}
\newtheorem{remark}[theorem]{Remark}

\newcommand{\ve}[0]{\textsf{\textit{v}}}
\newcommand{\te}[0]{\textsf{\textit{t}}}
\newcommand{\pe}[0]{\textsf{\textit{p}}}

\def\${|\!|\!|}

\DeclareMathOperator{\supp}{supp}
\def\T{\mathbb{T}}
\def\R{\mathbb{R}}

\def\Id{\textrm{Id}}

\def\<{{\langle}}
\def\>{{\rangle}}
\def\({{\Big(}}
\def\){{\Big)}}

\def\tr{\mathrm {tr}}

\def\dif{{\mathord{{\rm d}}}}

\def\min{{\mathord{{\rm min}}}}

\def\={&\!\!=\!\!&}
\def\curl{\mathop{\rm curl}\nolimits} 
\DeclareMathOperator*{\esssup}{esssup}

\def\mR{{\mathbb R}}

\def\1{{\mathbf{1}}}

\def\E{\mathbf E}

\def\geq{\geqslant}
\def\leq{\leqslant}
\def\ge{\geqslant}
\def\le{\leqslant}

\def\div{\mathord{{\rm div}}}
\def\iint{\int\!\!\!\int}

\def\<{{\langle}}
\def\>{{\rangle}}
\def\({{\Big(}}
\def\){{\Big)}}

\def\tr{\mathrm {Tr}}

\def\dif{{\mathord{{\rm d}}}}

\def\min{{\mathord{{\rm min}}}}

\def\geq{\geqslant}
\def\leq{\leqslant}
\def\ge{\geqslant}
\def\le{\leqslant}

\def\div{\mathord{{\rm div}}}
\def\iint{\int\!\!\!\int}

 \def\R{\mathbb R}
 \def\R{\mathbb R}    
\def\N{\mathbb N}  
   
\def\<{\langle} \def\>{\rangle}

\makeatletter

\newcommand{\Rmnum}[1]{\expandafter\@slowromancap\romannumeral #1@}
\makeatother

\allowdisplaybreaks

\begin{document}
	\fontsize{10.0pt}{\baselineskip}\selectfont
	
	\title[A proof of Onsager's conjecture for the stochastic 3D Euler equations]{A proof of Onsager's conjecture for the stochastic 3D Euler equations}

	\author{Huaxiang L\"u}
	\address[H. L\"u]{Academy of Mathematics and Systems Science,
		Chinese Academy of Sciences, Beijing 100190, China}
	\email{lvhuaxiang22@mails.ucas.ac.cn }
	
		\author{Lin L\"u}  
	\address[L. L\"u]{School of Mathematics and Statistics, Beijing Institute of Technology, Beijing 100081, China}
	\email{3120235976@bit.edu.cn}

	\author{Rongchan Zhu}
	\address[R. Zhu]{School of Mathematics and Statistics, Beijing Institute of Technology, Beijing 100081, China}
	\email{zhurongchan@126.com}

	\begin{abstract}
		This paper investigates the stochastic 3D Euler equations on a periodic domain $\mathbb{T}^3$, driven by a $GG^*$-Wiener process $B$ of trace class:
		\begin{align*}
			\mathrm{d} u+\div(u\otimes u)\,\mathrm{d} t+\nabla p\,\mathrm{d}t=\mathrm{d}B,  \quad \mathrm{div} u=0.
		\end{align*}
		For any $\vartheta<1/3$, we construct infinitely many global-in-time probabilistically strong and analytically weak solutions $u\in C([0,\infty),C^{\vartheta}(\mathbb{T}^3,\mathbb{R}^3))$. These solutions satisfy a $\textit{pathwise energy inequality}$ up to a stopping time $\mathfrak{t}$, which can be chosen arbitrarily large with high probability, i.e., it holds almost surely
		\begin{align*}
			\|u(t\wedge\mathfrak{t})\|_{L^2}^2< \|u(s\wedge\mathfrak{t})\|_{L^2}^2 +2\int_{s\wedge\mathfrak{t}}^{t\wedge\mathfrak{t}} \big\langle u(r), \dif B(r)  \big\rangle 
			+\tr\big(GG^*\big) (t\wedge\mathfrak{t}-s\wedge\mathfrak{t}),
		\end{align*}
		for any $0\leq s < t<\infty$. We also provide a brief proof of energy conservation for $\vartheta>1/3$, thereby confirming the Onsager theorem for the stochastic 3D Euler equations. The main difficulty of this work lies in deriving pathwise control of the stochastic integral while enhancing the solution’s regularity up to $1/3-$. Our construction is based on the convex integration method, which we adapt to the stochastic context by introducing a novel energy iteration and combining stochastic analysis arguments with a Wong--Zakai type estimate.
	\end{abstract}
	
	\subjclass[2020]{60H15; 35R60; 35Q31}
	\keywords{Stochastic Euler equations, Onsager theorem, Pathwise energy inequality, Convex integration.}
	
	\maketitle
	\tableofcontents
	\section{Introduction}
	
	\addtocontents{toc}{\protect\setcounter{tocdepth}{1}}
		\subsection{Background and Motivation}
	The incompressible Euler equations are a fundamental model in fluid dynamics, governing the time evolution of an incompressible, homogeneous, inviscid fluid. In the deterministic setting, the existence and uniqueness of classical solutions have been extensively studied in both two and three dimensions; see, for instance, \cite{BM02, Tem76, Yud63}. By contrast, irregular solutions may exhibit non-uniqueness and anomalous dissipation. In the seminal work \cite{Sch93}, Scheffer demonstrated the existence of non-trivial weak solutions with compact support in time (see also \cite{Shn97}). Subsequently, Shnirelman constructed non-unique weak solutions in $L^2$ with decreasing energy \cite{Shn00}. It is now well understood that such behavior is generic: non-uniqueness and anomalous dissipation occur in a class of low-regularity solutions to the Euler equations, in connection with Onsager's celebrated conjecture \cite{Ons49}:
	\begin{conjecture}
		Consider the 3D incompressible Euler equations on $[0,T]\times \T^3$.
		\begin{enumerate}[(a)]
			\item Any weak solution $u$  belonging to the H\"older space $C^{\beta}([0,T]\times \T^3)$ for $\beta>1/3$ conserves kinetic
			energy, i.e. $\|u(t)\|_{L^2}$ is conserved in time.
			\item For any $\beta<1/3$, there exist weak solutions $u\in C^{\beta}([0,T]\times \T^3)$ which dissipate kinetic energy.
		\end{enumerate}
	\end{conjecture}
	The first assertion was fully proven in \cite{CET94} using a commutator argument. As for the second assertion, the only known approach is the convex integration technique, first introduced in the context of the Euler equations by De Lellis and Sz\'ekelyhidi in \cite{DelSze13} to demonstrate the existence of infinitely many continuous solutions with dissipative energy. This breakthrough initiated a series of developments \cite{Buc15, BDLIS16, DLS14, DS17}, culminating in Isett's proof of the flexible aspect of Onsager's conjecture \cite{Ise18}. Although the weak solutions constructed by Isett are not strictly dissipative, this technical issue was addressed in \cite{BDLSV19}, allowing for the prescription of arbitrary positive energy profiles. Our overview of the historical developments is highly condensed, we refer the reader to the excellent survey \cite{BV19} for more details.
	For more recent works using convex integration to study the Euler equations, we refer to \cite{GKN23,GR23,NV23}, which establish Onsager's conjecture in  2D and the strong Onsager theorem. 
	
	Over the past few decades, extensive research has been dedicated to justifying the inclusion of stochastic perturbations in the Euler equations; see, e.g., \cite{BF99, BFM16, BP01, CC99, CFH19, GV!4, Kim09, MV00}.  One motivation is the expectation that suitable stochastic forcing may induce a regularization effect. In this direction, Glatt-Holtz and Vicol \cite{GV!4} established the local well-posedness of probabilistically strong solutions to the stochastic 3D Euler equations with nonlinear multiplicative noise, and further showed that the solutions are global with high probability when the multiplicative noise is linear. However, recent developments based on the convex integration method have revealed negative results for the stochastic Euler equations. Early applications of convex integration investigated the isentropic Euler system \cite{BFH20} and the full Euler system with linear multiplicative noise \cite{CFF19}. Hofmanov\'a, Zhu and the third named author studied the ill-posedness of dissipative martingale solutions to the stochastic 3D Euler equations in \cite{HZZ22a}, establishing the existence and non-uniqueness of strong Markov solutions. Later, they \cite{HZZ22b} constructed infinitely many statistically stationary solutions in $H^{\vartheta}$ for some $\vartheta>0$ to the stochastic 3D Euler equations, using a novel stochastic convex integration method. Recently, the second and third named authors \cite{LZ24} improved the regularity of such solutions to $C^{\vartheta}$ for some $\vartheta>0 $  in the case of additive noise; see also \cite{KK24} for further enhancements. Additionally, it is demonstrated in \cite{HLP22} that the 3D Euler equations perturbed by transport noise have more than one probabilistically strong solution in H\"older spaces. Nevertheless, the H\"older exponents achieved in these works remain far below the Onsager critical threshold $1/3$ due to the presence of noise. 
	
	Notably, Hofmanov\'a et. al \cite{HPZZ25} proved that if statistically stationary Leray--Hopf solutions to the stochastic 3D Navier--Stokes equations exhibit $H^{1/3-}$ regularity uniformly in the viscosity $\nu$, then the Kolmogorov 4/5 law holds with the dissipative length scale $\ell_D \sim \nu^{3/4-}$. This scaling coincides with the prediction of Kolmogorov’s 1941 turbulence theory; see \cite[Remark 1.2, Theorem 4.2]{HPZZ25}. Since the stochastic Navier--Stokes equations reduce to the stochastic Euler equations in the inviscid limit $\nu=0$, establishing the critical regularity $1/3-$ for the stochastic 3D Euler equations is not only mathematically interesting but also physically significant, as it is intimately related to the Kolmogorov turbulence theory.
	
	On the other hand, this work focuses on the energy inequality for the stochastic Euler equations. In the deterministic setting, smooth solutions 
	necessarily satisfy the energy equality $\|u(t)\|_{L^2}^2=\|u(0)\|_{L^2}^2$, whereas the energy inequality only holds for solutions with H\"older regularity below $1/3$. In the stochastic setting, if $u$ is a smooth solution (also see Theorem~\ref{rigid part} below for a relaxed regularity assumption), then It\^o's formula implies that it holds almost surely:
	\begin{align}\label{pathwise:energy:equality}
		\|u(t)\|_{L^2}^2= \|u(0)\|_{L^2}^2 +2 \int_{0}^{t} \big\langle u(r), \dif B(r)  \big\rangle +\tr\big(GG^*\big) t.
	\end{align}
	Compared to the deterministic setting, the pathwise energy equality \eqref{pathwise:energy:equality} contains additional terms arising from the stochastic noise, namely the martingale term $\int_{0}^{t} \big\langle u(s), \dif B(s )  \big\rangle$ and the trace term $\tr\big(GG^*\big) t$. It is natural to ask whether low-regularity solutions of the stochastic Euler equations satisfy the following $\textit{pathwise energy inequality}$ almost surely:
	\begin{align}\label{pathwise:energy:inequality}
		\|u(t)\|_{L^2}^2< \|u(0)\|_{L^2}^2 +2 \int_{0}^{t} \big\langle u(r), \dif B(r)  \big\rangle 
		+\tr\big(GG^*\big) t.
	\end{align}
	However, such a pathwise energy inequality involving a martingale term has not been established in the aforementioned‌ works. For instance, the construction in \cite{GV!4} yields smooth solutions, which necessarily conserve energy due to their high regularity. For the irregular solutions constructed in \cite{HLP22,HZZ22b, LZ24}, only an energy inequality in expectation can be proved:
	\begin{align}\label{energy:inequality:expectation}
	\E 	\|u(t)\|_{L^2}^2< \E \|u(0)\|_{L^2}^2 +\tr\big(GG^*\big) t,
	\end{align}
	which can be derived directly from the pathwise inequality \eqref{pathwise:energy:inequality}. Note that establishing \eqref{pathwise:energy:inequality} is highly nontrivial, since the stochastic integral involved is defined in terms of taking expectation. Most tools from stochastic analysis, such as Burkholder-Davis-Gundy's and Doob's inequalities, yield bounds only in expectation rather than pathwise. As a result, obtaining pathwise control of the martingale term in \eqref{pathwise:energy:inequality} poses the main difficulty.
	
	Overall, these developments and unresolved questions naturally lead to the following: \textit{Does the Onsager theorem remain valid for the 3D Euler equations under suitable stochastic perturbations?} The main objective of this work is to provide an affirmative answer to this question. Specifically, we aim to address the following issues:
	\begin{itemize}
		\item The existence and non-uniqueness of $C_tC^{\frac13-}_x$ solutions\footnote{For $\alpha\in (0,1)$, we write $C_x^{\alpha-}=\cup_{\varepsilon>0}C_x^{\alpha-\varepsilon}$.} to the stochastic 3D Euler equations satisfying the $\textit{pathwise energy inequality}$ \eqref{pathwise:energy:inequality}. 
		\item The existence and non-uniqueness of $C_tC^{\bar{\beta}-}_x$-solutions to the stochastic 3D Euler equations with arbitrary initial data in $C^{\bar{\beta}}_x$, for any $0<\bar{\beta}<1/3$.
	\end{itemize}
	This work provides the first resolution of Onsager's conjecture for the stochastic 3D Euler equations. It is also the first to successfully achieve pathwise control of the total energy (including the stochastic integral), among a series of works using convex integration methods.
	As a complement, we also verify that the pathwise energy equality \eqref{pathwise:energy:equality} holds for any solution in $C_tC_x^{\vartheta}$ with $\vartheta>1/3$ to the stochastic 3D Euler equations.
	
	\subsection{Main results}
	In this paper, we are concerned with the stochastic Euler equations on the torus $\mathbb{T}^3=\mathbb{R}^3/\mathbb{Z}^3$ driven by an additive noise
	\begin{equation}\label{eul1}
		\begin{aligned}
			\dif u+\div(u\otimes u)\,\dif t+\nabla p\,\dif t&=\dif B,
			\\	\div u&=0,
		\end{aligned}
	\end{equation}
	where $u\in \R^3$ represents the fluid velocity field and $p\in \R$ denotes the pressure field. Here, $B=\{B_t;0\leq t<\infty\}$ is a $GG^*$-Wiener process with spatial mean zero and divergence-free, on a given filtered probability space $(\Omega,\mathcal{F},(\mathcal{F}_{t})_{t\geq 0},\mathbf{P})$ and $G$ is a Hilbert--Schmidt operator from $U$ to $L^2$ for some Hilbert space $U$.
	Within this study, we focus on probabilistically strong and analytically weak solutions which satisfy the equations in the following sense.
	\begin{definition}\label{d:sol}
		Let $(\Omega,\mathcal{F},(\mathcal{F}_{t})_{t\geq 0},\mathbf{P})$ and $B$ be given as above. An $(\mathcal{F}_{t})_{t\geq 0}$-adapted stochastic process $u \in  C([0,\infty) \times \mathbb{T}^{3},\R^3)$ $\mathbf{P}$-a.s. is a probabilistically strong and analytically weak solution to the stochastic Euler system \eqref{eul1} provided
		\begin{enumerate}[(1)]
			\item $(\mathcal{F}_{t})_{t\geq 0}$ is the normal filtration generated by $B$, that is, the canonical right-continuous filtration augmented by all the $\mathbf{P}$-negligible events;
			\item for any $t\geq 0$ it holds $\mathbf{P}$-a.s.\footnote{$\langle, \rangle$ denotes the inner product in $L^2(\mathbb{T}^3,\mathbb{R}^3)$, i.e. $\langle f,  g\rangle= \int_{\T^3} f(x)g(x) \dif x $ for any $f,g \in L^2(\mathbb{T}^3,\mathbb{R}^3)$.}
			\begin{align*}
				\langle u(t),\psi \rangle  =\langle u(0),\psi \rangle + \int_{0}^{t} \langle   u,u \cdot \nabla \psi \rangle  \dif r +\langle B(t),\psi\rangle
			\end{align*}
			for every $\psi\in C^{\infty}(\mathbb{T}^{3},\R^3)$, $\div \psi=0$;
			\item  it holds $\mathbf{P}$-a.s. $\div u=0$ in the sense of distribution.
		\end{enumerate}
	\end{definition}
	
	We first present the rigid part of the Onsager theorem for system \eqref{eul1}: pathwise energy equality \eqref{pathwise:energy:equality} holds for solutions with H\"older regularity exceeding $1/3$. 
	\begin{theorem}\label{rigid part}
		Let $(\Omega,\mathcal{F},(\mathcal{F}_{t})_{t\geq 0},\mathbf{P},B)$ be a probability space and $p,q\in (1,\infty)$ satisfying $\frac{1}{p}+\frac{1}{q}=1$. Suppose that $\mathfrak{s}$ is a $\mathbf{P}$-a.s. strictly positive stopping time with $\E(\mathfrak{s}^p)<\infty$ and $u\in L^{3q}(\Omega;C([0,\mathfrak{s}],C^{\vartheta}(\T^3,\R^3)))$ is a probabilistically strong and analytically weak solution to \eqref{eul1} for some $\vartheta >\frac13$. Then, it holds $\mathbf{P}$-a.s.
		\begin{align}\label{energy:conserve}
			\|u(t\wedge\mathfrak{s})\|_{L^2}^2= \|u(0)\|_{L^2}^2 +2 \int_{0}^{t\wedge\mathfrak{s}} \big\langle u(s), \dif B(s )  \big\rangle +\tr\big(GG^*\big) (t\wedge\mathfrak{s}),
		\end{align}
		for any $t\in [0,\infty)$.
	\end{theorem}
	We provide a proof of Theorem~\ref{rigid part} in Section~\ref{energy:conservation} based on commutator estimates from \cite{CET94} and Itô's calculus. The presence of the noise also leads to two distinct types of energy equalities. By taking expectations on both sides of \eqref{energy:conserve}, we have for any $t\in [0,\infty)$
	\begin{align*}
		\E	\|u(t\wedge\mathfrak{s})\|_{L^2}^2= \E \|u(0)\|_{L^2}^2  +\tr\big(GG^*\big) \E (t\wedge\mathfrak{s}),
	\end{align*}
	which means that the noise introduces additional energy into the system \eqref{eul1}.
	
	The first main objective of this paper is to establish the existence of 		
	$1/3-$ H\"older continuous solutions to \eqref{eul1} satisfying the $\textit{pathwise energy inequality}$ \eqref{pathwise:energy:inequality}, thereby demonstrating the flexible side of the Onsager theorem for the stochastic 3D Euler equations.

	\begin{theorem}\label{Onsager:theorem}
		Suppose that $\tr\big((\mathrm{I}-\Delta)^{7/2+\gamma}GG^*\big)<\infty$ for some $\gamma>0$. For any given $T\in(0,\infty)$ and $\varkappa\in (0,1)$, there exists a $\mathbf{P}$-a.s. strictly positive stopping time $\mathfrak{t}$ satisfying $\mathbf{P}(\mathfrak{t}>T)\geq \varkappa$ such that the following holds true: For any $\vartheta\in(0,1/3)$, there exist infinitely many smooth functions $e$ on $[0,\infty)$ and a corresponding probabilistically strong and analytically weak solution $u\in C([0,\mathfrak{t}],C^\vartheta(\T^3,\R^3))   \cap C^\vartheta([0,\mathfrak{t}],C(\T^3,\R^3))$ $\mathbf{P}$-a.s. to \eqref{eul1} such that it holds $\mathbf{P}$-a.s. 
		\begin{align}\label{energy profile}
			e(t\wedge \mathfrak{t})+2\int_{0}^{t\wedge \mathfrak{t}} \big\langle u(s), \dif B(s) \big\rangle +\tr(GG^*)(t\wedge \mathfrak{t}) 	=\|u(t\wedge \mathfrak{t})\|_{L^2}^2,
		\end{align}
		for any $t\in[0,\infty)$. In particular, there exist infinitely many solutions $u$ such that the following pathwise energy inequality holds $\mathbf{P}$-a.s.  for any $0\leq s<t<\infty$:
		\begin{align}\label{energy:inequality}
			\|u(t\wedge\mathfrak{t})\|_{L^2}^2< \|u(s\wedge\mathfrak{t})\|_{L^2}^2 +2 \int_{s\wedge\mathfrak{t}}^{t\wedge\mathfrak{t}} \big\langle u(r), \dif B(r)  \big\rangle +\tr\big(GG^*\big) (t\wedge\mathfrak{t}-s\wedge\mathfrak{t}).
		\end{align}
	\end{theorem}
	
	\begin{remark}
		\textit{In Theorem~\ref{Onsager:theorem}, the energy function $e$ is only required to have suitable strictly positive lower and upper bounds, along with a finite derivative. Since these bounds for $e$ depend on the parameters in the convex integration scheme (see Proposition~\ref{p:iteration} below), we state our result Theorem~\ref{Onsager:theorem} as suitable functions $e$ exist, rather than arbitrary $e$.}
	\end{remark}
	
	From a probabilistic perspective, the energy profile \eqref{energy profile} naturally arises since subtracting the stochastic integral (a martingale) from the kinetic energy $\|u(t)\|_{L^2}^2$ yields the finite variation term $e$ on the left-hand side of \eqref{energy profile}.
	As discussed above, the $\textit{pathwise energy inequality}$ \eqref{energy:inequality} also implies for any $0\leq s<t<\infty$:
	\begin{align*}
		\E	\|u(t\wedge\mathfrak{t})\|_{L^2}^2< \E \|u(s\wedge\mathfrak{t})\|_{L^2}^2  +\tr\big(GG^*\big) \E (t\wedge\mathfrak{t}-s\wedge\mathfrak{t}).
	\end{align*}
	Our result shows that, despite the noise introducing additional energy, the Euler system \eqref{eul1} still exhibits energy dissipation (up to a stopping time) below the Onsager threshold.

	The proof of Theorem~\ref{Onsager:theorem} is given in Sections~\ref{s:in}-\ref{sec:inductive estimates}. Note that in Theorem~\ref{Onsager:theorem}, the initial data is part of the construction and not predetermined. Our second main result addresses the Cauchy problem for the stochastic Euler system \eqref{eul1} and establishes global existence.
	
	\begin{theorem}\label{Onsager:theorem:cauchy:problem}
		Suppose that $\tr\big((\mathrm{I}-\Delta)^{5/2+\gamma}GG^*\big)<\infty$, for some $\gamma>0$. Given $\bar{\beta}\in (0,1/3)$, let $u^{in}\in C^{\bar{\beta}}(\T^3,\R^3)$ $\mathbf{P}$-a.s. be a divergence-free initial condition independent of the Wiener process $B$. For any $\vartheta \in (0,\bar{\beta})$, 
		there exist infinitely many probabilistically strong and analytically weak solutions $u\in C([0,\infty),C^\vartheta(\T^3,\R^3)) \cap C^\vartheta([0,\infty),C(\T^3,\R^3)) $ $\mathbf{P}$-a.s. to \eqref{eul1} with initial condition $u|_{t=0}=u^{in}$.
	\end{theorem}
	
	Finally, we consider the solution $u$ constructed in Theorem~\ref{Onsager:theorem} and treat its value at the stopping time $\mathfrak{t}$ as a new initial condition $u(\mathfrak{t})$. By applying the arguments for the Cauchy problem outlined above and gluing convex integration solutions, we establish the following global existence result with pathwise energy dissipation before the stopping time $\mathfrak{t}$. 
	
	\begin{corollary}
		Suppose that $\tr\big((\mathrm{I}-\Delta)^{7/2+\gamma}GG^*\big)<\infty$ for some $\gamma>0$. For any given $T\in(0,\infty)$ and $\varkappa\in (0,1)$, there exists a $\mathbf{P}$-a.s. strictly positive stopping time $\mathfrak{t}$ satisfying $\mathbf{P}(\mathfrak{t}>T)\geq \varkappa$ with the following property.
		For any $\vartheta\in(0,1/3)$, there exist infinitely many probabilistically strong and analytically weak solutions $u\in C([0,\infty),C^\vartheta(\T^3,\R^3))   \cap C^\vartheta([0,\infty),C(\T^3,\R^3))$ $\mathbf{P}$-a.s. to \eqref{eul1}, such that $u$ satisfies $\mathbf{P}$-a.s. 
		\begin{align*}
			\|u(t\wedge \mathfrak{t})\|_{L^2}^2< 	\|u(s\wedge \mathfrak{t})\|_{L^2}^2+2\int_{s\wedge \mathfrak{t}}^{t\wedge \mathfrak{t}} \big\langle u(r), \dif B(r) \big\rangle +\tr(GG^*)(t\wedge \mathfrak{t}-s\wedge \mathfrak{t}) 	,
		\end{align*}
		for any $0\leq s <t<\infty$.
	\end{corollary}

	\subsection{Ideas of the proofs}
	Both our main results, Theorem~\ref{Onsager:theorem} and Theorem~\ref{Onsager:theorem:cauchy:problem} make use of the convex integration method. To illustrate the innovation of our ideas more clearly, we decompose the solution to \eqref{eul1} into two parts $u=v+z$ with $z:=B$, and $v$ solves the nonlinear and random PDE
	\begin{equation}\label{nonlinear1}
		\aligned
		\partial_t v+\div((v+z)\otimes (v+z))+\nabla p&=0,
		\\\div v&=0.
		\endaligned
	\end{equation}
	To apply the convex integration method and perform a pathwise analysis of the random PDE \eqref{nonlinear1}, we employ stopping times to control the growth of the noise. At each step $q\in \N$, we construct a pair $(v_q,\mathring{R}_q)$ satisfying the following system:
	\begin{align}\label{euler system1}
		\partial_t v_q +\div((v_q+z_q)\otimes (v_q+z_q))+\nabla p_q=\div
		\mathring{R}_q, \qquad 
		\div v_q=0,
	\end{align}
	where $v_q$ serves as an approximate solution to \eqref{nonlinear1}, $\mathring{R}_q$ is a trace-free symmetric matrix, and $z_q$ is a temporal mollification of the Wiener process $B$, introduced to address its time singularity.
	To reach the critical H\"older regularity $1/3-$, we adopt the iteration scheme developed in \cite{Ise18}. However, the presence of stochastic forcing introduces substantial new difficulties, making the analysis fundamentally different from the deterministic case. In particular, constructing the glued solutions to \eqref{euler system1} requires precise tracking of how the noise affects the pathwise estimates. To overcome this difficulty, we develop refined pathwise estimates that effectively absorb the influence of the noise and require a more delicate choice of parameters (see Subsection~\ref{sec:parameters} below) to guarantee convergence of the iteration scheme. These adjustments are essential and preclude a direct application of the deterministic results.
	
	Another novelty of this work lies in the derivation of the energy profile \eqref{energy profile}. To this end, we introduce a new energy iteration into the convex integration scheme, as stated in Proposition~\ref{p:iteration} (see \eqref{estimate:energy} below), where we use
	\begin{align}\label{stochatisc integral1}
		\int_{0}^{\cdot } \int_{\T^3} (v_{q}+z_q) \cdot \partial_t z_{q} \, \dif x \dif s 
	\end{align}
	as an approximation to the stochastic integral $\int_{0}^{\cdot} \big\langle u(s), \dif B(s) \big\rangle $. A central difficulty in our scheme is to control the energy iteration \eqref{estimate:energy}, which involves the time derivative of $ z_q$. To address this, we first carefully choose the temporal mollification parameters so that $\|\partial_t \nabla z_q\|_{C^0_{t,x}} \approx \|\nabla v_q\|_{C^0_{t,x}}$. When performing pathwise estimates, the introduction of \eqref{stochatisc integral1} leads to additional terms requiring special treatment, such as $ \int_{0}^{\cdot } \int_{\T^3} v_{q}\cdot \partial_t (z_{q+1}- z_q) \, \dif x \dif s $. By applying integration by parts twice and substituting $\partial_t v_q$ via the system \eqref{euler system1}, we rewrite this term as $ \int_{0}^{\cdot } \int_{\T^3} \tr ((\mathring{R}_q- (v_q+z_q)\otimes (v_q+z_q)) \nabla (z_{q+1}- z_q)^T )\,  \dif x \dif s $, where the temporal derivative $\partial_t (z_{q+1}- z_q) $ is transformed into the spatial derivative $\nabla (z_{q+1} - z_q)$. Thanks to the boundedness of $\mathring{R}_q$, $v_q$ and $z_q$, and the enhanced spatial regularity on the Wiener process $B$, this term can then be absorbed. Similar arguments will be employed repeatedly throughout the proof.
	
	After establishing the inductive estimates, we combine It\^o's calculus with pathwise bounds to prove the key convergence result, a Wong--Zakai type estimate (see \eqref{claim1} below):
	\begin{align*}
		\lim_{q\to 0}\E  \left| \int_{0}^{\cdot \wedge \mathfrak{t}} \int_{\mathbb{T}^3} (v_q+z_q) \cdot \partial_t z_q \dif x \dif s-  	 \int_{0}^{\cdot \wedge \mathfrak{t}} \big\langle u(s), \dif B(s) \big\rangle -\frac12 \tr(GG^*)(\cdot \wedge \mathfrak{t})  \right|^2 =0,
	\end{align*}
	which ultimately implies the energy profile \eqref{energy profile} by combining the inductive estimates in Proposition~\ref{p:iteration}.

	The proof of Theorem~\ref{Onsager:theorem:cauchy:problem} builds on the aforementioned convex integration scheme with necessary adjustments.
	In earlier works using convex integration to address Cauchy problems for the stochastic Euler equations (e.g., \cite{CDZ22, Lu24}), initial values were typically fixed outside the iteration by truncating perturbations near $t=0$, often resulting in solutions that are only $L^p$-integrable in time. Here, we use the idea from \cite{KMY22} to incorporate the initial data into the iteration via convolution.  The convex integration scheme is then adjusted to refine the initial condition during the gluing step (see Subsection~\ref{sec:esact solutions:new} below), ultimately recovering the prescribed initial data. The gluing procedure also guarantees that the perturbations vanish near $t=0$, thus removing the need for additional cutoffs as in \cite{ CDZ22, HZZ21markov, Lu24} and ensuring that the solutions are continuous in time. We then extend the convex integration solution by connecting it with another strong solution, ultimately constructing global-in-time solutions through countably many extensions.
	We also note that the methods in \cite{HLP22} may be difficult to apply directly to establish global solutions, as repeating the gluing step over larger time intervals of the form $[\mathfrak{t}_L,\mathfrak{t}_{L+1}]$ may cause the glued solutions to lose adaptedness.

	\subsection{Further relevant literature}
	We conclude this introductory section with a non-exhaustive list of papers where convex integration has been adapted to the stochastic setting for various equations. For instance, there are results for the stochastic Navier--Stokes equations \cite{CDZ22, CZZ24, HZZ21markov, HZZ23b, HZZ19, HZZ22b, LRS24,  LZ24a, LZ23b, Umb23, RS23, Yam22b, Yam22c, Yam24a}, stochastic SQG equations  \cite{BLW24,  HLZZ24, HZZ23c, WY24, Yam23, Yam22d}, stochastic power law fluids \cite{Ber24, LZ23}, stochastic Boussinesq system \cite{Yam22a}, stochastic MHD system \cite{CLZ24, Yam24b}. Among these, some works also apply convex integration to singular stochastic PDEs, see \cite{HLZZ24, HZZ23b, HZZ23c, LZ24a, LZ23b}. In particular, Hofmanov\'a, Zhu and the third named author studied the stochastic 3D Navier--Stokes equations perturbed by additive, linear multiplicative and nonlinear noise of cylindrical type in \cite{HZZ19}, establishing that  non-uniqueness in law holds in a class of analytically weak solutions. Furthermore, the existence of infinitely many global-in-time probabilistically strong solutions to the stochastic 3D Navier--Stokes equations driven by different types of noise has been established in \cite{HZZ21markov, HZZ23b, Umb23}. Additionally, sharp nonuniqueness for the stochastic $d$-dimensional ($d \geq  2$) Navier--Stokes equations, as well as  stationary solutions for the stochastic 3D Navier--Stokes equations, were respectively established in \cite{CDZ22} and \cite{HZZ22b}, where the authors developed a new stochastic version of the convex integration method to derive global solutions without using stopping times. These developments have greatly enhanced the understanding of the interplay between
	randomness and fluid dynamics.
	
	\subsection{Organization of the paper} 
	The paper is structured as follows: in Section~\ref{notation}, we introduce our notational conventions and formulate parameters used throughout the whole iteration.
	Sections~\ref{s:in}-\ref{sec:inductive esti} are devoted to the proof of our first main result, Theorem~\ref{Onsager:theorem}.
	First, in Section~\ref{s:in}, we state the main iterative proposition and demonstrate how Theorem~\ref{Onsager:theorem} follows from it. In
	Section~\ref{sec:begin}, we initiate the proof of the iterative proposition by implementing the mollification step and the gluing procedure. Section~\ref{sec:perturbation} focuses on constructing the new perturbation and analyzing the stress error term, preparing for the inductive estimates. In Section~\ref{sec:inductive esti}, we explain how the perturbation, stress error, and energy are controlled to finalize the proof of the iterative proposition.
	In Section~\ref{cauchy problem}, we adjust the aforementioned convex integration scheme to complete the proof of our second main result, Theorem~\ref{Onsager:theorem:cauchy:problem}. 
	For the reader's convenience, we prove that the analytically weak solutions to \eqref{eul1} with more than 1/3 H\"older regularity preserve energy balance in Section~\ref{energy:conservation}. Some technical tools for convex integration are gathered in Appendix~\ref{sec:appendix c}. In Appendix~\ref{sec:Mikado}, we recall the construction of Mikado flows needed for convex integration. Appendix~\ref{sec:esti:transport} provides standard estimates for transport equations, while finally, Appendix~\ref{appendix:proof:lemma 4.2} includes a proof of the local theory necessary for executing the gluing process for the Euler system.

	\section{Preliminaries}\label{notation}
	\subsection{Notations}\label{notation 1}
	Throughout the paper, we employ the notation $a\lesssim b$ if there exists a constant $c > 0$ such that $a\leq cb$. We let $\N_0:=\N \cup\{0\}$. We denote $L^p$ as the set of  standard $L^p$-integrable functions from $\mathbb{T}^3$ to $\mathbb{R}^3$. For $s>0$, $p>1$ the Sobolev space $W^{s,p}:=\{f\in L^p; \|f\|_{W^{s,p}}:= \|(\mathrm{I}-\Delta)^{{s}/{2}}f\|_{L^p}<\infty\}$. We set $L^{2}_{\sigma}:=\{f\in L^2; \int_{\mathbb{T}^{3}} f\,\dif x=0,\div f=0\}$. For $s>0$, we also denote $H^s:=W^{s,2}\cap L^2_\sigma$. 
	Given a Banach space $\left( Y, \|\cdot \|_{Y}\right) $ and $I\subset \R$, we write $C_I Y:=C(I;Y)$ as the space of continuous functions from $I$ to $Y$, equipped with the supremum norm $\|f\|_{C_IY}:=\sup_{s\in I}\|f(s)\|_{Y}$.  
	For $\kappa\in(0,1)$, we use $C^\kappa_I Y$ to denote the space of $\kappa$-H\"{o}lder continuous functions from $I$ to $Y$, endowed with the norm $$\|f\|_{C^\kappa_I Y}:=\sup_{s,r\in I,s\neq r}\frac{\|f(s)-f(r)\|_Y}{|r-s|^\kappa}+\|f\|_{C_I Y}.$$ 
	Whenever $I=[0,T]$, we simply write $C_TY:= C_{[0,T]}Y$ and $C_T^\kappa Y:= C_{[0,T]}^\kappa Y$.
	
	For $N\in \N_0 $, let $C^N( \T^3,\R^3)$ denote the space of $N$-times differentiable functions from $ \T^3$ to $\R^3$ equipped with the norm
	$$
	\|f\|_{C^N_{x}}:=\sum_{\substack{|\alpha|\leq N, \alpha\in\N^{3}_{0} }}\| D^\alpha f\|_{ L^\infty_x}.$$
	For $N\in \mathbb{N}_0$ and $\kappa \in (0,1)$, let $C^{N+\kappa}( \T^3,\R^3)$ denote the subspace of $C^N( \T^3,\R^3)$ whose $N$-th derivatives are $\kappa$-H\"{o}lder continuous, with the norm
	\begin{align*}
		\|f\|_{C^{N+\kappa}_{x}}:=	\|f\|_{C^N_{x}}+\sum_{\substack{|\alpha|= N,  \alpha\in\N^{3}_{0} }}[D^\alpha f]_{C^\kappa_{x}},
	\end{align*}
	where  $	[f]_{C^\kappa_{x}} :=\sup_{\substack{x\neq y, x,y\in \T^3 }} \frac{ \left|  f(x)- f(y)\right| }{|x-y|^\kappa}$ is the H\"{o}lder seminorm. We will write $C^N_{x}$ and $C^{N+\kappa}_{x}$ as shorthand for $C^N( \T^3,\R^3)$ and $C^{N+\kappa}( \T^3,\R^3)$.
	Moreover, we write $\|f(t)\|_{C^{N}_x}$ and $\|f(t)\|_{C^{N+\kappa}_x}$ when the time $t$ is fixed and the norms are computed for the restriction of $f$ to $t$-time slice. We may omit time $t$ if there is no danger of confusion. 
	
	We fix for the remainder of the paper two standard mollification kernels $ \psi  \in C_c^\infty( (0,1); [0,\infty))$ and $\varphi \in C_c^\infty( \T^3; [0,\infty))$
	and define for each $\varepsilon>0$ the rescaled kernel 
	\begin{align}\label{def:mollifiers}
		\psi_{\varepsilon}(t):=\frac{1}{\varepsilon}\psi \left( \frac{t}{\varepsilon} \right),\qquad 	\varphi_{\varepsilon}(x):=\frac{1}{\varepsilon^3}\varphi\left( \frac{x}{\varepsilon} \right). 
	\end{align}
	For any vector field $f$, we write $f*_t\psi_{\varepsilon}$ as the convolution over time and $f*_x\psi_{\varepsilon}$ as the convolution over space, and some useful mollification estimates are collected in Appendix~\ref{sec:appendix c}. 
	
	We also use $\mathring{\otimes}$ to denote the trace-free part of the tensor product. For a tensor $T$, we denote its traceless part by $\mathring{T}:=T-\frac13\tr(T)\rm{Id}$. By $\mathbb{P}_{\leq N}$ we denote the Fourier multiplier operator, which projects a function onto its Fourier frequencies $\leq N$ in absolute value. We write $\mathcal{S}^{3\times3}$ for the set of symmetric $3\times3$ matrices and $\mathcal{S}_0^{3\times3}$ for the set of symmetric trace-free $3\times3$ matrices. 
	
	Concerning the driving noise, we assume that $B$ is an $\mathbb{R}^3$-valued $GG^*$-Wiener process with zero spatial mean and divergence-free, defined on some probability space $(\Omega,\mathcal{F},\mathbf{P})$ and $G$ is a Hilbert-Schmidt operator from $U$ to $L^2_{\sigma}$ for some Hilbert space $U$. For a given probability measure $\mathbf{P}$ we denote by $\E$ the expectation under $\mathbf{P}$.

	\subsection{Parameters and their restrictions}\label{sec:parameters}
	Before we explain how the convex integration is set up, we would like to introduce some parameters commonly used in the iteration procedure.
	Given $0<\beta<1/3$, $b\in (1,2)$, $\alpha \in (0,1)$, and $a\gg1$, for all $q\in \N_0$ we define the frequency $\{\lambda_q  \}_{q\in \mathbb{N}_0} $ which diverges to $\infty$ given by ($\lceil x\rceil$ denotes the ceiling function)
	\begin{align*}
		\lambda_q=\lceil{a^{(b^{q})}}\rceil \, ,
	\end{align*}
	and a bounded amplitude sequence $\{\delta_q  \}_{q\in \mathbb{N}_0} $ which is decreasing to $0$ given by
	\begin{align*}
		\delta_{0}=16\lambda_{1}^{3\alpha},\qquad \delta_{1}=4\lambda_{1}^{3\alpha},\qquad \delta_q=\lambda_2^{2\beta}\lambda_{q}^{-2\beta}\lambda_{1}^{3\alpha}, 	\quad q\geq 2.
	\end{align*}
	
	In the Sections~\ref{s:in}--\ref{sec:inductive esti}, we always assume $0<\beta<1/3$, $b>1$ and close to $1$ such that 
	\begin{align}\label{choice:b}
		0<b-1<\min\left\{\frac{1-3\beta}{2\beta},\sqrt{\frac{1}{3\beta}}-1,\frac{1}{6\beta}-\frac12, 1 \right\}.
	\end{align}
	In addition, we require $\alpha>0$ to be sufficiently small in terms of $b,\beta$ satisfying 
	\begin{align}\label{def alpha}
		20b\alpha<\min \left\{(b-1)(1-2b\beta-\beta), \beta(b-1), \frac23-2b^2\beta, \frac13+\beta-2b\beta  \right\} .
	\end{align}
	Finally, we choose $a$ large enough to have $	2\leq a^{(b-1)\beta}
	\leq  a^{(b-1)(1-\beta)}$.
	In the sequel, we increase $a$ in order to absorb various implicit and universal constants. 
	
	In particular, we also define the space mollification parameters for all $q\in \N_0$ by
	\begin{align}\label{def l}
		\ell_q:=\frac{\delta_{q+1}^{\sfrac{1}{2}}}
		{\delta_q^{\sfrac12}\lambda_{q}^{1+6\alpha}}\in \left(\frac12 \lambda_{q}^{-1-(b-1)\beta-6\alpha}, \lambda_{q}^{-1-6\alpha}\right),
	\end{align}
	and the temporal mollification parameters for all $q\in \N_0$ by
	\begin{align}\label{def iotaq}
		 \iota_q :=\lambda_{q}^{-\sfrac43}.
	\end{align}
	If only a rough bound on $\ell_q$ is needed, then we will use 
	\begin{align}\label{ell:lambdaq}
		\lambda_{q}^{-\sfrac32}\leq \ell_q \leq 	\lambda_{q}^{-1}.
	\end{align}
	It follows from the definition \eqref{def l} that 
	\begin{align*}
		\frac{	\ell_{q+1}}{\ell_{q}}\leq 2\lambda_{q}^{-(b-1)(1-\beta+6\alpha)}\leq 2 a^{-(b-1)(1-\beta)}\leq 1.
	\end{align*}
	Hence, $\ell_{q}$ is decreasing.

	\section{Main iterative proposition and proof of Theorem~\ref{Onsager:theorem}}\label{s:in}
	This section provides an overview of the convex integration scheme and introduces our main iterative proposition, which serves to prove Theorem~\ref{Onsager:theorem}. More precisely, for a given stopping time $\mathfrak{t}$ (which can be chosen arbitrarily large) and a suitable smooth function $e$ (which depends on the parameters in Subsection~\ref{sec:parameters}), we construct a corresponding analytically weak and probabilistically strong solution to the Euler system \eqref{eul1} up to the stopping time $\mathfrak{t}$. This solution exhibits $1/3-$ H\"older regularity and satisfies $\mathbf{P}$-a.s. for any $t\in[0,\infty)$
	\begin{align}\label{energy profile1}
		e(t\wedge \mathfrak{t})+2\int_{0}^{t\wedge \mathfrak{t}} \big\langle u(s), \dif B(s) \big\rangle + \tr(GG^*)(t\wedge \mathfrak{t})	=\|u(t\wedge \mathfrak{t})\|_{L^2}^2.
	\end{align}
	 The proof for the main iteration relies on the convex integration scheme developed in \cite{BDLSV19} (see also \cite{Ise18}). One challenge in the stochastic setting is ensuring the effectiveness of convex integration in the presence of noise. To address this, we employ stopping times to control the growth of the noise term and perform pathwise analysis. Consequently, precise control of the interaction between the convex integration scheme and the noise term becomes crucial, particularly during gluing and perturbation procedures, to establish the required pathwise estimates as shown in \cite{BDLSV19}. Another challenge involves establishing the energy profile \eqref{energy profile1}. We address this by introducing a novel energy iteration and employing stochastic analysis methods to derive energy estimates, which cannot be achieved through purely pathwise analysis.
	
	In this and the following sections, we fix a probability space $(\Omega, \mathcal{F} ,\mathbf{P})$ with a $GG^*$-Wiener process $B$. Let $(\mathcal{F}_{t})_{t\geq0}$ be the normal filtration on $(\Omega,\mathcal{F})$ generated by $B$, namely, the canonical right-continuous filtration augmented by all the P-negligible sets (c.f. \cite[Section 2.1]{LR15}). In order to verify that the solution we construct is a probabilistically strong solution, it is essential that the solution is adapted to this filtration. 
	
	As the first step, we decompose a solution to the Euler system \eqref{eul1} into two parts. Let $u$ be any solution of \eqref{eul1}, and define $z:=B$. Then, the difference $v:=u-z$ satisfies the nonlinear equation
		\begin{equation}\label{nonlinear}
		\aligned
		\partial_t v+\div((v+z)\otimes (v+z))+\nabla p&=0,
		\\\div v&=0.
		\endaligned
	\end{equation}
	Here, $z$ is divergence-free and satisfies $z(0)=0$ by the assumptions on the noise, while $p$ denotes the pressure term associated with $v$. 
	In particular, applying \cite[Theorem 5.16]{DPZ92} together with the Kolmogorov continuity criterion yields the following result:
	\begin{proposition}\label{estimate for z}
		Suppose that $\tr((\mathrm{I}-\Delta)^{7/2+\gamma} GG^*)<\infty$ for some $\gamma>0$. Then for any $\delta \in (0,\frac{1}{2})$ and $T>0$
		$$
		\mathbf{E}\left[\|B\|_{C_T^{1/2-\delta}H^{7/2+\gamma}}\right]<\infty.
		$$
	\end{proposition}
	
	By the Sobolev embedding, we have that $\|\nabla^j f\|_{L^\infty}\leq C_S\|f\|_{H^{3/2+j+\gamma}}$ for $\gamma>0$, $j\in \{0,1,2\}$, and some constant $C_{S}\geq1$. For the sufficiently small $\alpha\in (0,1)$ given in \eqref{def alpha} and $L\in \N$, we define the following stopping time
	\begin{equation}\label{stopping time ps}
		\aligned
		\mathfrak{t}_L:=& \inf\{t\geq0,\|B\|_{C_t^{1/2-\alpha}H^{7/2+\gamma}}\geq L/C_S\}
		\wedge  L.
		\endaligned
	\end{equation}
	According to Proposition~\ref{estimate for z}, the stopping time $\mathfrak{t}_L$ is $\mathbf{P}$-a.s. strictly positive such that $\mathfrak{t}_L\to \infty$ almost surely as $L\to \infty$.

	\subsection{Outline of the convex integration scheme and main iterative proposition}\label{sec:main iteration}
	As is standard in convex integration schemes, we consider a modified version of \eqref{nonlinear} that includes a stress tensor error term $\mathring{R}_q$, which converges to 0. Specifically, at each step $q\in \mathbb{N}$, a pair $(v_q ,\mathring{R}_q)$ is constructed to solve the following system:
	\begin{equation}\label{euler1}
		\aligned
		\partial_t v_q +\div((v_q+z_q)\otimes (v_q+z_q))+\nabla p_q&=\div
		\mathring{R}_q,
		\\\div v_q&=0,
		\endaligned
	\end{equation}
	where $\mathring{R}_q\in \mathcal{S}_0^{3\times3}$, and we incorporate its trace part into the pressure. The term $z_q$ in \eqref{euler1} is obtained via the temporal mollification of the Wiener process $B$. To ensure its definition remains valid around $t=0$, we extend $B(t)=0$ for $t<0$ and define $z_q$ as
	\begin{align}\label{def zq}
		 z_q(t):= B*\psi_{\iota_q}(t) = \int_{0}^{\iota_q} \psi_{\iota_q}(s) B(t-s) \dif s.
	\end{align}
	Here, $\psi_{\iota_q}:= \frac{1}{\iota_q} \psi (\frac{\cdot}{\iota_q})$ is the one-sided temporal mollifier defined in \eqref{def:mollifiers} with support in $(0,\iota_q)$ to preserve the adaptedness of $z_q$ to the filtration $(\mathcal{F}_t)_{t\geq 0}$. This approximation enhances the temporal regularity of the noise $B$. By applying \eqref{stopping time ps} and Sobolev embedding, the approximation $z_q$ satisfies the following bounds for any $t\in[0,\mathfrak{t}_L]$:
	\begin{equation}\label{z ps}
		\aligned
		\| z_q(t)\|_{C_x^0} &\leq L, \qquad   \| z_q(t)\|_{C^1_x}\leq L, \qquad  \|z_q(t)\|_{C^2_x} \leq  L,
		\\    \| z_q\|_{C^{1/2-\alpha}_tC^0_x} & \leq  L,\qquad  \| z_q\|_{C^{1/2-\alpha}_tC^1_x} \leq  L, \qquad \|z_q\|_{C^{1/2-\alpha}_tC^2_x} \leq  L.
		\endaligned
	\end{equation}
	 Furthermore, we apply \eqref{stopping time ps} and mollification estimate \eqref{estimate:molli3} to derive for any $t\in [0,\mathfrak{t}_L]$
	\begin{equation}\label{zq+1-zq}
		\aligned
		 \|z_{q+1}(t)-z_q(t)\|_{C^1_x} &\leq \|B*\psi_{\iota_{q+1}}(t)-B(t)\|_{C^1_x}+ \|B*\psi_{\iota_q}(t)-B(t)\|_{C^1_x}
		 \\ &\lesssim \iota_{q+1}^{\sfrac12-\alpha} \|B\|_{C^{1/2-\alpha}_{[0,\mathfrak{t}_L]}C^1_x} + \iota_{q}^{\sfrac12-\alpha} \|B\|_{C^{1/2-\alpha}_{[0,\mathfrak{t}_L]}C^1_x} \lesssim L\lambda_{q}^{-\sfrac23+2\alpha}.
		 \endaligned
	\end{equation}
	In a similar manner, by utilizing \eqref{stopping time ps} and \eqref{estimate:molli3} again, 
	we obtain for any $t\in [0,\mathfrak{t}_L]$
	\begin{align}\label{zq-B C1}
		\|z_q(t)-B(t)\|_{C^1_x} \leq  \|B*\psi_{\iota_q}(t) -B(t)\|_{C^1_x}
		\lesssim  \iota_q^{\sfrac12-\alpha}\|B\|_{C^{1/2-\alpha}_{[0,\mathfrak{t}_L]}C^1_x} \leq L  \lambda_{q}^{-\sfrac23+2\alpha},
	\end{align}
	which further implies for any $t\in [0,\mathfrak{t}_L]$
	\begin{align}\label{zq-B C0}
		\|z_q(t)-B(t)\|_{C^0_x} \leq L  \lambda_{q}^{-\sfrac23+2\alpha}.
	\end{align}

	Under the above assumptions, the main ingredient in the proof of Theorem~\ref{Onsager:theorem} is the following iterative proposition.
	\begin{proposition}\label{p:iteration} 
		Assume $0<\beta<1/3$ and $\tr ((\mathrm{I}-\Delta)^{7/2+\gamma}GG^*)<\infty$ for some small $\gamma>0$. For any $L\in \N_0$ and the corresponding stopping time $\mathfrak{t}_L$ defined in \eqref{stopping time ps}, there exists a choice of parameters $a, b, \alpha$ depending on $\beta, L $ such that the following holds:
		
		Given a smooth function $e:[0,\infty)\to (0,\infty)$ such that $\lambda_{1}^{\sfrac{5\alpha}{2}}\leq \underline{e}\leq e(t)\leq \overline{e}\leq \lambda_{1}^{3\alpha}$ with $\|e\|_{C^1_t}\leq \tilde{e}$ for some constant $\tilde{e}>0$. Let $(v_{q},\mathring{R}_{q})$ for some $q\in\N$ be an $(\mathcal{F}_{t})_{t\geq 0}$-adapted solution to \eqref{euler1} on $[0,\mathfrak{t}_L]$ satisfying the inductive estimates
		\begin{align}
			\|v_q\|_{C_{[0,\mathfrak{t}_L]}C^0_x} &\leq 3\bar{M}L\lambda_{1}^{\sfrac{3\alpha}{2}} -\bar{M}L\delta_q^{\sfrac12},\label{itera:a}
			\\  \|v_q\|_{C_{[0,\mathfrak{t}_L]}C^1_x} &\leq  \bar{M} L  \lambda_{q}\delta_q^{\sfrac12},\label{itera:b}
			\\	\|\mathring{R}_q\|_{C_{[0,\mathfrak{t}_L]}C^0_{x}}  &\leq \bar{M}L^2 \delta_{q+1}\lambda_{q}^{-3\alpha},\label{itera:c}
		\end{align}
		where $\bar{M}$ is a universal constant defined in \eqref{def barM} below. Moreover, it holds for any $t\in [0,\mathfrak{t}_L]$
		\begin{align}\label{estimate:energy}
			L^2 \delta_{q+1} \lambda_{q}^{-\alpha}\leq e(t) - \| (v_q+z_q) (t)\|_{L^2}^2 +2 \int_{0}^{t} \int_{\T^3} (v_{q}+z_q) \cdot \partial_t z_{q}  \dif x \dif s  \leq L^2 \delta_{q+1}.
		\end{align}
		Then there exists an $(\mathcal{F}_t)_{t\geq 0}$-adapted solution $(v_{q+1}, \mathring{R}_{q+1})$ of \eqref{euler1} on $[0,\mathfrak{t}_L]$ satisfying \eqref{itera:a}, \eqref{itera:b}, \eqref{itera:c} and \eqref{estimate:energy} at the level $q+1$ and for $t\in [0,\mathfrak{t}_L]$
		\begin{equation}\label{vq+1-vq}
			\|v_{q+1}(t)-v_q(t)\|_{C_x^0}\leq \bar{M} L\delta_{q+1}^{\sfrac12}.
		\end{equation}
	\end{proposition}

	The proof of this result is detailed in Sections~\ref{sec:begin}, \ref{sec:perturbation} and \ref{sec:inductive estimates}. Based on Proposition~\ref{p:iteration} we may proceed with the proof of Theorem~\ref{Onsager:theorem}.
	
	\subsection{Proof of Theorem~\ref{Onsager:theorem}}\label{Proof:The:1.3}
	In this subsection, we employ Proposition~\ref{p:iteration} to complete the proof of Theorem~\ref{Onsager:theorem}.
	
	\emph{\underline{Step 1}.}	 
	We start the initial iteration with $(v_1,\mathring{R}_1):= (0,z_1 \mathring{\otimes}z_1)$.
	It is easy to check that they solve \eqref{euler1} and
	since $v_1=0$, \eqref{itera:a} and \eqref{itera:b} automatically hold. From the estimate \eqref{z ps}, it follows
	\begin{align*}
		\|\mathring{R}_1\|_{C_{ [0,\mathfrak{t}_L]}C^0_{x}}\leq \|z_1\|^2_{C_{ [0,\mathfrak{t}_L]}C^0_{x}}\leq L^2=L^2 \delta_2\lambda_{1}^{-3\alpha}.
	\end{align*}
	Moreover, it follows from $z_q(0)=0$ for every $q\in \N$ that
	\begin{align*}
	L^2	\delta_{2} \lambda_{1}^{-\alpha}
	\leq \lambda_{1}^{\sfrac{5\alpha}{2}}\leq 
		\underline{e}\leq  e(t) - \| z_1 (t)\|_{L^2}^2 +2 \int_{0}^{t} \int_{\T^3} z_1 \cdot \partial_t z_1  \dif x \dif s 
		=e(t)   \leq \overline{e}\leq \lambda_{1}^{3\alpha}\leq L^2\delta_{2},
	\end{align*}
	where we chose $a$ sufficiently large to have $L^2\leq a^{\sfrac{b\alpha}{2}}$ in the first inequality. Since the first iteration is established, using Proposition~\ref{p:iteration} yields a sequence $(v_q,\mathring{R}_q)$ satisfying \eqref{itera:a}--\eqref{vq+1-vq}. 
	
	Now, we assume there exists $q_0\in \N$ such that $b^q> bq$ holds for any $q\geq q_0$. Then we have for any $\vartheta<\beta$
	\begin{align*}
		\sum_{1\leq q}\lambda_q^{\vartheta-\beta}\leq \sum_{1\leq q<q_0}\lambda_q^{\vartheta-\beta}+\sum_{q_0 \leq q }\lambda_q^{\vartheta-\beta}\leq q_0-1 +\frac{a^{bq_0(\vartheta-\beta) }}{1-a^{b(\vartheta-\beta)}}\leq q_0,
	\end{align*}
	which boils down to $	2a^{b(\vartheta-\beta)}\leq 1$ by choosing $a$ large enough. With this choice of $q_0$, we use \eqref{itera:b}, \eqref{vq+1-vq} and interpolation to deduce for any $\vartheta<\beta$ and $t\in [0,\mathfrak{t}_L]$
  \begin{equation}\label{vq:convergence}
	 	\begin{aligned}
	 	\sum_{q=1}^{\infty}	\|v_{q+1}(t)-v_q(t)\|_{C^\vartheta_x}\lesssim& \sum_{q=1}^{\infty}	\|v_{q+1}(t)-v_q(t)\|_{C^0_x}^{1-\vartheta}\|v_{q+1}(t)-v_q(t)\|_{C^1_x}^{\vartheta}
	 	\\ \lesssim& \sum_{q=1}^{\infty} \bar{M} L  \delta_{q+1}^{\frac{1-\vartheta}{2}}\lambda_{q+1}^{\vartheta}\delta_{q+1}^{\frac{\vartheta}{2}}
	 	=\bar{M} L\lambda_{1}^{2\alpha} \lambda_{2}^{\beta} \sum_{q=2}^{\infty}\lambda_{q}^{\vartheta-\beta} \leq \bar{M}Lq_0 a^{2b\alpha+b^2\beta}.
	 \end{aligned}
  \end{equation}
	As a consequence, a limit $v=\lim_{q\to \infty} v_q$ exists and lies in $v\in C([0,\mathfrak{t}_L],C^\vartheta(\T^3,\R^3))$. Since $v_q$ is $(\mathcal{F}_{t})_{t\geq 0}$-adapted for every $q\in \mathbb{N}_0$, the limit $v$ is $(\mathcal{F}_{t})_{t\geq 0}$-adapted as well.
	By using mollification estimate \eqref{estimate:molli3}, we deduce for the same $\vartheta$ as above:
	\begin{align}\label{zq-B time}
		 \|z_q-B\|_{C^{\vartheta}_{[0,\mathfrak{t}_L]}C^0_x}\lesssim \iota_{q}^{\sfrac12-\vartheta-\alpha} \|B\|_{C^{\sfrac12-\alpha}_{[0,\mathfrak{t}_L]}C^0_x}\leq L\iota_{q}^{\sfrac12-\vartheta-\alpha}\to 0, \quad \mathrm{as} \quad q\to \infty.
	\end{align}
	Combining \eqref{zq-B C1} and \eqref{zq-B time} implies $\lim_{q\to \infty} z_q=B$ in $C([0,\mathfrak{t}_L],C^{\vartheta}(\T^3,\R^3)) \cap C^{\vartheta}([0,\mathfrak{t}_L],C(\T^3,\R^3))$. Furthermore, it follows from \eqref{itera:c} that $\lim_{q\to \infty} \mathring{R}_q=0$ in $C([0,\mathfrak{t}_L],C(\T^3,\R^3))$. 
	Thus, $v$ is an analytically weak solution to \eqref{nonlinear}. By the same argument as in \cite[page 234]{BDLSV19}, we can recover the temporal regularity of the solutions, namely, $v\in C^\vartheta([0,\mathfrak{t}_L],C(\T^3,\R^3))$. Letting $u=v+z$, we obtain an $(\mathcal{F}_{t})_{t\geq 0}$-adapted analytically weak solution to \eqref{eul1} of class $u\in C([0,\mathfrak{t}_L],C^\vartheta(\T^3,\R^3))\cap C^\vartheta([0,\mathfrak{t}_L],C(\T^3,\R^3))$ for any $\vartheta<1/3$. In addition, it follows from \eqref{vq:convergence} that there exists a deterministic constant $c_L$ dependent on $L$ such that $\|u(t)\|_{C^{\vartheta}_x}\leq c_L$ holds true for all $t\in [0,\mathfrak{t}_L]$.
	
	\emph{\underline{Step 2}.}		
	To prove the energy profile \eqref{energy profile}, we first verify the following claim:
	\begin{align}\label{claim1}
		\lim_{q\to 0}\E  \left|  \int_{0}^{t\wedge \mathfrak{t}_L} \int_{\mathbb{T}^3} (v_q+z_q) \cdot \partial_t z_q \dif x \dif s-	 \int_{0}^{t\wedge \mathfrak{t}_L} \big\langle u(s), \dif B(s) \big\rangle  -\frac12 \tr(GG^*)(t\wedge \mathfrak{t}_L)    \right|^2 =0,
	\end{align}
	which can be further simplified to estimate the following term
	\begin{align}\label{claim part1}
		 \E  \left|  \int_{0}^{t\wedge \mathfrak{t}_L} \int_{\mathbb{T}^3} v_q \cdot \partial_t z_q \dif x \dif s-	 \int_{0}^{t\wedge \mathfrak{t}_L} \big\langle v(s), \dif B(s) \big\rangle   \right|^2 .
	\end{align}
	We will mainly focus on the estimate of \eqref{claim part1}, since the remaining terms obtained by subtracting \eqref{claim part1} from \eqref{claim1} is easy to control. Indeed, by the facts $z_q(0)=B(0)=0$, $\lim_{q\to \infty}\|z_q-B\|_{C^0_{[0,\mathfrak{t}_L],x}}=0$ and $\|z_q\|_{C^0_{[0,\mathfrak{t}_L],x}}+\|B\|_{C^0_{[0,\mathfrak{t}_L],x}}\leq 2L$, using the integration by parts formula, we can bound the remaining parts in \eqref{claim1} as
	\begin{equation}\label{claim part2}
		\aligned
		&	\lim_{q\to \infty}  \E  \left|  \int_{0}^{t\wedge \mathfrak{t}_L} \int_{\mathbb{T}^3} z_q \cdot \partial_t z_q \dif x \dif s-	 \int_{0}^{t\wedge \mathfrak{t}_L} \big\langle B(s), \dif B(s) \big\rangle -\frac12 \tr(GG^*)(t\wedge \mathfrak{t}_L)  \right|^2 
			\\ &\quad  \lesssim \lim_{q\to \infty}  \E  \left| \frac12 \int_{0}^{t\wedge \mathfrak{t}_L} \left( \frac{\dif}{\dif s}\int_{\mathbb{T}^3}  |z_q|^2 \dif x \right) \dif s -\frac12 \left( \|B(t\wedge \mathfrak{t}_L)\|^2_{L^2}-\|B(0)\|^2_{L^2}\right) \right|^2 
			\\ & \quad  \lesssim  \lim_{q\to \infty}  \E \left|\|z_q(t\wedge \mathfrak{t}_L)\|^2_{L^2}-\|B(t\wedge \mathfrak{t}_L)\|^2_{L^2}\right|^2
				\\ & \quad  \lesssim  \lim_{q\to \infty}  \E \left| \|z_q-B\|_{C^0_{[0,\mathfrak{t}_L],x}} \left(\|z_q\|_{C^0_{[0,\mathfrak{t}_L],x}}+\|B\|_{C^0_{[0,\mathfrak{t}_L],x}}\right)\right|^2 =0,
			\endaligned
	\end{equation}
	where the last equality is justified by the dominated convergence theorem.
	
	  We next control \eqref{claim part1}. As the first step, we require to rewrite the two terms in \eqref{claim part1}. Through integrating by parts and replacing $\partial_t v_q$ with \eqref{euler1}, we can derive
	 \begin{equation}\label{claim2}
	 	\begin{aligned}
	 		  \int_{0}^{t\wedge \mathfrak{t}_L}& \int_{\mathbb{T}^3} v_q \cdot \partial_t z_q \dif x \dif s = \int_{\mathbb{T}^3} \big((v_q \cdot z_q )(t\wedge \mathfrak{t}_L) - (v_q \cdot z_q) (0) \big) \dif x - \int_{0}^{t\wedge \mathfrak{t}_L} \int_{\mathbb{T}^3} z_q \cdot \partial_t v_q \dif x \dif s
	 		  \\ & = \int_{\mathbb{T}^3} (v_q \cdot z_q )(t\wedge \mathfrak{t}_L)  \dif x 
	 		 - \int_{0}^{t\wedge \mathfrak{t}_L} \int_{\mathbb{T}^3} z_q \cdot \left(\div \left(\mathring{R}_q-(v_q+z_q)\otimes (v_q+z_q)\right)-\nabla p_q \right) \dif x \dif s 
	 		  \\ & = \int_{\mathbb{T}^3} (v_q \cdot z_q )(t\wedge \mathfrak{t}_L)  \dif x 
	 		+ \int_{0}^{t\wedge \mathfrak{t}_L} \int_{\mathbb{T}^3} \left(\mathring{R}_q-(v_q+z_q)\otimes (v_q+z_q)\right) : (\nabla z_q)^T \dif x \dif s,
	 	\end{aligned}
	 \end{equation}
	 where we use $:$ to denote the Frobenius inner product of two matrices $A,B$, defined by $A:B= \sum_{i,j}A_{i,j}B_{j,i}$ and we use $A^T$ to denote the transpose of $A$.
    In a similar manner, we note $\dif v_q =\div \big( \mathring{R}_q-(v_q+z_q)\otimes (v_q+z_q) \big) \dif t -\nabla p \dif t$ and employ the integration by parts formula again to derive
     \begin{equation}\label{claim3}
    	\begin{aligned}
    		 \int_{0}^{t\wedge \mathfrak{t}_L}  \big\langle &v_q(s), \dif B(s) \big\rangle = \int_{\mathbb{T}^3} \big((v_q \cdot B )(t\wedge \mathfrak{t}_L) - (v_q \cdot B) (0) \big) \dif x - \int_{\mathbb{T}^3}  \int_{0}^{t\wedge \mathfrak{t}_L} B (s)\cdot \dif v_q(s) \dif x 
    		\\ & = \int_{\mathbb{T}^3} (v_q \cdot B )(t\wedge \mathfrak{t}_L) \dif x 
    	+ \int_{0}^{t\wedge \mathfrak{t}_L} \int_{\mathbb{T}^3} \left(\mathring{R}_q-(v_q+z_q)\otimes (v_q+z_q)\right) :(\nabla B)^T \dif x \dif s.
    	\end{aligned}
    \end{equation}
    According to \cite[Proposition 2.1.10]{LR15}, the $GG^*$-Wiener process $B$ can be written as $B=\sum_k \sqrt{c_k}\beta_k e_k$ for an orthonormal basis $\{e_k\}_{k\in \N}$ of $L^2_\sigma$ consisting of eigenvectors of $GG^*$ with corresponding eigenvalues $c_k$ and the coefficients satisfy $\sum_k c_k<\infty$. Here, $\{\beta_k\}_{k\in \N}$  denotes a sequence of mutually independent standard real-valued Brownian motions. Then, by substituting \eqref{claim2} and \eqref{claim3} into \eqref{claim part1}, and using the estimates \eqref{z ps}, \eqref{zq-B C1}, \eqref{itera:a} and \eqref{itera:b}, we obtain for any $t\in [0,\infty)$ 
    \begin{equation*}
    	\begin{aligned}
    	&	\lim_{q\to \infty}	\E  \left|  \int_{0}^{t\wedge \mathfrak{t}_L} \int_{\mathbb{T}^3} v_q \cdot \partial_t z_q \dif x \dif s-	 \int_{0}^{t\wedge \mathfrak{t}_L} \big\langle v(s), \dif B(s) \big\rangle   \right|^2 
    	\\& \lesssim 	\lim_{q\to \infty} \E	\left|  \int_{0}^{t\wedge \mathfrak{t}_L} \int_{\mathbb{T}^3} v_q \cdot \partial_t z_q \dif x \dif s-	 \int_{0}^{t\wedge \mathfrak{t}_L} \big\langle v_q(s), \dif B(s) \big\rangle   \right|^2 + 		\lim_{q\to \infty} \E \left| 	 \int_{0}^{t\wedge \mathfrak{t}_L} \big\langle v_q(s)- v(s), \dif B(s) \big\rangle   \right|^2 
    		\\&  \lesssim	\lim_{q\to \infty}\E \left|  \int_{0}^{t\wedge \mathfrak{t}_L } \int_{\mathbb{T}^3} \left(\mathring{R}_q-(v_q+z_q)\otimes (v_q+z_q)\right) :\nabla (z_q-B)^T \dif x \dif s  \right|^2
    			\\ &\qquad+	\lim_{q\to \infty} \E	\left|  \int_{\mathbb{T}^3} \big((v_q\cdot z_q )(t\wedge \mathfrak{t}_L)-(v_q \cdot B )(t\wedge \mathfrak{t}_L) \big) \dif x   \right|^2 
    +	\lim_{q\to \infty} \sum_{k}c_k \E \left(\int_{0}^{t\wedge \mathfrak{t}_L}  \left\langle v_q(s)-v(s),e_k\right\rangle ^2 \dif s \right)
  		\end{aligned}
  \end{equation*}
\begin{equation*}
\begin{aligned}
    &\lesssim 	\lim_{q\to \infty} \E \left(  L^2\left(\|\mathring{R}_q \|_{C^0_{[0, \mathfrak{t}_L],x}}+\|v_q \|^2_{C^0_{[0, \mathfrak{t}_L],x}} + \|z_q \|^2_{C^0_{[0, \mathfrak{t}_L],x}}  \right)^2 \|z_q-B\|^2_{C_{[0, \mathfrak{t}_L]}^0C^1_x}\right)
    		\\ &\qquad + \lim_{q\to \infty} \E \left(  \|v_q \|^2_{C^0_{[0, \mathfrak{t}_L],x}}  \|z_q-B\|^2_{C_{[0, \mathfrak{t}_L],x}^0}\right) + 	\lim_{q\to \infty}\tr(GG^*) \E \left(\int_{0}^{t\wedge \mathfrak{t}_L}  \big\|v_q(s)-v(s)\big\|_{L^2}^2 \dif s \right)
    		\\ &\lesssim 	\lim_{q\to \infty} L^8\lambda_{1}^{9\alpha}\lambda_{q}^{-\sfrac43+4\alpha}+	\lim_{q\to \infty}  \tr(GG^*)\E \left(L \|v_q-v\|^2_{C^0_{[0, \mathfrak{t}_L],x}}\right)=0,
    	\end{aligned}
    \end{equation*}
    which establishes the claim \eqref{claim1}.

		\emph{\underline{Step 3}.}		
		Finally, we demonstrate that the energy profile \eqref{energy profile} holds. By combining \eqref{estimate:energy}, \eqref{claim1} and preceding discussion, we derive for any $t\in [0,\infty)$
	\begin{equation}\label{E:e-u}
		\aligned
		\E \bigg| & e(t\wedge \mathfrak{t}_L) -\|u(t\wedge \mathfrak{t}_L)\|^2_{L^2} +2\int_{0}^{t\wedge \mathfrak{t}_L} \big\langle u(s) ,\dif B(s)\big\rangle +  \tr(GG^*)(t\wedge \mathfrak{t}_L)  \bigg|
		\\ &\leq \lim_{q\to \infty}  \E \bigg| e(t\wedge \mathfrak{t}_L) -\|(v_q+z_q)(t\wedge \mathfrak{t}_L)\|^2_{L^2} +2\int_{0}^{t\wedge \mathfrak{t}_L} \int_{\mathbb{T}^3} (v_q+z_q)\cdot \partial_t z_q \dif x \dif s  \bigg|
		\\ &  + \lim_{q\to \infty}  \E   \bigg| \|(v_q+z_q)(t\wedge \mathfrak{t}_L)\|^2_{L^2} -  \|u(t\wedge \mathfrak{t}_L)\|^2_{L^2} \bigg|
		\\ &+ \lim_{q\to \infty}  \E   \bigg| 2\int_{0}^{t\wedge \mathfrak{t}_L} \big\langle u(s) ,\dif B(s)\big\rangle +  \tr(GG^*)(t\wedge \mathfrak{t}_L) -  2\int_{0}^{t\wedge \mathfrak{t}_L} \int_{\mathbb{T}^3} (v_q+z_q) \cdot \partial_t z_q \dif x \dif s   \bigg|=0.
		\endaligned
	\end{equation}
	By combining the continuity argument and \eqref{E:e-u}, we deduce that it holds $\mathbf{P}$-a.s.
	\begin{align}\label{eq:energy profile}
		e(t\wedge \mathfrak{t}_L) +2\int_{0}^{t\wedge \mathfrak{t}_L} \big\langle u(s), \dif B(s) \big\rangle	+\tr(GG^*)(t\wedge \mathfrak{t}_L) =\|u(t\wedge \mathfrak{t}_L)\|_{L^2}^2,
			\end{align} 
	for any $t\in[0,\infty)$.
		
		 In view of the definition \eqref{stopping time ps} of $\mathfrak{t}_L$ and Proposition~\ref{estimate for z}, we note that for a given $T>0$ and $\varkappa\in (0,1)$, we may possibly increase $L$ so that the set $\{\mathfrak{t}_L>T\}$ satisfies $\mathbf{P}(\mathfrak{t}_L>T)\geq \varkappa$. Therefore, by setting $\mathfrak{t}:=\mathfrak{t}_L$ on both sides of \eqref{eq:energy profile}, we obtain the desired energy profile \eqref{energy profile}. Moreover, there exists a deterministic constant $\bar{c}$ such that $\esssup_{\omega \in \Omega} \sup_{t\in [0,\mathfrak{t}]}\|u(t)\|_{C^{\vartheta}_x}\leq \bar{c}$. If we further require $e$ to be strictly decreasing, then \eqref{energy profile} immediately implies \eqref{energy:inequality}, thereby completing the proof of Theorem~\ref{Onsager:theorem}.  \qed

	\section{Proof of Proposition~\ref{p:iteration}---Step 1: Mollification and gluing}\label{sec:begin}
	In this section, we begin the proof of Proposition~\ref{p:iteration}. We conduct a pathwise analysis of the mollified random equation \eqref{euler1} to implement the convex integration procedure. Drawing inspiration from \cite{Ise18}, before adding the convex integration perturbation, it is useful to replace the approximate solution $(v_q,\mathring{R}_q)$ with another smooth solution $(\overline{v}_q,\mathring{\overline{R}}_q)$, such that $\overline{v}_q$  remains close to $v_q$, while ensuring that $\mathring{\overline{R}}_q$ vanishes on alternating intervals of size $\tau_q\approx (\lambda_{q}\delta_{q}^{\sfrac12})^{-1}$ within $[0,\mathfrak{t}_L]$. This gluing procedure is crucial for improving the regularity of the solutions, as discussed in Subsection~\ref{sec:esact solutions} and Subsection~\ref{sec:gluing}. 
	To achieve the desired pathwise estimates for $(\overline{v}_q,\mathring{\overline{R}}_q)$ while effectively controlling the noise term $z_{q}$, careful parameter adjustments are required. 
	
	The primary challenge in establishing energy inductive estimates arises from the presence of the integral $ \int_{0}^{\cdot} \int_{\T^3} (v_{q}+z_q) \cdot \partial_t z_{q}  \dif x \dif s $ in \eqref{estimate:energy}, which serves to approximate the stochastic integral $ \int_{0}^{\cdot}\big\langle u(s), \dif B(s) \big\rangle  $.
	This approximation necessitates a temporal mollification of the Wiener process  $B$, along with careful fine-tuning of the mollification parameter $\iota_{q}$ to ensure that $\|\partial_t z_{q}\|_{C^0_x}\approx \lambda_{q}\delta_{q}^{\sfrac12}$. As a result, more involved computations for $\partial_t z_q$ are necessary compared to the deterministic case, see Proposition~\ref{esti:mollification}, Proposition~\ref{estimate:energy2} and Proposition~\ref{Prop:addition energy} below. We begin this section by outlining the mollification procedure, adhering to standard techniques.

	\subsection{Mollification and preliminary estimates}\label{sec:mollification}
	To handle the loss of derivative problem, it is typical for convex integration schemes to replace $v_q$ with a mollified velocity field. To this end, we consider the spatial mollifier $\varphi_{\ell_q}$ given in Section~\ref{notation} with parameter $\ell_{q}$ defined in \eqref{def l}. Then, we define the mollified fields $(v_{\ell_q},z_{\ell_q},\mathring{R}_{\ell_q})$ as the convolution over space of the fields $(v_q,z_q,\mathring{R}_q)$ at the $q$ step:
	\begin{equation}\label{equation:mollified}
		\aligned
		v_{\ell_q}:&=v_q * \varphi_{\ell_q}, \qquad
		z_{\ell_q}:=z_q*\varphi_{\ell_q},
		\\ \mathring{R}_{\ell_q}:&=\mathring{R}_q*\varphi_{\ell_q}-((v_q+z_q)\mathring{\otimes}(v_q+z_q))*\varphi_{\ell_q}+(v_{\ell_q}+z_{\ell_q})\mathring{\otimes}(v_{\ell_q}+z_{\ell_q}).
		\endaligned
	\end{equation}
	Since $\varphi_{\ell_q}$ is a spatial mollifier, $v_{\ell_q}$, $z_{\ell_q}$ and $\mathring{R}_{\ell_q}$ are still $(\mathcal{F}_{t})_{t\geq 0}$-adapted. 
	In addition, we observe from \eqref{euler1} that $(v_{\ell_q} ,\mathring{R}_{\ell_q})$ satisfies on $[0,\mathfrak{t}_L]$:
	\begin{equation}\label{mollification}
		\aligned
		\partial_t v_{\ell_q} +\div((v_{\ell_q}+z_{\ell_q})\otimes (v_{\ell_q}+z_{\ell_q}))+\nabla p_{\ell_q}&=\div \mathring{R}_{\ell_q},
		\\\div v_{\ell_q} &=0.
		\\ 
		\endaligned
	\end{equation}
	Here, $p_{\ell_{q}}$ is the related pressure term.
	In addition, the standard mollification estimates, as outlined in Proposition~\ref{commutator esti}, yield the following results.
	\begin{proposition}\label{esti:mollification}
	The mollified fields obey the following bounds for any $ N\in \N_0$ and $t\in [0,\mathfrak{t}_L]$
		\begin{subequations}
			\begin{align}
				 	\|(v_{\ell_q}-v_q)(t)\|_{C^0_{x}}&\lesssim  \bar{M}L \delta_{q+1}^{\sfrac12}\ell_{q}^{\alpha}\ ,\label{vq-vl}
				 \\      \|\mathring{R}_{{\ell_q}}(t)\|_{C_x^{N+\alpha}}&\lesssim  \bar{M}L^2 \delta_{q+1}\ell_q^{-N+\alpha}\, ,\label{estimate Rl}
				 \\ \left| \int_{\mathbb{T}^3} |v_q+z_q|^2- |v_{\ell_q}+z_{\ell_q}|^2 \dif x \right|  &\lesssim  \frac14 L^2 \delta_{q+1} \ell_{q}^{\alpha} \, , \label{estimate energy vq-vl}
				 \\  \left| \int_{0}^{t} \int_{\T^3} ( v_{\ell_{q}}+z_{\ell_{q}})\cdot \partial_t z_{\ell_{q}} - (v_{q}+z_q) \cdot \partial_t z_{q}  \dif x \dif s  \right| &\lesssim \frac14 L^2 \delta_{q+1} \ell_{q}^{\alpha}\, ,  \label{estimate energy vq-vl cdot zlq}
			\end{align}
		\end{subequations}
		where the implicit constant may depend on $N$ and $\alpha$.
	\end{proposition}
	
	\begin{proof}
		The bound \eqref{vq-vl} directly follows from \eqref{def l}, \eqref{ell:lambdaq}, \eqref{itera:b} and \eqref{estimate:molli3}:
		\begin{align*}
			\|(v_{\ell_q}-v_q)(t)\|_{C^0_{x}} \lesssim {\ell_q} \|v_q\|_{C^0_{t}C^1_x}\lesssim \bar{M}L \ell_q  \delta_q^{\sfrac12}\lambda_{q} \leq\bar{M}L \delta_{q+1}^{\sfrac12}\ell_{q}^{\alpha}.
		\end{align*}
		Keeping the mollfication estimate \eqref{estimate:molli2} in mind, together with \eqref{z ps}, \eqref{itera:b} and \eqref{itera:c} permit to deduce for any $ N\in \N_0$ and $t\in [0,\mathfrak{t}_L]$
		\begin{align*}
			\|\mathring{R}_{{\ell_q}}(t)\|_{C_x^{N+\alpha}} &
			\lesssim  \|((v_q+z_q)\mathring{\otimes}  (v_q+z_q) )*\varphi_{\ell_q}-(v_{\ell_q}+z_{\ell_q})\mathring{\otimes}(v_{\ell_q}+z_{\ell_q})\|_{C_{t}^0C^{N+\alpha}_x} 
			+\ell_q^{-N-\alpha}\|\mathring{R}_q\|_{C^0_{t,x}}\nonumber
			\\  &\lesssim   \ell_q^{2-N-\alpha} \|v_q+z_q\|_{C^0_tC^1_x}^2  +\ell_q^{-N-\alpha}\|\mathring{R}_q\|_{C^0_{t,x}}
		   \lesssim \bar{M}^2L^2 \ell_q^{2-N-\alpha}\lambda_{q}^2\delta_q+ \bar{M}L^2\lambda_{q}^{-3\alpha}\delta_{q+1} \ell_q^{-N-\alpha}\nonumber
			\\ &\lesssim  (\bar{M}^2L^2 \ell_{q}^{-2\alpha}\lambda_{q}^{-12\alpha}+ \bar{M}L^2\lambda_{q}^{-3\alpha}\ell_{q}^{-2\alpha})\delta_{q+1} \ell_q^{-N+\alpha} \leq  \bar{M}L^2 \delta_{q+1} \ell_q^{-N+\alpha},\nonumber
		\end{align*}
		which implies \eqref{estimate Rl}. The last inequality is justified by $\lambda_{q}^{-3\alpha}\ell_{q}^{-2\alpha}\leq 1$ and choosing $a$ sufficiently large to have $\bar{M}\lambda_{q}^{-\alpha}\ll1$. Moving to \eqref{estimate energy vq-vl}, we use \eqref{def l}, \eqref{ell:lambdaq}, \eqref{z ps}, \eqref{itera:b} and mollification estimate \eqref{estimate:molli2} to obtain for any $t\in [0,\mathfrak{t}_L]$
		\begin{equation}\label{energy:comm}
		\begin{aligned}
			\left| \int_{\mathbb{T}^3} |v_q+z_q|^2- |v_{\ell_q}+z_{\ell_q}|^2 \dif x \right| &=\left| \int_{\mathbb{T}^3} |v_q+z_q|^2 *\varphi_{{\ell_q}}- |v_{\ell_q}+z_{\ell_q}|^2 \dif x \right| 
			\\ &\lesssim \big\||v_q+z_q|^2 *\varphi_{{\ell_q}}- |v_{\ell_q}+z_{\ell_q}|^2 \big\|_{C_{[0,\mathfrak{t}_L]}C^0_x} 
			\\ & \lesssim \ell_q^{2} \|v_q+z_q\|_{C_{[0,\mathfrak{t}_L]}C^1_x} ^2 
			\lesssim \bar{M}^2 L^2 \lambda_{q}^{-12\alpha}\delta_{q+1}\leq \frac14 L^2 \delta_{q+1}\ell_{q}^{\alpha}, 
		\end{aligned}
			\end{equation}
		which thanks to $\bar{M}\lambda_{q}^{-\alpha}\ll 1$. In contrast to the deterministic setting (see \cite[Proposition 2.2]{BDLSV19}), an additional estimate \eqref{estimate energy vq-vl cdot zlq} is required in our approach, which is used to derive the energy estimate \eqref{estimate:energy}. Following a similar procedure as in \eqref{energy:comm}, and using \eqref{z ps}, \eqref{itera:b}, \eqref{estimate:molli1} and \eqref{estimate:molli2} we obtain for any $t\in [0,\mathfrak{t}_L]$
		\begin{align*}
			 & \left| \int_{0}^{t} \int_{\T^3} ( v_{\ell_{q}}+z_{\ell_{q}})\cdot \partial_t z_{\ell_{q}} - (v_{q}+z_q) \cdot \partial_t z_{q} \, \dif x \dif s  \right| 
			 \\ &\quad =  \left| \int_{0}^{t} \int_{\T^3} ( v_{\ell_{q}}+z_{\ell_{q}})\cdot \partial_t z_{\ell_{q}} - \left((v_{q}+z_q) \cdot \partial_t z_{q} \right)*\varphi_{{\ell_q}} \dif x \dif s  \right| 
			 \\ & \quad \lesssim L \big\| (v_{\ell_{q}}+z_{\ell_{q}})\cdot \partial_t z_{\ell_{q}} - \left((v_{q}+z_q) \cdot \partial_t z_{q} \right)*\varphi_{{\ell_q}} \big\|_{C_{[0,\mathfrak{t}_L]}C^0_x}
			 \lesssim L \ell_{q}^2 \|v_{q}+z_q\|_{C_{[0,\mathfrak{t}_L]}C^1_x} \| \partial_t z_{q}\|_{C_{[0,\mathfrak{t}_L]}C^1_x} 
			 \\ & \quad \lesssim \bar{M}L^2 \ell_{q}^2  \lambda_{q}\delta_{q}^{\sfrac12} \iota_q^{-(\sfrac12+\alpha)} \|B_q\|_{C^{\sfrac12-\alpha}_{[0,\mathfrak{t}_L ]}C^1_x} 
		\lesssim \bar{M}L^3 \ell_{q}^2(\lambda_{q}\delta_{q}^{\sfrac12})^2\lambda_{q}^{2\alpha} \leq  \frac14 L^2 \delta_{q+1} \ell_{q}^{\alpha},
		\end{align*}
		where we used $\beta<1/3$ to have $\iota_{q}^{-\sfrac12}=\lambda_{q}^{\sfrac23}\leq \lambda_{q}^{1-\beta} \leq \lambda_{q}\delta_q^{\sfrac12}$ and chose $a$ sufficiently large to have $\bar{M}L\lambda_{q}^{-\alpha}\ll 1$ in the last line.
	\end{proof}

	\subsection{Exact solutions of the Euler system and their stability}\label{sec:esact solutions} 
	To execute the gluing procedure, we first construct a family of exact solutions $(\ve_i,\pe_i)$ to the Euler system \eqref{eq:euler exact1}, with initial conditions matching $v_{\ell_q}$ at specific times $\te_i$. Given the presence of $z_{\ell_{q}}$, it is necessary to carefully track the pathwise estimates for the exact solutions $\ve_i$ by adjusting the parameters to effectively control the noise component. Once the desired pathwise estimates for $\ve_i$ are obtained, the same method as in \cite[Proposition 3.3]{BDLSV19} can be applied to analyze their stability relative to $v_{\ell_{q}}$. Let us begin by constructing exact solutions to the Euler system.
	
	\subsubsection{Exact solutions}
	We first recall the following classical local existence result for the Euler system, which is an adaptation from \cite[Section 3.2]{BM02}.
	
	\begin{lemma}	\label{lem:local_existence}
		Let $\alpha \in (0,1)$ be given as in Subsection~\ref{sec:parameters} and $T>0 $. Let $v_0\in C^{\infty}(\T^3,\R^3)$ be a divergence-free initial data and $Z\in C([0,T], C^{\infty}(\T^3,\R^3))$.
		For $\tau \leq \min \left\{\frac{1}{4} \big(\|v_0\|_{C^{1+\alpha}_x}+ \|Z\|_{C_{ T}C^{2+\alpha}_x} \big)^{-1} ,T\right\} $,
		there exists a unique solution $v\in C([0,\tau ],C^{\infty}(\T^3,\R^3))$  to the Euler system
		\begin{equation*}
			\aligned
			\partial_t v+\div((v+Z)\otimes (v+Z))+\nabla p&=0,
			\\ \div v&=0,
			\\ v(0,\cdot)&=v_0.
			\endaligned
		\end{equation*}
		Moreover, $v$ obeys the following bounds for any $t \in [0,\tau ]$
		\begin{equation}\label{estimate: smooth u}
			\aligned
			\|v(t)\|_{C^{1+\alpha}_x}  &\lesssim_{\alpha}  	\|v_0\|_{C^{1+\alpha}_x}+\|Z\|_{C_T^0C^{2+\alpha}_x},
			\\ 
			\|v(t)\|_{C^{N+\alpha}_x} &\lesssim_{N,\alpha}   \|v_0\|_{C^{N+\alpha}_x} +  \tau  \|Z\|_{C_T^0C^{N+1+\alpha}_x}\big(\|v_0\|_{C^{1+\alpha}_x}+\|Z\|_{C_T^0C^{2+\alpha}_x}\big), \quad N\geq 2,
			\endaligned
		\end{equation}
		where the implict constant may depend on $N$ and $\alpha$.
	\end{lemma}
	The proof of Lemma~\ref{lem:local_existence} follows from standard techniques and we provide the details in Appendix~\ref{appendix:proof:lemma 4.2}.
	To construct the exact solutions to the Euler system and derive the related estimates by Lemma~\ref{lem:local_existence}, we first define the parameter $\tau_q$ and initial times $\te_i \, (i\in   [-1,\infty) \cap \mathbb{Z})$ as
	\begin{align}\label{def tauq}
		\tau_q:=\frac{1}{L\lambda_{q}^{1+6\alpha}\delta_q^{\sfrac12}}\, ,\qquad \te_i:=i\tau_q\,(i\geq 0),\qquad \te_{-1}:=0.
	\end{align}
	Invoking Lemma~\ref{lem:local_existence} with $T=\mathfrak{t}_L$ and $\tau= \tau_q$, we define $(\ve_i, \pe_i)$ for $i\geq 0$ to be the unique smooth solution to the Euler system with initial data $v_{\ell_q}(\te_{i-1},\cdot)$:
	\begin{equation}\label{eq:euler exact1}
		\aligned
		\partial_t \ve_i+\div ( (\ve_i+z_{\ell_q}) \otimes (\ve_i+z_{\ell_q}) )+\nabla \pe_i& =0 \, ,
		\\ 	\div \ve_i& =0\, ,
		\\  \ve_i(\te_{i-1},\cdot)&= v_{\ell_q}(\te_{i-1},\cdot)\, ,
		\endaligned
	\end{equation}
	on $\big([\te_{i-1},\te_{i+1}]\cap [0,\mathfrak{t}_L]\big)\times \mathbb{T}^3$. In fact, the local well-posedness of $\ve_i$ on this time scale is permissible because by combining \eqref{ell:lambdaq}, \eqref{z ps} and \eqref{itera:b} with mollification estimate \eqref{estimate:molli1}, we deduce 
	\begin{align*}
		(\te_{i+1}-\te_{i-1})  \|v_{\ell_q}\|_{C^0_{\mathfrak{t}_L}C_x^{1+\alpha}}&\lesssim  \bar{M} L \tau_q \lambda_{q}\delta_q^{\sfrac12} \ell_q^{-\alpha}  \leq \bar{M} \lambda_{q}^{-\alpha} \leq \frac18 \, ,
		\\ 	(\te_{i+1}-\te_{i-1})   \|z_{\ell_q}\|_{C^0_{\mathfrak{t}_L}C_x^{2+\alpha}}&\lesssim L\tau_q  \ell_{q}^{-\alpha}\leq  \lambda_{q}^{\beta-1} \leq \frac18 \, ,
	\end{align*}
	which verifies the conditions in Lemma~\ref{lem:local_existence}, and we used $\bar{M} \lambda_{q}^{-\alpha} \ll 1$ to absorb other constants. Since $z_{\ell_q}$ and $v_{\ell_q}$ are $(\mathcal{F}_t)_{t\geq 0}$-adapted, so are $\ve_i$ and $\pe_i$. Furthermore, by applying \eqref{z ps}, \eqref{itera:b} and \eqref{estimate: smooth u}, we deduce for any $t\in [\te_{i-1},\te_{i+1}]\cap [0,\mathfrak{t}_L]$ and $N\geq2$
	\begin{equation}\label{estimate:vi:N}
		\begin{aligned}
			\|\ve_i(t)\|_{C_x^{N+\alpha}} &\lesssim 	\|v_{\ell_q}\|_{C^0_{t}C_x^{N+\alpha}}
			+\tau_q \|z_{\ell_q}\|_{C_t^0C^{N+1+\alpha}_x}\big(\|v_{\ell_q}\|_{
				C^0_tC^{1+\alpha}_x}+\|z_{\ell_q}\|_{C_t^0C^{2+\alpha}_x}\big)
			\\&\lesssim  \bar{M}L \ell_q^{1-N-\alpha}\lambda_{q}\delta_q^{\sfrac12}+\bar{M}L^2\tau_q \ell_q^{1-N-\alpha}(\ell_{q}^{-\alpha}\lambda_{q}\delta_q^{\sfrac12}+\ell_q^{-\alpha})
			\\ &\lesssim (\ell_q^{2\alpha}\tau_q^{-1}+L\ell_q^{\alpha})\bar{M} \ell_q^{1-N+\alpha} \lesssim \bar{M} \tau_q^{-1} \ell_q^{1-N+\alpha} \, ,
		\end{aligned}
	\end{equation}
	where we used \eqref{ell:lambdaq} to have $L\tau_q\lambda_{q}\delta_{q}^{\sfrac12} =\lambda_{q}^{-6\alpha} \leq  \ell_{q}^{4\alpha}$ in the last line. For $N=1$ we have
	\begin{align}\label{estimate:vi:1}
		\|\ve_i(t)\|_{C^{1+\alpha }_x}  &\lesssim	\|v_{\ell_q}\|_{C^{1+\alpha }_x} +\|z_{\ell_{q}}\|_{C_t^0C^{2+\alpha }_x}
		\lesssim \bar{M}L \ell_q^{-\alpha}\lambda_{q}\delta_q^{\sfrac12}+L \ell_q^{-\alpha} \lesssim \bar{M} \tau_q^{-1} \ell_q^{\alpha} \, .
	\end{align}
	By the above discussion, we obtain that for any $t\in [\te_{i-1},\te_{i+1}]\cap [0,\mathfrak{t}_L]$ and $i\geq 0$, the exact solution $\ve_i$ to \eqref{eq:euler exact1} satisfies the following bounds for any $N\geq 1$
	\begin{align}\label{estiamte:vi}
		\|\ve_i(t)\|_{C_x^{N+\alpha}}
		\leq \bar{M} \tau_q^{-1} \ell_q^{1-N+\alpha}\, .
	\end{align}
	
	\subsubsection{Stability and estimates on $\ve_i-v_{\ell_q}$}
	We next show that for $t\in [\te_{i-1},\te_{i+1}]\cap [0,\mathfrak{t}_L]$, $\ve_i$ is close to $v_{\ell_q}$, and by the identity
	$$\ve_i-\ve_{i+1}=(\ve_i-v_{\ell_q})-(\ve_{i+1}-v_{\ell_q}),$$
	the vector field $ \ve_i$ is also close to $\ve_{i+1}$.
	\begin{proposition}\label{esti: v,p,D i-l}
		For any $t\in  [\te_{i-1},\te_{i+1}]\cap [0,\mathfrak{t}_L]$ and $N\geq 0$, we have
		\begin{subequations}
			 	\begin{align}
			 	\|v_{\ell_q}-\ve_i\|_{C^{N+\alpha}_x}&\lesssim \bar{M}L^2 \tau_q \delta_{q+1}\ell_q^{-1-N+\alpha} \ ,\label{v i- v l}
			 	\\
			 	\|D_{t,{\ell_q}}(v_{\ell_q}-\ve_i)\|_{C^{N+\alpha}_x}&\lesssim  \bar{M} L^2 \delta_{q+1}\ell_q^{-1-N+\alpha} \ ,  \label{Dtl v i- v l}
			 \end{align}
		\end{subequations}
		where we write $D_{t,{\ell_q}}:=\partial_t+(v_{\ell_q}+z_{\ell_q}) \cdot \nabla_x$ for the transport derivative and the implicit constant may depend on $N$ and $\alpha$.
	\end{proposition}
	\begin{proof}
		The proof follows essentially from \cite[Proposition 3.3]{BDLSV19}, but we need to treat the noise part $z_{\ell_q}$ more carefully. By using \eqref{z ps}, \eqref{itera:b} and mollification estimate \eqref{estimate:molli1}, we obtain for any $N\geq0$, $t\in[0,\mathfrak{t}_L]$
		\begin{equation}\label{estimate:vellzell}
			\begin{aligned}
				\|v_{\ell_{q}}(t)\|_{C^{N+1+\alpha}_x} &\lesssim \bar{M}\tau_q^{-1} \ell_q^{-N-\alpha} \lambda_{q}^{-4\alpha} ,\\ \|z_{\ell_{q}}(t)\|_{C^{N+1+\alpha}_x}&\lesssim L \ell_{q}^{-N-\alpha}\leq  \tau_q^{-1} \ell_q^{-N-\alpha} \lambda_{q}^{-4\alpha}.
			\end{aligned}
		\end{equation}
		We first consider \eqref{v i- v l} with $N=0$. Combining \eqref{mollification} with \eqref{eq:euler exact1} we have 
		\begin{align} \label{equation: vl-vi}
			\partial_t(v_{\ell_q}- \ve_i)+ (v_{\ell_q}+z_{\ell_q})\cdot \nabla(v_{\ell_q}- \ve_i)=( \ve_i-v_{\ell_q})\cdot \nabla(\ve_i+z_{\ell_q} )-\nabla(p_{\ell_q}- \pe_i)+\div \mathring{R}_{{\ell_q}}.
		\end{align}
		Taking divergence on both sides of \eqref{equation: vl-vi}, we obtain the equation for the pressure difference
		\begin{align}\label{laplace pl-pi}
			\Delta(p_{\ell_q}- \pe_i)=\div (( \ve_i-v_{\ell_q})\cdot \nabla(\ve_i+z_{\ell_q} )+( \ve_i-v_{\ell_q})\cdot \nabla(v_{\ell_q}+z_{\ell_q} ))+\div \div  \mathring{R}_{{\ell_q}}.
		\end{align}
		Using \eqref{estimate Rl}, \eqref{estiamte:vi}, \eqref{estimate:vellzell} and Schauder estimates, we deduce for any $t\in  [\te_{i-1},\te_{i+1}]\cap [0,\mathfrak{t}_L]$ 
		\begin{equation}\label{esti:pl-pi}
			\begin{aligned}
				\|\nabla(p_{\ell_q}-\pe_i)\|_{C^\alpha_x}&\lesssim (\|v_{\ell_q}\|_{C^{1+\alpha}_x}+\|\ve_i\|_{C^{1+\alpha}_x}+\|z_{\ell_q}\|_{C^{1+\alpha}_x})\|\ve_i-v_{\ell_q}\|_{C^\alpha_x}+\|\mathring{R}_{\ell_q}\|_{C^{1+\alpha}_x}
				\\&\lesssim (\bar{M}L\lambda_{q}\delta_q^{\sfrac12}\ell_q^{-\alpha}+\bar{M}\tau_q^{-1}\ell_q^{\alpha})\|\ve_i-v_{\ell_q}\|_{C^\alpha_x}+\|\mathring{R}_{\ell_q}\|_{C^{1+\alpha}_x}
				\\&\lesssim   \tau_q^{-1}\|\ve_i-v_{\ell_q}\|_{C^\alpha}+\bar{M}L^2\delta_{q+1}\ell_q^{-1+\alpha}.
			\end{aligned}
		\end{equation}
		The last line makes use of \eqref{ell:lambdaq} to have $L\tau_q\lambda_{q}\delta_{q}^{\sfrac12} =\lambda_{q}^{-6\alpha} \leq  \ell_{q}^{4\alpha}$ and $a$ large enough such that $\bar{M}\ell_{q}^{\alpha}\ll 1$.
		Hence, we use \eqref{estimate Rl} and \eqref{estiamte:vi} again to obtain 
		\begin{equation}\label{estimate: dtl alpha}
			\begin{aligned}
				\|D_{t,{\ell_q}}(v_{\ell_q}-\ve_i)\|_{C^\alpha_x}\lesssim   \tau_q^{-1}\|(\ve_i-v_{\ell_q})(t)\|_{C^\alpha_x}+\bar{M}L^2\delta_{q+1}\ell_q^{-1+\alpha}.
			\end{aligned}
		\end{equation}
		From $v_{\ell_q}(t_{i-1})=\ve_{i}(t_{i-1})$ and the estimate \eqref{eq:Gronwall1} for transport equations, it follows that
		\begin{align*}
			\|(v_{\ell_q}-\ve_i)(t)\|_{C^\alpha_x}\lesssim 	\bar{M}L^2 |t-\te_{i-1}|\delta_{q+1}\ell_q^{-1+\alpha}+\tau_q^{-1}\int_{\te_{i-1}}^{t}	\|(v_{\ell_q}-\ve_i)(s)\|_{C^\alpha_x} \dif s.
		\end{align*}
		Applying Gr\"{o}nwall's inequality and by the assumption $|t-\te_{i-1}|\leq 2\tau_q$, we obtain 
		\begin{align}\label{vl-vi alpha}
			\|(v_{\ell_q}-\ve_i)(t)\|_{C^\alpha_x}\lesssim \bar{M}L^2\tau_q \delta_{q+1}\ell_q^{-1+\alpha},
		\end{align}
		which proves \eqref{v i- v l} for the case $N=0$. 
		Then, we obtain \eqref{Dtl v i- v l} for the case $N=0$ as a consequence of  \eqref{estimate: dtl alpha}.
		
		Next, consider the case $N\geq1$ and let $\theta$ be a multi-index with $|\theta|=N$. We commute the derivative $\partial^\theta$ with the transport derivative $D_{t,{\ell_q}}=\partial_t+(v_{\ell_q}+z_{\ell_q})\cdot \nabla $ and use \eqref{estimate:vellzell}, \eqref{vl-vi alpha} to have
		\begin{equation}\label{Dt,lN}
			\begin{aligned}
				\|D_{t,{\ell_q}}&\partial^{\theta}(v_{\ell_q}-\ve_i)\|_{C^\alpha_x} 
				\\ &\lesssim \|\partial^{\theta} D_{t,{\ell_q}} (v_{\ell_q}-\ve_i)\|_{C^\alpha_x} + \sum_{0\leq j \leq N-1} \|v_{\ell_q}+z_{\ell_q}\|_{C^{j+1+\alpha}_x}  \|v_{\ell_q}-\ve_i\|_{C^{N-j+\alpha}_x}
				\\& \lesssim   \|\partial^{\theta} D_{t,{\ell_q}} (v_{\ell_q}-\ve_i)\|_{C^\alpha_x} +\|v_{\ell_q}+z_{\ell_q}\|_{C^{1+\alpha}_x} \|v_{\ell_q}-\ve_i\|_{C^{N+\alpha}_x}
				\\&\qquad \qquad \qquad \qquad\qquad \quad + \|v_{\ell_q}+z_{\ell_q}\|_{C^{N+1+\alpha}_x}  \|v_{\ell_q}-\ve_i\|_{C^{\alpha}_x}
				\\ &\lesssim \|\partial^{\theta} D_{t,{\ell_q}} (v_{\ell_q}-\ve_i)\|_{C^\alpha_x} +\tau_q^{-1} \|\ve_i-v_{\ell_q}\|_{C^{N+\alpha}_x}+\bar{M}L^2\delta_{q+1} \ell_q^{-1-N+\alpha},
			\end{aligned}
		\end{equation}
		where we used interpolation in the second inequality and the last line follows from \eqref{vl-vi alpha} and $\bar{M}\lambda_{q}^{-\alpha}\ll 1$. On the other hand, differentiating \eqref{equation: vl-vi} together with \eqref{estimate Rl}, \eqref{estiamte:vi}, \eqref{estimate:vellzell} and \eqref{vl-vi alpha} yields to 
		\begin{equation}\label{NDt,l}
			\begin{aligned}
				\|\partial^{\theta} D_{t,{\ell_q}} (v_{\ell_q}-\ve_i)\|_{C^\alpha_x}\lesssim& \|\ve_i-v_{\ell_q}\|_{C^{N+\alpha}_x}\|\ve_i+z_{\ell_q}\|_{C^{1+\alpha}_x}+ \|\ve_i-v_{\ell_q}\|_{C^{\alpha}_x}\| \ve_i+z_{\ell_q}\|_{C^{N+1+\alpha}_x}
				\\&+\|p_{\ell_q}-\pe_i\|_{C^{N+1+\alpha}_x}+\|\mathring{R}_{{\ell_q}}\|_{C^{N+1+\alpha}_x}
				\\\lesssim&\tau_q^{-1} \| \ve_i-v_{\ell_q}\|_{C^{N+\alpha}_x}+\|\nabla(p_{\ell_q}-\pe_i)\|_{C^{N+\alpha}_x}+  \bar{M}L^2 \delta_{q+1} \ell_q^{-1-N+\alpha}.
			\end{aligned}
		\end{equation}
		Similar to \eqref{esti:pl-pi}, by using again \eqref{estimate Rl}, \eqref{estiamte:vi}, \eqref{estimate:vellzell}, \eqref{vl-vi alpha} and Schauder estimates, it holds
		\begin{equation}\label{pl-iN}
			\begin{aligned}
				\|\nabla(p_{\ell_q}-\pe_i)\|_{C^{N+\alpha}_x}\lesssim& \|\ve_i-v_{\ell_q}\|_{C^{N+\alpha}_x}(\|\ve_i\|_{C^{1+\alpha}_x}+\|v_{\ell_q}\|_{C^{1+\alpha}_x}+\|z_{\ell_q}\|_{C^{1+\alpha}_x})+\|\mathring{R}_{{\ell_q}}\|_{C^{N+1+\alpha}_x}
				\\&+\|\ve_i-v_{\ell_q}\|_{C^{\alpha}_x}(\|\ve_i\|_{C^{N+1+\alpha}_x}+\|v_{\ell_q}\|_{C^{N+1+\alpha}_x}+\|z_{\ell_q}\|_{C^{N+1+\alpha}_x})
				\\ \lesssim& \tau_q^{-1}\|\ve_i-v_{\ell_q}\|_{C^{N+\alpha}_x}+ \bar{M}L^2 \delta_{q+1}\ell_q^{-1-N+\alpha}\, .
			\end{aligned}
		\end{equation}
		Substituting \eqref{NDt,l} and \eqref{pl-iN} into \eqref{Dt,lN}, we then obtain
		\begin{align*}
			\| D_{t,{\ell_q}} \partial^{\theta}(v_{\ell_q}-\ve_i)\|_{C^\alpha_x}\lesssim \tau_q^{-1}\|\ve_i-v_{\ell_q}\|_{C^{N+\alpha}_x}+ \bar{M}L^2 \delta_{q+1}\ell_q^{-1-N+\alpha}\, .
		\end{align*}
		Therefore, invoking once more \eqref{eq:Gronwall1} with Gr\"{o}nwall's inequality, we obtain for $t\in  [\te_{i-1},\te_{i+1}]\cap [0,\mathfrak{t}_L]$ 
		\begin{align}\label{vl-vi N+alpha}
			\|(v_{\ell_q}-\ve_i)(t)\|_{C^{N+\alpha}_x}&\lesssim \bar{M}L^2 \tau_q \delta_{q+1}\ell_q^{-1-N+\alpha},
		\end{align}
		which implies \eqref{v i- v l} for the case $N\geq 1$. As a consequence of \eqref{NDt,l}, \eqref{pl-iN} and \eqref{vl-vi N+alpha}, the estimate \eqref{Dtl v i- v l} also follows for the case $N\geq 1$.
	\end{proof}
	
	\subsubsection{Estimates on vector potentials}
As shown in Proposition~\ref{p:Rq} below, sharp estimates for $\ve_i-v_{\ell_q}$ in negative-order H\"older spaces are required. To achieve this, we recall the vector potentials associated with the exact solutions $\ve_i$. By the Helmholtz decomposition, a divergence-free field $V$ with mean zero can be written as a curl, i.e., $V=\nabla \times Z$ for an incompressible field $Z=\mathcal{B}V:=(-\Delta)^{-1}\mathrm{curl} V$ called the vector potential of $V$. Here, the operator $\mathcal{B}=(-\Delta)^{-1}\mathrm{curl}$ is known as the Biot-Savart operator. Define $b_i= \mathcal{B} \ve_{i}$. We aim to obtain estimates for the differences $b_i-b_{i+1}$, which will be established in Proposition~\ref{p:S_est} below.
	
	\begin{proposition}\label{p:S_est}
		For $t\in  [\te_{i},\te_{i+1}]\cap [0,\mathfrak{t}_L]$ and $N\geq 0$, we have
		\begin{subequations}
			 	\begin{align}
			 	\|b_i-b_{i+1}\|_{C^{N+\alpha}_x} &\lesssim  \bar{M}L^2 \tau_q\delta_{q+1}\ell_q^{-N+\alpha}\,,   \label{e:b_diff} \\
			 	\| D_{t,{\ell_q}} (b_i-b_{i+1})\|_{C^{N+\alpha}_x} &\lesssim  \bar{M}L^2 \delta_{q+1}\ell_q^{-N+\alpha}\,, \label{e:b_diff_Dt}
			 \end{align}
		\end{subequations}
		where $D_{t,{\ell_q}}$ is given as in Proposition~\ref{esti: v,p,D i-l} and the implicit constant may depend on $N$ and $\alpha$.
	\end{proposition}

	\begin{proof}
		We first denote $\tilde b_i := \mathcal{B} (\ve_i-v_{{\ell_q}})$ and observe that
		$b_{i}-b_{i+1}=\tilde b_i-\tilde b_{i+1}$. Hence, it suffices to estimate $\tilde b_i$ in place of $b_i-b_{i+1}$. Since $\nabla \mathcal{B}$ is a bounded operator on H\"older spaces, then it follows  directly from \eqref{v i- v l} that for $N\geq 1$ and $ [\te_{i},\te_{i+1}]\cap [0,\mathfrak{t}_L]$
		\begin{align}\label{e:N>=1}
			\|\nabla \tilde b_i\|_{C^{N-1+\alpha}_x}\lesssim  \|{\ve_i-v_{{\ell_q}}}\|_{C^{N-1+\alpha}_x}
			\lesssim  \bar{M}L^2\tau_q\delta_{q+1}\ell_q^{-N+\alpha} \, ,
		\end{align}
		which implies \eqref{e:b_diff} for the case $N\geq 1$.
		Let us move to the case $N=0$. First, we can derive the following Poission equation as in \cite[Page 244]{BDLSV19} 
		\begin{equation}\label{eq:poission}
			\begin{aligned}
				-\Delta\bigl(\partial_t\tilde b_i+(v_{\ell_q}+z_{\ell_q})\cdot\nabla\tilde b_i\bigr)=&-\nabla\div\big(\tilde b_i\cdot \nabla (v_{\ell_q}+z_{\ell_q})\big) -\curl\div\mathring{R}_{\ell_q}
				\\ &\quad -\curl\div \big(\tilde b_i\times\nabla(v_{\ell_q}+z_{\ell_q})
				+(\tilde b_i\times\nabla)(\ve_i+z_{\ell_q})^T\big).
			\end{aligned}
		\end{equation}	
		Since $ \Delta^{-1}\curl \div$, $\Delta^{-1} \nabla \div$, and $\Delta^{-1} \curl \div$ are bounded operators on H\"{o}lder space, we consider \eqref{eq:poission} and use \eqref{estimate Rl}, \eqref{estiamte:vi} and \eqref{estimate:vellzell} to obtain
		\begin{equation}\label{D(t,l )b N}
			\begin{aligned}
				\|\partial_t \tilde b_i+(v_{\ell_q}+z_{\ell_q})\cdot\nabla\tilde b_i\|_{C^{N+\alpha}_x}
			 \lesssim
				&(\|\ve_i\|_{C^{1+\alpha}_x}+\|v_{\ell_q}\|_{C^{1+\alpha}_x}+\|z_{\ell_q}\|_{C^{1+\alpha}_x})\|\tilde b_i\|_{C^{N+\alpha}_x}+\|\mathring{R}_{\ell_q}\|_{C^{N+\alpha}_x}
				\\&+(\|\ve_i\|_{C^{N+1+\alpha}}+\|v_{\ell_q}\|_{C^{N+1+\alpha}_x}+\|z_{\ell_q}\|_{C^{N+1+\alpha}_x})\|\tilde b_i\|_{C^{\alpha}_x}
				\\ \lesssim & \tau_q^{-1}\|\tilde b_i\|_{C^{N+\alpha}_x}+\tau_q^{-1}\ell_q^{-N}\|\tilde b_i\|_{C^{\alpha}_x}+\bar{M}L^2\delta_{q+1}\ell_q^{-N+\alpha}.
			\end{aligned}
		\end{equation}
		Since $\tilde{b}_i(\te_{i-1})=\mathcal{B}(\ve_{i}(\te_{i-1})-v_{\ell_q}(\te_{i-1}))=0$, taking $N=0$ in \eqref{D(t,l )b N} and using the estimate \eqref{eq:Gronwall1} we deduce for $t\in  [\te_{i-1},\te_{i+1}]\cap [0,\mathfrak{t}_L]$
		\begin{align*}
			\|\tilde{b}_i(t)\|_{C^{\alpha}_x}  \lesssim  \bar{M}L^2\tau_q \delta_{q+1}\ell_q^{\alpha} +  \tau_q^{-1} \int_{\te_{i-1}}^t \|\tilde b_i(s)\|_{C^{\alpha}_x}\, \dif s.
		\end{align*}
		Gr\"{o}nwall's inequality then gives for $t\in  [\te_{i},\te_{i+1}]\cap [0,\mathfrak{t}_L]$
		\begin{align}\label{e:b N=0}
			\|\tilde{b}_i(t)\|_{C^{\alpha}_x}\lesssim \bar{M}L^2\tau_q\delta_{q+1}\ell_q^{\alpha}.
		\end{align}
		The same estimate also holds for $\tilde{b}_{i+1}$. Hence, we achieve \eqref{e:b_diff} for any $N\geq 0$.
		Lastly, \eqref{e:b_diff_Dt} again follows from \eqref{e:N>=1}, \eqref{D(t,l )b N} and \eqref{e:b N=0}.
	\end{proof}
	\subsection{Gluing Procedure}\label{sec:gluing} In this subsection, we glue the exact solutions $\ve_{i}$ together to construct a new velocity field $\overline{v}_q$, and utilize the usual inverse divergence operator to construct the glued stress $\mathring{\overline{R}}_q$, whose temporal support is confined within pairwise disjoint intervals $I_i$ of length $\sfrac{\tau_q}{3}$. Building on the stability of $\ve_{i}$ established in Subsection~\ref{sec:esact solutions}, we then derive pathwise H\"{o}lder norms for the glued velocity field $\overline{v}_q$ and stress $\mathring{\overline{R}}_q$, demonstrating that $\overline{v}_q$ remains an approximate solution to the Euler system \eqref{euler1}. Lastly, we compute the energy difference between the original and glued velocity fields, with particular attention to the noise term $z_{\ell_{q}}$.
	
	\subsubsection{Definition of $\overline{v}_q$ and $\mathring{\overline{R}}_q$} Define the intervals $I_i, J_i$ $(i\geq 0)$ by
	\begin{align*}
		I_i:=[\te_i+\sfrac{\tau_q}{3},\te_i+\sfrac{2\tau_q}{3}]\cap[0,\mathfrak{t}_L], \qquad
		J_i:=(\te_i-\sfrac{\tau_q}{3},\te_i+\sfrac{\tau_q}{3})\cap  [0,\mathfrak{t}_L].
	\end{align*}
	Note that $\{I_i,J_i\}_{i\geq 0}$ is a decomposition of $[0,\mathfrak{t}_L]$ into pairwise disjoint intervals. Next, as in \cite[Section 4]{BDLSV19}, we define a partition of unity $\{\chi_i\}_{i\geq 0}$ in time with $\chi_i\in C^\infty_c(\R)$ and $0\leq \chi_i\leq1$ such that
	\begin{itemize}
		\item The cutoffs form a partition of unity $	\sum_i \chi_i \equiv 1$ on $[0,\mathfrak{t}_L]$;
		\item $	\supp \chi_i\subset I_{i-1} \cup J_i \cup I_i $, in particular $\supp \chi_i\cap \supp \chi_{i+2}=\emptyset$;
		\item $	\chi_i(t)=1$ for $t \in J_i $;
		\item $	\|{\partial_t^N \chi_i}\|_{C^0_t} \lesssim \tau_q^{-N} \label{e:dt:chi}$ for any $i$ and $N$.
	\end{itemize}
	Then we define glued velocity and pressure $(\overline v_q,\overline p_q^{(1)})$ by
	\begin{align}
		\label{eq:bar:v_q:def}
		\overline v_q(t,x):= \sum_i \chi_i(t) \ve_i(t,x)\,,  \qquad \overline p_q^{(1)}(t,x) =\sum_i \chi_i(t) \pe_i(t,x)
		\, .
	\end{align}
	Since the cutoffs $\chi_i$ only depend on time, the vector field $\overline v_q$ defined in \eqref{eq:bar:v_q:def} is divergence-free and $(\mathcal{F}_t)_{t\geq 0}$-adapted. Furthermore, if $t\in I_i$, then $\chi_i+\chi_{i+1}=1$ and $\chi_j=0$ for $j\neq i,i+1$, therefore on every $I_i$ interval we have
	\begin{align*}
		\overline v_q = \chi_i \ve_i + (1-\chi_i) \ve_{i+1} , \qquad  \overline p_q^{(1)} = \chi_i \pe_i + (1-\chi_i) \pe_{i+1} \,,
	\end{align*} 
	which also leads to
	\begin{equation}\label{computation:Rq}
		\aligned
		\partial_t \overline v_q&+\div((\overline v_q+z_{\ell_q}) \otimes(\overline v_q+z_{\ell_q}))+\nabla \overline p_q^{(1)}  
		\\=&\partial_t \chi_i(\ve_i- \ve_{i+1})  -\chi_i(1-\chi_i)\div((\ve_i -\ve_{i+1})\otimes (\ve_i -\ve_{i+1}))
		\\&+ \chi_i(\partial_t \ve_i+\div((\ve_i+z_{\ell_q})\otimes (\ve_i+z_{\ell_q}))+\nabla \pe_i)
		\\&+(1-\chi_i)(\partial_t \ve_{i+1}+\div((\ve_{i+1}+z_{\ell_q})\otimes (\ve_{i+1}+z_{\ell_q}))+\nabla \pe_{i+1})
		\\=&\partial_t \chi_i(\ve_i- \ve_{i+1})  -\chi_i(1-\chi_i)\div((\ve_i -\ve_{i+1})\otimes (\ve_i -\ve_{i+1})).
		\endaligned
	\end{equation}
	On the other hand, by the definition of the cutoff functions, on every $J_i$ interval we have 
	\begin{align*}
		\overline v_q = \ve_i , \qquad \overline p_q^{(1)} = \pe_i,
	\end{align*}
	therefore, on $\cup_i J_i$ we have 
	\begin{align*}
		\partial_t \overline v_q+\div((\overline v_q+z_{\ell_q}) \otimes(\overline v_q+z_{\ell_q}))+\nabla \overline p_q^{(1)}  =0.
	\end{align*}
	Using the inverse divergence operator $\mathcal{R}$ from Appendix~\ref{sec:appendix c}, for all $t \in I_i$ we thus define
	\begin{subequations}
		\begin{align}
			\mathring{\overline{R}}_q&=\partial_t\chi_i\mathcal{R}(\ve_i-\ve_{i+1})-\chi_i(1-\chi_i)(\ve_i-\ve_{i+1})\mathring{\otimes} (\ve_i-\ve_{i+1}) \, , \label{eq:bar:R_q:def}\\
			\overline{p}_q^{(2)}&=-\frac13 \chi_i(1-\chi_i)\left(|\ve_i-\ve_{i+1}|^2-\int_{\mathbb{T}^3}|\ve_i-\ve_{i+1}|^2\,\dif x\right) \, ,
		\end{align}
	\end{subequations}
	and $\mathring{\overline{R}}_q=0$, $\overline{p}_q^{(2)}=0$ for $t\notin \cup_i I_i$. Hence, $\mathring{\overline{R}}_q \in \mathcal{S}^{3\times 3}_0$ and $\supp \mathring{\overline{R}}_q \subset \cup_iI_i\times \mathbb{T}^3$.   
	We set $\overline{p}_q=\overline{p}_q^{(1)}+\overline{p}_q^{(2)}$, it follows from the preceding discussion that for all $(t,x)\in [0,\mathfrak{t}_L]\times \mathbb{T}^3$,
	\begin{equation}\label{gluing euler q}
		\aligned
		\partial_t \overline v_q+\div((\overline v_q+z_{\ell_q}) \otimes(\overline v_q+z_{\ell_q}))+\nabla \overline{p}_q&=\div \mathring{\overline{R}}_q,
		\\ \div \overline v_q &=0.
		\endaligned
	\end{equation}
	In addition, $\mathring{\overline{R}}_q$ is $(\mathcal{F}_t)_{t\geq 0}$-adapted.
	
	\subsubsection{Estimates on $\overline{v}_q$ and $ \mathring{\overline{R}}_q$} 
	We are now in a position to estimate the various H\"{o}lder norms of the glued velocity field $\overline{v}_q$ and stress $\mathring{\overline{R}}_q$.
	\begin{proposition}\label{p:vq:vell}
		The glued velocity field $\overline{v}_q$ enjoys the following bounds for any $N \geq 0$ and $t\in[0,\mathfrak{t}_L]$
		\begin{subequations}
			 	\begin{align}
			 	\|\overline v_q - v_{{\ell_q}}\|_{C^\alpha_x} &\lesssim \bar{M}L \delta_{q+1}^{\sfrac12}\ell_q^{\alpha} \label{e:vq:vell}, \\
			 	\|\overline{v}_q-v_{\ell_q}\|_{C^{1+N+\alpha}_x} &\lesssim \bar{M}\tau_q^{-1}\ell_q^{-N+\alpha}\, ,  \label{e:vq:vell:additional} 
			 	\\	\|\overline v_q\|_{C^{1+N}_x} &\lesssim \bar{M} \tau_q^{-1}\ell_q^{-N+\alpha},\label{e:vq:1+N} 
			 	\\ \|\overline{v}_q\|_{C^0_x}& \lesssim \bar{M}L \lambda_{1}^{\sfrac{3\alpha}{2}},\label{e:vq:0}
			 \end{align}
		\end{subequations}
		where the implicit constant may depend on $N$ and $\alpha$.
	\end{proposition}
	\begin{proof}
		Recalling the definition of $\overline{v}_q$ we have
		$$
		\overline{v}_q-v_{\ell_q}=\sum_i\chi_i(\ve_i-v_{\ell_q}).
		$$
		Then using \eqref{v i- v l} with $N=0$ implies 
		\begin{equation*}
			\|\overline{v}_q-v_{\ell_q}\|_{C^{\alpha}_x}\lesssim \bar{M}L^2 (\tau_q\delta_{q+1}^{\sfrac12}\ell_q^{-1})\delta_{q+1}^{\sfrac12}\ell_q^{\alpha}\leq \bar{M}L\delta_{q+1}^{\sfrac12}\ell_q^{\alpha}.
		\end{equation*}
		Furthermore, from \eqref{v i- v l} we find for any $N\geq 0$
		\begin{align*}
			\|\overline{v}_q-v_{\ell_q}\|_{C^{1+N+\alpha}_x}&\lesssim \bar{M}L^2 (\tau_q^2\delta_{q+1}\ell_q^{-2})\tau_q^{-1} \ell_q^{-N+\alpha}= \bar{M}\tau_q^{-1} \ell_q^{-N+\alpha},
		\end{align*}
		which implies \eqref{e:vq:vell:additional}. Thus, we combine \eqref{itera:b} and \eqref{e:vq:vell:additional} to obtain for any $N\geq0$
		\begin{align*}
			\|\overline{v}_q\|_{C^{1+N}_x}\lesssim \|v_{\ell_q}\|_{C^{1+N}_x}+\|v_{\ell_q}-\overline{v}_q\|_{C^{1+N+\alpha}_x}
			\lesssim \bar{M}\tau_q^{-1}\ell_q^{-N+\alpha}.
		\end{align*}
		Finally, we use \eqref{itera:a} and \eqref{e:vq:vell} to derive 
		\begin{align*}
			\|\overline{v}_q\|_{C^0_x}\leq  \|v_{\ell_q}\|_{C^0_x}+\|\overline{v}_q-v_{\ell_q}\|_{C^\alpha_x} \lesssim \bar{M}L\lambda_{1}^{\sfrac{3\alpha}{2}}+\bar{M}L\delta_{q+1}^{\sfrac12} \ell_q^{\alpha}\lesssim \bar{M}L \lambda_{1}^{\sfrac{3\alpha}{2}},
		\end{align*} 
		where the last inequality relies on $\delta_{q+1}^{\sfrac12}\leq \lambda_{1}^{\sfrac{3\alpha}{2}}$.
	\end{proof}
	
	\begin{proposition}\label{p:Rq}
		The glued stress $\mathring{\overline R}_q$ obeys the following bounds for any $t\in[0,\mathfrak{t}_L]$ and $N\in \N_0$
	\begin{subequations}
			\begin{align} 
			\|\mathring{\overline R}_q\|_{C^{N+\alpha}_x} &\lesssim \bar{M}L^2 \delta_{q+1}\ell_q^{-N+\alpha} \label{e:Rq:1}\, ,
			\\ \|(\partial_t +( \overline v_q+z_{\ell_q})\cdot \nabla) \mathring{\overline R}_q\|_{C^{N+\alpha}_x} &\lesssim \bar{M}L^2 \tau_q^{-1}\delta_{q+1}\ell_q^{-N+\alpha}\, , \label{e:Rq:Dt}
		\end{align}
	\end{subequations}
		 where the implicit constant may depend on $N$ and $\alpha$.
	\end{proposition}
	
	\begin{proof}
		Note that 
		$\curl (b_{i+1}-b_i)=\ve_{i+1}-\ve_i$  for $t\in I_i$, so that we may write:
		$$
		\mathring{\overline{R}}_q=\partial_t\chi_i(\mathcal{R}\curl)(b_i-b_{i+1})-\chi_i(1-\chi_i)(\ve_i-\ve_{i+1})\mathring\otimes (\ve_i-\ve_{i+1}).
		$$
		Since $\mathcal{R}\curl$ is a zero-order operator and bounded on H\"{o}lder space, it follows from \eqref{v i- v l} and \eqref{e:b_diff} that for any $N\in \N_0$ with $t\in I_i$
		\begin{align*}
			\| \mathring{\overline{R}}_q\|_{C^{N+\alpha}_x}&\lesssim  \tau_q^{-1} \|b_i-b_{i+1}\|_{C^{N+\alpha}_x}+\|\ve_i-\ve_{i+1}\|_{C^{N+\alpha}_x}\|\ve_i-\ve_{i+1}\|_{C^\alpha_x}
			\\&\lesssim \bar{M}L^2 \delta_{q+1}\ell_q^{-N+\alpha}+(\bar{M}^2L^4 \tau_q^2\delta_{q+1}\ell_q^{-2+\alpha})\delta_{q+1} \ell_q^{-N+\alpha} \lesssim \bar{M}L^2 \delta_{q+1}\ell_q^{-N+\alpha},
		\end{align*}
		where the last line is justified by taking $a$ large enough so that $\bar{M}^2L^4\tau_q^2\delta_{q+1}\ell_q^{-2+\alpha}= \bar{M}^2L^2 \ell_q^{\alpha}\leq L^2$. By computing explicitly the material derivative $D_{t,\ell_q}$ to $\mathring{\overline{R}}_q$, we arrive at
		\begin{align*}
			D_{t,\ell_q}\mathring{\overline{R}}_q
			&=\partial_t^2\chi_i(\mathcal{R}\curl)(b_i-b_{i+1})\\
			&\quad +\partial_t\chi_i(\mathcal{R}\curl)D_{t,{\ell_q}}(b_i-b_{i+1})+\partial_t \chi_i[(v_{\ell_q}+z_{\ell_q})\cdot\nabla,\mathcal{R}\curl](b_i-b_{i+1})\\
			&\quad -\partial_t(\chi_i(1-\chi_i))(\ve_i-\ve_{i+1})\mathring\otimes (\ve_i-\ve_{i+1})\\
			&\quad -\chi_i(1-\chi_i)\Bigl((D_{t,{\ell_q}}(\ve_i-\ve_{i+1}))\mathring\otimes (\ve_i-\ve_{i+1})+(\ve_i-\ve_{i+1})\mathring\otimes (D_{t,{\ell_q}}(\ve_i-\ve_{i+1}))\Bigr),
		\end{align*}
		where $[v\cdot\nabla,\mathcal{R}\curl]$ denotes the commutator. 
		We then apply the commutator estimate \eqref{estimate:molli2} with \eqref{v i- v l}, \eqref{Dtl v i- v l}, \eqref{estimate:vellzell}, \eqref{e:b_diff}, \eqref{e:b_diff_Dt} to obtain for any $N\geq0$ and $t\in I_i$
		\begin{align*}
			\|D_{t,{\ell_q}}\mathring{\overline{R}}_q\|_{C^{N+\alpha}_x}&\lesssim \tau_q^{-2}\|b_i-b_{i+1}\|_{C^{N+\alpha}_x}+ \tau_q^{-1}\|D_{t,{\ell_q}}(b_i-b_{i+1})\|_{C^{N+\alpha}_x}
			\\&\quad +\tau_q^{-1}\|v_{\ell_q}+z_{\ell_q}\|_{C^{1+\alpha}_x}\|b_i-b_{i+1}\|_{C^{N+\alpha}_x}+\tau_q^{-1} \|v_{\ell_q}+z_{\ell_q}\|_{C^{N+1+\alpha}_x}\|b_i-b_{i+1}\|_{C^{\alpha}_x}
			\\&\quad + \tau_q^{-1} \|\ve_i-\ve_{i+1}\|_{C^{N+\alpha}_x}\|\ve_i-\ve_{i+1}\|_{C^{\alpha}_x}
			\\&\quad+ \|D_{t,{\ell_q}}(\ve_i-\ve_{i+1})\|_{C^{N+\alpha}_x}\|\ve_i-\ve_{i+1}\|_{C^{\alpha}_x}+\|D_{t,{\ell_q}}(\ve_i-\ve_{i+1})\|_{C^{\alpha}_x}\|\ve_i-\ve_{i+1}\|_{C^{N+\alpha}_x}
			\\ & \lesssim \bar{M}L^2 \tau_q^{-1}\delta_{q+1}\ell_q^{-N+\alpha}+\bar{M}^2 L^2\tau_q^{-1}\delta_{q+1}\ell_q^{-N+2\alpha}
			+ (\bar{M}^2L^4\tau_q^2\delta_{q+1}\ell_{q}^{-2+\alpha}) \tau_q^{-1}\delta_{q+1}\ell_{q}^{-N+\alpha} 
			\\ &
		\lesssim  \bar{M}L^2 \tau_q^{-1}\delta_{q+1}\ell_q^{-N+\alpha},
		\end{align*}
		where we used again $\bar{M}^2L^4\tau_q^2\delta_{q+1}\ell_q^{-2+\alpha}\leq L^2$ in the last line.
		Finally, combining \eqref{e:vq:vell:additional}, \eqref{e:Rq:1} and the above estimate yields
		\begin{align*}
			\|(\partial_t +( \overline v_q+z_{\ell_q})\cdot \nabla) \mathring{\overline R}_q\|_{C^{N+\alpha}_x} &\lesssim \|D_{t,{\ell_q}}\mathring{\overline{R}}_q\|_{C^{N+\alpha}_x}+\|(\overline v_q-v_{\ell_q})\cdot \nabla \mathring{\overline R}_q\|_{C^{N+\alpha}_x}
			\\ & \lesssim  \|D_{t,{\ell_q}}\mathring{\overline{R}}_q\|_{C^{N+\alpha}_x}+ \|\overline{v}_q-v_{\ell_q}\|_{C^{N+\alpha}_x}\|\mathring{\overline R}_q\|_{C^{1+\alpha}_x}
			\\&\qquad +\|\overline{v}_q-v_{\ell_q}\|_{C^{\alpha}_x}\|\mathring{\overline R}_q\|_{C^{N+1+\alpha}_x}
			\\  & \lesssim  \bar{M}L^2 \tau_q^{-1}\delta_{q+1}\ell_q^{-N+\alpha}+ ( \bar{M}^2L^4\tau_q^2\delta_{q+1}\ell_q^{-2+\alpha}) \tau_q^{-1}\delta_{q+1}\ell_q^{-N+\alpha}
			\\ & \lesssim  \bar{M}L^2 \tau_q^{-1}\delta_{q+1}\ell_q^{-N+\alpha},
		\end{align*}
		which gives \eqref{e:Rq:Dt}.
	\end{proof}
	
	\subsubsection{Estimates on the energy of glued velocity} 
	To finish this section, let us show that the energy of $\overline{v}_q+z_{\ell_q}$ approximates that of $v_{\ell_{q}}+z_{\ell_q}$.
	\begin{proposition}\label{estimate:energy2}
		For any $t
		\in  [0,\mathfrak{t}_{L}]$, the difference of the energies between $\overline{v}_q+z_{\ell_q}$ and $v_{\ell_q}+z_{\ell_q}$ enjoys the following bound
		\begin{align}\label{energy difference1}
			\left| \int_{\T^3} 	\left| \overline{v}_{q}+z_{\ell_q}\right|^2 -  \left|v_{\ell_q} +z_{\ell_q}\right|^2  \dif x\right| \leq \frac12 L^2 \delta_{q+1}\ell_{q}^{\alpha}.
		\end{align}
	\end{proposition}
		\begin{proof}
		We first observe on $[\te_{i},\te_{i+1}]\cap [0,\mathfrak{t}_{L}]$,
		\begin{equation}\label{energy:barvq+zl-vl-zl}
			\aligned
			\left|\overline{v}_{q}+z_{\ell_q} \right|^2-& \left|v_{\ell_q} +z_{\ell_q} \right|^2 = \chi_i (\left|\ve_i \right|^2-\left|v_{\ell_q} \right|^2 )
			+ (1-\chi_i) (\left|\ve_{i+1} \right|^2-\left|v_{\ell_q}  \right|^2 )
			\\ &- \chi_i(1-\chi_i)\left| \ve_i-\ve_{i+1}\right|^2+ 2\chi_i(\ve_{i}-v_{\ell_{q}})\cdot z_{\ell_{q}} +2(1-\chi_i)(\ve_{i+1}-v_{\ell_{q}})\cdot z_{\ell_{q}} .
			\endaligned
		\end{equation}
	We will estimate each term in \eqref{energy:barvq+zl-vl-zl} separately. First, by multiplying both sides of equation \eqref{mollification} by $v_{\ell_{q}}$, and subsequently computing their inner products, we obtain
			\begin{align}\label{inner product1}
			\frac12  \frac{\dif}{\dif t} \int_{\T^3}  |v_{\ell_{q}}|^2 \dif x  = & - \int_{\T^3} v_{\ell_{q}} \cdot  \left( (v_{\ell_{q}}+z_{{\ell_q}}) \cdot \nabla (v_{\ell_{q}}+z_{{\ell_q}})\right)   \dif x  -  \int_{\T^3}  v_{\ell_{q}} \cdot \nabla \pe_{\ell_{q}}  \dif x  +  \int_{\T^3}   v_{\ell_{q}} \cdot \div \mathring{R}_{\ell_{q}} \dif x\nonumber
			\\ =&- \int_{\T^3}  (v_{\ell_{q}}+z_{{\ell_q}})  \cdot  \left( (v_{\ell_{q}}+z_{{\ell_q}}) \cdot \nabla (v_{\ell_{q}}+z_{{\ell_q}})\right)   \dif x    -  \int_{\T^3}  v_{\ell_{q}} \cdot \nabla \pe_{\ell_{q}}  \dif x  
			\\ &+ \int_{\T^3} z_{{\ell_q}} \cdot  \left( (v_{\ell_{q}}+z_{{\ell_q}}) \cdot \nabla (v_{\ell_{q}}+z_{{\ell_q}})\right)   \dif x +\int_{\T^3}   v_{\ell_{q}} \cdot \div \mathring{R}_{\ell_{q}} \dif x.\nonumber
		\end{align} 
		 Continuing in this manner to equation \eqref{eq:euler exact1}, we reach
		\begin{align}\label{inner product2}
			 \frac12  \frac{\dif}{\dif t} \int_{\T^3}  |\ve_i|^2 \dif x  = & - \int_{\T^3} \ve_i \cdot  \left( (\ve_i+z_{{\ell_q}}) \cdot \nabla (\ve_i+z_{{\ell_q}})\right)   \dif x  -  \int_{\T^3}  \ve_i \cdot \nabla \pe_i   \dif x  \nonumber
			 \\ =&- \int_{\T^3}  (\ve_i+z_{{\ell_q}})  \cdot  \left( (\ve_i+z_{{\ell_q}}) \cdot \nabla (\ve_i+z_{{\ell_q}})\right)   \dif x    -  \int_{\T^3}  \ve_i \cdot \nabla \pe_i   \dif x  
			 \\ &+ \int_{\T^3} z_{{\ell_q}} \cdot  \left( (\ve_i+z_{{\ell_q}}) \cdot \nabla (\ve_i+z_{{\ell_q}})\right)   \dif x. \nonumber
		\end{align} 
			Observe that the nonlinear term in the second line of \eqref{inner product1} and \eqref{inner product2} vanishes due to the divergence-free property of $v_{{\ell_q}}$, $ \ve_{i}$ and $z_{{\ell_q}}$, combined with integration by parts. We then combine  \eqref{inner product1}, \eqref{inner product2}, integration by parts and $v_{{\ell_q}}(\te_{i-1})=  \ve_{i}(\te_{i-1})$ to derive for any $t\in [\te_{i-1},\te_{i+1}]\cap [0,\mathfrak{t}_{L}]$
		\begin{equation}\label{eq:vl-zl}
			\aligned
			\int_{\T^3} \left| \ve_i(t) \right|^2 &  \dif x  -   \int_{\T^3}  \left| v_{\ell_{q}}(t) \right|^2 \dif x =2 \int_{\te_{i-1}}^{t } \int_{\T^3} \mathring{R}_{{\ell_q}} : \nabla (v_{\ell_q})^T \dif x \dif s
			\\ &+
			2 \int_{\te_{i-1}}^{t } \int_{\T^3} (v_{{\ell_q}}+z_{{\ell_q}}) \cdot \left( (v_{{\ell_q}}+z_{{\ell_q}})\cdot \nabla z_{{\ell_q}} \right) - (\ve_{i}+z_{{\ell_q}}) \cdot \left( (\ve_{i}+z_{{\ell_q}})\cdot \nabla z_{{\ell_q}} \right) \dif x \dif s.
			\endaligned
		\end{equation}
		To control \eqref{eq:vl-zl}, we first employ \eqref{itera:b}, \eqref{estimate Rl} and \eqref{e:Rq:1} to deduce for any $t\in [\te_{i-1},\te_{i+1}]\cap [0,\mathfrak{t}_{L}]$
		\begin{align*}
			2 \left| \int_{\te_{i-1}}^{t } \int_{\T^3} \mathring{R}_{{\ell_q}} : (\nabla v_{\ell_q})^T  \dif x \dif s\right| & \lesssim (t-\te_{i-1}) \|\mathring{R}_{{\ell_q}} \|_{C_{[0,\mathfrak{t}_L]}C^0_x} \|  v_{\ell_q}\|_{C_{[0,\mathfrak{t}_L]}C^1_x}
			\\ &\lesssim \bar{M}^2L^3 \tau_{q}\delta_{q+1} \lambda_{q}\delta_{q}^{\sfrac12} \lambda_{q}^{\alpha}\leq \frac{1}{18}L^2 \delta_{q+1}\ell_{q}^{\alpha},
		\end{align*} 
		which requires $\bar{M}^2\lambda_{q}^{-\alpha}\leq 1$ in the last inequality.
		We then combine \eqref{v i- v l}, \eqref{estimate:vellzell} and \eqref{e:vq:0} to estimate the second term in \eqref{eq:vl-zl} as
		\begin{align*}
			&2 \left| \int_{\te_{i-1}}^{t } \int_{\T^3} (v_{{\ell_q}}+z_{{\ell_q}}) \cdot \left( (v_{{\ell_q}}+z_{{\ell_q}})\cdot \nabla z_{{\ell_q}} \right) - (\ve_{i}+z_{{\ell_q}}) \cdot \left( (\ve_{i}+z_{{\ell_q}})\cdot \nabla z_{{\ell_q}} \right) \dif x \dif s\right|
			\\ & \quad \lesssim \int_{\te_{i-1}}^{t } \int_{\T^3} \left| (v_{{\ell_q}}-\ve_i) \cdot \left( (v_{{\ell_q}}+z_{{\ell_q}})\cdot \nabla z_{{\ell_q}} \right) \right| +  \left|     (\ve_{i}+z_{{\ell_q}}) \cdot \left( (v_{{\ell_q}}- \ve_{i})\cdot \nabla z_{{\ell_q}} \right)  \right|   \dif x\dif s
			\\ & \quad \lesssim (t-\te_i) \| v_{{\ell_q}}- \ve_{i}\|_{C_{[0,\mathfrak{t}_L]}C^0_x} \| \nabla z_{{\ell_q}} \|_{C_{[0,\mathfrak{t}_L]}C^0_x} \left( \| v_{{\ell_q}}+ z_{{\ell_q}}\|_{C_{[0,\mathfrak{t}_L]}C^0_x} + \| \ve_{i}+ z_{{\ell_q}}\|_{C_{[0,\mathfrak{t}_L]}C^0_x}  \right)
			\\ & \quad \lesssim \bar{M}^2 L^4  \tau_{q}^2 \delta_{q+1}\ell_{q}^{-1} \lambda_{1}^{2\alpha} \leq  \bar{M}^2 L^2 (\lambda_{q}\delta_{q}^{\sfrac12} )^{-1} \lambda_{q}^{-2\alpha} \delta_{q+1}^{\sfrac12}   \leq \frac{1}{18}L^2 \delta_{q+1}\ell_{q}^{\alpha},
		\end{align*}
		where we used $b\beta+\beta<1$ to have $(\lambda_{q}\delta_{q}^{\sfrac12})^{-1}\leq \delta_{q+1}^{\sfrac12} $ in the last inequality. The same estimate also applies to the second term on the right-hand side of \eqref{energy:barvq+zl-vl-zl} .

		Moving to the third term of \eqref{energy:barvq+zl-vl-zl}, we use \eqref{v i- v l} to obtain for $t\in  [\te_{i},\te_{i+1}]\cap [0,\mathfrak{t}_{L}]$
		\begin{align*}
			\| \ve_{i}(t)-\ve_{i+1}(t)\|_{L^2}^2 
			\lesssim	\| \ve_{i}(t)-v_{\ell_{q}}(t)\|_{C^0_x}^2  +	\| \ve_{i+1}(t)-v_{\ell_{q}}(t)\|_{C^0_x}^2 
				\leq  (\bar{M}L^2 \tau_q \delta_{q+1}\ell_{q}^{-1+\alpha})^2 \leq \frac{1}{18} L^2 \delta_{q+1}\ell_{q}^{\alpha}.
		\end{align*}
		Considering the remaining terms in \eqref{energy:barvq+zl-vl-zl}, it follows from $\ve_{i}-v_{\ell_{q}}= \curl \tilde{b}_i $, integration by parts, \eqref{estimate:vellzell} and \eqref{e:b N=0} that for any $t \in [\te_{i},\te_{i+1}]\cap [0,\mathfrak{t}_L]$
		\begin{align*}
			\left|  \int_{\T^3} (\ve_{i}-v_{\ell_{q}})\cdot z_{\ell_{q}} \dif x   \right|=
		 \left|\int_{\T^3} \curl \tilde{b}_i  \cdot z_{\ell_{q}} \dif x\right| 
			=\left|\int_{\T^3} \tilde{b}_i \cdot  \curl  z_{\ell_{q}} \dif x\right| \lesssim \bar{M} L^3 \tau_{q}\delta_{q+1}\ell_{q}^{\alpha} \leq\frac{1}{18}L^2 \delta_{q+1}\ell_{q}^{\alpha}.
		\end{align*}
		The same estimate also applies to $ \int_{\T^3} (\ve_{i+1}-v_{\ell_{q}})\cdot z_{\ell_{q}} \dif x $.		
		Let us recollect the above estimates to obtain for any $t \in [\te_{i},\te_{i+1}]\cap [0,\mathfrak{t}_L]$
		\begin{align*}
		&\left|\int_{\T^3} 	\left|\overline{v}_{q}+z_{\ell_q} \right|^2- \left|v_{\ell_q} +z_{\ell_q} \right|^2 \dif x\right| 
		\leq  \left| \int_{\T^3} \left|\ve_i \right|^2-\left|v_{\ell_q} \right|^2 \dif x \right|
		+ \left| \int_{\T^3} \left|\ve_{i+1} \right|^2-\left|v_{\ell_q}  \right|^2 \dif x \right|
		\\ &\qquad \qquad +  \int_{\T^3} 	\left| \ve_i-\ve_{i+1}\right|^2 \dif x+  \left| \int_{\T^3} 2\chi_i(\ve_{i}-v_{\ell_{q}})\cdot z_{\ell_{q}} +2(1-\chi_i)(\ve_{i+1}-v_{\ell_{q}})\cdot z_{\ell_{q}} \dif x \right| \leq \frac12 L^2 \delta_{q+1}\ell_{q}^{\alpha},
		\end{align*}
		which further implies \eqref{energy difference1} for any $t\in [0,\mathfrak{t}_L]$. Hence, we conclude the proof of Proposition~\ref{estimate:energy2}. 
	\end{proof}
	Finally, we need to control $\int_{0}^{t} \int_{\T^3} (\overline{v}_q +z_{\ell_{q}}) \cdot \partial_t z_{\ell_{q}} - (v_{\ell_{q}}+z_{\ell_{q}}) \cdot \partial_t z_{\ell_{q}}  \dif x \dif s$, which will be used in Subsection~\ref{sec:cutoffs flow map} to estimate the energy gaps.
	\begin{proposition}\label{Prop:addition energy}
		 For any $t
		 \in  [0,\mathfrak{t}_{L}]$, we have:
		 \begin{align} \label{energy difference2}
		 	\left| \int_{0}^{t} \int_{\T^3} (\overline{v}_q +z_{\ell_{q}}) \cdot \partial_t z_{\ell_{q}} \dif x \dif s - \int_{0}^{t} \int_{\T^3} (v_{\ell_{q}}+z_{\ell_{q}}) \cdot \partial_t z_{\ell_{q}}  \dif x \dif s \right|\leq \frac14 L^2\delta_{q+1}\ell_{q}^{\alpha}.
		 \end{align}
	\end{proposition}
	\begin{proof}
		 	Recall that $\ve_{i}-v_{\ell_{q}}= \curl \tilde{b}_i $ on $[\te_{i},\te_{i+1}]\cap [0,\mathfrak{t}_L]$. By combining \eqref{zq+1-zq}, \eqref{e:b N=0} and mollification estimate \eqref{estimate:molli2}, we obtain for any $t \in [\te_{i},\te_{i+1}]\cap [0,\mathfrak{t}_L]$	
		 \begin{equation}\label{energy:add1}
		 	\aligned
		 	&\left|  \int_{\T^3} (\overline{v}_q +z_{\ell_{q}}) \cdot \partial_t z_{\ell_{q}} \dif x \dif s - \int_{\T^3} (v_{\ell_{q}}+z_{\ell_{q}}) \cdot \partial_t z_{\ell_{q}}  \dif x \right| 
		 	\\ &=\left|  \int_{\T^3}  \chi_i(\ve_{i}-v_{\ell_{q}}) \cdot \partial_t  z_{\ell_{q}} + (1- \chi_i)(\ve_{i+1}-v_{\ell_{q}}) \cdot \partial_t  z_{\ell_{q}}  \dif x \right|
		 	\\ &= \left|  \int_{\T^3}  \chi_i \tilde{b}_i \cdot \partial_t \curl  z_{\ell_{q}} + (1- \chi_i)\tilde{b}_{i+1} \cdot \partial_t  \curl  z_{\ell_{q}}  \dif x  \right|
		 	\\ & \lesssim \iota_q^{-(\sfrac12+\alpha)} \|B\|_{C^{\sfrac12-\alpha}_{[0,\mathfrak{t}_L ]}C^1_x}\left(\|\tilde{b}_i\|_{C_{[0,\mathfrak{t}_L]}C^0_x} +  \|\tilde{b}_{i+1}\|_{C_{[0,\mathfrak{t}_L]}C^0_x}\right) \lesssim \bar{M} L^3 \lambda_{q}^{\sfrac23+2\alpha} \tau_{q} \delta_{q+1}\ell_{q}^{\alpha} \leq \frac14 L\delta_{q+1}\ell_{q}^{\alpha},
		 	\endaligned
		 \end{equation}
		 which requires the condition $\bar{M}L\lambda_{q}^{-\alpha}\ll1$ in the last inequality. Indeed, the estimate \eqref{energy:add1} holds for any $ t\in  [0,\mathfrak{t}_L]$. Therefore, the estimate \eqref{energy difference2} follows from multiplying \eqref{energy:add1} by $L$.
	\end{proof}

	\section{Proof of Proposition~\ref{p:iteration}---Step 2: Construction of the perturbation}\label{sec:perturbation}
	In this section, we construct the new perturbation and the associated stress error term. Before delving into the details, we introduce the flow maps $\Phi_i$ and key cutoff functions $\eta_i$, which play a crucial role in canceling the glued stress $\mathring{\overline R}_q$ and controlling the energy, as discussed in Subsection~\ref{sec:cutoffs flow map}. Additionally, we define the energy gap decomposition $\rho_{q,i}$ in this subsection to prescribe the energy profile. The amplitudes are introduced in Subsection~\ref{sec:amplitude}, along with relevant pathwise estimates. With these foundations in place, we construct the new velocity $v_{q+1}$ in Subsection~\ref{sec:principal w}. To achieve the optimal regularity, we employ the Mikado flows from \cite{BV19} (see also \cite{DS17}) to construct the perturbation $w_{q+1}$. After defining $v_{q+1}:=\overline{v}_q+w_{q+1} $, we immediately derive the term $\div \mathring{R}_{q+1}$, as $(v_{q+1},\mathring{R}_{q+1})$ must satisfy \eqref{euler1}. Finally, we decompose $\mathring{R}_{q+1}$ into distinct error parts in Subsection~\ref{sec: new Reynolds stress}.

	\subsection{Cutoffs, Energy gap decomposition and Flow maps}\label{sec:cutoffs flow map}
	We start by introducing `squiggling' space-time cutoffs constructed in \cite[Section 5.2]{BDLSV19}, which we recall in Appendix~\ref{sec:appendix c}. It follows from \cite[Section 5.2]{BDLSV19} that the cutoff functions $\eta_i$  satisfy the following properties: 
	\begin{enumerate}[(i)]
		\item \label{eq:eta:i:support1}$\eta_i \in C^{\infty}([0,\mathfrak{t}_{L}]\times \mathbb{T}^3;[0,1])$, and $\eta_i \, \eta_j \equiv 0$ for every  $i \neq j$,
		\item $\supp \eta_i \subset  \tilde{I_i} \times  \T^3 $, where $\tilde I_i := J_i \cup I_i \cup J_{i+1} $, and $\eta_i \equiv 1$ on $I_i  \times  \T^3 $,
		\item \label{eq:eta:i:special} There exists a universal constant $c_\eta>0$ independent of $q$ (one can take $c_\eta=1/5$) such that for all $t\in[0,\mathfrak{t}_L]$, $c_\eta \leq \sum_i \int_{\T^3} \eta_i^2(t,x) \dif x \leq 1$,
		\item  \label{eq:eta:i:derivative}$\|{\partial_t^n \eta_i}\|_{C^m_x} \lesssim_{n,m} \tau_q^{-n}$,  for all $n,m\geq 0$.
	\end{enumerate}
	
	To prescribe the energy profile, we first define the energy gap
	\begin{align}\label{def:energy gap}
		\rho_q(t)=\frac13  \left[e(t)-L^2 \frac{\delta_{q+2}}{2}-  \| (\overline{v}_q+z_{\ell_q}) (t )\|_{L^2}^2 +2 \int_{0}^{t} \int_{\T^3} (\overline{v}_q+z_{\ell_q}) \cdot \partial_t z_{\ell_{q}} \dif x \dif s  \right],
	\end{align}
	and decompose $\rho_q$ by setting
	\begin{align}\label{def:rho q,i}
		\rho_{q,i}(t,x):=\frac{\eta_i^2(t,x)}{\sum \int_{\T^3} \eta_j^2(t,y) \dif y} \rho_q(t) .
	\end{align}
	\begin{remark}
		 	\textit{Here, the energy gap \eqref{def:energy gap} differs from its counterpart in \cite[Page 252]{BDLSV19}. This distinction arises naturally, as evidenced by the proof of \eqref{estimate:rho3} below, where the temporal derivative of the term $\int_{\T^3} |\overline{v}_q+z_{\ell_q}|^2 \dif x$ cannot be absorbed. To address this, we introduce an additional term involving $\partial_t z_{\ell_q}$ in \eqref{def:energy gap} to eliminate it. This modification also necessitates adding the term $\int_{0}^{t} \int_{\T^3} (v_q+z_q) \cdot \partial_t z_q \dif x \dif s $ to the energy iterative estimate \eqref{estimate:energy}, since the bound of \eqref{def:energy gap} essentially depends on \eqref{estimate:energy}.}
	\end{remark}
	From the construction, it follows $\sum_i \int_{\T^3} \rho_{q,i} =\rho_q$ for all $t\in [0,\mathfrak{t}_L]$. Furthermore, we have the following estimates for $\rho_q$ and $\rho_{q,i}$:
	\begin{proposition}\label{prop:rho}
		The energy gap $\rho_q$ and $\rho_{q,i}$ enjoy the following bounds for any $t\in [0,\mathfrak{t}_L]$ and $N\in \N_0$
		\begin{subequations}
				\begin{align}
				\frac{L^2\delta_{q+1}}{6\lambda_{q}^\alpha}\leq \rho_q(t)  &\leq L^2 \delta_{q+1}, \label{estimate:rho1}
				\\ \|\rho_{q,i}(t)\|_{C^N_x} &\lesssim L^2 \delta_{q+1},\label{estimate:rho2}
				\\ \|\partial_t \rho_{q,i}\|_{C^N_x} &\lesssim L^2 \tau_q^{-1}\delta_{q+1} ,\label{estimate:rho3}
			\end{align}
		\end{subequations}
		where the implicit constant may depend on $N$.
	\end{proposition}
	\begin{proof}
		Observe that the estimate \eqref{estimate:rho1} is an apparent consequence of \eqref{estimate:energy}, \eqref{estimate energy vq-vl}, \eqref{estimate energy vq-vl cdot zlq}, \eqref{energy difference2} and Proposition~\ref{estimate:energy2}. More precisely, we have
		\begin{equation}
			\aligned
			\frac{L^2\delta_{q+1}}{2\lambda_{q}^{\alpha}} &\leq L^2 \delta_{q+1} \lambda_{q}^{-\alpha}-3 L^2 \delta_{q+1}\ell_{q}^{\alpha} \leq L^2 \delta_{q+1} \lambda_{q}^{-\alpha}-L^2\frac{\delta_{q+2}}{2}-2L^2\delta_{q+1} \ell_{q}^{\alpha} \leq 3 \rho_q(t)
			\\&= e(t)-L^2 \frac{\delta_{q+2}}{2}- \| (v_q+z_q) (t)\|_{L^2}^2  + 2\int_{0}^{t} \int_{\T^3} (v_q+z_q) \cdot \partial_t z_{q} \dif x \dif s
		\\ &\quad + \left( \| (v_q+z_q) (t)\|_{L^2}^2- \| (v_{\ell_q}+z_{\ell_q}) (t)\|_{L^2}^2\right)
		+\left( \| (v_{\ell_q}+z_{\ell_q}) (t)\|_{L^2}^2- \| (\overline{v}_q+z_{\ell_q}) (t)\|_{L^2}^2 \right)
		\\ &\quad +2\int_{0}^{t} \int_{\T^3} \left(( v_{\ell_{q}}+z_{\ell_{q}})\cdot \partial_t z_{\ell_{q}} - (v_{q}+z_q) \cdot \partial_t z_{q} \right) \dif x \dif s  
		\\ &\quad +2\int_{0}^{t} \int_{\T^3}  \left( (\overline{v}_q +z_{\ell_{q}})\cdot \partial_t z_{\ell_{q}} -  (v_{\ell_{q}}+z_{\ell_{q}})\cdot \partial_t z_{\ell_{q}}\right) \dif x \dif s
			\\ & \leq L^2 \delta_{q+1}+2L^2\delta_{q+1} \ell_{q}^{\alpha}\leq 3L^2 \delta_{q+1}, 
			\endaligned
		\end{equation}
		where we used $\delta_{q+2}\leq 2\delta_{q+1}\ell_{q}^{\alpha}$ and $(\ell_{q}\lambda_{q})^{\alpha}\ll 1/6$ in the first line. By the Leibniz rule and properties \eqref{eq:eta:i:special} and \eqref{eq:eta:i:derivative} for the cutoff  functions $\eta_i$
		\begin{align}\label{property eta}
			\| \eta_i\|_{C^N_x}\lesssim 1, \qquad c_\eta \leq \sum_i \int_{\T^3} \eta_i^2(t,x) \dif x \leq 1\,,
		\end{align} 
		the bound \eqref{estimate:rho2} directly follows. Let us move to \eqref{estimate:rho3} now. Going back to \eqref{gluing euler q} we have
		\begin{align}\label{inner}
		\partial_t(\overline v_q+z_{\ell_{q}})=-(\overline v_q+z_{{\ell_q}}) \cdot \nabla(\overline v_q+z_{{\ell_q}})-\nabla \overline{p}_q +\partial_t z_{{\ell_q}}+\div \mathring{\overline{R}}_q.
		\end{align}
		Multiplying both sides of equation \eqref{inner} by $\overline{v}_q+z_{\ell_{q}}$, then calculating the inner product and integrating by parts, we obtain
		\begin{equation}\label{inner product3}
			\aligned
			\frac12  \frac{\dif}{\dif t} \int_{\T^3}  |\overline{v}_q+z_{\ell_{q}}|^2 \dif x 
			&=- \int_{\T^3}  (\overline{v}_q+z_{{\ell_q}})  \cdot  \left( (\overline{v}_q+z_{{\ell_q}}) \cdot \nabla (\overline{v}_q+z_{{\ell_q}})\right)   \dif x    -  \int_{\T^3} (\overline{v}_q+z_{\ell_{q}})  \cdot \nabla \overline{p}_q  \dif x  
			\\ &\quad + \int_{\T^3} (\overline{v}_q+z_{\ell_{q}}) \cdot  \partial_t z_{{\ell_q}}  \dif x +\int_{\T^3}  (\overline{v}_q+z_{\ell_{q}})  \cdot \div \mathring{\overline{R}}_q \dif x
			\\ &= \int_{\T^3} (\overline{v}_q+z_{\ell_{q}}) \cdot  \partial_t z_{{\ell_q}}  \dif x - \int_{\T^3} \mathring{\overline{R}}_q  : \nabla (\overline{v}_q+z_{\ell_{q}})^T    \dif x.
			\endaligned
		\end{equation}  
		By the definition \eqref{def:energy gap} of $\rho_q$ and \eqref{inner product3}, we have
		\begin{align*}
			\left|   \frac{\dif}{\dif t} \rho_q   \right| &= \frac13 \left|\frac{\dif}{\dif t}  e -   \frac{\dif}{\dif t} \int_{\T^3}  |\overline{v}_q+z_{\ell_{q}}|^2 \dif x + 2  \int_{\T^3} (\overline{v}_q+z_{\ell_{q}}) \cdot  \partial_t z_{{\ell_q}}  \dif x \right|
		\leq \tilde{e}+\left|  \int_{\T^3} \mathring{\overline{R}}_q: \nabla (\overline{v}_q+z_{\ell_{q}})^T  \dif x \right|.
		\end{align*}
	 Thus, by using \eqref{estimate:vellzell}, \eqref{e:vq:1+N}, \eqref{e:Rq:1} and choosing $a$ sufficiently large satisfying $\tilde{e}\leq a^{b(1-\beta-2b\beta)}\leq \delta_{q+1}\lambda_{q}\delta_{q}^{\sfrac12}$,
		we deduce
		\begin{align}\label{partial t rho}
			|\partial_t \rho_q|\leq \tilde{e}+\bar{M}^2 L^2\tau_q^{-1}\delta_{q+1}\ell_{q}^{2\alpha}\lesssim \delta_{q+1}\lambda_{q}\delta_{q}^{\sfrac12}+ L^2 \tau_q^{-1}\delta_{q+1}\ell_{q}^{\alpha} \lesssim L^2 \tau_q^{-1}\delta_{q+1}\ell_{q}^{\alpha}.
		\end{align}
		Finally, owning to \eqref{property eta}, \eqref{partial t rho}, $\|\partial_t \eta_i\|_{C^N_x}\lesssim \tau_q^{-1}$ and Leibniz rule for the derivative of the product, we obtain estimate \eqref{estimate:rho3}.
	\end{proof}

	Similarly to \cite{BDLSV19}, we define the map $\Phi_i: \Omega\times[\te_{i-1},\te_{i+1}] \times \R^3 \to \R^3$ as the $\T^3$ periodic solution to the transport equation
	\begin{subequations}
		\label{eq:Phi:i:def}
		\begin{align}
			(\partial_t + (\overline v_q+z_{\ell_q})  \cdot \nabla) \Phi_i &=0 \, ,\\  
			\Phi_i\left(\te_{i-1},x\right) &= x \, .
		\end{align}
	\end{subequations}
	From the adaptedness of $z_{{\ell_q}}$ and $\overline{v}_q$, it follows that $\Phi_i$ is also $(\mathcal{F}_t)_{t\geq 0}$-adapted.
	For the remainder of this section, it is convenient to denote the material derivative as $D_{t,q}: = \partial_t +( \overline v_q+z_{\ell_q})  \cdot \nabla_x$.
	Since the estimates for the transport equation are standard (cf.~\cite[Proposition D.1]{BDLIS16}, we put further details on the estimates of $\Phi_i$ in Appendix~\ref{sec:esti:transport}. We summarize them as follows:
	\begin{proposition}\label{Lemma:transport}
		For any $t \in  [\te_{i-1},\te_{i+1}]\cap[0,\mathfrak{t}_L]$ and $N\geq 0$, we have
		\begin{subequations}
			\label{eq:Phi:i:bnd}
			\begin{align}
				\|\nabla \Phi_i(t) - \mathrm{Id}\|_{C^0_x} &  \leq \bar{M} \ell_q^{\alpha}\ll \frac{1}{10}
				\label{eq:Phi:i:bnd:a}\, , \\
				\|\nabla \Phi_i\|_{C^N_x} + \|(\nabla \Phi_i)^{-1}\|_{C^N_x} &\lesssim \ell_q^{-N}
				\label{eq:Phi:i:bnd:b} ,\\
				\|D_{t,q} \nabla \Phi_i\|_{C^N_x} &\lesssim \tau_q^{-1} \ell_q^{-N},
				\label{eq:Phi:i:bnd:c}
			\end{align}
		\end{subequations}
		where the implicit constant may depend on $N$.
	\end{proposition}

	\subsection{Amplitudes}\label{sec:amplitude}
	Following \cite[Section 6.5]{BDLSV19}, we first define a stress supported in $\supp (\eta_i)\subset \tilde{I}_i\times \T^3$ as
	\begin{equation}\label{e:tildeR_def}
		\tilde R_{q,i} = \frac{\nabla\Phi_i R_{q,i} \nabla\Phi_i^T}{\rho_{q,i}}= \nabla\Phi_i \left( \Id -\frac{\eta_i^2 \mathring{\overline R}_q}{	\rho_{q,i}}   \right) \nabla\Phi_i^T,
	\end{equation}
	where we denote $	R_{q,i}$ as
	\begin{align}\label{R_q,i}
		R_{q,i}=\rho_{q,i}\Id- \eta_i^2 \mathring{\overline R}_q.
	\end{align}
	 Observe that $\supp ({\mathring{\overline R}}_q) \subset   \cup_i I_i\times \T^3 $ and $\eta_i \equiv 1$ on $ I_i$, $\eta_i \eta_j \equiv 0$ for $i\neq j$, we deduce that  the stress $\tilde R_{q,i}$ satisfies the identity
	\begin{align}
		\sum_i \rho_{q,i} (\nabla \Phi_i)^{-1} \tilde R_{q,i} (\nabla \Phi_i)^{-T} = \left( \sum_i \rho_{q,i} \right) \Id - \sum_i \eta_i^2 {\mathring{\overline R}}_q =\left( \sum_i \rho_{q,i} \right) \Id - {\mathring{\overline R}}_q \, ,
		\label{eq:tilde:R:q:i:identity}
	\end{align}
	which is useful in canceling the glued stress. 
	In addition, we give more estimates on $\tilde R_{q,i}$:
	\begin{proposition}\label{estimate:R:q,i}
		For any $t\in \tilde{I}_i$ and $N\in \N_0$, we have
		\begin{subequations}
			\begin{align}
				\|\tilde R_{q,i}\|_{C^N_x}&\lesssim  \ell_q^{-N},\label{es:Rq,i}
				\\ \| D_{t,q}\tilde R_{q,i}\|_{C^N_x}&\lesssim \tau_q^{-1} \ell_q^{-N},\label{es:Dt,q Rq,i}
			\end{align}
		\end{subequations}
		where the implicit constant may depend on $N$.
	\end{proposition}
	\begin{proof}
		The first estimate \eqref{es:Rq,i} can be derived through direct computation. 
		More precisely, recalling the definition of $\rho_{q,i}$ and applying  \eqref{e:Rq:1}, \eqref{estimate:rho1} and \eqref{R_q,i}, we obtain
		\begin{align}\label{Rq,i/rhoq.i}
			\left\|  \frac{R_{q,i}}{\rho_{q,i}}\right\| _{C^N_x} \lesssim 1+ \frac{1}{\rho_q} \|\mathring{\overline R}_q\|_{C^N_x} \lesssim \bar{M} (\ell_q\lambda_{q})^{\alpha} \ell_q^{-N}\lesssim \ell_q^{-N},
		\end{align}
		where we chose $a$ large enough to have $\bar{M} (\lambda_q\ell_q)^{\alpha}\ll 1$ in the last inequality. 
	We then combine \eqref{eq:Phi:i:bnd:b} with \eqref{Rq,i/rhoq.i} to deduce
		\begin{align*}
			\|\tilde R_{q,i}\|_{C^N_x}\lesssim  \|\nabla\Phi_i \|_{C^N_x} \left\|\frac{R_{q,i}}{\rho_{q,i}} \right\|_{C^0_x} + \|\nabla\Phi_i \|_{C^0_x}\left\|\frac{R_{q,i}}{\rho_{q,i}} \right\|_{C^N_x} \lesssim \ell_q^{-N}.
		\end{align*}
		
		We now proceed to estimate \eqref{es:Dt,q Rq,i}. By applying the material derivative $D_{t,q}$ to $\rho_{q,i}^{-1} R_{q,i}$, we obtain
		\begin{equation}\label{Dt,q R}
			D_{t,q} (\rho_{q,i}^{-1} R_{q,i}) = -\partial_t \left( \frac{ \sum \int_{\T^3}\eta_j^2}{\rho_q}\right) \mathring{\overline R_q} - \left( \frac{ \sum \int_{\T^3}\eta_j^2}{\rho_q}\right) D_{t,q} \mathring{\overline R}_q \, .
		\end{equation}
		Making use of Leibniz rule for the derivative of the product, \eqref{estimate:rho1} and \eqref{partial t rho}, we derive
		\begin{align}\label{chain rule rho}
			\left| \partial_t \left( \frac{ \sum \int_{\T^3}\eta_j^2}{\rho_q}\right)\right| \lesssim \sup_{i} \left| \frac{\partial_t\eta_i}{\rho_q}\right| +\left| \frac{\partial_t\rho_q}{\rho_q^2}\right|\lesssim \frac{\delta_{q+1}^{-1}\lambda_{q}^{\alpha}\tau_q^{-1}}{L^2}+\frac{\delta_{q+1}^{-2}\lambda_{q}^{2\alpha}\tau_q^{-1}\delta_{q+1}\ell_{q}^{\alpha}}{L^2}\leq \frac{\delta_{q+1}^{-1}\lambda_{q}^{\alpha}\tau_q^{-1}}{L^2}.
		\end{align}
		Therefore, by substituting \eqref{chain rule rho} into \eqref{Dt,q R} and utilizing \eqref{e:Rq:1} and \eqref{e:Rq:Dt}, we deduce
		\begin{equation}\label{e:Dt_ratio}
			\begin{aligned}
				\|D_{t,q} (\rho_{q,i}^{-1}R_{q,i})\|_{C^N_x}
				&\lesssim \frac{ \delta_{q+1}^{-1} \lambda_{q}^{\alpha} \tau_q^{-1}}{L^2} \|\mathring{\overline R}_q\|_{C^N_x}  + \frac{\delta_{q+1}^{-1}  \lambda_{q}^{\alpha} }{L^2} \|D_{t,q} \mathring{\overline R}_q\|_{C^N_x}
				\\&\lesssim \bar{M}(\lambda_{q}\ell_{q})^{\alpha}   \delta_{q+1}^{-1}  \tau_q^{-1} \delta_{q+1} \ell_q^{-N} \lesssim
				\tau_q^{-1} \ell_q^{-N}\, .
			\end{aligned}
		\end{equation}
		By applying the material derivative $D_{t,q}$ to \eqref{e:tildeR_def} once more, we derive
		\begin{align*}
			D_{t,q}  \tilde R_{q,i} = D_{t,q} \nabla\Phi_i (\rho_{q,i}^{-1} R_{q,i}) \nabla\Phi_i^T +
			\nabla\Phi_i D_{t,q} (\rho_{q,i}^{-1} R_{q,i}) \nabla\Phi_i^T + \nabla\Phi_i (\rho_{q,i}^{-1} R_{q,i})(D_{t,q} \nabla\Phi_i)^T\, .
		\end{align*}
		Therefore, combining \eqref{eq:Phi:i:bnd:b}, \eqref{eq:Phi:i:bnd:c} and \eqref{e:Dt_ratio} yields
		\begin{align*}
			\|D_{t,q} \tilde R_{q,i}\|_{C^N_x} \lesssim & \|D_{t,q} \nabla\Phi_i\|_{C^N_x}\|(\rho_{q,i}^{-1} R_{q,i})\|_{C^0_x} 
			+\|D_{t,q} \nabla\Phi_i\|_{C^0_x} \|(\rho_{q,i}^{-1} R_{q,i})\|_{C^N_x}+ \|D_{t,q} (\rho_{q,i}^{-1} R_{q,i})\|_{C^N_x} 
			\\ &+ \|D_{t,q} (\rho_{q,i}^{-1} R_{q,i})\|_{C^0_x} \|\nabla\Phi_i\|_{C^N_x}+ \|D_{t,q} \nabla\Phi_i \|_{C^0_x} \|(\rho_{q,i}^{-1} R_{q,i})\|_{C^0_x} \|\nabla\Phi_i\|_{C^N_x}
			\lesssim  \tau_q^{-1} \ell_q^{-N},
		\end{align*}
		which implies \eqref{es:Dt,q Rq,i}.
	\end{proof}
	Moreover, we use \eqref{e:Rq:1} together with \eqref{eq:Phi:i:bnd:a} and \eqref{eq:Phi:i:bnd:b} to have for $t\in \tilde{I}_i$
	\begin{equation}\label{es:Rq,i-Id}
		\aligned
		\|	\tilde R_{q,i}-\mathrm{Id}\|_{C^0_x} &\lesssim \|\nabla\Phi_i \nabla\Phi_i ^T-\Id\|_{C^0_x} +  \|\nabla\Phi_i \|_{C^0_x}^2 \|\mathring{\overline R}_q\|_{C^0_x} \frac{\delta_{q+1}^{-1}  \lambda_q^{\alpha}}{L^2}
		\\ & \lesssim( \|\nabla\Phi_i \|_{C^0_x} +1)  \|\nabla\Phi_i -\Id\|_{C^0_x}  + \bar{M} (\lambda_q\ell_q)^{\alpha} \leq \frac12,
		\endaligned
	\end{equation}
	which requires $a$ large enough to have $\bar{M} (\lambda_q\ell_q)^{\alpha}\ll 1$. This implies that $\tilde R_{q,i}$ restricted to $\supp \eta_i$ obeys the conditions of Lemma~\ref{l:linear_algebra}. Then, we define the amplitude functions
	for $i\geq 0$, $\xi \in \Lambda_{i}$ and $(t,x)\in [0,\mathfrak{t}_L]\times \T^3$ as
	\begin{align}\label{def:amplitude a}
		a_{(\xi,i)}(t,x)=\rho_{q,i}(t,x)^{\sfrac12} \gamma_{\xi}(\tilde R_{q,i}(t,x)),
	\end{align}
	where $\Lambda_{i}\subset \mathbb{S}^2\cap \mathbb{Q}^3$ and $\gamma_\xi \in C^\infty(\overline{B_{\sfrac 12}}(\mathrm{Id}))$ are the functions given in Lemma~\ref{l:linear_algebra}. Since $\rho_{q,i}^{\sfrac12}$ is a multiple of $\eta_i$, this shows that the support sets of $ a_{(\xi,i)}$ are pairwise disjoint. Moreover, the amplitude functions $a_{(\xi,i)}$ inherit the related bounds of $\rho_{q,i}^{\sfrac12}$ and $\tilde R_{q,i}$, and we have the following estimates.
	\begin{proposition}\label{es:amplitude}
		The amplitude functions satisfy the following bounds for any $t\in \tilde{I}_i$ and $N\in \N_0$
	\begin{subequations}
			\begin{align}
			\| a_{(\xi,i)}\|_{C^N_x}&\lesssim \frac{M}{C_{\Lambda}}L \delta_{q+1}^{\sfrac12}\ell_q^{-N}\, , \label{estimate:a CN}
			\\ \| D_{t,q} a_{(\xi,i)}\|_{C^N_x}&\lesssim  \frac{M}{C_{\Lambda}}L \tau_q^{-1}\delta_{q+1}^{\sfrac12} \ell_q^{-N}\, ,\label{estimate:Da CN}
		\end{align}
	\end{subequations}
		where $M$, $C_\Lambda$ are universal constants given in \eqref{eq:Onsager:M:def} and the implicit constant may depend on $N$.
	\end{proposition}
	\begin{proof}
		By applying the Leibniz rule for the product derivative, we obtain for all $N\in \N$
		\begin{align}\label{chain:a xi}
			\| a_{(\xi,i)}\|_{C^N_x} \lesssim \sum_{m=0}^{N} \| \rho_{q,i}^{\sfrac12}\|_{C^m_x} \|\gamma_{\xi}(\tilde R_{q,i})\|_{C^{N-m}_x} \, .
		\end{align}
		Keeping \cite[Proposition C.1]{BDLIS16} in mind, and applying \eqref{es:Rq,i}  together with \eqref{eq:Onsager:M:def}, we obtain for $N-m\geq 1$
		\begin{align}\label{chain:a xi:2}
			\|\gamma_{\xi}(\tilde R_{q,i})\|_{C^{N-m}_x}\lesssim \|\gamma_\xi\|_{C^1_x}\|\tilde R_{q,i}\|_{C^{N-m}_x}+  \|\gamma_\xi\|_{C^{N-m}_x}\|\tilde R_{q,i}\|^{N-m}_{C^{1}_x} \lesssim \frac{M}{C_{\Lambda}} \ell_q^{-(N-m)}.
		\end{align}
		From the properties of the cutoff function and energy gap, namely $	\| \eta_i\|_{C^m_x}\lesssim 1$ and $|\rho_q^{\sfrac12}|\leq L \delta_{q+1}^{\sfrac12}$, it follows that $ \|\rho_{q,i}^{\sfrac12}\|_{C^m_x}\lesssim L \delta_{q+1}^{\sfrac12}$. Substituting these estimates into \eqref{chain:a xi}, we obtain for all $N\in \N$
		\begin{align*}
			\| a_{(\xi,i)}\|_{C^N_x} \lesssim \frac{M}{C_{\Lambda}} L \delta_{q+1}^{\sfrac12}\ell_q^{-N}.
		\end{align*}
		Due to $\|\gamma_\xi\|_{C^0}\leq \frac{M}{C_{\Lambda}}$, this estimate also applies when $N=0$.
		Next, we observe that
		\begin{align*}
			D_{t,q} (\rho_{q,i}^{\sfrac12})=\left[  \partial_t \left( \frac{\eta_i}{(\sum \int_{\T^3}\eta_j^2)^\frac12} \right) +
			\frac{ (\overline{v}_q+z_{\ell_q}) \cdot \nabla \eta_i }{(\sum \int_{\T^3}\eta_j^2)^\frac12} \right]  \rho_q^{\sfrac12}+ \frac{\eta_i}{(\sum \int_{\T^3}\eta_j^2)^\frac12} \partial_t( \rho_q^{\sfrac12}).
		\end{align*}
	By employing \eqref{estimate:rho1} and \eqref{partial t rho}, we obtain $|  \partial_t( \rho_q^{\sfrac12})|\lesssim \left| \frac{\partial_t \rho_q}{\rho_q^{\sfrac12}} \right| \lesssim L \tau_q^{-1} \delta_{q+1}^{\sfrac12}$. Combining this bound with
		\eqref{estimate:vellzell}, \eqref{e:vq:1+N} and \eqref{e:vq:0} yields for any $N\in \N$
		\begin{align*}
			\|  D_{t,q} (\rho_{q,i}^{\sfrac12})\|_{C^N_x} \lesssim & (\|\partial_t \eta_i\|_{C^N_x } + \|\partial_t \eta_i\|_{C^0_x } \| \eta_i\|_{C^N_x } )\|\rho_q^{\sfrac12}\|_{C^0_t}+ \|\eta_i\|_{C^N_x}\|\rho_q^{\sfrac12}\|_{C^1_t}
			\\ &+ (\|\overline{v}_q +z_{\ell_q}\|_{C^0_x } \|\nabla \eta_i\|_{C^N_x }+  \|\overline{v}_q+z_{\ell_q} \|_{C^N_x } \|\nabla \eta_i\|_{C^0_x}) \|\rho_q^{\sfrac12}\|_{C^0_t}
			\\ \lesssim & L \tau_q^{-1}\delta_{q+1}^{\sfrac12}+\bar{M} L
			\tau_q^{-1}\ell_q^{-(N-1)+\alpha} \delta_{q+1}^{\sfrac12} \lesssim L \tau_q^{-1} \delta_{q+1}^{\sfrac12} \ell_q^{-N},
		\end{align*}
		where the last line is justified by $\bar{M}\ell_{q}^{\alpha}\ll 1$.
		Since $\|\overline v_q+ z_{\ell_{q}}\|_{C^0_x} \lesssim \bar{M} \lambda_{1}^{\sfrac{3\alpha}{2}}$, we can also estimate $ \|  D_{t,q} (\rho_{q,i}^{\sfrac12})\|_{C^0_x}\lesssim L\tau_q^{-1} \delta_{q+1}^{\sfrac12}$.
		Using the chain rule together with \eqref{es:Rq,i}, \eqref{es:Dt,q Rq,i} yields for any $N\geq 0$
		\begin{equation}\label{Dt,q a}
			\aligned
			\| D_{t,q} a_{(\xi,i)}\|_{C^N_x} &\lesssim  \|  D_{t,q} (\rho_{q,i}^{\sfrac12})\|_{C^N_x} \|\gamma_{\xi}(\tilde R_{q,i})\|_{C^{0}_x}+  \|  D_{t,q} (\rho_{q,i}^{\sfrac12})\|_{C^0_x} \|\gamma_{\xi}(\tilde R_{q,i})\|_{C^{N}_x}
			\\&\, \,+ \|  \rho_{q,i}^{\sfrac12}\|_{C^N_x} \|  D_{t,q} \gamma_{\xi}(\tilde R_{q,i})\|_{C^0_x}+  \|  \rho_{q,i}^{\sfrac12}\|_{C^0_x} \|  D_{t,q} \gamma_{\xi}(\tilde R_{q,i})\|_{C^N_x} 
            \lesssim  \frac{M}{C_{\Lambda}} L \tau_q^{-1}   \delta_{q+1}^{\sfrac12} \ell_q^{-N}, 
           \endaligned
		\end{equation}
		which gives \eqref{estimate:Da CN}.
	\end{proof}

	\subsection{Construction of the velocity increment $w_{q+1}$}\label{sec:principal w}
	With the previous preparation at hand, we proceed with the construction of the new perturbation $w_{q+1}$ in this subsection. Then,	 
	the velocity field at the level $q+1$ is constructed as 
	\begin{align}\label{def:v q+1}
		v_{q+1}:=\overline{v}_q+w_{q+1}.
	\end{align}
	To this end, the building blocks of the perturbation are the Mikado flows constructed in \cite{DS17} and presented \cite[Section 6.4]{BV19}, which we recall in Appendix~\ref{sec:Mikado}. 
	In the sequel, we consider the Mikado building blocks as defined in \eqref{eq:Mikado:def} with $\lambda = \lambda_{q+1}$, $\xi \in \Lambda_{i}$, i.e. 
	\begin{align*}
		W_{(\xi)}(x)= W_{\xi,\lambda_{q+1}}(x) .
	\end{align*}
	For the index sets $\Lambda_i$ of Lemma~\ref{l:linear_algebra}, we overload notation and write $\Lambda_i = \Lambda_{i \, \textrm{mod} \,2}$ for any $i\in \mathbb{Z}$. With this notation, we now define the principal part of the perturbation as 
	\begin{align}\label{eq:Onsager:w:q+1:p}
		w_{q+1}^{(p)}(t,x) = \sum_{i} \sum_{\xi \in \Lambda_i} a_{(\xi,i)}(t,x) (\nabla \Phi_i(t,x))^{-1} W_{(\xi)}(\Phi_i(t,x)) \, .
	\end{align}
	Note that both $\overline{v}_q$ and $z_{\ell_q}$ are divergence-free, it follows from \cite[(6.46)]{BV19} that
	\begin{align}
		D_{t,q} \left( (\nabla \Phi_i)^{-1} W_{(\xi)}(\Phi_i)\right) = \left( \nabla( \overline v_q+z_{\ell_q})\right) ^T (\nabla \Phi_i)^{-1} W_{(\xi)}(\Phi_i) \, .
		\label{eq:Onsager:Lie:advect}
	\end{align}
	
	In order to ensure $w_{q+1}$ is divergence-free, we aim to construct an incompressibility corrector $w_{q+1}^{(c)}$ such that the resulting function $w_{q+1}^{(p)}+w_{q+1}^{(c)}$ is the curl of a vector field. As mentioned in~\cite[Section 5]{DS17},  the following identity holds for any smooth vector field $V$
	\begin{align*}
		(\nabla \Phi_i)^{-1} \left( (\curl V)\circ \Phi_i \right)= \curl\left( (\nabla \Phi_i)^T (V\circ \Phi_i)\right).
	\end{align*}
	By the identity \eqref{eq:Mikado:curl} and definition \eqref{eq:Mikado:curl:2}, we have $W_{(\xi)} = \curl V_{(\xi)}$ and it follows from the above identity that 
	\begin{align}\label{W:V}
		(\nabla \Phi_i)^{-1} (W_{(\xi)} \circ \Phi_i) = \curl\left( (\nabla \Phi_i)^T (V_{(\xi)}\circ \Phi_i)\right) \,.
	\end{align}
	In view of \eqref{eq:Onsager:w:q+1:p} and \eqref{W:V}, it is natural to define the incompressibility corrector as
	\begin{align}
		w_{q+1}^{(c)}(t,x) = \sum_i \sum_{\xi \in \Lambda_i} \nabla a_{(\xi,i)}(t,x) \times \left( (\nabla \Phi_i(t,x))^T (V_{(\xi)} (\Phi_i(t,x) )  \right),
		\label{eq:Onsager:w:q+1:c}
	\end{align}
	and one may check that the new perturbation $w_{q+1}= w_{q+1}^{(p)} + w_{q+1}^{(c)} $ satisfies
	\begin{align}
		w_{q+1} = \curl \left( \sum_i \sum_{\xi \in \Lambda_i} a_{(\xi,i)} \, (\nabla \Phi_i)^T (V_{(\xi)}\circ \Phi_i) \right) ,
		\label{eq:Onsager:w:q+1}
	\end{align}
	so that it is divergence-free. Since the coeﬀicients $a_{(\xi,i)}$ and $\Phi_i$ are $(\mathcal{F}_t)_{t\geq 0}$-adapted, we deduce that $w_{q+1}$ is also $(\mathcal{F}_t)_{t\geq 0}$-adapted.

	\subsection{Definition of Reynolds Stress $\mathring{R}_{q+1}$}
	\label{sec: new Reynolds stress}
	
	Recalling the system \eqref{gluing euler q} and substituting $v_{q+1}=\overline v_q+w_{q+1}$ into $\eqref{euler1}$ at the level $q+1$, we obtain that
	\begin{equation}\label{Rno1}
		\begin{aligned}
			&\div \mathring{R}_{q+1}-\nabla p_{q+1}
			\\& =\underbrace{(\partial_t+(\overline v_q+z_{\ell_q})\cdot \nabla)w_{q+1}^{(p)}}
			_{\div(R^{\textrm{trans}})}+ \underbrace{ \div (w_{q+1}^{(p)} \otimes w_{q+1}^{(p)}+\mathring{\overline{R}}_q)}
			_{\div(R^{\textrm{osc}})+\nabla p^{\textrm{osc}}}-\nabla \overline{p}_{q}
			\\&+\underbrace{w_{q+1}\cdot \nabla(\overline v_q +z_{\ell_q})}_{\div(R^{\textrm{Nash}})}+\underbrace{(\partial_t+(\overline v_q+z_{\ell_q})\cdot \nabla ) w_{q+1}^{(c)}+\div(w_{q+1}^{(c)}\otimes w_{q+1}+w_{q+1}^{(p)} \otimes w_{q+1}^{(c)})}_{\div(R^{\textrm{cor}})+\nabla p^{\textrm{cor}}}
			\\&+\underbrace{\div(v_{q+1} \otimes (z_{q+1}-z_{{\ell_q}} )+(z_{q+1} -z_{{\ell_q}} )\otimes v_{q+1}+z_{q+1}\otimes z_{q+1}-z_{{\ell_q}} \otimes z_{{\ell_q}})}_{\div(R^{\textrm{com}})+\nabla p^{\textrm{com}}} .
		\end{aligned}
	\end{equation}
	By using the inverse divergence operator $\mathcal{R}$ introduced in Section~\ref{sec:appendix c}, we define
	\begin{equation*}
		\aligned
		R^{\textrm{trans}}&:=\mathcal{R}\left((\partial_t+(\overline v_q+z_{\ell_q})\cdot \nabla)w_{q+1}^{(p)}\right),
		\\ R^{\textrm{Nash}}&:=\mathcal{R} ( w_{q+1}\cdot \nabla(\overline v_q +z_{\ell_q})),
		\\ R^{\textrm{cor}}&:=\mathcal{R} \left((\partial_t+(\overline v_q+z_{\ell_q})\cdot \nabla ) w_{q+1}^{(c)}\right)+\left(w_{q+1}^{(c)}\mathring \otimes w_{q+1}+w_{q+1}^{(p)}\mathring \otimes w_{q+1}^{(c)}\right),
		\\ R^{\textrm{com}}&:=v_{q+1}\mathring \otimes (z_{q+1}- z_{{\ell_q}}) +(z_{q+1}-z_{{\ell_q}})\mathring \otimes v_{q+1}+z_{q+1}\mathring \otimes z_{q+1}-z_{{\ell_q}}\mathring \otimes z_{{\ell_q}},
		\\ p^{\textrm{cor}}&:=\frac{1}{3}(2w_{q+1}^{(c)}\cdot w_{q+1}^{(p)}+|w_{q+1}^{(c)}|^2),
		\\ p^{\textrm{com}}&:=\frac{1}{3}(v_{q+1} \cdot z_{q+1}-v_{q+1} \cdot z_{{\ell_q}} +z_{q+1}\cdot v_{q+1}-z_{{\ell_q}} \cdot v_{q+1}+z_{q+1} \cdot z_{q+1}-z_{{\ell_q}}\cdot z_{{\ell_q}}).
		\endaligned
	\end{equation*}
	
	In order to define the remaining oscillation error from the second line in \eqref{Rno1}, we first note that $\eta_i$ have mutually disjoint supports and for $\xi \neq \xi'\in \Lambda_{i}$, $W_{(\xi)}\otimes W_{(\xi')}\equiv0$. Then, using the identity \eqref{eq:tilde:R:q:i:identity} and the spanning property of the Mikado flows \eqref{eq:Mikado:4}, we obtain the following identity from \cite[(6.47)]{BV19}
	\begin{equation}\label{eq:w:q+1:is:good}
		\begin{aligned}
			&w_{q+1}^{(p)} \otimes w_{q+1}^{(p)} 
			\\&= \left( \sum_i \rho_{q,i} \right) \Id - \mathring{\overline{R}}_q + \sum_i \sum_{\xi \in \Lambda_i} a_{(\xi,i)}^2 (\nabla \Phi_i)^{-1} \left( \left(\mathbb{P}_{\geq \sfrac{\lambda_{q+1}}{2}}(W_{(\xi)} \otimes W_{(\xi)}) \right)\circ \Phi_i \right)  (\nabla \Phi_i)^{-T} ,  
		\end{aligned}
	\end{equation}
	where $\mathbb
	{P}_{\neq 0}f$ denotes the projection of $f$ onto its nonzero frequencies, i.e. $\mathbb
	{P}_{\neq 0}  = f- \fint_{\T^3} f$. We have also used that since  $W_{(\xi)}\otimes W_{(\xi)}$ is $(\sfrac{\T}{\lambda_{q+1}})^3$-periodic, the identity $\mathbb
	{P}_{\neq 0} (W_{(\xi)} \otimes W_{(\xi)}) = \mathbb
	{P}_{\geq \sfrac{\lambda_{q+1}}{2}} (W_{(\xi)} \otimes W_{(\xi)})$ holds. Hence, we denote oscillation error by
	\begin{align*}
		R^{\textrm{osc}}:=\mathcal{R}\left( \div \left(    \sum_i \sum_{\xi \in \Lambda_i} a_{(\xi,i)}^2 (\nabla \Phi_i)^{-1} \left( \left(\mathbb{P}_{\geq \sfrac{\lambda_{q+1}}{2}}(W_{(\xi)} \otimes W_{(\xi)}) \right)\circ \Phi_i \right)  (\nabla \Phi_i)^{-T}  \right)  \right),
	\end{align*}
	and the related pressure is given by $	p^{\textrm{osc}} := \sum_i \rho_{q,i}$.
	With the above notation, we define the Reynolds stress at the level $q+1$ by
	\begin{equation}\label{Reynold1}
		\mathring{R}_{q+1}=R^{\textrm{trans}}+R^{\textrm{osc}}+R^{\textrm{Nash}}+R^{\textrm{cor}}+R^{\textrm{com}},
	\end{equation}
	and pressure at the level $q+1$ by
	\begin{equation*}
		p_{q+1}=\overline p_q-p^{\textrm{osc}}-p^{\textrm{cor}}-
		p^{\textrm{com}}.
	\end{equation*}
	
	\section{Proof of Proposition~\ref{p:iteration}---Step 3: Inductive estimates}\label{sec:inductive esti}\label{sec:inductive estimates}
	In the present section we collect all the necessary estimates to complete the proof of Proposition~\ref{p:iteration}. We will verify that $v_{q+1}$ and $\mathring{R}_{q+1}$ satisfy the inductive estimates \eqref{itera:a}, \eqref{itera:b},  \eqref{itera:c}, and \eqref{vq+1-vq} in Subsection~\ref{Estimates on v q+1} and Subsection~\ref{estimate on Rq+1}. Lastly, the energy estimate \eqref{estimate:energy} will be justified in Subsection~\ref{estimate on energy}.
	\subsection{Estimates on $v_{q+1}$}\label{Estimates on v q+1}
	We begin by establishing the bounds on perturbation $w_{q+1}$. Noting that $\eta_i$ has disjoint supports, the same holds for $a_{(\xi,i)}$. Therefore, for any
	fixed $t\in [0,\mathfrak{t}_L]$, the sum over $i$ in the definitions of $w_{q+1}^{(p)}(t)$ in \eqref{eq:Onsager:w:q+1:p} and $w_{q+1}^{(c)}(t)$ in \eqref{eq:Onsager:w:q+1:c} is finite. Thus, we arrive at the following proposition.
	\begin{proposition}\label{Pro:6.1}
		The perturbation $w_{q+1}^{(p)}$ and $w_{q+1}^{(c)}$ satisfy the following bounds for any $t\in [0,\mathfrak{t}_L]$
		\begin{subequations}
				\begin{align}
				\|w_{q+1}^{(p)}\|_{C^0_{t,x}}+	\|w_{q+1}^{(c)}\|_{C^0_{t,x}} &\leq \frac12 \bar{M} L \delta_{q+1}^{\sfrac12},\label{estimate on wp}
				\\	\|w_{q+1}^{(p)}\|_{C^0_tC^1_{x}}+	\|w_{q+1}^{(c)}\|_{C^0_tC^1_{x}} &\leq \frac12 \bar{M} L \lambda_{q+1} \delta_{q+1}^{\sfrac12}.\label{estimate on wc}
			\end{align}
		\end{subequations}
	\end{proposition}
	\begin{proof}
		We first use \eqref{eq:Phi:i:bnd:b}, \eqref{estimate:a CN}, \eqref{def:CLambda} and  \eqref{eq:Mikado:bounds} to have for any $t\in [0,\mathfrak{t}_L]$
		\begin{align}\label{esti:principal w 0}
			\|w_{q+1}^{(p)}\|_{C^0_{t,x}}  \lesssim \sup_{i} \sum_{\xi \in \Lambda_i}  \| a_{(\xi,i)}\|_{C^0_{t,x}} \|(\nabla\Phi_i)^{-1}\|_{C^0_{t,x}}  \|W_{(\xi)}( \Phi_i )\|_{C^0_{t,x}} \leq \frac{|\Lambda|ML}{C_\Lambda} \delta_{q+1}^{\sfrac12}\leq \frac14 \bar{M}L \delta_{q+1}^{\sfrac12},
		\end{align}
		where $|\Lambda|$ is the cardinality of the set $\Lambda_0\cup \Lambda_1$ as given in Appendix~\ref{sec:Mikado} and $\bar{M} $ is a universal constant satisfying
		\begin{align}\label{def barM}
			100|\Lambda|M<C_\Lambda\bar{M}.
		\end{align}
		The estimates for $w_{q+1}^{(c)}$ follow a similar pattern to those for $w_{q+1}^{(p)}$. First, we observe that the condition
		$16\alpha<(b-1)(1-\beta)$, together with the choice of $a$ sufficiently large, implies for all $q\geq 2$
		\begin{align}\label{eq:Onsager:ell:gap}
			\frac{\ell_q^{-1-\alpha}}{\lambda_{q+1}}\leq  \frac{\delta_q^{\sfrac 12} \lambda_q^{1+ 8\alpha}}{\delta_{q+1}^{\sfrac 12} \lambda_{q+1}} \leq 2 \lambda_{q}^{8\alpha - (b-1)(1-\beta)} \leq 2\lambda_{q}^{- \sfrac{(b-1)(1-\beta)}{2}}  \ll 1 \, .
		\end{align}
		Since $a$ large enough, we have $\frac{\ell_1^{-1-\alpha}}{\lambda_{2}}\leq 2\lambda_{1}^{1-b+8\alpha}\ll 1$, then the estimate \eqref{eq:Onsager:ell:gap} also holds for $q=1$.
		We then use \eqref{eq:Phi:i:bnd:b}, \eqref{estimate:a CN}, \eqref{def:CLambda} and \eqref{eq:Mikado:bounds} to deduce for any $t\in [0,\mathfrak{t}_L]$
		\begin{equation}\label{esti:corrector w 0}
			\begin{aligned}
				\|w_{q+1}^{(c)}\|_{C^0_{t,x}} \lesssim \sup_{i} \sum_{\xi \in \Lambda_i}  \| \nabla a_{(\xi,i)}\|_{C^0_{t,x}} \|\nabla\Phi_i\|_{C^0_{t,x}}  \|V_{(\xi)}( \Phi_i )\|_{C^0_{t,x}} 
				\lesssim \frac{|\Lambda|ML}{C_\Lambda \lambda_{q+1}} \delta_{q+1}^{\sfrac12}\ell_q^{-1} \leq \frac{1}{4} \bar{M}L \delta_{q+1}^{\sfrac12},
			\end{aligned}
		\end{equation}
		where the last inequality is justified by \eqref{eq:Onsager:ell:gap}. Combining \eqref{esti:principal w 0} with \eqref{esti:corrector w 0}, we derive \eqref{estimate on wp}.

		Turning to the $C^1_x$-norm of $w_{q+1}^{(p)}$, by applying \eqref{eq:Phi:i:bnd:b}, \eqref{estimate:a CN}, \eqref{def:CLambda} and \eqref{eq:Mikado:bounds}, we obtain for any $t\in [0,\mathfrak{t}_L]$
		\begin{equation}\label{esti:principal w 1}
			\begin{aligned}
				\|w_{q+1}^{(p)}\|_{C^0_tC^1_{x}} &\lesssim\sup_{i} \sum_{\xi \in \Lambda_i} \bigg(  \| \nabla a_{(\xi,i)}\|_{C^0_{t,x}} \|(\nabla \Phi_i )^{-1}\|_{C^0_{t,x}} \|W_{(\xi)}( \Phi_i )\|_{C^0_{t,x}} 
				\\ &\qquad \qquad\qquad+\| a_{(\xi,i)}\|_{C^0_{t,x}}  \|( \nabla \Phi_i )^{-1}\|_{C^0_tC^1_{x}} \|W_{(\xi)}( \Phi_i )\|_{C^0_{t,x}} 
				\\& \qquad \qquad \qquad + \| a_{(\xi,i)}\|_{C^0_{t,x}}\|( \nabla \Phi_i )^{-1}\|_{C^0_{t,x}}  \|W_{(\xi)}( \Phi_i )\|_{C^0_tC^1_{x}} \bigg) 
				\\ &\lesssim \frac{|\Lambda|ML}{C_\Lambda}(\ell_q^{-1} +\lambda_{q+1}) \delta_{q+1}^{\sfrac12}
				\lesssim \frac{|\Lambda|M L}{C_\Lambda} \delta_{q+1}^{\sfrac12} \lambda_{q+1} 
				\leq  \frac14 \bar{M}L \delta_{q+1}^{\sfrac12} \lambda_{q+1}, 
			\end{aligned}
		\end{equation}
		where we used \eqref{eq:Onsager:ell:gap} in the third inequality. Lastly, we   use \eqref{eq:Phi:i:bnd:b}, \eqref{estimate:a CN}, \eqref{eq:Onsager:ell:gap}, \eqref{def:CLambda} and \eqref{eq:Mikado:bounds} to have for any $t\in [0,\mathfrak{t}_L]$
		\begin{equation}\label{esti:corrector w 1}
			\begin{aligned}
				\|w_{q+1}^{(c)}\|_{C^0_tC^1_{x}}&\lesssim  \sup_{i} \sum_{\xi \in \Lambda_i} \bigg(\| \nabla a_{(\xi,i)}\|_{C^0_tC^1_{x}} \| \nabla \Phi_i\|_{C^0_{t,x}}  \|V_{(\xi)}( \Phi_i )\|_{C^0_{t,x}}
				\\& \qquad \qquad \qquad +\| \nabla a_{(\xi,i)}\|_{C^0_{t,x}} \| \nabla \Phi_i\|_{C^0_tC^1_{x}}  \|V_{(\xi)}( \Phi_i )\|_{C^0_{t,x}}
				\\& \qquad \qquad \qquad +\| \nabla a_{(\xi,i)}\|_{C^0_{t,x}} \| \nabla \Phi_i\|_{C^0_{t,x}}  \|V_{(\xi)}( \Phi_i )\|_{C^0_tC^1_{x}}\bigg)
				\\&\lesssim \frac{|\Lambda|ML}{C_\Lambda}\left( \frac{\ell_q^{-2}}{\lambda_{q+1}}+\ell_q^{-1}\right) \delta_{q+1}^{\sfrac12}
				= \frac{|\Lambda|ML}{C_\Lambda}  
				\delta_{q+1}^{\sfrac12} \ell_q^{-1}\left( 1+\frac{\ell_q^{-1}}{\lambda_{q+1}}\right) 
				\leq \frac14 \bar{M}L\lambda_{q+1} \delta_{q+1}^{\sfrac12} ,
			\end{aligned}
		\end{equation}
		which combined with \eqref{esti:principal w 1}, yields \eqref{estimate on wc}.
	\end{proof}
	
	With Proposition~\ref{Pro:6.1} in hand, we begin estimating $v_{q+1}$. Recalling the definition of $v_{q+1}$ in \eqref{def:v q+1}, we write
	\begin{align*}
		v_{q+1}=\overline{v}_q+w_{q+1}=v_q+(v_{\ell_q}-v_q)+(\overline v_q-v_{\ell_q})+w_{q+1}.
	\end{align*} 
	We then combine \eqref{vq-vl}, \eqref{e:vq:vell} and \eqref{estimate on wp} to deduce for any $t\in [0,\mathfrak{t}_L]$
	\begin{align*}
		\|v_{q+1}(t)-v_q(t)\|_{C^0_{x}}&\leq \|v_{\ell_q}(t)-v_q(t)\|_{C^0_{x}}+\|\overline{v}_q(t)-v_{\ell_q}(t)\|_{C^0_{x}}+\|w_{q+1}(t)\|_{C^0_{x}}
		\\ &\leq  2\bar{M} L \delta_{q+1}^{\sfrac12}\ell_{q}^{\alpha}+\frac{1}{2} \bar{M} L\delta_{q+1}^{\sfrac12}\leq \bar{M}L\delta_{q+1}^{\sfrac12},
	\end{align*}
	which verifies \eqref{vq+1-vq} at the level $q+1$. Combining \eqref{itera:a} with \eqref{vq+1-vq}, we deduce for any $t\in [0,\mathfrak{t}_L]$
	\begin{align*}
		\|v_{q+1}\|_{C^0_{t,x}} &\leq \|v_q\|_{C^0_{t,x}}+	\|v_{q+1}-v_q\|_{C^0_{t,x}}
		\\& \leq 3\bar{M}L\lambda_{1}^{\sfrac{3\alpha}{2}}-\bar{M}L\delta_{q}^{\sfrac12}+\bar{M}L\delta_{q+1}^{\sfrac12}\leq 3\bar{M}L\lambda_{1}^{\sfrac{3\alpha}{2}}-\bar{M}L\delta_{q+1}^{\sfrac12}.
	\end{align*}
	Here, the last inequality is justified by $2\delta_{q+1}^{\sfrac12}\leq \delta_{q}^{\sfrac12}$, which requires $a$ to be sufficiently large such that  $2<a^{(b-1)\beta}$.
	Hence, \eqref{itera:a} holds at the level $q+1$.
	Using \eqref{e:vq:1+N}, and \eqref{estimate on wc} we have  for any $t\in [0,\mathfrak{t}_L]$
	\begin{align*}
		\|v_{q+1}\|_{C^0_tC^1_{x}} &\leq \|\overline{v}_q\|_{C^0_tC^1_{x}}+\|w_{q+1}\|_{C^0_tC^1_{x}} \leq \bar{M}  \tau_q^{-1} +\frac{1}{2} \bar{M}L \lambda_{q+1}\delta_{q+1}^{\sfrac12}
		\leq  \bar{M}L \lambda_{q+1}\delta_{q+1}^{\sfrac12},
	\end{align*}
	where we used $6\alpha<(b-1)(1-\beta)$ and chose $a$ sufficiently large to have 
	\begin{align*}
		\frac{ \tau_q^{-1} }{\lambda_{q+1}\delta_{q+1}^{\sfrac12}}\lesssim L \lambda_{q}^{6\alpha-(b-1)(1-\beta)} \ll L.
	\end{align*}
	Hence, we verified that \eqref{itera:b} holds at the level $q+1$.

	\subsection{Estimates on $\mathring{R}_{q+1}$}\label{estimate on Rq+1}
	In this section, we demonstrate that the stress $\mathring{R}_{q+1}$ defined in \eqref{Reynold1} satisfies the
	estimate:
	\begin{align*}
		\|\mathring{R}_{q+1}(t)\|_{C^0_{x}} \leq \bar{M}L^2 \delta_{q+2}\lambda_{q+1}^{-3\alpha},
	\end{align*}
	for any $t\in[0,\mathfrak{t}_L]$, which implies \eqref{itera:c} at the level $q+1$. In order to apply Proposition~\ref{prop.inv.div} to each stress given in \eqref{Reynold1} and obtain the desired bounds, we decompose the function $\phi_{(\xi)}$ in \eqref{eq:phi:rotate} as a Fourier series and use the fast decay of the Fourier coefficients $\hat{f}_{\xi}(k)$ to estimate $\mathcal{R}(aW_{(\xi)}(\Phi_i))$. Without providing all the details, we refer to the following estimates from \cite[Page 231]{BV19}
	\begin{equation}	\label{eq:Mikado:stationary:phase:1}
		\begin{aligned}
			\|{\mathcal{R} \left( a \; (W_{(\xi)} \circ \Phi_i ) \right) }\|_{C^\alpha_x}  + \lambda_{q+1} & \|{\mathcal{R} \left( a \; (V_{(\xi)} \circ \Phi_i ) \right) }\|_{C^\alpha_x}
			\lesssim \frac{  \|{a}\|_{C^0_x}}{\lambda_{q+1}^{1-\alpha}} +   \frac{ \|{a}\|_{C^{m+\alpha}_x}+\|{a}\|_{C^0_x} \|\nabla \Phi_i\|_{C^{m+\alpha}_x} }{\lambda_{q+1}^{m-\alpha}} \, ,
		\end{aligned}
	\end{equation}
	\begin{align}
		\left\| {\mathcal{R} \left( a \; \left( \Big( \mathbb{P}_{\geq \sfrac{\lambda_{q+1}}{2}}(\phi_{(\xi)}^2) \Big) \circ\Phi_i \right) \right) }\right\| _{C^\alpha_x}   
		\lesssim \frac{  \|{a}\|_{C^0_x}}{\lambda_{q+1}^{1-\alpha}} +   \frac{ \|{a}\|_{C^{m+\alpha}_x}+\|{a}\|_{C^0_x} \| \nabla \Phi_i\|_{C^{m+\alpha}_x} }{\lambda_{q+1}^{m-\alpha}} \, ,
		\label{eq:Mikado:stationary:phase:2}
	\end{align}
	where the implicit constants are independent of $q$. 
	
	Before estimating each error term given in \eqref{Reynold1} separately, we present two frequently used estimates. First, the condition $20b\alpha <(b-1)(1-2b\beta-\beta)$ implies for any $q\geq 2$
	\begin{align}\label{choice:alpha}
		\frac{\delta_{q+1}^{\sfrac12}\lambda_q\delta_{q}^{\sfrac12}}{\lambda_{q+1}^{1-20\alpha}\delta_{q+2}}=\frac{\lambda_{q}^{1-\beta}\lambda_{q+1}^{-\beta}}{\lambda_{q+1}^{1-20\alpha}\lambda_{q+2}^{-2\beta}}\lesssim \lambda_{q}^{20b\alpha-(b-1)(1-2b\beta-\beta)}\ll 1,
	\end{align}
	which also holds when $q=1$ since 
	\begin{align*}
		\frac{\delta_{2}^{\sfrac12}\lambda_1\delta_{1}^{\sfrac12}}{\lambda_{2}^{1-20\alpha}\delta_{3}}\lesssim \lambda_{1}^{20b\alpha-(b-1)(1-2b\beta)}\ll1.
	\end{align*}
	Another essential estimate required during the proof process is
	\begin{align}\label{Rsmall}
		\frac{\ell_q^{-N-\alpha} }{\lambda_{q+1}^{N-1}}\leq 1,
	\end{align}
	which holds for sufficiently small $\alpha$ and large $N\in \N$ in terms of $b,\beta$ and sufficiently large $a$ but independent of $q$. Such choice is equivalent to $\lambda_{q+1}^{(N-1)-(N+\alpha)\beta}\geq \lambda_{q}^{(N+\alpha)(1-\beta+6\alpha)}$. Since $\lambda_{q}^b=\lceil{a^{(b^{q})}}\rceil^b \lesssim \lceil{a^{(b^{q+1})}}\rceil =\lambda_{q+1}$, it is enough to enforce the condition 
	\begin{align*}
		b((N-1)-(N+\alpha)\beta)>(N+\alpha)(1-\beta+6\alpha).
	\end{align*}
	In order to show that for sufficiently small $\alpha$, we can choose such an $N$, it suffices to verify the existence of $N$ such that $b((N-1)-N\beta)>N(1-\beta)$. This is equivalent to $(b-1)(N-1)(1-\beta)>(1-\beta)+b\beta$, which can be satisfied for some large $N$.
	In the sequel, we only apply \eqref{eq:Mikado:stationary:phase:1} and \eqref{eq:Mikado:stationary:phase:2} with $m=N$-th derivative estimates.
	
	\begin{proposition}[Estimates for transport error] \label{Estimates for transport error}
		The transport error enjoys for any $t\in[0,\mathfrak{t}_L]$:
		\begin{align}\label{esti:transport}
			\|R^{\rm{trans}}(t)\|_{C^\alpha_x}\leq  \frac15 \bar{M}L^2 \delta_{q+2}\lambda_{q+1}^{-3\alpha}.
		\end{align}
	\end{proposition}
	\begin{proof}
		Recalling the definition of $w_{q+1}^{(p)}$ in \eqref{eq:Onsager:w:q+1:p} and the Lie-advection identity \eqref{eq:Onsager:Lie:advect}, we write the transport error in \eqref{Reynold1} as 
		\begin{equation}\label{Rtransport}
			\aligned
			R^{\textrm{trans}}=\sum_{i}\sum_{\xi \in \Lambda_i} \mathcal{R}& \left( a_{(\xi,i)} (\nabla (\overline{v}_q+  z_{\ell_q}))^T (\nabla \Phi_i)^{-1} W_{(\xi)}(\Phi_i) \right) 
			\\ &+ \sum_{i}\sum_{\xi \in \Lambda_i} \mathcal{R}\left( D_{t,q} a_{(\xi,i)} (\nabla \Phi_i)^{-1} W_{(\xi)}(\Phi_i) \right) 
			=: R^{\textrm{trans}}_1+R^{\textrm{trans}}_2.
			\endaligned
		\end{equation}
		Let us control each term separately, starting with $R^{\textrm{trans}}_1$. By using \eqref{def tauq}, \eqref{e:vq:1+N}, \eqref{eq:Phi:i:bnd:b} and \eqref{estimate:a CN}, we deduce for any $t\in [0,\mathfrak{t}_L]$ and $m\in \mathbb{N}_0$
		\begin{equation}\label{RT1:part1}
			\aligned
			\|a_{(\xi,i)}\nabla (\overline{v}_q+  z_{\ell_q})(\nabla \Phi_i)^{-1}\|_{C^{m}_x} &\lesssim \| a_{(\xi,i)}\|_{C^{m}_x}\|\nabla (\overline{v}_q+  z_{\ell_q})\|_{C^0_x} 
			+\|a_{(\xi,i)}\|_{C^0_x} \|\nabla (\overline{v}_q+  z_{\ell_q})\|_{C^m_x} 
			\\ &\quad+\|a_{(\xi,i)}\|_{C^0_x} \|\nabla (\overline{v}_q+  z_{\ell_q})\|_{C^0_x}\| (\nabla \Phi_i)^{-1}\|_{C^m_x}   
			\\ &\lesssim \frac{M}{C_{\Lambda}}\bar{M}L\ell_{q}^{\alpha} \delta_{q+1}^{\sfrac12}\tau_q^{-1}\ell_q^{-m}\leq  \frac{M}{C_{\Lambda}}L^2\delta_{q+1}^{\sfrac12}\lambda_{q}^{1+6\alpha}\delta_{q}^{\sfrac12} \ell_q^{-m},
			\endaligned
		\end{equation}
		which requires $\bar{M}\ell_{q}^{\alpha}\ll 1$ in the last inequality. By combining \eqref{RT1:part1} with interpolation, applying \eqref{eq:Mikado:stationary:phase:1} with $m=N$ and $a$ replaced by $a_{(\xi,i)}\nabla (\overline{v}_q+  z_{\ell_q})(\nabla \Phi_i)^{-1}$, we obtain for any $t\in [0,\mathfrak{t}_L]$
		\begin{align}\label{esti:RT1} 
			\|R^{\textrm{trans}}_1(t)\|_{C^\alpha_x} &\lesssim \sup_{i} \sum_{\xi \in \Lambda_i} 	\frac{\|a_{(\xi,i)}\nabla (\overline{v}_q+  z_{\ell_q})(\nabla \Phi_i)^{-1} \|_{C^0_{t,x}}}{\lambda_{q+1}^{1-\alpha}}\nonumber
			\\ &\quad+\sup_{i} \sum_{\xi \in \Lambda_i}\frac{\|a_{(\xi,i)}\nabla (\overline{v}_q+  z_{\ell_q})(\nabla \Phi_i)^{-1}\|_{C^0_tC^{N+\alpha}_x}}{\lambda_{q+1}^{N-\alpha}}
			\\ &\quad+\sup_{i} \sum_{\xi \in \Lambda_i}\frac{\|a_{(\xi,i)}\nabla (\overline{v}_q+  z_{\ell_q})(\nabla \Phi_i)^{-1}\|_{C^0_{t,x}}\ell_q^{-N-\alpha}}{\lambda_{q+1}^{N-\alpha}}\nonumber
			\\ & \lesssim  \frac{\bar{M} L^2\delta_{q+1}^{\sfrac12}\lambda_q\delta_{q}^{\sfrac12}}{\lambda_{q+1}^{1-7\alpha}}+ \frac{\bar{M} L^2\delta_{q+1}^{\sfrac12}\lambda_q\delta_{q}^{\sfrac12}\ell_q^{-N-\alpha}}{\lambda_{q+1}^{N-7\alpha}}
			= \frac{\bar{M}L^2\delta_{q+1}^{\sfrac12}\lambda_q\delta_{q}^{\sfrac12}}{\lambda_{q+1}^{1-7\alpha}}\left(1+\frac{\ell_q^{-N-\alpha} }{\lambda_{q+1}^{N-1}}\right). \nonumber
		\end{align}
		
		Let us move to the term $R^{\textrm{trans}}_2$. First, we use \eqref{eq:Phi:i:bnd:b} and \eqref{estimate:Da CN} to deduce for any $t\in [0,\mathfrak{t}_L]$ and $m\in \mathbb{N}_0$
		\begin{align*}
			\|D_{t,q} a_{(\xi,i)} (\nabla \Phi_i)^{-1} \|_{C^m_x} \lesssim \|&D_{t,q} a_{(\xi,i)} \|_{C^m_x}\|(\nabla \Phi_i)^{-1}\|_{C^0_x} 
			\\ &+  \|D_{t,q} a_{(\xi,i)} \|_{C^0_x}\|(\nabla \Phi_i)^{-1}\|_{C^m_x} \lesssim \frac{M}{C_\Lambda}L\delta_{q+1}^{\sfrac12} \tau_q^{-1}\ell_q^{-m}.
		\end{align*}
		In a similar manner to \eqref{esti:RT1}, we apply \eqref{eq:Mikado:stationary:phase:1} with $m=N$ and use \eqref{eq:Phi:i:bnd:b} along with interpolation to deduce for any $t\in [0,\mathfrak{t}_L]$
		\begin{equation}\label{esti:RT2}
			\begin{aligned}
				\|R^{\textrm{trans}}_2(t)\|_{C^\alpha_x}  &\lesssim \sup_{i} \sum_{\xi \in \Lambda_i} \frac{\|D_{t,q} a_{(\xi,i)} (\nabla \Phi_i)^{-1}  \|_{C^0_{t,x}}}{\lambda_{q+1}^{1-\alpha}}
				+ \sup_{i} \sum_{\xi \in \Lambda_i} \frac{\|D_{t,q} a_{(\xi,i)} (\nabla \Phi_i)^{-1} \|_{C^0_tC^{N+\alpha}_x}}{\lambda_{q+1}^{N-\alpha}} 
				\\ & \quad + \sup_{i} \sum_{\xi \in \Lambda_i} \frac{\|D_{t,q} a_{(\xi,i)} (\nabla \Phi_i)^{-1} \|_{C^0_{t,x}}\ell_q^{-N-\alpha}}{\lambda_{q+1}^{N-\alpha}}
				\\& \lesssim   \frac{\bar{M}L\delta_{q+1}^{\sfrac12}\tau_q^{-1}}{\lambda_{q+1}^{1-\alpha}}+ \frac{\bar{M}L\delta_{q+1}^{\sfrac12}\tau_q^{-1}\ell_q^{-N-\alpha}}{\lambda_{q+1}^{N-\alpha}} \lesssim \frac{\bar{M}L^2\delta_{q+1}^{\sfrac12}\lambda_{q}\delta_{q}^{\sfrac12}}{\lambda_{q+1}^{1-7\alpha}}\left(1+\frac{\ell_q^{-N-\alpha} }{\lambda_{q+1}^{N-1}}\right).
			\end{aligned}
		\end{equation}
		By combining \eqref{esti:RT1} with \eqref{esti:RT2}, we obtain for any $t\in [0,\mathfrak{t}_L]$
		\begin{align}\label{calculation for Rtransport}
			\|	R^{\textrm{trans}}(t)\|_{C^\alpha_x}&\stackrel{\eqref{Rsmall}}{\lesssim} \frac{ \bar{M}L^2\delta_{q+1}^{\sfrac12}\lambda_q\delta_{q}^{\sfrac12}}{\lambda_{q+1}^{1-7\alpha}}
			\stackrel{\eqref{choice:alpha}}{\lesssim} \bar{M}L^2  \delta_{q+2}\lambda_{q+1}^{-4\alpha}\leq \frac15 \bar{M}L^2 \delta_{q+2}\lambda_{q+1}^{-3\alpha},
		\end{align}
		where the extra power $\lambda_{q+1}^{-\alpha}$ is used to absorb the implicit constants.
	\end{proof}
	
	\begin{proposition}[Estimates for oscillation error] \label{Estimates for oscillation error}
	The oscillation error enjoys for any $t\in[0,\mathfrak{t}_L]$:
		\begin{align}\label{esti:oscillation}
			\|R^{\rm{osc}}(t)\|_{C^\alpha_x}\leq  \frac15 \bar{M}L^2 \delta_{q+2}\lambda_{q+1}^{-3\alpha}.
		\end{align}
	\end{proposition}
	\begin{proof}
		Recalling the definition of $R^{\rm{osc}}$ in \eqref{Reynold1}, it follows from the construction $(\xi \cdot \nabla) \phi_{(\xi)} = 0$ that we also have $(\xi \cdot \nabla) \mathbb{P}_{\geq \sfrac{\lambda_{q+1}}{2}}( \phi_{(\xi)}^2 )  = 0$. We then observe that the divergence of the high-frequency term $\left(\mathbb{P}_{\geq \sfrac{\lambda_{q+1}}{2}}(W_{(\xi)} \otimes W_{(\xi)}) \right)\circ \Phi_i$ may vanish. Similar to the calculations in \cite[Page 232]{BV19}, we have
		\begin{align*}
			\div & \left( a_{(\xi,i)}^2 (\nabla \Phi_i)^{-1} \left( \left(\mathbb{P}_{\geq \sfrac{\lambda_{q+1}}{2}}(W_{(\xi)} \otimes W_{(\xi)}) \right)\circ \Phi_i \right)  (\nabla \Phi_i)^{-T}   \right)  
			\\ &= \left( \left(\mathbb{P}_{\geq \sfrac{\lambda_{q+1}}{2}}(\phi_{(\xi)}^2) \right)\circ \Phi_i \right)    \div \left( a_{(\xi,i)}^2 (\nabla \Phi_i)^{-1}    (\xi \otimes \xi) (\nabla \Phi_i)^{-T}\right) ,
		\end{align*}
		which further leads to
		\begin{align*}
			R^{\rm{osc}} = \sum_i \sum_{\xi \in \Lambda_i} \mathcal{R} \left( \left( \left(\mathbb{P}_{\geq \sfrac{\lambda_{q+1}}{2}}( \phi_{(\xi)}^2) \right)\circ \Phi_i \right) \div \left( a_{(\xi,i)}^2 (\nabla \Phi_i)^{-1} (\xi\otimes \xi) (\nabla \Phi_i)^{-T}  \right)  \right)    \, .
		\end{align*}
		We then combine \eqref{eq:Phi:i:bnd:b} with \eqref{estimate:a CN} to deduce for any $m\in \mathbb{N}_0$
		\begin{align*}
			\left\|  \div \left( a_{(\xi,i)}^2 (\nabla \Phi_i)^{-1} (\xi\otimes \xi) (\nabla \Phi_i)^{-T} \right) \right\|_{C^m_x} & \lesssim \|a_{(\xi,i)}^2 \|_{C^{m+1}_x}+ \|a_{(\xi,i)}^2 \|_{C^0_x}\| (\nabla \Phi_i)^{-1}\|_{C^{m+1}_x}
			\\&\lesssim \|a_{(\xi,i)} \|_{C^{m+1}_x}\|a_{(\xi,i)}\|_{C^0_x}+\ell_q^{-m-1}\|a_{(\xi,i)}\|_{C^0_x}^2
	\lesssim \frac{M^2}{C_\Lambda^2} L^2 \delta_{q+1}\ell_q^{-m-1}.
		\end{align*} 
		Hence, applying \eqref{eq:Mikado:stationary:phase:2} with $m=N$ and using \eqref{eq:Phi:i:bnd:b}, \eqref{Rsmall}, along with interpolation, we derive for any $t\in[0,\mathfrak{t}_L]$
		\begin{equation}\label{calculation for Oscillation}
			\begin{aligned}
				\|R^{\rm{osc}}(t)\|_{C^\alpha_x} &\lesssim  \sup_{i} \sum_{\xi \in \Lambda_i} \frac{1}{\lambda_{q+1}^{1-\alpha}}\left\|  \div \left( a_{(\xi,i)}^2 (\nabla \Phi_i)^{-1} (\xi\otimes \xi) (\nabla \Phi_i)^{-T} \right)\right\|_{C^0_{t,x}}
				\\ &\quad+\sup_{i} \sum_{\xi \in \Lambda_i}\frac{1}{\lambda_{q+1}^{N-\alpha}}\left\|  \div \left( a_{(\xi,i)}^2 (\nabla \Phi_i)^{-1} (\xi\otimes \xi) (\nabla \Phi_i)^{-T} \right)\right\|_{C^0_tC^{N+\alpha}_x} 
				\\&\quad+\sup_{i} \sum_{\xi \in \Lambda_i}\frac{\ell_q^{-N-\alpha}}{\lambda_{q+1}^{N-\alpha}}\left\|  \div \left( a_{(\xi,i)}^2 (\nabla \Phi_i)^{-1} (\xi\otimes \xi) (\nabla \Phi_i)^{-T} \right)\right\|_{C^0_{t,x}}
				\\ & \lesssim  \frac{\bar{M}^2L^2\delta_{q+1}\ell_q^{-1}}{\lambda_{q+1}^{1-\alpha}}+\frac{\bar{M}^2L^2\delta_{q+1}\ell_q^{-N-1-\alpha}}{\lambda_{q+1}^{N-\alpha}} \leq \frac{\bar{M}^2L^2\delta_{q+1}\ell_q^{-1}}{\lambda_{q+1}^{1-2\alpha}}\left( 1+\frac{\ell_q^{-N-\alpha}}{\lambda_{q+1}^{N-1}} \right)
				\\ &\stackrel{\eqref{Rsmall}}{\lesssim} \frac{\bar{M}^2L^2\delta_{q+1}^{\sfrac12}\lambda_{q}\delta_{q}^{\sfrac12}}{\lambda_{q+1}^{1-8\alpha}}
				\stackrel{\eqref{choice:alpha}}{\lesssim} \bar{M}L^2 \delta_{q+2}\lambda_{q+1}^{-4\alpha}\leq \frac15 \bar{M}L^2 \delta_{q+2}\lambda_{q+1}^{-3\alpha},
			\end{aligned}
		\end{equation}
		where the additional factor $\lambda_{q+1}^{-\alpha}$ is utilized to absorb the implicit constants.
	\end{proof}

	\begin{proposition}[Estimates for Nash error] \label{Estimates for Nash error}
		The Nash error enjoys for any $t\in[0,\mathfrak{t}_L]$:
		\begin{align}\label{esti:Nash}
			\|R^{\rm{Nash}}(t)\|_{C^\alpha_x}\leq  \frac15 \bar{M}L^2 \delta_{q+2}\lambda_{q+1}^{-3\alpha}.
		\end{align}
	\end{proposition}
	\begin{proof}
		
		Recalling the definitions of $w_{q+1}^{(p)}$ and $w_{q+1}^{(c)}$ provided in \eqref{eq:Onsager:w:q+1:p} and \eqref{eq:Onsager:w:q+1:c}, respectively, we can decompose the Nash error into two components:
		\begin{align*}
			\mathcal{R}	\left( w_{q+1}\cdot \nabla (\overline{v}_q+z_{\ell_q}) \right) &= \mathcal{R}\left( 	w_{q+1}^{(p)}\cdot \nabla (\overline{v}_q+z_{\ell_q})\right)  + \mathcal{R}\left( 	w_{q+1}^{(c)}\cdot \nabla (\overline{v}_q+z_{\ell_q})\right) 
			\\&=  \sum_{i} \sum_{\xi \in \Lambda_i}  \mathcal{R} \left( a_{(\xi,i)} (\nabla \Phi_i)^{-1} W_{(\xi)}(\Phi_i) \cdot \nabla (\overline{v}_q+z_{\ell_q})\right) 
			\\  &\, \, +  \sum_i \sum_{\xi \in \Lambda_i} \mathcal{R} \left( \nabla a_{(\xi,i)} \times \left( (\nabla \Phi_i)^T (V_{(\xi)} (\Phi_i ))  \right) \cdot \nabla (\overline{v}_q+z_{\ell_q}) \right) 
			=: R^{\textrm{Nash}}_1+R^{\textrm{Nash}}_2.
		\end{align*}
		Let us first focus on the term $R^{\textrm{Nash}}_1$. It follows from \eqref{RT1:part1} that for any $m\in \mathbb{N}_0$
		\begin{align*}
			\| a_{(\xi,i)} (\nabla \Phi_i)^{-1}  \cdot \nabla (\overline{v}_q+z_{\ell_q})\|_{C^m_x} \leq \frac{M}{C_{\Lambda}}L^2\delta_{q+1}^{\sfrac12}\lambda_{q}^{1+6\alpha}\delta_{q}^{\sfrac12} \ell_q^{-m}.
		\end{align*}
		Similar to \eqref{esti:RT1}, by using \eqref{def tauq}, \eqref{eq:Phi:i:bnd:b} and \eqref{Rsmall}, and applying  \eqref{eq:Mikado:stationary:phase:1} with $m=N$, we deduce that for any $t\in[0,\mathfrak{t}_L]$
		\begin{equation}\label{calculation for Nash1}
			\begin{aligned}
				\|R^{\textrm{Nash}}_1(t)\|_{C^\alpha_x}
				& \lesssim  \frac{\bar{M} L^2\delta_{q+1}^{\sfrac12}\delta_{q}^{\sfrac12}\lambda_q}{\lambda_{q+1}^{1-7\alpha}}+\frac{\bar{M}L^2\delta_{q+1}^{\sfrac12}\delta_{q}^{\sfrac12}\lambda_q\ell_q^{-N-\alpha}}{\lambda_{q+1}^{N-7\alpha}}
				\\ & = \frac{\bar{M}L^2\delta_{q+1}^{\sfrac12}\lambda_{q}\delta_{q}^{\sfrac12}}{\lambda_{q+1}^{1-7\alpha}}\left( 1+\frac{\ell_q^{-N-\alpha}}{\lambda_{q+1}^{N-1}} \right)
				\stackrel{\eqref{Rsmall}}{\lesssim} \frac{\bar{M}L^2\delta_{q+1}^{\sfrac12}\lambda_{q}\delta_{q}^{\sfrac12}}{\lambda_{q+1}^{1-7\alpha}}
				\\ & \stackrel{\eqref{choice:alpha}}{\lesssim} \bar{M}L^2\delta_{q+2}\lambda_{q+1}^{-4\alpha}\leq \frac{1}{10} \bar{M}L^2 \delta_{q+2}\lambda_{q+1}^{-3\alpha}.  
			\end{aligned}
		\end{equation}
		Turning to the second term $R^{\textrm{Nash}}_2$, we first use \eqref{e:vq:1+N}, \eqref{eq:Phi:i:bnd:b} and \eqref{estimate:a CN} again to deduce that for any $m\in \mathbb{N}_0$
		\begin{align*}
			\| \nabla a_{(\xi,i)} \times (\nabla \Phi_i)^{-1}  \cdot \nabla (\overline{v}_q+z_{\ell_q})\|_{C^m_x} 
			&\lesssim \|\nabla a_{(\xi,i)}\|_{C^m_x} \|\nabla (\overline{v}_q+z_{\ell_q})\|_{C^0_x} 
		  +\| \nabla a_{(\xi,i)}\|_{C^0_x} \|\nabla (\overline{v}_q+z_{\ell_q})\|_{C^m_x} 
			\\  &\quad +\|\nabla a_{(\xi,i)}\|_{C^0_x}\| (\nabla \Phi_i)^{-1} \|_{C^m_x} \|\nabla (\overline{v}_q+z_{\ell_q})\|_{C^0_x} 
			\\ & \lesssim   \frac{M}{C_\Lambda}L\tau_q^{-1} \delta_{q+1}^{\sfrac12}\ell_q^{-m-1}.
		\end{align*}
		Compared to \eqref{calculation for Nash1}, we get additional $\ell_{q}^{-1}$ from $\nabla a_{(\xi,i)} \times (\nabla \Phi_i)^{-1}  \cdot \nabla (\overline{v}_q+z_{\ell_q})$, but we also have an extra $\lambda_{q+1}^{-1}$. Hence, by applying \eqref{eq:Mikado:stationary:phase:1} with $m=N$ and using \eqref{eq:Phi:i:bnd:b}, \eqref{eq:Onsager:ell:gap}, \eqref{choice:alpha}, \eqref{Rsmall} along with interpolation, we deduce for $t\in[0,\mathfrak{t}_L]$
		\begin{equation}\label{calculation for Nash2}
			\begin{aligned}
				\|R^{\textrm{Nash}}_2(t)\|_{C^\alpha_x}
				&	\lesssim  \frac{\bar{M}L\delta_{q+1}^{\sfrac12}\tau_q^{-1}\ell_q^{-1}}{\lambda_{q+1}^{2-\alpha}}+\frac{\bar{M}L\delta_{q+1}^{\sfrac12}\tau_q^{-1}\ell_q^{-N-1-\alpha}}{\lambda_{q+1}^{N+1-\alpha}} 
			\stackrel{\eqref{eq:Onsager:ell:gap}}{\lesssim} \frac{\bar{M}L\tau_q^{-1}\delta_{q+1}^{\sfrac12}}{\lambda_{q+1}^{1-\alpha}}\left(1+ \frac{\ell_q^{-N-\alpha}}{\lambda_{q+1}^{N-1}} \right)
				\\  &  \stackrel{\eqref{Rsmall}}{\lesssim} \frac{\bar{M}L^2\delta_{q+1}^{\sfrac12}\lambda_{q}\delta_{q}^{\sfrac12}}{\lambda_{q+1}^{1-7\alpha}} \stackrel{ \eqref{choice:alpha}}{\lesssim} \bar{M}L^2 \delta_{q+2}\lambda_{q+1}^{-4\alpha}\leq \frac{1}{10} \bar{M}L^2 \delta_{q+2}\lambda_{q+1}^{-3\alpha},
			\end{aligned}
		\end{equation}
		where we used the extra factor $\lambda_{q+1}^{-\alpha}$ to absorb the implicit constants in the last inequality. Combining \eqref{calculation for Nash1} with \eqref{calculation for Nash2} yields \eqref{esti:Nash}.
	\end{proof}

	\begin{proposition}[Estimates for corrector error] \label{Estimates for corrector error}
	The corrector error enjoys for any $t\in[0,\mathfrak{t}_L]$:
		\begin{align}\label{esti:Corrector}
			\|R^{\rm{cor}}(t)\|_{C^\alpha_x}\leq  \frac{3}{10} \bar{M}L^2 \delta_{q+2}\lambda_{q+1}^{-3\alpha}.
		\end{align}
	\end{proposition}
	\begin{proof}
		
		Recall that the corrector error consists of two components. One is the transport derivative of $w^{(c)}_{q+1}$ by the flow of $\overline{v}_{q}+z_{\ell_q}$. The other is $w_{q+1}^{(c)}\mathring{\otimes} w_{q+1}+w_{q+1}^{(p)} \mathring{\otimes} w_{q+1}^{(c)}$, which is easier to estimate. Therefore, using \eqref{esti:corrector w 0}, \eqref{esti:corrector w 1} and interpolation we first deduce
		\begin{equation}\label{estim:wc:alpha}
			\begin{aligned}
				\|w_{q+1}^{(c)}\|_{C^\alpha_x} \lesssim 		\|w_{q+1}^{(c)}\|_{C^1_x}^{\alpha}\|w_{q+1}^{(c)}\|_{C^0_x}^{1-\alpha}
				 \lesssim \bar{M}L(\lambda_{q+1}\delta_{q+1}^{\sfrac12})^{\alpha}\left( \frac{\delta_{q+1}^{\sfrac{1}{2}}\ell_q^{-1}}{\lambda_{q+1}} \right)^{1-\alpha}\lesssim \frac{\bar{M}L\delta_{q+1}^{\sfrac12}\ell_q^{-1}}{\lambda_{q+1}^{1-2\alpha}}.
			\end{aligned}
		\end{equation}
		Then we combine \eqref{esti:principal w 0}, \eqref{esti:principal w 1} with \eqref{estim:wc:alpha} to deduce for any $t\in[0,\mathfrak{t}_L]$
		\begin{equation}\label{calculation for Corrector}
			\begin{aligned}
				\|w_{q+1}^{(c)}\mathring{\otimes} w_{q+1}+w_{q+1}^{(p)} \mathring{\otimes} w_{q+1}^{(c)}\|_{C^\alpha_x} &\lesssim \|w_{q+1}^{(c)}\|_{C^\alpha_x}^2+ \|w_{q+1}^{(c)}\|_{C^\alpha_x} \|w_{q+1}^{(p)}\|_{C^0_x}+\|w_{q+1}^{(c)}\|_{C^0_x} \|w_{q+1}^{(p)}\|_{C^\alpha_x}
				\\ & \lesssim \frac{\bar{M}^2L^2\delta_{q+1}\ell_q^{-2}}{\lambda_{q+1}^{2-4\alpha}}+\frac{\bar{M}^2L^2\delta_{q+1}\ell_q^{-1}}{\lambda_{q+1}^{1-2\alpha}}\stackrel{\eqref{eq:Onsager:ell:gap}}{\lesssim }  \frac{\bar{M}^2L^2\delta_{q+1}^{\sfrac12}\delta_{q}^{\sfrac12}\lambda_q}{\lambda_{q+1}^{1-10\alpha}} 
				\\ &\stackrel{\eqref{choice:alpha}}{\leq } \bar{M}^2L^2 \delta_{q+2}\lambda_{q+1}^{-4\alpha}\leq  \frac{1}{10} \bar{M}L^2 \delta_{q+2}\lambda_{q+1}^{-3\alpha}.
			\end{aligned}
		\end{equation}
		For the remaining part of the corrector error, we express it as
		\begin{align*}
			\mathcal{R}\left( (\partial_t+(\overline v_q+z_{\ell_q})\cdot \nabla ) w_{q+1}^{(c)}\right)=&\sum_i \sum_{\xi\in \Lambda_i} \mathcal{R} \left( D_{t,q} \nabla a_{(\xi, i)}\times\left(  (\nabla \Phi_i)^{T} V_{\xi}(\Phi_i)\right) \right) 
			\\ &+\sum_i \sum_{\xi\in \Lambda_i} \mathcal{R} \left( \nabla a_{(\xi, i)}\times\left(  D_{t,q} (\nabla \Phi_i)^{T} V_{\xi}(\Phi_i)\right) \right)
			=: R^{\textrm{cor}}_1 + R^{\textrm{cor}}_2 .
		\end{align*}
		We first consider $R^{\textrm{cor}}_1 $ and use \eqref{e:vq:1+N}, \eqref{estimate:a CN} and \eqref{estimate:Da CN} to deduce for any $m\in \mathbb{N}_0$
		\begin{align*}
			\|D_{t,q}\nabla a_{(\xi,i)}\|_{C^m_x} \lesssim  \|D_{t,q} a_{(\xi,i)}\|_{C^{m+1}_x} &+\|\overline{v}_q+z_{\ell_q}\|_{C^{m+1}_x} \|a_{(\xi,i)}\|_{C^1_x} 
			\\ &+ \|\overline{v}_q+z_{\ell_q}\|_{C^1_x} \|a_{(\xi,i)}\|_{C^{m+1}_x}
			  \lesssim  \frac{M}{C_\Lambda} L \tau_q^{-1}\delta_{q+1}^{\sfrac12}\ell_q^{-m-1}.
		\end{align*}
		Then, using \eqref{eq:Phi:i:bnd:b} and the above estimate, we obtain for any $m\in \mathbb{N}_0$
		\begin{align*}
			\|D_{t,q} \nabla a_{(\xi, i)}\times (\nabla \Phi_i)^{T} \|_{C^m_x}  \lesssim 	\|D_{t,q} \nabla a_{(\xi, i)} \|_{C^m_x} +	\|D_{t,q} \nabla a_{(\xi, i)} \|_{C^0_x} \|\nabla \Phi_i\|_{C^m_x}
		 \lesssim \frac{M}{C_\Lambda} L\tau_q^{-1}\delta_{q+1}^{\sfrac12}\ell_q^{-m-1}.
		\end{align*}
		Similar to \eqref{calculation for Nash2}, we apply \eqref{eq:Mikado:stationary:phase:1} with $m=N$ and use \eqref{eq:Phi:i:bnd:b}, \eqref{eq:Onsager:ell:gap}, \eqref{Rsmall} and interpolation to obtain for any $t\in[0,\mathfrak{t}_L]$
		\begin{equation}\label{calculation for Corrector1}
			\begin{aligned}
				\|R^{\textrm{cor}}_1(t) \|_{C^\alpha_x}
				&\lesssim \frac{\bar{M}L\tau_q^{-1}\delta_{q+1}^{\sfrac12}\ell_q^{-1}}{\lambda_{q+1}^{2-\alpha}}+ \frac{\bar{M}L\tau_q^{-1}\delta_{q+1}^{\sfrac12}\ell_q^{-N-1-\alpha}}{\lambda_{q+1}^{N+1-\alpha}}
				\\&   \mathop{\lesssim} \limits_{\eqref{eq:Onsager:ell:gap}}^{\eqref{Rsmall}} \frac{\bar{M}L^2\lambda_{q}\delta_{q}^{\sfrac12}\delta_{q+1}^{\sfrac12}}{\lambda_{q+1}^{1-7\alpha}} \stackrel{\eqref{choice:alpha}}{\lesssim} \bar{M}L^2 \delta_{q+2}\lambda_{q+1}^{-4\alpha}\leq \frac{1}{10} \bar{M}L^2 \delta_{q+2}\lambda_{q+1}^{-3\alpha}.
			\end{aligned}
		\end{equation}
		Moving on to estimate $R^{\textrm{cor}}_2$, it follows from \eqref{eq:Phi:i:bnd:c} and \eqref{estimate:a CN} that for any $m\in \mathbb{N}_0$
		\begin{align*}
			\| \nabla a_{(\xi, i)}\times D_{t,q}(\nabla \Phi_i)^{T} \|_{C^m_x}\lesssim &\| \nabla a_{(\xi, i)}\|_{C^0_x}\| D_{t,q}(\nabla \Phi_i)^{T} \|_{C^m_x}
			\\&+\| \nabla a_{(\xi, i)}\|_{C^m_x}\| D_{t,q}(\nabla \Phi_i)^{T} \|_{C^0_x}\lesssim \frac{M}{C_\Lambda} L \tau_q^{-1} \delta_{q+1}^{\sfrac12}\ell_q^{-m-1}.
		\end{align*}
		In a similar manner to \eqref{calculation for Corrector1}, we apply \eqref{eq:Mikado:stationary:phase:1} with $m=N$ and use \eqref{eq:Phi:i:bnd:b}, \eqref{eq:Onsager:ell:gap}, \eqref{Rsmall} together with interpolation to derive for any $t\in[0,\mathfrak{t}_L]$
		\begin{equation}\label{calculation for Corrector2}
			\begin{aligned}
				\|R^{\textrm{cor}}_2(t)\|_{C^\alpha_x} 
				\lesssim &\frac{\bar{M} L \tau_q^{-1} \delta_{q+1}^{\sfrac12}\ell_q^{-1}}{\lambda_{q+1}^{2-\alpha}}+ \frac{\bar{M}L\tau_q^{-1}\delta_{q+1}^{\sfrac12}\ell_q^{-N-1-\alpha}}{\lambda_{q+1}^{N+1-\alpha}} 
				\leq \frac{1}{10} \bar{M}L^2 \delta_{q+2}\lambda_{q+1}^{-3\alpha},
			\end{aligned}
		\end{equation}
		where we utilized the additional power $\lambda_{q+1}^{-\alpha}$ to absorb the implicit constants. Combining \eqref{calculation for Corrector}, \eqref{calculation for Corrector1} and \eqref{calculation for Corrector2} yields \eqref{esti:Corrector}.
	\end{proof}

	\begin{proposition}[Estimates for commutator  error]\label{Estimates for commutator  error} 
		The commutator error enjoys for any $t\in[0,\mathfrak{t}_L]$:
		\begin{align}\label{esti:Commutator }
			\|R^{\rm{com}}(t)\|_{C^0_x}\leq  \frac{1}{10} \bar{M}L^2 \delta_{q+2}\lambda_{q+1}^{-3\alpha}.
		\end{align}
	\end{proposition}
	\begin{proof}
		
		Finally, we estimate $R^{\rm{com}}$ defined in \eqref{Reynold1}. By using \eqref{z ps}, \eqref{zq+1-zq} and mollification estimate \eqref{estimate:molli3}, we obtain the following for any $t\in[0,\mathfrak{t}_L]$
		\begin{equation}\label{estimate:zq-zl}
			\aligned
			\|z_{q+1}(t)-z_{\ell_{q}}(t)\|_{C^0_x}  & \leq  \|z_{q+1}(t)-z_q(t)\|_{C^0_x} +	\|z_{\ell_q}(t)-z_q(t)\|_{C^0_x} 
			\\ &\lesssim L \lambda_{q}^{-\sfrac23+2\alpha} + \ell_q \|z_q\|_{C^0_{t}C^1_x}  \lesssim L( \ell_q +\lambda_{q}^{-\sfrac23+2\alpha}).
			\endaligned
		\end{equation}
		Note that $\ell_{q}\leq \lambda_{q}^{-1} \leq \lambda_{q}^{-\sfrac23}$, it suffices to control $\lambda_{q}^{-\sfrac23}$. 
		By using $b<\sqrt{\frac{1}{3\beta}}$ and $9b\alpha<\sfrac23-2b^2\beta$, we have for $q\geq 1$
		\begin{align}\label{lmabda 2/3 deltaq+2}
			\frac{\lambda_{q}^{-\sfrac23} }{\delta_{q+2}\lambda_{q+1}^{-9\alpha}} \lesssim \lambda_{q}^{-\sfrac23}\lambda_{q+1}^{9\alpha}\lambda_{q+2}^{2\beta} \lesssim \lambda_{q}^{9b\alpha+2b^2\beta-\sfrac23}\leq1,
		\end{align}
		which ensures that \eqref{estimate:zq-zl} can be bounded by $L\delta_{q+2}\lambda_{q+1}^{-5\alpha}$.
		We then use \eqref{estimate:zq-zl} to obtain for any $t\in[0,\mathfrak{t}_L]$
		\begin{equation}\label{calculation for Commutator}
			\begin{aligned}
				\|R^{\rm{com}}(t)\|_{C^0_x} &\lesssim  \|v_{q+1}(t)\|_{C^0_x}	\|z_{q+1}(t)-z_{\ell_q}(t)\|_{C^0_x}+\|z_{q+1}(t)\|_{C^0_x}	\|z_{q+1}(t)-z_{\ell_q}(t)\|_{C^0_x}
				\\ &\quad +\|z_{\ell_q}(t)\|_{C^0_x}	\|z_{q+1}(t)-z_{\ell_q}(t)\|_{C^0_x}
				\\ & \lesssim  (\bar{M}L\lambda_{1}^{\sfrac{3\alpha}{2}}+L) L\delta_{q+2}\lambda_{q+1}^{-5\alpha}\leq \frac{1}{10} \bar{M} L^2\delta_{q+2}\lambda_{q+1}^{-3\alpha},
			\end{aligned}
		\end{equation}
		where we utilized the additional factor $\lambda_{q+1}^{-\alpha}$ to absorb the implicit constants.
	\end{proof}
	Summarizing the estimates \eqref{esti:transport}, \eqref{esti:oscillation}, \eqref{esti:Nash}, \eqref{esti:Corrector} and \eqref{esti:Commutator }, we deduce for any $t\in [0,\mathfrak{t}_L]$
	\begin{align*}
		\|\mathring{R}_{q+1}(t)\|_{C^0_{x}}\leq \bar{M} L^2 \delta_{q+2}\lambda_{q+1}^{-3\alpha}.
	\end{align*}

	\subsection{Estimates on the energy}\label{estimate on energy}
To complete the proof of Proposition~\ref{p:iteration}, it remains to verify that the iterative energy bound \eqref{estimate:energy} holds at the level $q+1$. This is established in the following:
	\begin{proposition}\label{estimate:final:energy}
		It holds for $t\in[0,\mathfrak{t}_L]$ 
		\begin{align}\label{1/4}
			\left|    e(t)-L^2\frac{\delta_{q+2}}{2}-\|(v_{q+1}+{z_{q+1}})(t )\|_{L^2}^2  +2 \int_{0}^{t} \int_{\T^3} (v_{q+1}+z_{q+1}) \cdot \partial_t z_{q+1}  \dif x \dif s  \right|  \leq L^2 \delta_{q+2}\ell_{q}^{\alpha}.
		\end{align}
	\end{proposition} 
	\begin{proof}
		By the definition of $v_{q+1}=\overline{v}_q+w_{q+1}^{(p)}+w_{q+1}^{(c)}$, we find
		\begin{align}\label{Eq+1}
			\begin{aligned}
				&  e(t)-L^2\frac{\delta_{q+2}}{2}-\|(v_{q+1}+{z_{q+1}})(t)\|_{L^2}^2+2 \int_{0}^{t} \int_{\T^3} (v_{q+1}+z_{q+1}) \cdot \partial_t z_{q+1}  \dif x \dif s
				\\ &=  e(t)-L^2\frac{\delta_{q+2}}{2}-\|(\overline{v}_q+{z_{\ell_q}})(t)\|_{L^2}^2+2 \int_{0}^{t} \int_{\T^3} (\overline{v}_q+z_{\ell_q}) \cdot \partial_t z_{\ell_{q}}  \dif x \dif s -
				\int_{\mathbb{T}^3} |w_{q+1}^{(p)}|^2(t)\dif x 
				\\&\quad  -2 \int_{\mathbb{T}^3}(z_{q+1}-z_{\ell_{q}})\cdot (\overline{v}_q+z_{\ell_{q}}) (t)\dif x -\int_{\mathbb{T}^3}|z_{q+1}-z_{\ell_{q}}|^2 (t)\dif x 
				\\& \quad - \int_{\mathbb{T}^3}\left( |w_{q+1}^{(c)}|^2+2w_{q+1}^{(p)}\cdot w_{q+1}^{(c)}\right)(t) \dif x-  \int_{\mathbb{T}^3}(\overline{v}_q+z_{q+1})\cdot w_{q+1} (t)\dif x   
				\\ &\quad +2 \int_{0}^{t} \int_{\T^3} (v_{q+1}+z_{q+1}) \cdot \partial_t z_{q+1}  \dif x \dif s -2 \int_{0}^{t} \int_{\T^3} (\overline{v}_q+z_{\ell_q}) \cdot \partial_t z_{\ell_{q}}  \dif x \dif s  .
			\end{aligned}
		\end{align}
		Let us first focus on the second line. Taking trace on both sides of \eqref{eq:w:q+1:is:good} and using the fact that $ \mathring{\overline{R}}_q $ is traceless, we deduce for any $t\in [0,\mathfrak{t}_L]$
		\begin{align*}
			|w_{q+1}^{(p)}|^2=3	\sum_i \rho_{q,i} + \sum_i \sum_{\xi \in \Lambda_i} \tr \left[ a_{(\xi,i)}^2 (\nabla \Phi_i)^{-1} \left( \left(\mathbb{P}_{\geq \sfrac{\lambda_{q+1}}{2}}(W_{(\xi)} \otimes W_{(\xi)}) \right)\circ \Phi_i \right)  (\nabla \Phi_i)^{-T} \right].
		\end{align*}
		Recalling $\sum_i  \int_{\T^3} \rho_{q,i} \dif x =\rho_q$ and integrating on both sides, we deduce
		\begin{align*}
		\int_{\T^3}	|w_{q+1}^{(p)}(t)|^2 \dif x&=   e(t)-L^2\frac{\delta_{q+2}}{2}-\|(\overline{v}_{q}+{z_{\ell_q}})(t)\|_{L^2}^2+2 \int_{0}^{t} \int_{\T^3} (\overline{v}_q+z_{\ell_q}) \cdot \partial_t z_{\ell_{q}}  \dif x \dif s 
			\\ &+ \sum_i \sum_{\xi \in \Lambda_i} \int_{\T^3} \tr \left[ a_{(\xi,i)}^2 (\nabla \Phi_i)^{-1} \left( \left(\mathbb{P}_{\geq \sfrac{\lambda_{q+1}}{2}}(W_{(\xi)} \otimes W_{(\xi)}) \right)\circ \Phi_i \right)  (\nabla \Phi_i)^{-T} \right](t) \dif x.
		\end{align*}
		Then we use \eqref{eq:Phi:i:bnd:b}, \eqref{estimate:a CN}, \eqref{choice:alpha} and \eqref{esti:integral} to obtain for any $t\in [0,\mathfrak{t}_L]$
		\begin{equation}\label{energy1}
			\aligned
			\sum_i \sum_{\xi \in \Lambda_i} & \left| \int_{\mathbb{T}^3} a_{(\xi,i)}^2 \tr \left[  (\nabla \Phi_i)^{-1} \xi \otimes \xi  (\nabla \Phi_i)^{-T}  \right]\left( \Big( \mathbb{P}_{\geq \sfrac{\lambda_{q+1}}{2}}(\phi_{(\xi)}^2) \Big) \circ\Phi_i \right) (t ) \dif x\right| 
			\\  &\lesssim 	\sup_i \sum_{\xi \in \Lambda_i}  \frac{\|a_{(\xi,i)}^2 (\nabla \Phi_i)^{-1} (\nabla \Phi_i)^{-T} \|_{C^1_x} }{\lambda_{q+1}}+\sup_i \sum_{\xi \in \Lambda_i}\frac{\ell_q^{-1}\|a_{(\xi,i)}^2 (\nabla \Phi_i)^{-1} (\nabla \Phi_i)^{-T} \|_{C^0_x} }{\lambda_{q+1}}
			\\& \lesssim \bar{M}^2L^2 \delta_{q+1}\frac{\ell_q^{-1}}{\lambda_{q+1}}= \bar{M}^2L^2 \frac{\delta_{q}^{\sfrac12}\delta_{q+1}^{\sfrac12}\lambda_{q}^{1+6\alpha}}{\lambda_{q+1}}\leq \frac15 L^2 \delta_{q+2} \ell_q^{\alpha} ,
			\endaligned
		\end{equation}
		where we used \eqref{choice:alpha} and $\bar{M}^2\lambda_{q}^{-\alpha}\ll 1$ in the last inequality. This completes the bound for the second line of \eqref{Eq+1}. 
		
		Going back to \eqref{Eq+1}, we continue to control the remaining terms. We first use \eqref{e:vq:0} and \eqref{estimate:zq-zl} to derive for any $t\in [0,\mathfrak{t}_L]$
		\begin{equation}\label{energy2}
			\aligned
			&  2    \left| 	 \int_{\mathbb{T}^3}(z_{q+1}-z_{\ell_{q}})(\overline{v}_q+z_{\ell_{q}}) (t)\dif x  \right| +   \int_{\mathbb{T}^3}|z_{q+1}-z_{\ell_{q}}|^2 (t)\dif x 
			\\& \qquad \lesssim \|z_{q+1}-z_{\ell_{q}}\|_{C^0_{t,x}}\|\overline{v}_q+z_{\ell_{q}}\|_{C^0_{t,x}}+\|z_{q+1}-z_{\ell_{q}}\|_{C^0_{t,x}}^2
			\\ &\qquad \lesssim  L ( \ell_q +\lambda_{q}^{-\sfrac23+2\alpha})(\bar{M}L\lambda_{1}^{\sfrac{3\alpha}{2}}+L)
		\lesssim \bar{M} L^2\lambda_{q}^{-\sfrac23+4\alpha}\stackrel{\eqref{lmabda 2/3 deltaq+2}}{\leq}  \bar{M} L^2 \delta_{q+2}\lambda_{q+1}^{-3\alpha}\leq \frac15 L^2 \delta_{q+2}\ell_{q}^{\alpha},
			\endaligned
		\end{equation}
		where we chose $a$ large enough to have $\bar{M}^2\lambda_{q+1}^{-\alpha}\ll1$ in the last inequality.
		Then, the estimates \eqref{esti:principal w 0}, \eqref{esti:corrector w 0} and \eqref{choice:alpha} imply for any $t\in [0,\mathfrak{t}_L]$
		\begin{equation}\label{energy3}
				\aligned
			\left| 	  \int_{\mathbb{T}^3}\left(|w_{q+1}^{(c)}|^2+2w_{q+1}^{(p)}\cdot w_{q+1}^{(c)} \right)(t) \dif x\right| 
			& \lesssim \bar{M}^2 L^2 \delta_{q+1}\frac{\ell_q^{-2}}{\lambda_{q+1}^2}+  \bar{M}^2 L^2  \delta_{q+1} \frac{\ell_q^{-1}}{\lambda_{q+1}}
			\\ &\lesssim \bar{M}^2  L^2 \frac{\delta_{q}^{\sfrac12}\delta_{q+1}^{\sfrac12}\lambda_{q}^{1+6\alpha}}{\lambda_{q+1}} \leq \frac15 L^2 \delta_{q+2}\ell_{q}^{\alpha}.
			\endaligned
		\end{equation}
		
		Moving on to estimate $\int_{\mathbb{T}^3}(\overline{v}_q+z_{q+1})\cdot w_{q+1} \dif x $, we recall that $w_{q+1}$ can be written as the curl of a vector field (c.f. \eqref{eq:Onsager:w:q+1})
		\begin{align*}
			w_{q+1} = \sum_i \sum_{\xi \in \Lambda_i}  \curl \left( a_{(\xi,i)} \, (\nabla \Phi_i)^T (V_{(\xi)}\circ \Phi_i) \right).
		\end{align*}
		Then integrating by parts  and using the estimates \eqref{e:vq:1+N}, \eqref{eq:Phi:i:bnd:b}, \eqref{estimate:a CN}, \eqref{choice:alpha} and \eqref{eq:Mikado:bounds} we obtain for any $t\in [0,\mathfrak{t}_L]$
		\begin{equation}\label{curl wq+1}
			\aligned
			&    \left|\int_{\mathbb{T}^3}(\overline{v}_q+z_{q+1})\cdot w_{q+1} (t)\dif x  \right| 
		 \lesssim  \sup_i \sum_{\xi \in \Lambda_i} \| a_{(\xi,i)} \, (\nabla \Phi_i)^T (V_{(\xi)}\circ \Phi_i)\| _{C_{t,x}^0}\|\overline v_q+z_{q+1}\|_{C^0_tC^1_{x}}
			\\ &  \qquad \qquad \lesssim \frac{\bar{M}^2L \delta_{q+1}^{\sfrac12}\tau_q^{-1}}{\lambda_{q+1}} = \frac{ \bar{M}^2 L^2\delta_{q}^{\sfrac12}\delta_{q+1}^{\sfrac12}\lambda_{q}^{1+6\alpha}}{\lambda_{q+1}}\leq \bar{M}^2L^2 \delta_{q+2}\lambda_{q+1}^{-5\alpha}
			\leq \frac15 L^2  \delta_{q+2}  \ell_q^{\alpha} , 
			\endaligned
		\end{equation}
		where we used again $\bar{M}^2\lambda_{q+1}^{-\alpha}\ll1$ in the last inequality. 
	
		Let us now focus on the last line of \eqref{Eq+1}. By using $v_{q+1}=\overline{v}_q+w_{q+1}$ and performing a direct computation, we derive
		\begin{equation}\label{energy:addition}
			\aligned
			& \int_{0}^{t} \int_{\T^3} (v_{q+1}+z_{q+1}) \cdot \partial_t z_{q+1}  \dif x \dif s - \int_{0}^{t} \int_{\T^3} (\overline{v}_q+z_{\ell_q}) \cdot \partial_t z_{\ell_{q}}  \dif x \dif s  
			 \\ &\quad =  \int_{0}^{t} \int_{\T^3} \overline{v}_q \cdot (\partial_t z_{q+1} -\partial_t z_{\ell_q} )	\dif x \dif s  + \int_{0}^{t} \int_{\T^3} w_{q+1} \cdot \partial_t z_{q+1} \dif x \dif s
			 \\ &\qquad+\frac12  \int_{0}^{t} \left( \frac{\dif}{\dif t}  \int_{\T^3}  |z_{q+1}|^2 \dif x \right) \dif s - \frac12  \int_{0}^{t} \left( \frac{\dif}{\dif t}  \int_{\T^3}  |z_{\ell_q}|^2 \dif x \right)\dif s 
			 \\& \quad= \int_{0}^{t} \int_{\T^3} \overline{v}_q \cdot \partial_t( z_{q+1} - z_{q} )	\dif x \dif s  + \int_{0}^{t} \int_{\T^3} \overline{v}_q \cdot \partial_t( z_{q} - z_{\ell_q} )	\dif x \dif s  + \int_{0}^{t} \int_{\T^3} w_{q+1} \cdot \partial_t z_{q+1} \dif x \dif s
			 \\&\qquad+ \frac12 \left(\|z_{q+1}(t)\|^2_{L^2}-\|z_{\ell_q}(t)\|^2_{L^2}\right).
			 \endaligned
		\end{equation}
		We next control each term in \eqref{energy:addition} separately, starting with $	 \int_{0}^{t} \int_{\T^3} \overline{v}_q \cdot \partial_t( z_{q+1} - z_{q} )	\dif x \dif s$. Integrating by parts with respect to the temporal variable and substituting  $\partial_t \overline{v}_q$ by the equation \eqref{gluing euler q}, we obtain
		\begin{align*}
			 \int_{0}^{t} & \int_{\T^3} \overline{v}_q \cdot \partial_t( z_{q+1} - z_{q} )	\dif x \dif s  = \int_{\T^3} \overline{v}_q \cdot (z_{q+1}- z_{q}  )  \bigg|_0^t	\dif x 
			  -\int_{0}^{t} \int_{\T^3}  \partial_t \overline{v}_q \cdot ( z_{q+1} - z_{q} )	\dif x \dif s  
			 \\ &=  \int_{\T^3} \overline{v}_q \cdot (z_{q+1}- z_{q}  )  \bigg|_0^t	\dif x 
			-\int_{0}^{t} \int_{\T^3}  (\div \mathring{\overline{R}}_q-\div((\overline v_q+z_{\ell_q}) \otimes(\overline v_q+z_{\ell_q}))-\nabla \overline{p}_q)\cdot ( z_{q+1} - z_{q} )	\dif x \dif s  
			 \\ &= \int_{\T^3} \overline{v}_q (t) \cdot (z_{q+1}(t)  - z_{q} (t) ) 	\dif x 
			 + \int_{0}^{t} \int_{\T^3}   \big(\mathring{\overline{R}}_q - (\overline v_q+z_{\ell_q}) \otimes(\overline v_q+z_{\ell_q})\big): \nabla (z_{q+1}-z_q)^T  \dif x \dif s ,
		\end{align*}
		where the last line arises from integrating by parts with respect to the spatial variable. Applying \eqref{zq+1-zq}, \eqref{estimate:vellzell}, \eqref{e:vq:0} and \eqref{e:Rq:1} we have
	  \begin{equation}\label{energy:addition1}
	  	\aligned
	  	\bigg|  \int_{0}^{t} \int_{\T^3} \overline{v}_q \cdot \partial_t (z_{q+1} - z_{q} )	\dif x \dif s \bigg| & \lesssim L (\|\mathring{\overline{R}}_q\|_{C^0_{t,x}}+\|\overline{v}_q\|^2_{C^0_{t,x}}+\|z_{\ell_{q}}\|^2_{C^0_{t,x}}) \|z_{q+1}-z_q\|_{C^0_tC^1_x}
	  \\&	\leq \bar{M}L^4\lambda_q^{-2/3+4\alpha} \stackrel{\eqref{lmabda 2/3 deltaq+2}}{\leq}  \bar{M} L^4 \delta_{q+2}\lambda_{q+1}^{-3\alpha}\leq \frac{1}{20} L^2 \delta_{q+2}\ell_{q}^{\alpha},
	  \endaligned
	  \end{equation}
	  where we used $\bar{M}L^2 \lambda_{q+1}^{-\alpha}\ll 1$ in the last inequality. Moving to the second term on the right-hand side of \eqref{energy:addition}, we use \eqref{z ps}, \eqref{e:vq:0} and \eqref{estimate:molli2} to obtain
	  \begin{equation}\label{energy:addition2}
	  	\aligned
	  	\bigg|  \int_{0}^{t} \int_{\T^3} \overline{v}_q \cdot \partial_t( z_{q} - z_{\ell_q} )	\dif x \dif s  \bigg| & \lesssim L\|\overline{v}_q\|_{C^0_{t,x}}\| \partial_t z_q - \partial_t z_q * \varphi_{\ell_{q}}\|_{C^0_x} 
	  	\lesssim L \ell_{q}^2  \|\overline{v}_q\|_{C^0_{t,x}} \|B*\psi_{\iota_{q}}\|_{C^1_tC^2_x} 
	  	\\ &  \lesssim \bar{M} L^2 \ell_{q}^2 \iota_{q}^{-(\sfrac12+\alpha)} \lambda_{q}^{2\alpha}\|B\|_{C^{1/2-\alpha}_tC^2_x} \leq  \frac{ \bar{M} L^3\delta_{q+1}}{\lambda_{q}^{\sfrac43}\delta_q} \leq \frac{1}{20} L^2 \delta_{q+2}\ell_{q}^{\alpha},
	  	\endaligned
	  \end{equation}
	  where we used the relation $3b^2\beta+6\alpha<1$ to have $\lambda_{q}^{-2b\beta+2\beta-\sfrac43+4\alpha}\leq \lambda_{q}^{-2b^2\beta}$ in the last inequality.
	  
	  Let us focus on the third term on the right-hand side of \eqref{energy:addition}. By using the same computation as \eqref{curl wq+1} and the estimates \eqref{eq:Phi:i:bnd:b}, \eqref{estimate:a CN} and \eqref{eq:Mikado:bounds}, we obtain for any $t\in [0,\mathfrak{t}_L]$
		\begin{equation}\label{energy:addition3}
			\aligned
			  \left|\int_{\mathbb{T}^3} w_{q+1} \cdot \partial_t z_{q+1} \dif x  \right| 
		&	\lesssim  \sup_i \sum_{\xi \in \Lambda_i} \| a_{(\xi,i)} \, (\nabla \Phi_i)^T (V_{(\xi)}\circ \Phi_i)\| _{C_{t,x}^0}\| B*\psi_{\iota_{q+1}}\|_{C^1_tC^1_{x}}
			\\ & \lesssim  \frac{\bar{M}L \delta_{q+1}^{\sfrac12}}{\lambda_{q+1}} \iota_{q+1}^{-(\sfrac12+\alpha)} \|B\|_{C^{1/2-\alpha}_tC^1_x} \lesssim \frac{\bar{M}L^2 \delta_{q+1}^{\sfrac12}}{\lambda_{q+1}^{\sfrac13-2\alpha}}
			\\&  \leq \bar{M}L^2 \delta_{q+2}\lambda_{q+1}^{-5\alpha}
			\leq \frac{1}{20} L^2 \delta_{q+2}  \ell_q^{\alpha} , 
			\endaligned
		\end{equation}
		where the last line makes use of the relation $7\alpha<\frac13+\beta-2b\beta$ to have $\lambda_{q+1}^{-\sfrac13-\beta+7\alpha}\leq \lambda_{q+1}^{-2b\beta}$. 
		
		For tha last line of \eqref{energy:addition}, we use \eqref{z ps} and \eqref{estimate:zq-zl} to derive for any $t \in [0,\mathfrak{t}_L]$
			\begin{equation}\label{energy:addition4}
			\aligned
	\big|	\|z_{q+1}(t)\|^2_{L^2}-\|z_{\ell_q}(t)\|^2_{L^2} \big|  &\leq 	\|z_{q+1}(t)-z_{\ell_{q}}(t)\|_{C^0_x} (	\|z_{q+1}(t)\|_{L^2}+\|z_{\ell_q}(t)\|_{L^2})
			\\ & \lesssim L^2 (\ell_{q}+ \lambda_{q}^{-\sfrac23+2\alpha} ) \stackrel{\eqref{lmabda 2/3 deltaq+2}}{\leq}  \frac{1}{20} L^2 \delta_{q+2}\ell_{q}^{\alpha}.
			\endaligned
		\end{equation}
		We then collect the estimates \eqref{energy:addition1}, \eqref{energy:addition2}, \eqref{energy:addition3} and \eqref{energy:addition4} to obtain
		\begin{align}\label{energy5}
			 \left|   \int_{0}^{t} \int_{\T^3} (v_{q+1}+z_{q+1}) \cdot \partial_t z_{q+1}  \dif x \dif s - \int_{0}^{t} \int_{\T^3} (\overline{v}_q+z_{\ell_q}) \cdot \partial_t z_{\ell_{q}}  \dif x \dif s   \right| \leq  \frac15L^2 \delta_{q+2}\ell_{q}^{\alpha}.
		\end{align}
		Finally, by combining the estimates \eqref{energy1},  \eqref{energy2},  \eqref{energy3}, \eqref{curl wq+1} and \eqref{energy5}, we obtain \eqref{1/4}.
	\end{proof}
	
	With the estimate \eqref{1/4} in hand and by the fact $\ell_{q}^{\alpha}+\lambda_{q+1}^{-\alpha}\leq \frac12$, we have for any $t \in [0,\mathfrak{t}_L]$
	\begin{align*}
	L^2	\delta_{q+2}\lambda_{q+1}^{-\alpha} & \leq L^2\frac{\delta_{q+2}}{2}- L^2 \delta_{q+2}\ell_{q}^{\alpha}
		\\ &	\leq   e(t)- \|(v_{q+1}+{z_{q+1}})(t )\|_{L^2}^2  + 2\int_{0}^{t} \int_{\T^3} (v_{q+1}+z_{q+1}) \cdot \partial_t z_{q+1}  \dif x \dif s\leq L^2 \delta_{q+2}.
	\end{align*}
	This concludes the proof of Proposition~\ref{p:iteration}.
	
	\section{Construction of global-in-time solutions with prescribed initial condition}\label{cauchy problem}
	This section is dedicated to proving our second main result, Theorem~\ref{Onsager:theorem:cauchy:problem}.  Our goal is to establish the existence of non-unique, global-in-time, probabilistically strong and analytically weak solutions in $C_tC^{\bar{\beta}-}_x$ to the Euler system \eqref{eul1}, applicable for every given divergence-free initial condition $u^{in}\in C^{\bar{\beta}}_x$, $0<\bar{\beta}<1/3$. 
	To this end, we adjust the convex integration scheme developed in Sections~\ref{s:in}-\ref{sec:inductive estimates}, following the approach in \cite{KMY22}, to incorporate a convolution approximation of the initial data $u^{in}$ in the iterative equation \eqref{euler1:new}. This allows us to recover the prescribed initial data $u^{in}$ in the limit.
	To establish global existence and address constraints imposed by the stopping time, we employ a gluing procedure for convex integration solutions, as outlined in \cite{HZZ21markov}. Specifically, for any given initial data $u^{in}\in C^{\bar{\beta}}_x$, we first construct probabilistically strong solutions up to a suitable stopping time using convex integration. We then use the final value at this stopping time as a new initial condition to reapply the convex integration scheme, thereby constructing another strong solution that extends beyond the stopping time. By repeating these steps, we extend the convex integration solutions as probabilistically strong solutions defined over the entire time interval $[0,\infty)$.

	In this section, the notations generally follow the previous conventions, with any necessary modifications explained below. First, we fix a probability space $(\Omega, \mathcal{F} ,\mathbf{P})$ with a $GG^*$-Wiener process $B$. Let $u^{in}\in C^{\bar{\beta}}_x$ $\mathbf{P}$-a.s. be an arbitrary random initial condition to \eqref{eul1} independent of the given Wiener process $B$. We denote $(\mathcal{F}_{t})_{t\geq0}$ as the augmented joint canonical filtration on $(\Omega,\mathcal{F})$ generated by $B$ and $u^{in}$ (c.f. \cite[Section 2.1]{LR15}). Thus, $B$ is an $(\mathcal{F}_{t})_{t\geq0}$-Wiener process and $u^{in}$ is $\mathcal{F}_0$-measurable. 
	We still decompose the Euler system \eqref{eul1} into $u:=v+z$, where $z:=B$ with $B(0)=0$, but we incorporate initial data into the random PDE. More precisely, the difference $v:=u-z$ solves the following  random PDE equation with the same initial data:
	\begin{equation}\label{nonlinear:new}
		\aligned
		\partial_t v+\div((v+z)\otimes (v+z))+\nabla p&=0,
		\\\div v&=0,
		\\ v(0)&=u^{in}.
		\endaligned
	\end{equation}
	Here, $z$ is divergence-free due to the assumptions on the noise, and $p$ denotes the pressure term associated with $v$. As before, the iteration is indexed by a parameter $q\in \mathbb{N}$, and at each $q$ step, we construct a pair $(v_q,\mathring{R}_q)$ to solve the following system:
	\begin{equation}\label{euler1:new}
		\aligned
		\partial_t v_q +\div((v_q+z_q)\otimes (v_q+z_q))+\nabla p_q&=\div
		\mathring{R}_q,
		\\\div v_q&=0,
		\\ v_q(0)&=u^{in}*\varphi_{\ell_{q-1}},
		\endaligned
	\end{equation}
	where $z_q=\mathbb{P}_{\leq f(q)} z$ with $f(q)=\lambda_{q}^{\sfrac23}$ and $\mathbb{P}_{\leq f(q)} $ is the Fourier
	multiplier operator defined in Subsection~\ref{notation 1}. In the right-hand side of \eqref{euler1:new}, $\mathring{R}_q\in \mathcal{S}_0^{3\times3}$ and we place the trace part into the pressure term. The parameters $\lambda_{q},\delta_{q}$, and the mollification parameter $\ell_{q}$ retain the same structure as in Subsection~\ref{sec:parameters}, but the values of the determining parameters $a,b,\alpha$ may differ, as additional conditions must be satisfied. Details regarding the selection of these parameters are provided in Subsection~\ref{sec:mollification:new} below.
	
	Let $L>1$ sufficiently large be given and define the stopping time
	\begin{equation}\label{stopping time ps:new}
		\aligned
		\mathfrak{t}_L:=& \inf\{t\geq0,\|z(t)\|_{H^{5/2+\gamma}}\geq L/C_S\}
		\wedge  L.
		\endaligned
	\end{equation}
	By Sobolev embedding, the following estimates also hold on $[0,\mathfrak{t}_L]$
	\begin{equation}\label{z ps:new}
		\| z_q(t)\|_{L^\infty} \leq L, \qquad
		\| z_q(t)\|_{C^1_x} \leq L, \qquad 
		\|z_q(t)\|_{C^2_x}  \leq L\lambda_{q}^{\sfrac23}.
	\end{equation}
	Moreover, without loss of generality, we can suppose that
	\begin{align}\label{eq:u0} 
		\|u^{in}\|_{C_x^{\bar{\beta}}}\leq N,
		\quad
		\mathbf{P}-a.s. 
	\end{align}
	for some finite constant $N$. 
	Indeed, for a general initial condition $u^{in} \in {C^{\bar{\beta}}_x}$  $\mathbf{P}$-a.s.,  one defines $\Omega_N := \{ N-1 \leq \| u^{in} \|_{C^{\bar{\beta}}_x} < N\} \in \mathcal{F}_0$.
	Then, given the existence of infinitely many solutions $u^N$ on each $\Omega_N$ one can define $u := \sum_{N \in \N} u^N \mathbf{1}_{\Omega_N}$, solving the equation with initial condition $u^{in}$.
	We maintain this additional assumption on the initial condition throughout the convex integration step in Proposition~\ref{p:iteration1}. 
	We also denote 
	\begin{align}\label{def:varsigma}
		\varsigma_q=\delta_{q}, \ q\in \mathbb{N} \setminus \{2\}, \qquad \varsigma_2=K\delta_2,
	\end{align}
	where $K\geq 1$ is a large
	constant used in the proof of Theorem~\ref{Onsager:theorem:cauchy:problem} to distinguish different solutions. 
	
	Under the above assumptions, our main iteration reads as follows. 
	
	\begin{proposition}\label{p:iteration1} 
		Let $L,N\geq 1$ and assume \eqref{eq:u0}. Suppose that $\tr ((\mathrm{I}-\Delta)^{5/2+\gamma}GG^*)<\infty$ for some $\gamma>0$. Let $0<\bar{\beta}<1/3$, for any $0<\beta<\bar{\beta}$, there exists a choice of parameters $a,b,\alpha$ depending on $\beta,\bar{\beta}$ with the following properties:
		
		Let $(v_{q},\mathring{R}_{q})$ for some $q\in\N$ be an $(\mathcal{F}_{t})_{t\geq 0}$-adapted solution to \eqref{euler1:new} satisfying the following estimates for any $t\in [0,\mathfrak{t}_L]$:
		\begin{align}
			\|v_q(t)\|_{C^0_x} &\leq 3\tilde{M}M_L\lambda_{1}^{\sfrac{3\alpha}{2}}-\tilde{M}M_L\delta_q^{\sfrac12},\label{itera:a:new}
			\\  \|v_q(t)\|_{C^1_{x}} &\leq  \tilde{M}M_L  \lambda_{q}\delta_q^{\sfrac12},\label{itera:b:new}
			\\	\|\mathring{R}_q(t)\|_{C^0_x}  &\leq \tilde{M}M_L^2 \delta_{q+1}\lambda_{q}^{-3\alpha},\label{itera:c:new}
		\end{align}
		where $\tilde{M}$ is a universal constant given by \eqref{def:barM} and $M_{L}=(L+N)^2$. Then there exists an $(\mathcal{F}_t)_{t\geq 0}$-adapted process $(v_{q+1}, \mathring{R}_{q+1})$ which solves \eqref{euler1:new} and satisfies the inductive estimates \eqref{itera:a:new}-\eqref{itera:c:new} at the $q+1$ step and we have
		\begin{equation}\label{vq+1-vq:new}
			\|v_{q+1}(t)-v_q(t)\|_{C_x^0}\leq \tilde{M}M_L\delta_{q+1}^{\sfrac12}.
		\end{equation}
		Moreover, it holds
		\begin{align}\label{estimate:energy:new}
			\big|\|v_{q+1}(t)\|_{L^2}^2-\|v_q(t)\|_{L^2}^2-3M_L^2\varsigma_{q+1} \big|\leq M_L^2 \lambda_{1}^{\sfrac{3\alpha}{2}} \delta_{q+1}^{\sfrac12}.
		\end{align}
		for any $t\in[\te_1\wedge \mathfrak{t}_L,\mathfrak{t}_L]$, where $\te_1:=\tau_q$ is given as in \eqref{def tauq:new} below.
	\end{proposition}
	\begin{remark}\label{verify:vq0}
	\textit{Due to the specific definition of $v_q(0)$ in \eqref{euler1:new}, it is necessary to verify that the iterative estimates \eqref{itera:a:new} and \eqref{itera:b:new} hold for $v_q(0)$. This verification is carried out after selecting the parameters in Subsection~\ref{sec:mollification:new}; see estimates \eqref{vq0:C1} and \eqref{vq0:C0} below.} 
	\end{remark}
	
	The proof of Proposition~\ref{p:iteration1} is presented in the following subsections. Based on Proposition~\ref{p:iteration1}, we can proceed to establish Theorem~\ref{Onsager:theorem:cauchy:problem}.
	
	\subsection{Proof of Theorem~\ref{Onsager:theorem:cauchy:problem}}
	The proof is primarily similar to Theorem~\ref{Onsager:theorem} and \cite[Theorem 1.1]{HZZ21markov}.
	
	\emph{\underline{Step 1}.}		
	In this step, we check that the initial iteration is valid.
	Let $\mathfrak{t}_L$ be a stopping time defined by \eqref{stopping time ps:new}, which can be made arbitrarily large by choosing $L$ large. Assume that the additional assumption \eqref{eq:u0} holds for some $N\geq1$.
	To apply Proposition~\ref{p:iteration1} iteratively and obtain a sequence of solutions $(v_q,\mathring{R}_q)$ to the system \eqref{euler1:new}, we start the initial iteration with $(v_1,\mathring{R}_1):=(u^{in}*\varphi_{\ell_0},(v_1+z_1)  \mathring{\otimes}(v_1+z_1) )$.
	It is easy to check that they solve \eqref{euler1:new}. Given $\|u^{in}\|_{C^{\bar{\beta}}_x}\leq N\leq M_L$, it follows that for any $t\in [0,\mathfrak{t}_L]$
	\begin{align*}
		\|v_1\|_{C^0_{t,x}}\leq  M_L \leq 3\tilde{M}M_L \lambda_{1}^{\sfrac{3\alpha}{2}}-2\tilde{M}M_L \lambda_{1}^{\sfrac{3\alpha}{2}}=3\tilde{M}M_L \lambda_{1}^{\sfrac{3\alpha}{2}}-\tilde{M}M_L\delta_{1}^{\sfrac12},
	\end{align*}
	which gives \eqref{itera:a:new}.
	By \eqref{vq0:C1}, we have for any $t\in [0,\mathfrak{t}_L]$
	\begin{align*}
		\|v_1\|_{C^0_tC^1_{x}}\lesssim \ell_0^{\bar{\beta}-1} \|u^{in}\|_{C^{\bar{\beta}}_x} \leq M_L\ell_1^{\bar{\beta}-1}  \leq \tilde{M} M_L\lambda_{1}\delta_{1}^{\sfrac12},
	\end{align*}
	which gives \eqref{itera:b:new}. Finally, it follows from \eqref{z ps:new} that for any $t\in [0,\mathfrak{t}_L]$
	\begin{align*}
		\|\mathring{R}_1\|_{C^0_{t,x}}\leq \|v_1+z_1\|^2_{C^0_{t,x}}\leq (N+L)^2\leq \tilde{M} M_L\delta_2\lambda_{1}^{-3\alpha},
	\end{align*}
	which gives \eqref{itera:c:new}.
	
	\emph{\underline{Step 2}.}		
	Since the first iteration is established, Proposition~\ref{p:iteration1} yields a sequence $(v_q,\mathring{R}_q)$ satisfying \eqref{itera:a:new}-\eqref{estimate:energy:new}. We then use the same computations as in Subsection~\ref{Proof:The:1.3} together with \eqref{itera:b:new}, \eqref{vq+1-vq:new} and interpolation to deduce that $v_q$ converges in $C([0,\mathfrak{t}_L],C^{\bar{\vartheta}}(\T^3,\R^3))\cap C^{\bar{\vartheta}}([0,\mathfrak{t}_L],C(\T^3,\R^3))$ towards a limit $v$ with initial data $v(0)=u^{in}$ for any $\bar{\vartheta}<\bar{\beta}$. Since $v_q$ is $(\mathcal{F}_{t})_{t\geq 0}$-adapted for every $q\in \mathbb{N}_0$, the limit $v$ is $(\mathcal{F}_{t})_{t\geq 0}$-adapted as well.
	Furthermore, it follows from \eqref{itera:c:new} that $\lim_{q\to \infty} \mathring{R}_q=0$ in $C([0,\mathfrak{t}_L],C(\T^3,\R^3))$. Thus $v$ is an analytically weak solution to \eqref{nonlinear:new}. Hence letting $u=v+z$ we obtain an $(\mathcal{F}_{t})_{t\geq 0}$-adapted analytically weak solution to \eqref{eul1} of class $u\in C([0,\mathfrak{t}_L],C^{\bar{\vartheta}}(\T^3,\R^3))\cap C^{\bar{\vartheta}}([0,\mathfrak{t}_L],C(\T^3,\R^3))$ with initial data $u(0)=u^{in}$.
	
	Next, we prove the non-uniqueness of the constructed solutions. In view of \eqref{estimate:energy:new}, we have for any $t\in [\te_1\wedge \mathfrak{t}_L,\mathfrak{t}_L]$
	\begin{align*}
		\big|\|v(t)\|_{L^2}^2-\|v_1(t)\|_{L^2}^2-3KM_L^2\lambda_{1}^{3\alpha} \big|&=	\big|\|v(t)\|_{L^2}^2-\|v_1(t)\|_{L^2}^2-3M_L^2 \varsigma_{2} \big|
		\\ &\leq \left|\sum_{q=1}^\infty(\|v_{q+1}(t)\|_{L^2}^2-\|v_{q}(t)\|_{L^2}^2-3M_L^2\varsigma_{q+1})\right|+3 M_L^2  \sum_{q\geq2} \varsigma_{q+1}
		\\
		&\leq  M_L^2  \lambda_{1}^{\sfrac{3\alpha}{2}}\sum_{q=1}^{\infty}\delta_{q+1}^{\sfrac12}+3M_L^2  \sum_{q\geq 2}\varsigma_{q+1}=:M_L^2 c\lambda_{1}^{3\alpha},
	\end{align*}
	where $c$ is a constant independent of $K, a, b ,\beta, \alpha$. 
	By choosing different $K$ and $K'$ so that $3|K-K'|>2c$ we deduce that the corresponding solutions $v_K$ and $v_{K'}$ have different $L^2$-norms. Therefore, the solutions $u_K=v_K+z$ and $u_{K'}=v_{K'}+z$ are different in the pathwise sense as well.
	
	\emph{\underline{Step 3}.}		 Based on the previous discussions, we next briefly review the arguments in \cite[Theorem 1.1]{HZZ21markov} to construct global solutions. To this end, we first define $\hat{z}(t)=z(t+\mathfrak{t}_L)-e^{-t}z(\mathfrak{t}_L)$, $\hat{B}_t=B_{t+\mathfrak{t}_L}-B_{\mathfrak{t}_L}$, $\hat{\mathcal{F}}_t=\sigma(\hat{B}_s, s\leq t)\vee \sigma (u(\mathfrak{t}_L))$ and the stopping time $	\hat{\mathfrak{t}}_{L+1}:= \inf\{t\geq0,\|z(t)\|_{H^{5/2+\gamma}}\geq 2(L+1)/C_S\}\wedge (L+1)$,
	which satisfies $\mathfrak{t}_{L+1}-\mathfrak{t}_{L}\leq \hat{\mathfrak{t}}_{L+1}$.
	Then, we can use the value $u(\mathfrak{t}_L)$ as a new initial condition in \emph{Step 1} to construct a solution $\bar{u}_1\in C([0,\hat{\mathfrak{t}}_{L+1}], C^{\vartheta_1}(\T^3,\R^3))$ for any $\vartheta_1<\bar{\vartheta}$ to \eqref{eul1} with $B$ replaced by $\hat{B}$, which is adapted to $\hat{\mathcal{F}}_t$. 
	
	Following the arguments in \cite[Theorem 1.1]{HZZ21markov}, one can check $u_1(t)=u(t) \mathbf{1}_{\{t\leq \mathfrak{t}_L\}}+\bar{u}_1(t-\mathfrak{t}_L) \mathbf{1}_{\{t> \mathfrak{t}_L\}}$ satisfies the system \eqref{eul1} before $\mathfrak{t}_{L+1}$ and is adapted to the natural filtration $(\mathcal{F}_t)_{t\geq0}$, hence $u_1$ is a probabilistically strong solution to \eqref{eul1}. Moreover, 
	by the arbitrariness of $\vartheta_1<\bar{\vartheta}<\bar{\beta}$, we deduce that $u_1$ belongs to $C([0,\mathfrak{t}_{L+1}],C^\vartheta(\T^3,\R^3))$ for any $\vartheta<\bar{\beta}$. Now, we can iterate the above steps: starting from $u(\mathfrak{t}_{L+k})$ and constructing solutions $u_{k+1}$ before the stopping time $\mathfrak{t}_{L+k+1}$. We obtain $\bar{u}:=\sum_{k=1}^{\infty}u\mathbf{1}_{\{t\leq \mathfrak{t}_L\}}+ u_{k} \mathbf{1}_{\{\mathfrak{t}_{L+k-1} \leq t\leq \mathfrak{t}_{L+k}\}}$ is a probabilistically strong solution belonging to
	$C([0,\infty),C^\vartheta(\T^3,\R^3))$ for any $\vartheta<\bar{\beta}$, and the time regularity can also be recovered. Hence, we conclude the proof of Theorem~	\ref{Onsager:theorem:cauchy:problem}. \qed
	
	\vspace{1em} 
	
	In the following subsections, we prove Proposition~\ref{p:iteration1} following the same structure outlined in Sections~\ref{sec:begin}, \ref{sec:perturbation} and \ref{sec:inductive esti}. The main differences arise from the gluing and perturbation steps as we need to modify the initial data of the exact Euler system (as seen in \eqref{eq:euler exact2}) and redefine the cutoffs $\eta_i$ to ensure $v_{q+1}(0)=u^{in}*\varphi_{\ell_{q}}$.
	Since most of the assumptions and calculations remain unchanged, we will primarily focus on highlighting the differences in the subsequent subsections.
	
	\subsection{Choice of parameters and mollification step}\label{sec:mollification:new}
	During the proof of Proposition~\ref{p:iteration1}, we always assume $0<\beta<\bar{\beta}<1/3$, $b>1$ and close to $1$ such that 
	\begin{align}\label{choice:b:new}
		0<b-1<\min\left\{\frac{1-3\beta}{2\beta},\sqrt{\frac{1}{3\beta}}-1,\frac{\bar{\beta}-\beta}{3}, \frac{\bar{\beta}}{\beta}-1, \frac{1}{9} \right\}.
	\end{align}
	In addition, we require $\alpha>0$ to be sufficiently small in terms of $b,\beta$ satisfying 
	\begin{align}\label{parameter:alpha1:new}
		20b\alpha<\min \left\{(b-1)(1-2b\beta-\beta), \beta(b-1), \frac23-2b^2\beta, \bar{\beta}-b\beta \right\} .
	\end{align}
	Finally, we choose $a$ large enough to have $2\leq a^{(b-1)\beta}
	\leq  a^{(b-1)(1-\beta)}$ and $\tilde{M}  \leq a^{\sfrac{b\alpha}{2}}$.
	We point out that by choosing different $K$ we get different $a$.
	
	Before the mollification procedure, we need to verify that the inductive assumptions \eqref{itera:a:new} and \eqref{itera:b:new} hold for $t=0$ and all $q\geq1$.
	Indeed, by the mollification estimate \eqref{estimate:molli1} and the bound $\|u^{in}\|_{C^{\bar{\beta}}_x} \leq M_{L}$, we have for all $q\geq 1$:
	\begin{align}\label{vq0:C11}
		\|u^{in}*\varphi_{\ell_q}\|_{C^1_x}\lesssim \ell_{q}^{\bar{\beta}-1} \|u^{in}\|_{C^{\bar{\beta}}_x}\lesssim M_L\ell_{q}^{\bar{\beta}-1}\leq M_L\lambda_{q}\delta_{q}^{\sfrac12}.
	\end{align}
	More precisely, by the choice of the parameters \eqref{choice:b:new} and \eqref{parameter:alpha1:new}, we have $b-1<\bar{\beta}-\beta<1$ and $3\alpha<\beta(b-1)$.
	Combining this with \eqref{ell:lambdaq}, we have
	\begin{align*}
		\ell_{q}^{\bar{\beta}-\beta}\leq \lambda_q^{-(\bar{\beta}-\beta)} \le \lambda_q^{-3\beta (b-1)} \le \lambda_q^{-\beta(b-1) - 6\alpha } \lesssim \delta_{q+1}^{\sfrac12}\delta_q^{-\sfrac12}\lambda_q^{-6\alpha}=\lambda_{q}\ell_{q}\leq  (\lambda_q\ell_{q})^{1-\beta},
	\end{align*}
	which gives \eqref{vq0:C11}, as
	\begin{align*}
		\ell_{q}^{\bar{\beta}-1} = 	\ell_{q}^{\bar{\beta}-\beta } \ell_{q}^{-(1-\beta)}\lesssim \lambda_q^{1-\beta} \lesssim  \lambda_q\delta_q^{\sfrac12}.
	\end{align*}
	Since $\ell_{q}$ is decreasing, we also have
	\begin{align}\label{vq0:C1}
		\|v_q(0)\|_{C^1_x}= \|u^{in}*\varphi_{\ell_{q-1}} \|_{C^1_x} \lesssim \|u^{in}\|_{C^{\bar{\beta}}_x} \ell_{q-1}^{\bar{\beta}-1} \leq M_L\ell_{q}^{\bar{\beta}-1} \leq  M_L\lambda_{q}\delta_{q}^{\sfrac12},
	\end{align}
	which verifies \eqref{itera:b:new} at $t=0$. Similarly, using \eqref{estimate:molli1} and $\|u^{in}\|_{C^{\bar{\beta}}_x} \leq M_{L}$ again, it follows that for any $q\geq 1$
	\begin{align}\label{vq0:C0}
		\|v_q(0)\|_{C^0_x}\leq   M_L \leq 3\tilde{M}M_L \lambda_{1}^{\sfrac{3\alpha}{2}}-2\tilde{M}M_L \lambda_{1}^{\sfrac{3\alpha}{2}}=3\tilde{M}M_L \lambda_{1}^{\sfrac{3\alpha}{2}}-\tilde{M}M_L\delta_{1}^{\sfrac12}.
	\end{align}
	
	In order to guarantee smoothness throughout the construction, we replace $(v_q,z_q,\mathring{R}_q)$ by a mollified
	field $(v_{\ell_q},z_{\ell_q},\mathring{R}_{\ell_q})$, which is given by \eqref{equation:mollified} and satisfies on $t\in [0,\mathfrak{t}_L]$
	\begin{equation}\label{mollification:new}
		\aligned
		\partial_t v_{\ell_q} +\div((v_{\ell_q}+z_{\ell_q})\otimes (v_{\ell_q}+z_{\ell_q}))+\nabla p_{\ell_q}&=\div \mathring{R}_{\ell_q},
		\\ \div v_{\ell_q} &=0,
		\\ v_{\ell_q}(0)&=u^{in}*\varphi_{\ell_{q}}*\varphi_{\ell_{q-1}}.
		\endaligned
	\end{equation}
	Moreover, using standard mollification estimates, we derive the following proposition:
	\begin{proposition}\label{esti:mollification:new}
		For any $t\in [0,\mathfrak{t}_L]$ and $N\geq 0$, we have
		\begin{subequations}
			\begin{align}
				\|v_{\ell_q}(t)-v_q(t)\|_{C^0_{x}}&\lesssim \tilde{M} M_L \delta_{q+1}^{\sfrac12}\ell_{q}^{\alpha}\ ,\label{vq-vl:new}
				\\      \|\mathring{R}_{{\ell_q}}(t)\|_{C_x^{N+\alpha}}&\lesssim \tilde{M}M_L^2 \delta_{q+1}\ell_q^{-N+\alpha}\, ,\label{estimate Rl:new}
			\end{align}
		\end{subequations}
		where the implicit constant may depend on $N$ and $\alpha$.
	\end{proposition}
	\begin{proof}
		The proof follows the same line as in Proposition~\ref{esti:mollification}.
	\end{proof}

	\subsection{Gluing step}\label{sec:esact solutions:new} 
The gluing procedure is similar to that in Subsection~\ref{sec:esact solutions} and Subsection~\ref{sec:gluing}. We  begin by constructing exact solutions $\ve_{i}$ to the Euler system and then combine these solutions to derive the glued solution $\overline{v}_q$. Notably, to propagate the initial data throughout the iteration, we adjust the initial value to $u^{in}*\varphi_{{\ell_q}}$ at $t=0$, which allows us to construct the next iteration $v_{q+1}(0)=u^{in}*\varphi_{{\ell_q}}$. This adjustment also necessitates additional estimates on $\ve_{i}-v_{\ell_{q}}$ compared to Subsection~\ref{sec:esact solutions}, and we will detail these differences below. 
	
	\subsubsection{Exact solutions}
	We first construct the exact solutions to the Euler system.
	Similar to \eqref{def tauq}, we define the parameter $\tau_q$ and initial times $\te_i \, (i\in   [-1,\infty) \cap \mathbb{Z})$ by
	\begin{align}\label{def tauq:new}
		\tau_q:=\frac{1}{M_L \lambda_{q}^{1+6\alpha}\delta_q^{\sfrac12}}\, ,\qquad \te_i:=i\tau_q\,(i\geq 0),\qquad \te_{-1}=0.
	\end{align}
	For $i\geq 1$, we define $(\ve_i, \pe_i)$ to be the smooth solutions to the Euler system \eqref{eq:euler exact1} on  $\big([\te_{i-1},\te_{i+1}]\cap [0,\mathfrak{t}_L]\big)\times \mathbb{T}^3$.
	For $ i=0 $, we define $(\ve_0,\pe_0)$ to be the smooth solution to the Euler system on $\big([\te_{-1},\te_{1}]\cap [0,\mathfrak{t}_L]\big)\times \mathbb{T}^3$ starting from $u^{in}*\varphi_{\ell_q}$:
	\begin{equation}\label{eq:euler exact2}
		\aligned
		\partial_t \ve_0+\div ( (\ve_0+z_{\ell_q}) \otimes (\ve_0+z_{\ell_q}) )+\nabla \pe_0& =0,
		\\ 	\div \ve_0& =0 ,
		\\  \ve_0(0,\cdot)& =u^{in}*\varphi_{\ell_q}.
		\endaligned
	\end{equation}
	To match the initial data $v_{q+1}(0)=u^{in}*\varphi_{\ell_{q}}$, we modify the value of $\ve_{0}(0)$ to be $u^{in}*\varphi_{\ell_{q}}$, which leads to $\ve_{0}(0)\neq v_{\ell_{q}}(0)$. This is the main distinction from Subsection~\ref{sec:esact solutions}. From the argument in Subsection~\ref{sec:esact solutions}, it follows that $\ve_i$ is well-defined on $\big([\te_{i-1},\te_{i+1}]\cap [0,\mathfrak{t}_L]\big)\times \mathbb{T}^3$ for $i\geq 1$, and the estimate \eqref{estimate:vi:N} also follows. We only need to justify the case $i=0$. By using \eqref{vq0:C1} along with standard mollification estimate \eqref{estimate:molli1} we have
	\begin{align*}
		(\te_{i+1}-\te_{i-1})\|u^{in}*\varphi_{\ell_q}\|_{C^{1+\alpha}_x}&\lesssim  \tilde{M} M_L \tau_q \lambda_{q}\delta_q^{\sfrac12} \ell_q^{-\alpha}  \leq \tilde{M} \lambda_{q}^{-\alpha} \leq \frac12 \, ,
	\end{align*}
	which ensures $\ve_{0}$ is well-defined on $\big([\te_{-1},\te_{1}]\cap [0,\mathfrak{t}_L]\big)\times \mathbb{T}^3$.
	Furthermore, we use \eqref{estimate: smooth u} and \eqref{vq0:C1} to deduce for any $t\in [\te_{-1},\te_{1}]\cap [0,\mathfrak{t}_L]$ and $N\geq2$
	\begin{equation*}
		\begin{aligned}
			\|\ve_0(t)\|_{C_x^{N+\alpha}} &\lesssim 	
			\|u^{in}*\varphi_{\ell_q}\|_{C^{N+\alpha}_x}+\tau_q \|z_{\ell_q}\|_{C_t^0C^{N+1+\alpha}_x}\big(\|u^{in}*\varphi_{\ell_q}\|_{
				C^0_tC^{1+\alpha}_x}+\|z_{\ell_q}\|_{C_t^0C^{2+\alpha}_x}\big)
			\\&\lesssim  M_L \ell_q^{1-N-\alpha}\lambda_{q}\delta_q^{\sfrac12}+M_L^2\tau_q \ell_q^{1-N-\alpha}\lambda_{q}^{\sfrac23}(\ell_{q}^{-\alpha}\lambda_{q}\delta_q^{\sfrac12}+\ell_{q}^{-\alpha}\lambda_{q}^{\sfrac23})\leq \tilde{M}  \tau_q^{-1} \ell_q^{1-N+\alpha} \, ,
		\end{aligned}
	\end{equation*}
	where the last line is justified by the same argument as \eqref{estimate:vi:N}. Using \eqref{estimate: smooth u} and \eqref{vq0:C1} again, we obtain for any $t\in [\te_{-1},\te_{1}]\cap [0,\mathfrak{t}_L]$
	\begin{align*}
		\|\ve_0(t)\|_{C^{1+\alpha }_x}  &\lesssim	\|u^{in}*\varphi_{\ell_q}\|_{C^{1+\alpha }_x} +\|z_{\ell_{q}}\|_{C_t^0C^{2+\alpha }_x}
		\lesssim  M_L \ell_q^{-\alpha}\lambda_{q}\delta_q^{\sfrac12}+M_L \ell_q^{-\alpha}\lambda_{q}^{\sfrac23}\leq \tau_q^{-1} \ell_q^{\alpha} \, .
	\end{align*}
	By the above discussion, we deduce that for any $i\geq 0$ and $t\in [\te_{i-1},\te_{i+1}]\cap [0,\mathfrak{t}_L]$, the exact solution $\ve_i$ satisfies the following bounds for any $N\geq 1$ 
	\begin{align}\label{estiamte:vi:new}
		\|\ve_i(t)\|_{C_x^{N+\alpha}}
		\leq \tilde{M}\tau_q^{-1} \ell_q^{1-N+\alpha}\, .
	\end{align}
	Since we have verified the same bounds for $v_{\ell_{q}}$ and $\ve_i$ as in Section~\ref{sec:begin}, we can deduce the same stability estimates as in Proposition~\ref{esti: v,p,D i-l} and Proposition~\ref{p:S_est}.
	\begin{proposition}\label{esti: v,p,D i-l:new}
		For any $t\in  [\te_{i-1},\te_{i+1}]\cap [0,\mathfrak{t}_L]$ and $N\geq 0$, we have
		\begin{subequations}
			\begin{align}
			 	\|\ve_i- v_{\ell_q}\|_{C^{N+\alpha}_x}&\lesssim \tilde{M}M_L^2 \tau_q \delta_{q+1}\ell_q^{-1-N+\alpha} \ ,\label{vi- vl:new}
			 	\\	\|(\partial_t+(v_{\ell_q}+z_{\ell_q})\cdot \nabla)(v_{\ell_q}-\ve_i)\|_{C^{N+\alpha}_x}&\lesssim  \tilde{M} M_L^2 \delta_{q+1}\ell_q^{-1-N+\alpha} \ , \label{Dtl v i- v l:new}
			 	 \end{align}
			 	\end{subequations}
			 	where the implicit constant may depend on $N$ and $\alpha$.
			 	Let $b_i= \mathcal{B} \ve_{i}$ be the vector potential defined as in Subsection~\ref{sec:esact solutions}. For any $t\in  [\te_{i},\te_{i+1}]\cap [0,\mathfrak{t}_L]$ and $N\geq 0$, we have
			 	\begin{subequations}
			 		\begin{align}
			 	\|b_i-b_{i+1}\|_{C^{N+\alpha}_x} &\lesssim  \tilde{M}M_L^2 \tau_q\delta_{q+1}\ell_q^{-N+\alpha}\,,   \label{e:b_diff:new} 
			 	\\ \| (\partial_t+(v_{\ell_q}+z_{\ell_q})\cdot \nabla)(b_i-b_{i+1})\|_{C^{N+\alpha}_x} &\lesssim  \tilde{M}M_L^2 \delta_{q+1}\ell_q^{-N+\alpha}\,, \label{e:b_diff_Dt:new}
			 \end{align}
		\end{subequations}
		where the implicit constant may depend on $N$ and $\alpha$.
	\end{proposition}
	\begin{proof}
		The proof follows a similar approach to that of Proposition~\ref{esti: v,p,D i-l} and Proposition~\ref{p:S_est}, with the main distinction being the mismatch in initial values when $i=0$. Therefore, we only need to estimate this case. Let us first focus on \eqref{vi- vl:new}. By using the same calculations as in \eqref{estimate:vellzell}-\eqref{estimate: dtl alpha}, we deduce
		\begin{align*}
			\|(v_{\ell_q}-\ve_i)(t)\|_{C^\alpha_x}\lesssim 	\|(v_{\ell_q}-\ve_i)(\te_{i-1})\|_{C^\alpha_x}+ \tilde{M}M_L^2 \tau_q\delta_{q+1}\ell_q^{-1+\alpha}+\tau_q^{-1}\int_{\te_{i-1}}^{t}	\|(v_{\ell_q}-\ve_i)(s)\|_{C^\alpha_x} \dif s.
		\end{align*}
		Using $v_{\ell_q}(\te_{i-1})=\ve_i(\te_{i-1})$ and Gr\"onwall's inequality, we obtain \eqref{vi- vl:new} for the case $N=0$ and $i\neq 0$. 
		
		Note that the case $N=0$, $i=0$ requires a further estimate because $v_{\ell_q}(0)\neq \ve_0(0)$. By using mollification estimates and \eqref{vq0:C1}, we obtain
		\begin{equation}\label{bound:vlq0-vq0}
			\begin{aligned}
				\|v_{\ell_q}(0)- \ve_0(0)\|_{C^\alpha_x}= \|u^{in}*\varphi_{\ell_{q}}*\varphi_{\ell_{q-1}}-u^{in}*\varphi_{\ell_{q}}\|_{C^\alpha_x}
				\lesssim \ell^{\bar{\beta}-\alpha}_{q-1}\|u^{in}\|_{C^{\bar{\beta}}_x}
				\lesssim \tilde{M} M_L\ell^{\bar{\beta}-\alpha}_{q-1}.
			\end{aligned}
		\end{equation}
		Using $b-1< \frac{\bar{\beta} - \beta}{2}<1$, we have $-\bar{\beta}\leq - 4\beta(b-1)-\beta\leq -\beta(b-1)-b^2\beta$. Combining this with $10b\alpha<\beta(b-1)$ implies 
		\begin{align*}
			\ell_{q-1}^{\bar{\beta}-\alpha}\leq 	\lambda_{q-1}^{-\bar{\beta}+2\alpha}  \lesssim  \lambda_{q-1}^{-\beta b^2}\lambda_q^{2\alpha-\beta(b-1)/b} \le \lambda_{q-1}^{-\beta b^2} \lambda_q^{-8\alpha}  \leq   \lambda_{q+1}^{-\beta} \lambda_q^{-3\alpha} \leq M_L\tau_q \delta_{q+1}\ell_q^{-1+\alpha}.
		\end{align*}
		Therefore, we obtain
		\begin{align*}
			\|v_{\ell_q}(0)- \ve_0(0)\|_{C^\alpha_x}\leq \tilde{M}M_L^2 \tau_q \delta_{q+1}\ell_q^{-1+\alpha},
		\end{align*}
		and using again Gr\"{o}nwall's inequality implies that \eqref{vi- vl:new} also holds for $N=0$ and $i=0$. 
		
		Next, we consider $N\geq 1$, invoking once more the computations \eqref{Dt,lN}-\eqref{pl-iN} and \eqref{eq:Gronwall1} gives 
		\begin{align*}
			\|(v_{\ell_q}-\ve_i)(t)\|_{C^{N+\alpha}_x}&\lesssim 	\|(v_{\ell_q}-\ve_i)(\te_{i-1})\|_{C^{N+\alpha}_x}+ \tilde{M}M_L^2 \tau_q\delta_{q+1}\ell_q^{-1-N+\alpha}
			\\ &\quad+\tau_q^{-1}\int_{\te_{i-1}}^{t}	\|(v_{\ell_q}-\ve_i)(s)\|_{C^{N+\alpha}_x} \dif s.
		\end{align*}
		By using $v_{\ell_q}(\te_{i-1})=\ve_i(\te_{i-1})$ and applying Gr\"onwall's inequality, \eqref{vi- vl:new} is clearly established for $N\geq1$ and $i\neq 0$. 
		For the case $N\geq 1$ and $i=0$, we again require a further estimate similar to \eqref{bound:vlq0-vq0}. Utilizing mollification estimates, \eqref{vq0:C1}, $\ell_{q}\leq \lambda_{q}^{-1}$ and $3\alpha<\bar{\beta}-b\beta$, we deduce  
		\begin{align*}
			\|v_{\ell_{q}}(0)-\ve_{0}(0)\|_{C^{N+\alpha}_x}&\leq  \|u^{in}*\varphi_{{\ell_q}}*\varphi_{{\ell_{q-1}}}\|_{C^{N+\alpha}_x}+\|u^{in}*\varphi_{{\ell_q}}\|_{C^{N+\alpha}_x}
			\lesssim  \ell_q^{\bar{\beta}-N-\alpha}\|u^{in}\|_{C^{\bar{\beta}}_x}
			\\ &\lesssim \tilde{M} M_L\lambda_{q}^{-\bar{\beta}+3\alpha}\ell_{q}^{-N+\alpha}
			\lesssim \tilde{M} M_L\lambda_{q}^{-b\beta}\ell_{q}^{-N+\alpha}
			\lesssim \tilde{M} M_L^2\tau_q\delta_{q+1}\ell_{q}^{-1-N+\alpha}.
		\end{align*}
		Then,  by applying Gr\"{o}nwall's inequality, we conclude that \eqref{vi- vl:new} also holds for $N\geq 1$ and $i=0$. Following the same argument as in Proposition~\ref{esti: v,p,D i-l}, the estimate \eqref{Dtl v i- v l:new} directly obeys.
		
		Turning to the estimate \eqref{e:b_diff:new}, we first observe from \eqref{e:N>=1} and \eqref{vi- vl:new} that \eqref{e:b_diff:new} holds for $N\geq 1$ and $i\geq 0$. Next, we focus on the case $N=0$. As with Proposition~\ref{p:S_est}, we define $\tilde{b}_i:= \mathcal{B}(\ve_{i}-v_{\ell_{q}})$ and invoke the same computations as in \eqref{D(t,l )b N} to derive
		\begin{equation}\label{D(t,l )b N:new}
			\begin{aligned}
				\|\partial_t \tilde b_i+(v_{\ell_q}+z_{\ell_q})\cdot\nabla\tilde b_i\|_{C^{N+\alpha}_x}
				\lesssim \tau_q^{-1}\|\tilde b_i\|_{C^{N+\alpha}_x}+\tau_q^{-1}\ell_q^{-N}\|\tilde b_i\|_{C^{\alpha}_x}+\tilde{M}M_L^2\delta_{q+1}\ell_q^{-N+\alpha}.
			\end{aligned}
		\end{equation}
		Using the estimate  \eqref{eq:Gronwall1} for transport equations, we obtain for any $t\in [\te_{i-1},\te_{i+1}]\cap [0,\mathfrak{t}_L]$
		\begin{align}\label{tilde:bi}
			\|\tilde b_i(t)\|_{C^\alpha}\lesssim 	\|\tilde b_i(\te_{i-1})\|_{C^\alpha}+\tilde{M}M_L^2 \tau_q \delta_{q+1}\ell_q^{\alpha}+ \tau_q^{-1}\int_{\te_{i-1}}^{t} \|\tilde b_i(s)\|_{C^\alpha} \dif s.
		\end{align}
		For $i\neq0$, using $\tilde{b}_i(\te_{i-1})=0$ and applying Gr\"onwall's inequality we obtain
		\begin{align}\label{tilde b:0}
			\|\tilde b_i(t)\|_{C^\alpha}\lesssim \tilde{M}M_L^2\tau_q \delta_{q+1}\ell_q^{\alpha}.
		\end{align}
		
		The case $i=0$ is similar to \eqref{bound:vlq0-vq0} and requires to control the initial data $\tilde{b}_0(0)$.
		From \eqref{vq0:C1}, \eqref{estimate:molli1} and the boundedness of the zero-order operator $\nabla \mathcal{B}$ on H\"{o}lder spaces, it follows that
		\begin{equation*}
			\begin{aligned}
			\|\tilde{b}_0(0)\|_{C^\alpha_x}  =	\|\mathcal{B}v_{\ell_q}(0)- \mathcal{B}\ve_0(0)\|_{C^\alpha_x} 
				&= \|(\mathcal{B}u^{in})*\varphi_{\ell_{q}}*\varphi_{\ell_{q-1}}-\mathcal{B}u^{in}*\varphi_{\ell_{q}}\|_{C^\alpha_x}
				\\&\lesssim \ell^{1+\bar{\beta}-\alpha}_{q-1}\|\nabla \mathcal{B}u^{in}\|_{C^{\bar{\beta}}_x}
				\lesssim     \|u^{in}\|_{C^{\bar{\beta}}_x} \ell^{1+\bar{\beta}-\alpha}_{q-1} 
				\lesssim \tilde{M} M_L\ell^{1+\bar{\beta}-\alpha}_{q-1}.
			\end{aligned}
		\end{equation*}
		To match with \eqref{tilde b:0} for $i\neq 0$, we use $b-1<\frac{\bar{\beta}-\beta}{3}<1$, $(2b+1)\beta<1$ and $10\alpha<\frac{b-1}{b}$ to derive
		\begin{align*}
			-\bar{\beta}<-(b-1)(1+\beta+2b\beta)-10b\alpha-\beta<-b+1-b(2b-1)\beta-10b\alpha,
		\end{align*}
		which implies 
		\begin{align*}
			\lambda_{q-1}^{-\bar{\beta}}\leq \lambda_{q-1}\lambda_{q}^{-1-(2b-1)\beta-10\alpha}\leq \lambda_{q-1}\frac{\delta_{q+1}}{\lambda_{q}^{1+6\alpha}\delta_{q}^{\sfrac12}}\lambda_{q}^{-4\alpha} \leq  \lambda_{q-1} \tau_q\delta_{q+1}\ell_q^{\alpha}\ell_{q-1}^{\alpha}.
		\end{align*}
		This is enough since by \eqref{ell:lambdaq} we find
		\begin{align*}
			\|\tilde{b}_0(0)\|_{C^{\alpha}_x}\lesssim \tilde{M}M_L \ell_{q-1}^{1+\bar{\beta}-\alpha} \leq \tilde{M}M_L \lambda_{q-1}^{-1-\bar{\beta}}\ell_{q-1}^{-\alpha}\leq \tilde{M} M_L\tau_q\delta_{q+1}\ell_q^{\alpha}.
		\end{align*}
		Hence, by applying Gr\"onwall's inequality, we establish \eqref{tilde b:0} for the case $N=0$ and $i=0$. Furthermore, since $\tilde{b}_i-\tilde{b}_{i+1}=b_i-b_{i+1}$, we conclude that \eqref{e:b_diff:new} holds for any $N=0$ and $i\geq0$. Finally, using \eqref{e:b_diff:new} and \eqref{D(t,l )b N:new} once again, we obtain \eqref{e:b_diff_Dt:new}.
	\end{proof}
	
	\subsubsection{Gluing exact solutions}\label{gluing:exact:new}
	This section is exactly the same as Subsection~\ref{sec:gluing}, therefore, we omit the routine computations and proofs. We define the intervals $I_i, J_i$ $(i\geq 0)$ by $I_i:=[\te_i+\sfrac{\tau_q}{3},\te_i+\sfrac{2\tau_q}{3}]\cap[0,\mathfrak{t}_L]$, $J_i:=(\te_i-\sfrac{\tau_q}{3},\te_i+\sfrac{\tau_q}{3})\cap  [0,\mathfrak{t}_L]$ and let $\{\chi_i\}_{i\geq 0}$ be the partition of unity as defined in Subsection~\ref{sec:gluing}.
	We then define the glued velocity $\overline v_q$ by
	\begin{align}
		\label{eq:bar:v_q:def:new}
		\overline v_q(t,x):= \sum_i \chi_i(t) \ve_i(t,x)		\, .
	\end{align}
	Using the same computation as \eqref{computation:Rq}, we can show for all $(t,x)\in [0,\mathfrak{t}_L]\times \mathbb{T}^3$,
	\begin{equation}\label{gluing euler q:new}
		\aligned
		\partial_t \overline v_q+\div((\overline v_q+z_{\ell_q}) \otimes(\overline v_q+z_{\ell_q}))+\nabla \overline{p}_q&=\div \mathring{\overline{R}}_q,
		\\ \div \overline v_q &=0,
		\\ \overline{v}_q(0) &= u^{in}*\varphi_{\ell_{q}},
		\endaligned
	\end{equation}
	where the pressure $\overline{p}_q$ and stress $\mathring{\overline{R}}_q$ coincide with those in Subsection~\ref{sec:gluing}.
	
	Moreover, we have shown that the same estimates apply for $v_q$, $v_{\ell_{q}}$ and $\ve_i$ as in Section~\ref{sec:begin}. Therefore, the results for Proposition~\ref{p:vq:vell} and Proposition~\ref{p:Rq} remain valid, and the proof here requires no modifications. We summarize these estimates as follows.
	\begin{proposition}\label{p:vq:vell:new}
		For any $t\in[0,\mathfrak{t}_L]$, the glued velocity field $\overline v_q$ satisfies the following estimates
		\begin{subequations}
			 	\begin{align}
			 	\|\overline v_q - v_{{\ell_q}}\|_{C^\alpha_x} &\lesssim \tilde{M}M_L \delta_{q+1}^{\sfrac12}\ell_q^{\alpha} \label{e:vq:vell:new}, 
			 	\\	\|\overline v_q\|_{C^{1+N}_x} &\lesssim \tilde{M} \tau_q^{-1}\ell_q^{-N+\alpha},\label{e:vq:1+N:new} 
			 	\\ \|\overline{v}_q\|_{C^0_x}& \lesssim \tilde{M}M_L \lambda_{1}^{\sfrac{3\alpha}{2}},\label{e:vq:0:new}
			 \end{align}
		\end{subequations}
		for all $N \geq 0$. The new glued Reynolds stress $\mathring{\overline{R}}_q$ satisfies the estimates
	\begin{subequations}
			\begin{align}
			\|\mathring{\overline R}_q\|_{C^{N+\alpha}_x} &\lesssim \tilde{M}M_L^2 \delta_{q+1}\ell_q^{-N+\alpha} \label{e:Rq:1:new}\, ,
			\\ \|(\partial_t +( \overline v_q+z_{\ell_q})\cdot \nabla) \mathring{\overline R}_q\|_{C^{N+\alpha}_x} &\lesssim \tilde{M}M_L^2 \tau_q^{-1}\delta_{q+1}\ell_q^{-N+\alpha}\, , \label{e:Rq:Dt:new}
		\end{align}
	\end{subequations}
		for all $N \geq 0$, where the implicit constant may depend on $N$ and $\alpha$.
	\end{proposition}

	\subsection{Perturbation step}\label{sec:perturbation:new}
	Compared to Section~\ref{sec:perturbation},  we need to reconstruct the perturbation $w_{q+1}$ such that $v_{q+1}(0)=\overline{v}_q(0)+ w_{q+1}(0) =u^{in}*\varphi_{\ell_{q}}$, which boils down to $w_{q+1}(0)=0$. To achieve this, we only modify the definition of the cutoffs $\eta_i$, while the building blocks remain the Mikado flows. In particular,
	we combine the ‘squiggling' space-time cutoffs used in Section~\ref{sec:perturbation} with the ‘straight' cutoffs from \cite{Ise18} to ensure $w_{q+1}(0)=0$. The construction of other components, such as the flow maps $\Phi_i$, the stress $\tilde{R}_{q,i}$ and the amplitudes $a_{(\xi,i)}$ follow the same procedure as in Section~\ref{sec:perturbation}. Therefore, we will primarily discuss the steps that require modification.
	\subsubsection{Cutoffs, energy decomposition, and amplitudes} Let us begin by defining a family of smooth, nonnegative cutoff functions $\{\tilde{\eta}_i\}_{i\geq 0}$ by
	\begin{align}
		\tilde{\eta}_i (t,x) \coloneq \begin{cases}
			\bar\eta_0(t) &   i =  0, \\
			\eta_i(t,x) &  i\ge 1,
		\end{cases} \label{e:defn-cutoff-total}
	\end{align}
	where  $\bar\eta_0$ is `straight' cutoffs defined as follows: let $\bar\eta_0\in C_c^\infty(J_0\cup I_0 \cup J_{1};[0,1])$, be identically $1$ on $I_0$, satisfy
	\begin{align*}
		\supp \bar\eta_0 = I_0 + \Big[  -\frac{\tau_q}6,\ \frac{\tau_q}6\Big]= \Big[\frac{\tau_q}{3}-\frac{\tau_q}6 , \frac{2\tau_q}{3}+\frac{\tau_q}6 \Big],
	\end{align*}
	and have the derivative estimates for $N\ge 0$:
	\[ \| \partial_t ^N \bar\eta_i\|_{C^0_t} \lesssim \tau_q^{-N}. \]
	For $i\geq 1$, $\eta_i$ are `squiggling' space-time cutoffs, which we used in Section~\ref{sec:perturbation}. In addition, the cutoffs $\tilde{\eta}_i$ still satisfy the properties \eqref{eq:eta:i:support1}-\eqref{eq:eta:i:derivative} outlined in Subsection~\ref{sec:cutoffs flow map} on the interval $[\te_1\wedge \mathfrak{t}_L,\mathfrak{t}_L]$.
	
	We next modify the energy decomposition in \eqref{def:rho q,i} by combining $\eta_i$ with $\varsigma_{q+1}$:
	\begin{equation}\label{def:rho:q,i}
		\tilde{\rho}_{q,i}(t,x):= \frac{\tilde{\eta}_i^2(t,x)}{\zeta(t)+\sum_{j\geq 1} \int_{\T^3} \tilde{\eta}_j^2(t,y)\,dy} M_L^2\varsigma_{q+1},
	\end{equation}
	where $\varsigma_q$ is defined in \eqref{def:varsigma} and $\zeta$ is a smooth function given by
	\begin{align*}
		\zeta(t)=\begin{cases}
			1, & t\leq  \te_1,\\
			\in (0,1), & t\in (\te_1, \te_1+\sfrac{\tau_q}{3}),\\
			0, & t\geq  \te_1+\sfrac{\tau_q}{3},
		\end{cases}
	\end{align*}
	with the derivative bounded by $\|\partial_t^{n}\zeta\|_{C^0_t}\lesssim_n \tau_q^{-n}$. Such function $\zeta$ ensures that $\rho_{q,i}$ is well-defined even for the times when $\sum_{j\geq 1} \int_{\T^3} \tilde{\eta}_j^2(t,y)\,dy=0$.
	From the definition \eqref{def:rho:q,i}, it follows $\sum_i \int_{\T^3} \tilde{\rho}_{q,i} =M_L^2\varsigma_{q+1}$ for all $t\in[\te_1\wedge \mathfrak{t}_L,\mathfrak{t}_L]$. By tracing back the properties \eqref{eq:eta:i:special} and \eqref{eq:eta:i:derivative}, along with $\|\partial_t \zeta \|_{C^0_t}\lesssim \tau_q^{-1}$, we obtain the following estimates for any $N\geq 0$:
	\begin{align}\label{rho q,i m}
		\|	\tilde{\rho}_{q,i}^{\sfrac12}\|_{C^N_x}\lesssim K^{\sfrac12} M_L \delta_{q+1}^{\sfrac12}.
	\end{align}
	
	Next, we introduce the localized versions of Reynolds stress as described in Subsection~\ref{sec:amplitude} 
	\begin{equation}\label{e:tildeR_def:new}
		\tilde{R}_{q,i} = \frac{\nabla\Phi_i R_{q,i} \nabla\Phi_i^T}{\tilde{\rho}_{q,i}}= \nabla\Phi_i \left( \Id -\frac{\tilde{\eta}_i^2 \mathring{\overline R}_q}{	\tilde{\rho}_{q,i}}   \right) \nabla\Phi_i^T,
	\end{equation}
	where $\Phi_i$ is the solution of the transport equation defined as in \eqref{eq:Phi:i:def}, with $\overline{v}_q$ and $z_{\ell_{q}}$ replaced by their counterparts in this section. Moreover, Proposition~\ref{Lemma:transport} also holds because we have the same bound for $\overline{v}_q$ and $z_{\ell_{q}}$. Using the same computations as in \eqref{es:Rq,i-Id}, we obtain
	\begin{align*}
		\|	\tilde R_{q,i}-\mathrm{Id}\|_{C^0_x} &\lesssim \|\nabla\Phi_i \nabla\Phi_i ^T-\Id\|_{C^0_x} +  \frac{\|\nabla\Phi_i \|_{C^0_x}^2 \|\mathring{\overline R}_q\|_{C^0_x}}{M_L^2\delta_{q+1}} \leq \tilde{M} \ell_q^{\alpha}
		\leq \frac12,
	\end{align*}
	which implies $\tilde R_{q,i}(t,x) \in \overline{B_{\sfrac 12}}(\mathrm{Id})$ for all $(t,x)$. Consequently, the amplitude functions $\tilde{a}_{(\xi,i)}$ also modified using the formula \eqref{def:amplitude a} with $\rho_{q,i}$ replaced by $\tilde{\rho}_{q,i}$, i.e.,
	\begin{align}\label{def:amplitude a:new}
		\tilde{a}_{(\xi,i)}(t,x)=\tilde{\rho}_{q,i}(t,x)^{\sfrac12} \gamma_{\xi}(\tilde R_{q,i}(t,x)),
	\end{align}
	where $\gamma_\xi \in C^\infty(\overline{B_{\sfrac 12}}(\mathrm{Id}))$ are the functions given in Lemma~\ref{l:linear_algebra}. 
	Moreover, the stress $\tilde R_{q,i}$ and amplitude functions $\tilde{a}_{(\xi,i)}$ inherit the same bounds as in Subsection~\ref{sec:amplitude}, and we have the following estimates:
	
	\begin{proposition}
		For any $t\in \tilde{I}_i$ and $N\geq0$, the stress $\tilde R_{q,i}$ satisfies the estimates
		\begin{subequations}
			\begin{align}
				\|\tilde R_{q,i}\|_{C^N_x}&\lesssim \ell_q^{-N},\label{es:Rq,i:new}
				\\ \| D_{t,q}\tilde R_{q,i}\|_{C^N_x}&\lesssim \tau_q^{-1} \ell_q^{-N},\label{es:Dt,q Rq,i:new}
			\end{align}
		\end{subequations}
		and the amplitude functions $\tilde{a}_{(\xi,i)}$ satisfy
		\begin{subequations}
			\begin{align}
				\| 	\tilde{a}_{(\xi,i)}\|_{C^N_x}&\lesssim \frac{MK^{\sfrac12}}{C_{\Lambda}}  M_L \delta_{q+1}^{\sfrac12}\ell_q^{-N}\, , \label{estimate:a CN:new}
				\\ \| D_{t,q} 	\tilde{a}_{(\xi,i)}\|_{C^N_x}&\lesssim  \frac{MK^{\sfrac12}}{C_{\Lambda}} M_L\tau_q^{-1}\delta_{q+1}^{\sfrac12} \ell_q^{-N}\, ,\label{estimate:Da CN:new}
			\end{align}
		\end{subequations}
		where $M$, $C_\Lambda$ are universal constants given in \eqref{eq:Onsager:M:def} and the implicit constant may depend on $N$.
	\end{proposition}
	\begin{proof}
		The proof closely follows that of Proposition~\ref{estimate:R:q,i} and  Proposition~\ref{es:amplitude}, as we have established identical bounds for $\overline{v}_q$, $z_{\ell_{q}}$, $\Phi_i$ and $\mathring{\overline{R}}_q$. The main distinction lies in the adjusted definition of $\tilde{\rho}_{q,i}$, which we will briefly explain below. Recalling to the definition of $\tilde{\rho}_{q,i}$, we apply \eqref{e:Rq:1:new} to obtain
		\begin{align*}
			\left\|  \frac{R_{q,i}}{\tilde{\rho}_{q,i}}\right\| _{C^N_x} \lesssim 1+ \frac{1}{M_L^2\varsigma_{q+1}} \|\mathring{\overline R}_q\|_{C^N_x} \lesssim \tilde{M} \ell_q^{-N+\alpha}\leq \ell_q^{-N},
		\end{align*}
		which requires $a$ large enough to have $\tilde{M} \ell_q^{\alpha}\ll 1$. Then, by applying the Leibniz rule for the derivative of the product and Proposition~\ref{Lemma:transport}, we obtain \eqref{es:Rq,i:new}.
		
		Next, we differentiate $\tilde{\rho}_{q,i}^{-1} R_{q,i}$ by the material derivative $D_{t,q}$ to obtain
		\begin{equation}
			D_{t,q} (\tilde{\rho}_{q,i}^{-1} R_{q,i}) = -\partial_t \left( \frac{\zeta + \sum_{j\geq 1}\int_{\T^3} \tilde{\eta}_j^2}{M_L^2\varsigma_{q+1}}\right) \mathring{\overline R}_q - \frac{\zeta+\sum_{j\geq 1}\int_{\T^3}\tilde{\eta}_j^2}{M_L^2\varsigma_{q+1}} D_{t,q} \mathring{\overline R}_q \, .
		\end{equation}
		Similar to \eqref{e:Dt_ratio}, by applying $|\partial_t\zeta|\leq \tau_q^{-1}$, along with \eqref{e:Rq:1:new} and \eqref{e:Rq:Dt:new}, we can deduce
		\begin{equation*}
			\|D_{t,q} (\tilde{\rho}_{q,i}^{-1}R_{q,i})\|_{C^N_x}
			\lesssim \frac{\delta_{q+1}^{-1} \tau_q^{-1} }{M_L^2}\|\mathring{\overline R}_q\|_{C^N_x}  + \frac{\delta_{q+1}^{-1}}{M_L^2} \|D_{t,q} \mathring{\overline R}_q\|_{C^N_x}
			\lesssim
			\tau_q^{-1} \ell_q^{-N}\, .
		\end{equation*}
		Then, by using the same routine computation as in Proposition~\ref{estimate:R:q,i}, we obtain \eqref{es:Dt,q Rq,i:new}.
		
		Let us proceed to the estimates for $\tilde{a}_{(\xi,i)}$. Similar to \eqref{chain:a xi} and \eqref{chain:a xi:2}, the bound \eqref{estimate:a CN:new} readily follows as a consequence of the chain rule. We next observe that 
		\begin{align*}
			D_{t,q} (\tilde{\rho}_{q,i}^{\sfrac12})=\left[  \partial_t \left( \frac{\tilde{\eta}_i}{(\zeta+\sum_{j\geq 1}\int_{\T^3} \tilde{\eta}_j^2)^\frac12} \right) +
			\frac{ (\overline{v}_q+z_{\ell_q}) \cdot \nabla \tilde{\eta}_i }{(\zeta+\sum_{j\geq 1}\int_{\T^3}\tilde{\eta}_j^2)^\frac12} \right] M_L \varsigma_{q+1}^{\sfrac12}.
		\end{align*}
		By combining \eqref{e:vq:1+N:new} with \eqref{e:vq:0:new} we can estimate for any $N\geq 1$
		\begin{align*}
			\|  D_{t,q} (\tilde{\rho}_{q,i}^{\sfrac12})\|_{C^N_x} &\lesssim  (\|\partial_t \tilde{\eta}_i\|_{C^N_x } + \|\partial_t \tilde{\eta}_i\|_{C^0_x } \| \tilde{\eta}_i\|_{C^N_x }+ \|\partial_t \zeta\|_{C^0_x } \| \tilde{\eta}_i\|_{C^N_x })M_L \varsigma_{q+1}^{\sfrac12}
			\\ &\quad + (\|\overline{v}_q +z_{\ell_q}\|_{C^0_x } \|\nabla \tilde{\eta}_i\|_{C^N_x }+  \|\overline{v}_q+z_{\ell_q} \|_{C^N_x } \|\nabla \tilde{\eta}_i\|_{C^0_x}) M_L \varsigma_{q+1}^{\sfrac12}
			\\ &\lesssim M_L \tau_q^{-1}\varsigma_{q+1}^{\sfrac12}+\tilde{M}M_L
			\tau_q^{-1}\ell_q^{-(N-1)+\alpha} \varsigma_{q+1}^{\sfrac12} \lesssim K^{\sfrac12} M_L\tau_q^{-1} \delta_{q+1}^{\sfrac12} \ell_{q}^{-N},
		\end{align*}
		where we used $\tilde{M}\ell_{q}^{\alpha}\ll 1$.
		Since $\|\overline v_q\|_{C^0_x} \lesssim \tilde{M} M_L \lambda_{1}^{\sfrac{3\alpha}{2}}$, we also have $ \|  D_{t,q} (\tilde{\rho}_{q,i}^{\sfrac12})\|_{C^0_x}\lesssim K^{\sfrac12} M_L \tau_q^{-1} \delta_{q+1}^{\sfrac12} $.
		Analogous to \eqref{Dt,q a}, applying the chain rule together with \eqref{rho q,i m}, \eqref{es:Rq,i:new} and \eqref{es:Dt,q Rq,i:new} yields  \eqref{estimate:Da CN:new}.
	\end{proof}
	
	\subsubsection{Definition of $v_{q+1}$ and $\mathring{R}_{q+1}$}
	The construction of the perturbation follows the approach outlined in Subsection~\ref{sec:principal w} with $a_{(\xi,i)}$ replaced by $\tilde{a}_{(\xi,i)}$, while the building blocks of the perturbation remain the Mikado flows. Specifically, the total perturbation is given by
	\begin{align*}
		\tilde w_{q+1} &= \curl \left( \sum_i \sum_{\xi \in \Lambda_i} \tilde{a}_{(\xi,i)} \, (\nabla \Phi_i)^T (V_{(\xi)}\circ \Phi_i) \right) 
		\\ &=\sum_{i} \sum_{\xi \in \Lambda_i} 	\tilde a_{(\xi,i)}(\nabla \Phi_i)^{-1} W_{(\xi)}(\Phi_i) 
		+ \sum_i \sum_{\xi \in \Lambda_i} \nabla 	\tilde a_{(\xi,i)} \times \left( (\nabla \Phi_i)^T (V_{(\xi)} (\Phi_i)  \right) =:	\tilde w_{q+1}^{(p)} +\tilde w_{q+1}^{(c)} ,
	\end{align*}
	so that it is divergence-free and has zero mean. Additionally, due to the construction of the cutoffs $\tilde{\eta}_i$, it follows that the support of $\tilde{w}_{q+1}$ is away from zero. Thus, the new velocity field $v_{q+1}:= \overline{v}_q+\tilde{w}_{q+1}$ satisfies
	\begin{align*}
		v_{q+1}(0)=\overline{v}_q(0)+\tilde{w}_{q+1}(0)= u^{in}*\varphi_{{\ell_q}}.
	\end{align*}
	The new stress $\mathring{R}_{q+1}$ is defined in the same spirit as in Subsection~\ref{sec: new Reynolds stress}, specifically by replacing the two correctors
	$w_{q+1}^{(p)}$ and $w_{q+1}^{(c)}$ with their respective modified counterparts $\tilde w_{q+1}^{(p)}$ and $\tilde w_{q+1}^{(c)}$ as shown in \eqref{Rno1}. We will omit the specific expressions for $\mathring{R}_{q+1}$.
	
	\subsection{Inductive estimates step}
	To conclude the proof of Proposition~\ref{p:iteration1}, we will show that $v_{q+1}$ and $\mathring{R}_{q+1}$ satisfy the inductive estimates \eqref{itera:a:new}-\eqref{estimate:energy:new}.
	
	We first consider $v_{q+1}$. Compared to Subsection~\ref{Estimates on v q+1}, the only minor difference is that the bound for the amplitudes
	$\tilde{a}_{(\xi,i)}$ now includes an additional $K^{\sfrac12}$. Therefore, we use a universal constant $\tilde{M}$ to control it. The proof follows the lines established in Subsection~\ref{Estimates on v q+1}, and the omitted computations can be found there. We utilize \eqref{esti:principal w 0} and \eqref{esti:corrector w 0} to derive
	\begin{align}\label{esti:w:0}
		\|\tilde{w}_{q+1}\|_{C^0_{t,x}}  \leq \frac{|\Lambda|MK^{\sfrac12}}{C_\Lambda }M_L \delta_{q+1}^{\sfrac12} +\frac{|\Lambda|MK^{\sfrac12}}{C_\Lambda \lambda_{q+1}} M_L\delta_{q+1}^{\sfrac12}\ell_q^{-1} \leq \frac{1}{2} \tilde{M}  M_L \delta_{q+1}^{\sfrac12},
	\end{align}
	where $ \tilde{M}$ is a universal constant satisfying
	\begin{align}\label{def:barM}
		100	|\Lambda|MK^{\sfrac12}\leq C_\Lambda \tilde{M}.
	\end{align}
	From \eqref{esti:principal w 1} and \eqref{esti:corrector w 1}, it follows that for any $t\in[0,\mathfrak{t}_L]$
	\begin{align}\label{esti:w:1}
		\|\tilde{w}_{q+1}(t)\|_{C^1_{x}} \leq \frac{|\Lambda|MK^{\sfrac12}  }{C_\Lambda } M_L\left( \delta_{q+1}^{\sfrac12} \lambda_{q+1} +	
		\delta_{q+1}^{\sfrac12} \ell_q^{-1} \left(1+\frac{\ell_q^{-1}}{\lambda_{q+1}}\right)\right) 
		\leq \frac12 \tilde{M} M_L \lambda_{q+1} \delta_{q+1}^{\sfrac12} .	\end{align} 
	Then, by applying the same computations as in Subsection~\ref{Estimates on v q+1}, we can deduce \eqref{itera:a:new}, \eqref{itera:b:new} and \eqref{vq+1-vq:new}.

	Let us now turn to the Reynolds stress
	$\mathring{R}_{q+1}$. Note that we have established the same estimates for $\overline{v}_q$, $z_{\ell_{q}}$ and $\Phi_i$ as in Section~\ref{sec:begin} and Section~\ref{sec:perturbation}. And the additional factor $K^{\sfrac12}$ in the bounds of $\tilde{a}_{(\xi,i)}$ can be absorbed into $\tilde{M}$, as previously discussed. Therefore, the proof for the estimate \eqref{itera:c:new} of $\mathring{R}_{q+1}$ follows exactly the same line as in Subsection~\ref{estimate on Rq+1}. Putting together Proposition~\ref{Estimates for transport error}, Proposition~\ref{Estimates for oscillation error}, Proposition~\ref{Estimates for Nash error}, Proposition~\ref{Estimates for corrector error} and Proposition~\ref{Estimates for commutator  error} from  Subsection~\ref{estimate on Rq+1}, we can establish \eqref{itera:c:new}.
	
	Finally, we control the energy similarly to that in Subsection~\ref{estimate on energy}. By definition, we find
	\begin{equation}\label{eq:deltaE ps}
		\aligned
		\big|\|v_{q+1}\|_{L^2}^2-\|v_q\|_{L^2}^2-3M_L^2\varsigma_{q+1}\big|
		&\leq \left| \int_{\mathbb{T}^3} \big(|\tilde{w}_{q+1}^{(p)}|^2 -3\varsigma_{q+1}M_L^2\big) \dif x \right| +2\left| \int_{\mathbb{T}^3}\tilde{w}_{q+1}^{(p)}\tilde{w}_{q+1}^{(c)}  \dif x \right|
		\\ &+  \int_{\mathbb{T}^3} | \tilde{w}_{q+1}^{(c)}|^2 \dif x  
		+ 2\left| \int_{\mathbb{T}^3}\overline v_q \tilde{w}_{q+1}\dif x \right|  
		+\left|\|\overline  v_q\|_{L^2}^2-\|v_q\|_{L^2}^2\right| .
		\endaligned
	\end{equation}
	The first term on the right-hand side of \eqref{eq:deltaE ps} can be made arbitrarily small using the same argument presented in Subsection~\ref{estimate on energy}. More precisely, by taking the trace of both sides of \eqref{eq:w:q+1:is:good} and using the fact that $ \mathring{\overline{R}}_q $ is traceless, along with \eqref{eq:Phi:i:bnd:b}, \eqref{estimate:a CN:new} and \eqref{esti:integral}, we deduce for $t\in [\te_1\wedge \mathfrak{t}_L,\mathfrak{t}_L]$
	\begin{equation*}
		\aligned
		\left| \int_{\mathbb{T}^3} |\tilde{w}_{q+1}^{(p)}|^2 -3M_L^2\varsigma_{q+1} \dif x \right| &
		 \leq \sum_i \sum_{\xi \in \Lambda_i} \left| \int_{\mathbb{T}^3} \tilde{a}_{(\xi,i)}^2 \tr \left[  (\nabla \Phi_i)^{-1} \xi \otimes \xi  (\nabla \Phi_i)^{-T}  \right]\left( \Big( \mathbb{P}_{\geq \sfrac{\lambda_{q+1}}{2}}(\phi_{(\xi)}^2) \Big) \circ\Phi_i \right)  \dif x\right| 
		\\& \lesssim \tilde{M}^2M_L^2 \delta_{q+1}\frac{\ell_q^{-1}}{\lambda_{q+1}}\leq \frac13 M_L^2 \lambda_{1}^{\sfrac{3\alpha}{2}} \delta_{q+1}^{\sfrac12},
		\endaligned
	\end{equation*}
	where we chose $a$ large enough to ensure $\tilde{M}^2 \ell_{q}^{-1}\lambda_{q+1}^{-1}\ll1$ in the last inequality. Returning to \eqref{eq:deltaE ps}, we control the remaining parts similarly as in Subsection~\ref{estimate on energy}. We utilize \eqref{e:vq:1+N:new}, \eqref{estimate:a CN:new}, \eqref{esti:w:0} and \eqref{esti:w:1} to obtain the following for $t\in [\te_1\wedge \mathfrak{t}_L,\mathfrak{t}_L]$
	\begin{align*}
		2&\left| \int_{\mathbb{T}^3} \tilde{w}_{q+1}^{(p)}\tilde{w}_{q+1}^{(c)}  \dif x \right|+ \int_{\mathbb{T}^3} | \tilde{w}_{q+1}^{(c)}|^2 \dif x +	2 \left| \int_{\mathbb{T}^3}\overline v_q\cdot \tilde{w}_{q+1} \dif x \right|
		\\ &\quad 
		\lesssim \tilde{M}^2M_L^2 \delta_{q+1}\frac{\ell_q^{-1}}{\lambda_{q+1}}+ \tilde{M}^2M_L^2 \delta_{q+1}\frac{\ell_q^{-2}}{\lambda_{q+1}^2}+\frac{\tilde{M}^2M_L^2 \delta_{q+1}^{\sfrac12}\tau_q^{-1}}{\lambda_{q+1}} \leq \frac13 M_L^2 \lambda_{1}^{\sfrac{3\alpha}{2}} \delta_{q+1}^{\sfrac12}.
	\end{align*} 
	For the last  term in \eqref{eq:deltaE ps}, applying \eqref{vq-vl:new}, \eqref{e:vq:vell:new} and \eqref{e:vq:0:new} implies for any $t\in[\te_1\wedge \mathfrak{t}_L, \mathfrak{t}_L]$
	\begin{align*}
		\left| \|\overline v_q(t)\|_{L^2}^2-\|v_q(t)\|_{L^2}^2\right| &\lesssim \|\overline v_q(t)-v_q(t)\|_{L^2}(\|\overline v_q(t)\|_{L^2}+\|v_q(t)\|_{L^2})
		\\ &\lesssim \tilde{M}M_L \lambda_{1}^{\sfrac{3\alpha}{2}}  (\|\overline v_q(t)-v_{\ell_q}(t)\|_{C^0_x}+\| v_q(t)-v_{\ell_q}(t)\|_{C^0_x})
		\\&	\lesssim \tilde{M}^2M_L^2\ell_q^{\alpha}  \lambda_{1}^{\sfrac{3\alpha}{2}} \delta_{q+1}^{\sfrac12} \leq \frac13 M_L^2 \lambda_{1}^{\sfrac{3\alpha}{2}} \delta_{q+1}^{\sfrac12}.
	\end{align*}
	By combining the above estimates, we obtain \eqref{estimate:energy:new}, thus concluding the proof of Proposition~\ref{p:iteration1}.
	
	\section{Proof of energy conservation for H\"older exponent beyond 1/3}\label{energy:conservation}
	In this section, we show that the energy balance is preserved when the H\"older regularity of the solution exceeds $1/3$.
	\begin{proof}[Proof of Theorem~\ref{rigid part}]
		Recall that $B=\sum_k \sqrt{c_k} \beta_k e_k$, as used in the proof of Theorem~\ref{Onsager:theorem}. As in \eqref{def:mollifiers}, let $\varphi_{\varepsilon}$ be a smooth radial mollifier in space of length $\varepsilon$. For any $\varepsilon>0$, we write:
		\begin{align*}
			u_{\varepsilon}=u*\varphi_{\varepsilon},\qquad (u\otimes u)_{\varepsilon}=(u\otimes u)*\varphi_{\varepsilon}, \qquad B_{\varepsilon}=B*\varphi_{\varepsilon}, \qquad e_k^{\varepsilon}=e_k*\varphi_{\varepsilon}.
		\end{align*}
		Observe from \eqref{eul1} that  $	\dif u_{\varepsilon}+\mathbb{P}\div[(u\otimes u)_{\varepsilon} ]\,\dif t=\dif B_{\varepsilon}$.
		Since it holds almost surely $u_{\varepsilon}\in L^2(\mathbb{T}^3)$ and $\mathbb{P}\div[(u\otimes u)_{\varepsilon} ] \in L^2(\mathbb{T}^3)$, we apply It\^o's formula (c.f. \cite[Theorem 6.1.1]{LR15}) to obtain for any $t\in [0,\infty)$:
		\begin{equation}\label{2}
			\aligned
			\|u_{\varepsilon}(t\wedge\mathfrak{s})\|_{L^2}^2-\|u_{\varepsilon}(0)\|_{L^2}^2=&  2 \int_{0}^{t\wedge\mathfrak{s}} \big\langle u_{\varepsilon}(s), \dif B_{\varepsilon}(s)  \big\rangle
			\\ &- 2 \int_{0}^{t\wedge\mathfrak{s}} \big\langle u_{\varepsilon}(s), \mathbb{P}\div[(u\otimes u)_{\varepsilon} (s)] \big\rangle \dif s +\big(\sum_k c_k \|e_k^{\varepsilon}\|_{L^2}^2\big)(t\wedge\mathfrak{s}).
			\endaligned
		\end{equation}
		
		We will control each term of \eqref{2} separately. Let us now focus on the first term on the right-hand side of \eqref{2}. By applying the mollification estimates \eqref{estimate:molli1} and \eqref{estimate:molli3}, we obtain
		\begin{align}\label{mollification bounds}
			\|e_k^{\varepsilon} -e_k\|_{L^2}\lesssim {\varepsilon} \|e_k\|_{H^1}, \quad \|u_{\varepsilon}(s)\|_{C^0_x} \lesssim  \|u(s)\|_{C^0_x},\quad \|u_{\varepsilon}(s)-u(s)\|_{C^0_x} \lesssim  \varepsilon^{\vartheta}\|u(s)\|_{C^{\vartheta}_x}.
		\end{align} 
		From \eqref{mollification bounds}, triangle inequality and It\^o's isometry, it follows that
		\begin{equation}
			\aligned
			&\E \bigg|  \int_{0}^{t\wedge\mathfrak{s}} \big\langle u_{\varepsilon}(s), \dif B_{\varepsilon}(s)  \big\rangle -  \int_{0}^{t\wedge\mathfrak{s}} \big\langle u(s), \dif B(s)  \big\rangle  \bigg|^2
			\\ & \quad \lesssim  	\E \bigg|  \int_{0}^{t\wedge\mathfrak{s}} \big\langle u_{\varepsilon}(s), \dif B_{\varepsilon}(s)  - \dif B(s)\big\rangle \bigg|^2 + \E \bigg| \int_{0}^{t\wedge\mathfrak{s}} \big\langle u_{\varepsilon}(s)-u(s), \dif B(s)  \big\rangle  \bigg|^2
			\\ &\quad \lesssim \sum_{k} c_k \E \int_{0}^{t\wedge\mathfrak{s}} \bigg|  \big\langle  u_{\varepsilon} (s), e_k^{\varepsilon} -e_k \big\rangle  \bigg|^2 \dif s+ \sum_{k} c_k \E \int_{0}^{t\wedge\mathfrak{s}} \bigg|  \big\langle  u_{\varepsilon} (s)-u(s) , e_k \big\rangle  \bigg|^2 \dif s \label{first term}
			\\ &\quad\lesssim  \varepsilon^2 \sum_{k} c_k \|e_k\|_{H^{1}}^2 \E \left( \int_{0}^{t\wedge\mathfrak{s}} \|u_{\varepsilon}(s)\|_{L^2}^2 \dif s \right) + \sum_{k} c_k \E\left( \int_{0}^{t\wedge\mathfrak{s}}  \|u_{\varepsilon} (s)-u (s)\|_{L^2}^2 \dif s \right)
			\\ & \quad\lesssim  \varepsilon^2 \tr \big((\mathrm{I}-\Delta)GG^*\big)   \E \left( \mathfrak{s} \|u\|_{C_{[0,\mathfrak{s}]}C^0_x}^2  \right) + \varepsilon^{2\vartheta} \tr(GG^*) \E\left( \mathfrak{s}  \| u\|^2_{C_{[0,\mathfrak{s}]}C^{\vartheta} _x} \right).
			\endaligned
		\end{equation}
		By the conditions $u\in L^{3q}(\Omega;C([0,\mathfrak{s}],C^{\vartheta}(\T^3,\R^3)))$,    $\E(\mathfrak{s}^p)<\infty$ for $\frac{1}{p}+\frac{1}{q}=1$ and H\"older inequality, we have 
		\begin{align}\label{bounded:moment}
			\E \left( \mathfrak{s}  \| u\|^2_{C_{[0,\mathfrak{s}]}C^{\vartheta} _x} \right) \leq 	 \E \left(   \| u\|^{2q}_{C_{[0,\mathfrak{s}]}C^{\vartheta} _x} \right)^{1/q} \E(\mathfrak{s}^p)^{1/p}\leq \E \left(   \| u\|^{3q}_{C_{[0,\mathfrak{s}]}C^{\vartheta} _x} \right)^{2/3q} \E(\mathfrak{s}^p)^{1/p}<\infty.
		\end{align}
		Hence, we deduce that \eqref{first term} converges to $0$, as $\varepsilon \to 0$. Moving on to estimate the second term on the right-hand side of \eqref{2}, by the symmetry of the Leray projector $\mathbb{P}$ and integration by parts, it follows that
		\begin{align*}
			-	\int_{\T^3} u_{\varepsilon} \cdot \mathbb{P}\div[(u\otimes u)_{\varepsilon}] \dif x = \int_{\T^3} \tr[(u\otimes u)_{\varepsilon} \nabla u_{\varepsilon}] \dif x,
		\end{align*}
		and through a direct computation, we derive:
		\begin{align*}
			\int_{\T^3} \tr [(u_{\varepsilon}\otimes u_{\varepsilon})\nabla u_{\varepsilon}] \dif x= \int_{\T^3}  u_{\varepsilon} \cdot (u_{\varepsilon} \cdot \nabla u_{\varepsilon}) \dif x =0.
		\end{align*}
		We apply the mollification estimates \eqref{estimate:molli3} and \eqref{estimate:molli2} to deduce for any $t\in [0,\mathfrak{s}]$
		\begin{align}\label{esti:com}
			\|(u\otimes u)_{\varepsilon}(t)-(u_{\varepsilon}\otimes u_{\varepsilon})(t)\|_{C^0_x}\lesssim \varepsilon^{2\vartheta} \|u\|_{C_{[0,\mathfrak{s}]}C^\vartheta_x}^2, \qquad \|\nabla u_{\varepsilon}(t)\|_{C^0_x}\lesssim \varepsilon^{\vartheta-1}\|u\|_{C_{[0,\mathfrak{s}]}C^{\vartheta}_x}.
		\end{align}
		By the same argument as \eqref{bounded:moment}, we derive
		\begin{align}\label{moment:for:u}
				\E \left( \mathfrak{s}  \| u\|^3_{C_{[0,\mathfrak{s}]}C^{\vartheta} _x} \right) \leq  \E \left(   \| u\|^{3q}_{C_{[0,\mathfrak{s}]}C^{\vartheta} _x} \right)^{1/q} \E(\mathfrak{s}^p)^{1/p}<\infty.
		\end{align}
		Combining \eqref{esti:com} with \eqref{moment:for:u}, the second term on the right-hand side of \eqref{2} can be bounded as
		\begin{equation}\label{second term}
			\aligned
			\E\bigg| \int_{0}^{t\wedge\mathfrak{s}} \langle u_{\varepsilon}(s), \mathbb{P}\div[(u\otimes u)_{\varepsilon}](s) \rangle \dif s  \bigg|
			&=  \E 
			\left|\int_{0}^{t \wedge \mathfrak{s}}\int_{\T^3} \tr [((u\otimes u)_{\varepsilon}-u_{\varepsilon}\otimes u_{\varepsilon})\nabla u_{\varepsilon}] \dif x \dif s  \right| 
			\\ &\lesssim \varepsilon^{3\vartheta-1}   \E \big(  \mathfrak{s} \|u\|_{C_{[0,\mathfrak{s}]}C^\vartheta_x}^3 \big) , 
			\endaligned
		\end{equation}
		which converges to zero as $\varepsilon \to 0$. For the last term on the  right-hand side of \eqref{2}, by applying the mollification estimate \eqref{mollification bounds} and the dominated convergence theorem, we deduce:
		\begin{align}\label{third term}
			\lim_{\varepsilon \to 0} \sum_k c_k \|e_k^{\varepsilon}\|_{L^2}^2 =\sum_k c_k \lim_{\varepsilon \to 0}\|e_k^{\varepsilon}\|_{L^2}^2  = \sum_k c_k \|e_k\|_{L^2}^2= \tr (GG^*).
		\end{align}
		Therefore, by combining \eqref{2}, \eqref{first term}, \eqref{second term} and \eqref{third term}, we obtain for any $t\in [0,\infty)$:
		\begin{align*}
			&\E \bigg| 	\|u(t\wedge\mathfrak{s})\|_{L^2}^2-\|u(0)\|_{L^2}^2 -2 \int_{0}^{t\wedge\mathfrak{s}} \big\langle u(s), \dif B(s )  \big\rangle -\tr\big(GG^*\big) (t\wedge\mathfrak{s}) \bigg|
			\\ & \quad \leq \lim_{\varepsilon \to 0} \E \bigg| 	\|u(t\wedge\mathfrak{s})\|_{L^2}^2- 	\|u_{\varepsilon}(t\wedge\mathfrak{s})\|_{L^2}^2 \bigg| + \lim_{\varepsilon \to 0} \E \bigg|  \|u_{\varepsilon}(0)\|_{L^2}^2-	\|u(0)\|_{L^2}^2 \bigg| 
			\\ &\quad + \lim_{\varepsilon \to 0} \E\bigg|2 \int_{0}^{t\wedge\mathfrak{s}} \big\langle u_{\varepsilon}(s), \mathbb{P}\div[(u\otimes u)_{\varepsilon}(s)] \big\rangle \dif s  \bigg|
			+ \lim_{\varepsilon \to 0} \E \bigg| \big(\sum_k c_k \|e_k^{\varepsilon}\|_{L^2}^2\big)(t\wedge\mathfrak{s}) - \tr(GG^*) (t\wedge\mathfrak{s})  \bigg|
			\\ &  \quad +   \lim_{\varepsilon \to 0} \E \bigg| 2 \int_{0}^{t\wedge\mathfrak{s}} \big\langle u_{\varepsilon}(s), \dif B_{\varepsilon}(s)  \big\rangle - 2  \int_{0}^{t\wedge\mathfrak{s}} \big\langle u(s), \dif B(s)  \big\rangle  \bigg|=0.
		\end{align*}
		This estimate, combined with the continuity argument, implies \eqref{energy:conserve} holds $\mathbf{P}$-a.s. 	for any $t\in [0,\infty)$.	
	\end{proof}

	\appendix
	\renewcommand{\appendixname}{Appendix~\Alph{section}}
	\renewcommand{\theequation}{A.\arabic{equation}}

	\renewcommand{\theequation}{A.\arabic{equation}}
	\section{Some technical tools of convex integration}\label{sec:appendix c}
	In this section, we recall some critical tools used during the iteration.\\
	
	\noindent$\bullet$ {\bf Mollification estimates}
	
	We first recall the following quadratic commutator estimate from \cite[Lemma 1]{CDS12}.
	\begin{proposition}\label{commutator esti}
		Let $f,g\in C^{\infty}( \mathbb{T}^3,\R^3)$ and $\varphi_{\varepsilon}$ be a space standard mollifier as defined in \eqref{def:mollifiers}. Then for any $r,s\geq0$
		\begin{align}
			\|f*\varphi_{\varepsilon}\|_{C^{r+s}_{x}} &\lesssim \varepsilon^{-s}\|f\|_{C^r_{x}}\, ,\label{estimate:molli1}
			\\	\|f-f*\varphi_{\varepsilon}\|_{C^{r}_{x}} & \lesssim \varepsilon^{s}\|f\|_{C^{r+s}_{x}}\, ,\label{estimate:molli3}
		\end{align}
		and for any $r,s\in(0,1]$, $l\geq 0$
		\begin{align}\label{estimate:molli2}
			\|(f*\varphi_\varepsilon)(g*\varphi_{\varepsilon})-(fg)*\varphi_{\varepsilon}\|_{C^{l}_{x}} \lesssim \varepsilon^{r+s-l}\|f\|_{C^r_x}\|g\|_{C^s_x} \, .
		\end{align}
	\end{proposition}
	In the proof of Proposition~\ref{p:Rq}, we find the following commutator estimate from \cite[Proposition D.1]{BDLSV19} is useful. 
	\begin{proposition} \label{CZ_comm}
		Suppose $\kappa \in (0,1)$ and $N\in \mathbb{N}_0$. Let $\mathcal{T}$ be a Calder\'on-Zygmund operator and $g \in C^{N+1+\kappa}(\T^3,\R^3)$ be a divergence-free vector field. Then, we have
		\begin{align*}
			\|[ g \cdot \nabla, \mathcal{T}] f \|_{C^{N+\kappa}_x} \lesssim \|g\|_{C^{1+\kappa}_x} \|f\|_{C^{N+\kappa}_x} + \|g\|_{C^{N+1+\kappa}_x} \|f\|_{C^\kappa_x}
		\end{align*}
		for any $f \in C^{N+\kappa}(\T^3,\R^3)$, where the implicit constants only depend on $\kappa,N$.
	\end{proposition}
	
	\noindent$\bullet$ {\bf ‘Squiggling’ space-time cutoffs $\eta_i$}
	
	We recall the construction of the ‘squiggling’ cutoffs in \cite[Section 4.4]{KMY22}, which is an adaptation from \cite[Lemma 5.3]{BDLSV19}. Let $\varepsilon \in (0,\frac13)$,  $\varepsilon_0\ll 1$ and define for $1\leq i$ the sets
	\begin{align*}
		\tilde{S}_{i}&:=  \Big[\te_i+\frac{\varepsilon\tau_q}3, \te_i+ \frac{(3-\varepsilon)\tau_q}3\Big] \subset \R ,\\
		\tilde{\Omega}_i&:= \Big\{\Big(t+\frac{2\varepsilon\tau_q}3\sin (2\pi x_1),x \Big) : t \in \tilde{S}_{i} , \ x \in\mathbb T^3 \Big\}\subset  \R\times \T^3 .
	\end{align*}
	Then, using the mollifiers $\varphi_{\varepsilon_0}$ and $\psi_{\varepsilon_0\tau_q} $ in \eqref{def:mollifiers} to mollify $\mathbf{1}_{\tilde{\Omega}_i} $ in space and time as follows:
	\begin{align*}
		\eta_i (t,x)&:= \mathbf{1}_{\tilde{\Omega}_i}*_t\psi_{\varepsilon_0\tau_q} *_x\varphi_{\varepsilon_0} = \frac1{\varepsilon_0 \tau_q} \frac1{\varepsilon_0^3} \iint \mathbf{1}_{\tilde{\Omega}_i} (s,y) \psi \big(\frac{t-s}{\varepsilon_0\tau_q} \big)\varphi\Big( \frac{x-y}{\varepsilon_0}\Big) \dif s \dif y. 
	\end{align*}
	One may check the propoties \eqref{eq:eta:i:support1}-\eqref{eq:eta:i:derivative} about the cutoffs $\eta_i$ follow by taking $\varepsilon\in (0,\frac13)$ and $\varepsilon_0\ll 1$.

	\noindent$\bullet$ {\bf An inverse-divergence operator}
	
	We use the inverse-divergence operator $\mathcal{R}$ from \cite{DelSze13} which acts on vector fields $v$ with $\int_{\mathbb{T}^3}v\dif x=0$ as
	\begin{equation*}
		(\mathcal{R}v)^{kl}=(\partial_k\Delta^{-1}v^l+\partial_l\Delta^{-1}v^k)-\frac{1}{2}(\delta_{kl}+\partial_k\partial_l\Delta^{-1})\div\Delta^{-1}v,
	\end{equation*}
	for $k,l\in\{1,2,3\}$. The above inverse-divergence operator has the property that  $\mathcal{R}v(x)$ is a symmetric trace-free matrix for each $x\in\mathbb{T}^3$, and $\mathcal{R}$ is a right inverse of the div operator, i.e. $\div(\mathcal{R} v)=v$. For general $f$, we overload notation and denote $\mathcal{R}f= \mathcal{R}(f-\int_{\mathbb{T}^3} f)$. The main properties of $\mathcal{R}$ are proven in \cite[Section 4]{DelSze13}.
	
	\noindent$\bullet$ {\bf A stationary phase lemma}
	
	We also recall the following stationary phase lemma adapted to our setting (see for example \cite[Lemma 5.7]{BV19} and \cite [Lemma 2.2]{DS17}) which makes rigorous the fact that the inverse-divergence $\mathcal{R}$ obeys the same elliptic regularity estimates as $|\nabla|^{-1}$. We refer the reader to \cite{DS17} for the proof of the following stationary phase lemma. 
	
	\begin{proposition}\label{prop.inv.div}
		Let $\alpha \in (0,1)$ and $N \geq 1$. Let $a \in C^\infty(\mathbb{T}^3)$, $\Phi \in C^\infty(\mathbb{T}^3,\mathbb{R}^3)$ be smooth functions and assume that there exists a constant $\hat C$ such that
		\[\hat C^{-1} \leq |\nabla \Phi| \leq \hat C\]
		holds on $\mathbb{T}^3$. Then
		\begin{align}\label{esti:integral}
			\left| \int_{\mathbb{T}^3} a(x)e^{i\lambda \xi \cdot \Phi(x)} \dif x\right| \lesssim \frac{\|a\|_{C^N}+\|a\|_{C^0} \|\nabla \Phi\|_{C^N}}{\lambda^N},
		\end{align}
		and for the operator $\mathcal{R}$ recalled above, we have
		\begin{equation}
			\left\|\mathcal{R}\left(a(x)e^{i\lambda \xi\cdot \Phi(x)}\right)\right\|_{C^\alpha}\lesssim \frac{\|a\|_{C^0}}{\lambda^{1-\alpha}}+\frac{\|a\|_{C^{N+\alpha}}+\|a\|_{C^0}\|\nabla \Phi\|_{C^{N+\alpha}}}{{\lambda^{N-\alpha}}},
		\end{equation}
		where the implicit constants depend on $\hat C$, $\alpha$ and $N$, but not on the frequency $\lambda$.
	\end{proposition}

	\renewcommand{\theequation}{B.\arabic{equation}}
	\section{Mikado flows}\label{sec:Mikado}
	In this part, we recall the construction and the main properties of the Mikado flows from~\cite{BV19} which is adapted to the convex integration scheme in Proposition~\ref{p:iteration}.
	We point out that the construction is entirely deterministic, meaning that none of the functions below depends on $\omega$. Let us begin with the following geometric lemma which can be found in \cite[Lemma 6.6]{BV19}.
	\begin{lemma}\label{l:linear_algebra}
		Denote by $\overline{B_{\sfrac 12}}(\mathrm{Id})$ the closed ball of radius $1/2$ around the identity matrix $\Id$, in the space of symmetric $3\times 3$ matrices.
		There exist mutually disjoint sets  $\{\Lambda_i\}_{i=0,1}  \subset\mathbb S^2\cap \mathbb Q^3$ such that for each $\xi \in \Lambda_i$ there exists a $C^\infty$-smooth  functions $\gamma_{\xi}: \overline{B_{\sfrac 12}}(\mathrm{Id})\rightarrow \mathbb R$ such that
		\begin{align*}
			R=\sum_{\xi\in \Lambda_{i}} \gamma^2_{\xi}(R)(\xi\otimes \xi)
		\end{align*} 
		for every symmetric matrix $R$ satisfying $|R-\mathrm{Id}| \leq \sfrac12$, and for each $i \in \{0,1\}$. 
		Moreover, for $i \in \{0,1\}$, and each $\xi\in \Lambda_i $, let use define $A_{\xi}\in\mathbb S^2\cap \mathbb Q^3$ to be an orthogonal vector to $\xi$. Then for each $\xi\in \Lambda_i$, we have that $\{\xi,A_{\xi},\xi\times A_{\xi}\}\subset \mathbb S^2\cap \mathbb Q^3$ form an orthonormal basis for $\mathbb R^3$. Furthermore, we label by $n_*$  the smallest natural such that 
		\begin{equation*}
			\left\{n_*\, \xi,~n_* \, A_{\xi},~n_* \, \xi\times A_{\xi}\right\}\subset  \mathbb{Z}^3\,
		\end{equation*}
		for every $\xi \in \Lambda_i$ and for every $i \in \{0,1\}$. For a sufficiently large constant $C_\Lambda \geq 1$ to be chosen in the sequel, it is convenient to denote $M$ geometric constant such that
		\begin{align}	\label{eq:Onsager:M:def}
			M \geq C_\Lambda \sup_{\xi \in \Lambda_0 \cup \Lambda_1} ( \|\gamma_\xi\|_{C^0} +  \sum_{j\leq N}\| D^j  \gamma_\xi\|_{C^0} ) \, ,
		\end{align}
		holds for $n$ large enough. This parameter is universal.
	\end{lemma}
	
	Next, we recall the construction and properties of the Mikado flows from \cite[Section 6.4]{BV19}. Let $\Psi:\mathbb R^2\rightarrow \mathbb R$ be a $C^\infty$ smooth function with support contained in a ball of radius 1 around the origin. We normalize $\Psi$  such that $\phi=- \Delta  \Psi$  obeys
	\begin{equation*}
		\int_{\R^2} \phi^2(x_1,x_2)\,\dif x_1 \dif x_2= 4 \pi^2\,.
	\end{equation*}
	By definition we know $\int_{\R^2}\phi \,\dif x=0$.
	Moreover, since $\supp \Psi, \phi \subset  \T^2$, we abuse notation and still denote by $\Psi, \phi$ the $\T^2$-periodized versions of $\Psi$ and $\phi$.   
	Then, for any large  $\lambda \in \N$ and every $\xi \in \Lambda_i$, we introduce the functions
	\begin{subequations}
		\label{eq:phi:rotate}
		\begin{align}
			\Psi_{(\xi)}(x)&:=\Psi_{\xi, \lambda}(x):=\Psi(n_* \lambda (x-\alpha_{\xi})\cdot A_\xi,n_* \lambda (x-\alpha_\xi) \cdot (\xi\times A_{\xi}))\, ,\\
			\phi_{(\xi)}(x)&:=\phi_{\xi,\lambda}(x):=\phi(n_* \lambda (x-\alpha_{\xi})\cdot A_\xi,n_* \lambda (x-\alpha_\xi)\cdot (\xi\times A_{\xi}))\,, 
		\end{align}
	\end{subequations}
	where $\alpha_\xi \in \R^3$ are {\em shifts} to ensure that the functions $\{\phi_{(\xi)} \}_{\xi \in \Lambda_i}$ and $\{\Psi_{(\xi)} \}_{\xi \in \Lambda_i}$ have mutually disjoint support. In addition, we choose the constant $C_\Lambda$ in \eqref{eq:Onsager:M:def} as 
	\begin{align}\label{def:CLambda}
		C_\Lambda = 32n_*|\Lambda|(\|\phi\|_{C^1}+\|\Psi\|_{C^2}),
	\end{align}
	where $|\Lambda|$ is the cardinality of the set $\Lambda_0 \cup \Lambda_1$.
	
	Note that since $n_*  A_\xi$ and $n_* \xi \times A_\xi \in \mathbb{Z}^3$, and $\lambda \in \N$, the functions $\Psi_{(\xi)}$ and $\phi_{(\xi)}$ are $(\sfrac{\T}{\lambda})^3$-periodic. By construction we have that $\{\xi, A_\xi,\xi \times A_\xi\}$ are an orthonormal basis or $\R^3$, and hence $\xi \cdot \nabla \Psi_{(\xi)}(x) = \xi \cdot \nabla \phi_{(\xi)}(x) = 0$.  From the normalization of $\phi$ we have that $\int_{\T^3} \phi_{(\xi)}^2 \dif x = 1$ and $\phi_{(\xi)}$ has zero mean on $(\sfrac{\T}{\lambda})^3$.  Since $\phi=-\Delta  \Psi$  we have that $(n_*\lambda)^2 \phi_{(\xi)} = - \Delta \Psi_{(\xi)}$. 
	
	With this notation, the {\em Mikado flows} $W_{(\xi)} \colon \T^3 \to \R^3$ are defined as
	\begin{align}
		W_{(\xi)}(x) := W_{\xi,\lambda}(x) := \xi \, \phi_{(\xi)}(x) \, .
		\label{eq:Mikado:def}
	\end{align}
	Moreover, by the choice of $\alpha_\xi$ we have that 
	\begin{align}
		W_{(\xi)} \otimes W_{(\xi')} \equiv 0, \quad \mbox{for} \quad \xi \neq \xi' \in   \Lambda_i \,,
		\label{eq:Mikado:2}
	\end{align}
	for $i\in \{0,1\}$, and by normalization of $\phi_{(\xi)}$ we obtain
	\begin{align*}
		\int_{\T^3} W_{(\xi)}(x) \otimes W_{(\xi)}(x)\, \dif x=\xi\otimes \xi\,. 
	\end{align*}
	Lastly, these facts combined with Lemma~\ref{l:linear_algebra} impliy that
	\begin{align}
		\sum_{\xi \in \Lambda_i} \gamma_{\xi}^2(R) \int_{\T^3} W_{(\xi)}(x) \otimes W_{(\xi)}(x) \dif x = R \, ,
		\label{eq:Mikado:4}
	\end{align}
	for every $i \in \{0,1\}$ and any symmetric matrix $R \in \overline B_{\sfrac 12}(\mathrm{Id})$.
	
	To conclude this section we note that $W_{(\xi)}$ may be written as the $\curl$ of a vector field, a fact which is useful in defining the incompressibility corrector in Section~\ref{sec:principal w}. Indeed, by $\xi \cdot \nabla \Phi_{(\xi)} = 0$, $-\frac{1}{(n_* \lambda)^2} \Delta \Phi_{(\xi)} = \phi_{(\xi)}$ and the identity $\curl \curl =\nabla \div - \Delta $ we obtain
	\begin{align}
		W_{(\xi)}=   - \xi \left( \frac{1}{(n_* \lambda)^2} \Delta \Psi_{(\xi)} \right) = \curl\left( \frac{1}{(n_* \lambda)^2} \curl (\xi \Psi_{(\xi)}) \right) = \curl\left( \frac{1}{(n_* \lambda)^2} \nabla \Psi_{(\xi)} \times \xi \right)   \,.
		\label{eq:Mikado:curl}
	\end{align}
	For notational simplicity, we   define 
	\begin{align}
		V_{(\xi)} =  \frac{1}{(n_* \lambda)^2} \nabla \Psi_{(\xi)} \times \xi 
		\label{eq:Mikado:curl:2}
	\end{align}
	so that $\curl V_{(\xi)} = W_{(\xi)}$. With this notation we have the bounds 
	\begin{align}
		\|W_{(\xi)}\|_{C^N} + \lambda \|V_{(\xi)}\|_{C^N} \lesssim \lambda^{N}
		\label{eq:Mikado:bounds}
	\end{align}
	for $N\geq 0$.

	\renewcommand{\theequation}{C.\arabic{equation}}
	\section{Estimates for transport equations}\label{sec:esti:transport}
	We first recall some standard estimates for solutions to the transport equation on $[t_0,T]$:
	\begin{equation}\label{eq:phi}
		\aligned
		(\partial_t+v\cdot\nabla)f&=g ,
		\\ f(t_0,x)&=f_0.
		\endaligned
	\end{equation}
	The following proposition is given in \cite[Appendix B]{BDLSV19} and the proof follows by interpolation from the corresponding result in \cite[Appendix D]{BDLIS16}.
	\begin{proposition}
		Assume $(t-t_0)\|v\|_{C^0_{[t_0,t]}C^1_x}\leq 1$. Then, any solution f of \eqref{eq:phi} satisfies
		\begin{equation}\label{eq:Gronwall1}
			\|f(t)\|_{C_x^\alpha}\leq e^{\alpha}\left( \|f_0\|_{C_x^\alpha}+\int_{t_0}^{t}\|{g(\tau)}\|_{C_x^\alpha} \dif \tau \right),
		\end{equation}
		for all $\alpha\in[0,1)$. More generally, for any $N\geq1$ and $\alpha\in[0,1)$
		\begin{equation}\label{eq:GronwallN}
			\aligned
			\|f(t)\|_{C_x^{N+\alpha}}& \lesssim \|f_0\|_{C_x^{N+\alpha}}+(t-t_0)\|{v}\|_{C^0_{[t_0,t]}C^{N+\alpha}_x} \|f_0\|_{C_x^1} 
			\\ &\quad +\int_{t_0}^{t} \left( \|g(\tau)\|_{C_x^{N+\alpha}}+(t-\tau)\|{v}\|_{C^0_{[t_0,t]}C^{N+\alpha}_x} \|g(\tau)\|_{C_x^1} \right) \dif \tau,
			\endaligned
		\end{equation}
		where the implicit constant depends on $N$ and $\alpha$.
		Define $\Phi$ to be the solution of \eqref{eq:phi} with $g=0$ and $\Phi(t_0,x)=x$. Under the same assumptions as above, we have
		\begin{align}
			\| \nabla \Phi(t)-\mathrm{Id}\|_{C^0_x} & \lesssim  e^{(t-t_0)\|v\|_{C^0_{[t_0,t]}C^1_x}}-1\lesssim(t-t_0)\|v\|_{C^0_{[t_0,t]}C^1_x} \, , \label{eq:B5}
			\\ \|  \Phi(t)\|_{C^N_x} & \lesssim (t-t_0)\|v\|_{C^0_{[t_0,t]}C^N_x}, \, N\geq 2 .\label{eq:B6}
		\end{align}
	\end{proposition}
	
	\begin{proof}[Proof of Proposition~\ref{Lemma:transport}]
		First, \eqref{eq:Phi:i:bnd:a} is a direct consequence from \eqref{e:vq:1+N}, \eqref{eq:B5} and the fact that for any $t\in [\te_{i-1},\te_{i+1}]\cap [0,\mathfrak{t}_L]$
		\begin{align*}
			\tau_q\|{\overline{v}_q+z_{\ell_q}\|_{C^0_{[\te_{i-1},t]}C_x^1}}\leq 	\tau_q \bar{M} \tau_q^{-1}\ell_{q}^{\alpha}= \bar{M}\ell_{q}^{\alpha}\ll1.
		\end{align*}
		Using \eqref{estimate:vellzell}, \eqref{e:vq:1+N} and \eqref{eq:B6} we obtain for any $N\geq1$ and $t\in [\te_{i-1},\te_{i+1}]\cap [0,\mathfrak{t}_L]$
		\begin{align}\label{eq:B7}
			\|  \nabla  \Phi_i(t)\|_{C^N_x}\lesssim \tau_q\|{\overline{v}_q+z_{\ell_q}\|_{C^0_{[\te_{i-1},t]}C_x^{N+1}}}\lesssim \tau_q  \bar{M}  \tau_q^{-1}\ell_q^{-N+\alpha} \leq \ell_q^{-N},
		\end{align}
		where the last inequality is justified by $\bar{M}\ell_{q}^{\alpha}\ll 1$.
		We also observe from \eqref{eq:Phi:i:bnd:a} that $(\nabla \Phi_i)^{-1}$ is well-defined on $[\te_{i-1},\te_{i+1}]$ and we have  $\|(\nabla \Phi_i)^{-1}\|_{C^0_x}\lesssim1$. Moreover, differentiating both sides of $\nabla \Phi_i (\nabla \Phi_i)^{-1}=\Id$ we obtain for any $t\in [\te_{i-1},\te_{i+1}]\cap [0,\mathfrak{t}_L]$ and $N\geq 0$
		\begin{align}\label{eq:B8}
			\|(\nabla \Phi_i)^{-1}\|_{C^{N+1}_x}\lesssim  \|(\nabla \Phi_i)^{-1}\|_{C^{N}_x}   \|  \nabla  \Phi_i\|_{C^1_x}+\|(\nabla \Phi_i)^{-1}\|_{C^{0}_x} \|  \nabla  \Phi_i\|_{C^{N+1}_x}.
		\end{align}
		Substituting \eqref{eq:B7} into \eqref{eq:B8}, we then obtain \eqref{eq:Phi:i:bnd:b}.
		In order to establish \eqref{eq:Phi:i:bnd:c}, applying a gradient to \eqref{eq:Phi:i:def} we observe that
		\begin{align*}
			D_{t,q} \nabla \Phi_i = - \nabla \Phi_i \, D (\overline v_q+z_{\ell_q})  \, .
		\end{align*}
		Therefore, we use \eqref{e:vq:1+N} and \eqref{eq:Phi:i:bnd:b} to deduce for any $t\in [\te_{i-1},\te_{i+1}]\cap [0,\mathfrak{t}_L]$
		\begin{align*}
			\|D_{t,q} \nabla \Phi_i \|_{C^N_x} &\lesssim \| \nabla \Phi_i\|_{C^0_x}\|\overline v_q+z_{\ell_q}\|_{C^{N+1}_x}+\| \nabla \Phi_i\|_{C^N_x}\|\overline v_q+z_{\ell_q}\|_{C^1_x} \lesssim \bar{M} \tau_q^{-1} \ell_q^{-N+\alpha}\leq \tau_q^{-1} \ell_q^{-N} \, ,
		\end{align*}
		which gives \eqref{eq:Phi:i:bnd:c}.
	\end{proof}

	\renewcommand{\theequation}{D.\arabic{equation}}
	\section{Proof of Lemma~\ref{lem:local_existence}}\label{appendix:proof:lemma 4.2}
	For completeness, we give the proof of Lemma~\ref{lem:local_existence} in this appendix. The idea follows from \cite[Lemma 4.1]{BHP23}, but the details of calculations are different due to the appearance of $Z$ in the advective term. We now consider the Euler system
	\begin{equation}\label{euler system}
		\aligned
		\begin{cases}
			\partial_t v+ (v+Z)\cdot \nabla (v+Z)+\nabla p=0,
			\\ \div v =0,
			\\ v(0)=v_0,
		\end{cases}
		\endaligned
	\end{equation}
	where $v_0$ is a divergence-free initial condition and $Z\in C([0,T], C^{\infty}(\T^3,\R^3))$ for some $T>0$.
	
	It is well-known (c.f. \cite[Corollary 3.2]{BM02}) that given a divergence-free initial data $v_0\in H^m$ for some $m>\frac{d}{2}+1$, there exists a maximal time $T^*$ (depending on $v_0$ and $Z$) and a unique solution $v\in C([0,T^*),H^m)$ to the Euler system \eqref{euler system}. Moreover, for $m>\frac{d}{2}+1$, if $\|v(t)\|_{H^m}+\|Z(t)\|_{H^m}<\infty$ on $[0,T)$ for some $T\in (0,\infty)$, then $T^*>T$. It follows that if $\|v(t)\|_{C^m_x}+\|Z(t)\|_{C^m_x}<\infty$ on $[0,T)$, then $T^*>T$. Therefore, to establish the well-posedness of the solution $v$ stated in Lemma~\ref{lem:local_existence}, it suffices to show that for the given $v_0$, $Z$ and $ \tau$ in Lemma~\ref{lem:local_existence}, the maximal time $T^*>\tau$. To achieve this, we first present the following lemma.
	
	\begin{lemma}\label{lemma:e1}
		Given a small $\kappa\in (0,1)$ and $\tilde{T}>T>0$, let $v_0\in C^{\infty}(\T^3,\R^3)$ be a divergence-free initial condition and $Z\in C([0,\tilde{T}], C^{\infty}(\T^3,\R^3))$. Suppose that $v\in C([0,T),C^{\infty}(\T^3,\R^3))$ is a solution to \eqref{euler system} with $\|v(t)\|_{C_x^{1+\kappa}}<\infty$ on $[0,T)$, then the maximal time $T^*>T$. 
	\end{lemma}
	\begin{proof}
		To obtain this result, we first estimate the $C_x^{N+\kappa}$-norm for $v$. Let $\theta$ be a multi-index with $|\theta|=N \in \mathbb{N}$, we have
		\begin{align*}
			( \partial_t+(v+Z)\cdot \nabla ) D^{\theta}v= -D^{\theta}[(v+Z) \cdot \nabla Z] -\sum_{\substack{0\leq |\theta_2|\leq N-1, \\ |\theta_1|+|\theta_2| =N} 
			}D^{\theta_1}(v+Z)\cdot \nabla D^{\theta_2}v-\nabla D^{\theta} p.
		\end{align*}
		We use interpolation to estimate the first term on the right-hand side as
		\begin{align}\label{estimate for u1}
			\|D^{\theta}[(v+Z) \cdot \nabla Z]\|_{C^{\kappa}_x} &\lesssim \|v+Z\|_{C^{N+\kappa}_x}\|Z\|_{C^{1+\kappa}_x}+\|v+Z\|_{C^{\kappa}_x}\|Z\|_{C^{N+1+\kappa}_x}.
		\end{align}
		Similarly, the $C_x^{\kappa}$-norm for the second term can be bounded by 
		\begin{align}\label{estimate for u2}
			\|v+Z\|_{C^{N+\kappa}_x}\|v\|_{C^{1+\kappa}_x}+\|v+Z\|_{C^{1+\kappa}_x}\|v\|_{C^{N+\kappa}_x}.
		\end{align}
		Using the equation for the pressure $-\Delta p=\tr[\nabla(v+Z)\nabla(v+Z)]$ and Schauder estimates, we obtain
		\begin{align}\label{estimate for u3}
			\|	\nabla D^{\theta} p\|_{C^{\kappa}_x}&\lesssim \|\tr[\nabla(v+Z)\nabla(v+Z)]\|_{C^{N-1+\kappa}_x} \lesssim \|v+Z\|_{C^{N+\kappa}_x}\|v+Z\|_{C^{1+\kappa}_x}.
		\end{align}
		By combining the estimates \eqref{estimate for u1}, \eqref{estimate for u2} and \eqref{estimate for u3} above, and using the estimate \eqref{eq:Gronwall1} for the transport equations, we obtain for any $t\in[0,T)$ and $N\in \mathbb{N}$
		\begin{equation}\label{estimate: uN}
			\aligned
			\|&v(t)\|_{C^{N+\kappa}_x}\lesssim 	\|v_0\|_{C^{N+\kappa}_x} + \int_{0}^{t} \|v(s)\|_{C^{N+\kappa}_x}\big(\|v(s)\|_{C^{1+\kappa}_x}+\|Z(s)\|_{C^{1+\kappa}_x}\big) \dif s
			\\ &+\int_{0}^{t}  \left(\|Z(s)\|_{C^{N+\kappa}_x}\big(\|v(s)\|_{C^{1+\kappa}_x}+\|Z(s)\|_{C^{1+\kappa}_x}\big) + \|Z(s)\|_{C^{N+1+\kappa}_x}\big(\|v(s)\|_{C^{\kappa}_x}+\|Z(s)\|_{C^{\kappa}_x}\big)\right)  \dif s
			\\ &\qquad \lesssim \|v_0\|_{C^{N+\kappa}_x}+ T\big(\|v\|_{C_{T-}C^{1+\kappa}_x}+\|Z\|_{C_{T}C^{{1+\kappa}}_x}\big)\|Z\|_{C_{T}C^{N+1+\kappa}_x}
			 \\ &\qquad \qquad  +\big(\|v\|_{C_{T-}C^{1+\kappa}_x}+\|Z\|_{C_{T}C^{1+\kappa}_x}\big) \int_{0}^{t} \|v(s)\|_{C^{N+\kappa}_x} \dif s .
			\endaligned
		\end{equation}
	  Here, we denote $\|v\|_{C_{T-}C^{1+\kappa}_x}:=\sup_{t\in[0,T)} \|v(t)\|_{C^{1+\kappa}_x} <\infty $, which differs from the notation $\|\cdot\|_{C_T}$ in Subsection~\ref{notation 1} and denote $\|Z\|_{C_TC^{j+\kappa}_x}:=\sup_{t\in[0,T]} \|Z(t)\|_{C^{j+\kappa}_x}<\infty $ for any $j\in \mathbb{N}$. Then, it follows from \eqref{estimate: uN} and Gr\"onwall's inequality that for $N> \frac{d}{2}+1$ and any $t\in[0,T)$ 
		\begin{align*}
			\|v(t)\|_{C^{N+\kappa}_x}	\lesssim 	\big(\|v_0\|_{C^{N+\kappa}_x} +T\big( &\|v\|_{C_{T-}C^{1+\kappa}_x}+\|Z\|_{C_{T}C^{{1+\kappa}}_x}\big)\|Z\|_{C_{T}C^{N+1+\kappa}_x}\big)e^{T\big(\|v\|_{C_{T-}C^{1+\kappa}_x}+\|Z\|_{C_{T}C^{1+\kappa}_x}\big)}<\infty .
		\end{align*}		
		Therefore, $T^*>T$.
	\end{proof}
	We next recall the following nonlinear Gr\"onwall's inequality, which serves as a technical tool for the proof of Lemma~\ref{lem:local_existence}.
	\begin{lemma}\label{gronwall lemma}
		Assume that $A,C\geq 0$ are two constants and $f$ is a continuous non-negative function such
		that
		\begin{align*}
			f(t)\leq A+\int_{0}^{t} (f(s)+C)^2 \dif s.
		\end{align*}
		Then for any $t\in\left( 0,\frac{1}{2(A+C)}\right) $ we have 
		\begin{align*}
			f(t)\leq A+ \frac{A+C}{1-(A+C)t}-(A+C) \leq 2A+C.
		\end{align*}
	\end{lemma}
	\begin{proof}
		Let $F(t)=\int_{0}^{t} (f(s)+C)^2 \dif s$. Then $F'(t)\leq (A+C+F(t))^2$, or equivalently 
		\begin{align*}
			\frac{\dif}{\dif t}\left( 	-\frac{1}{A+C+F(t)}\right) \leq 1,
		\end{align*}
		which implies $-(A+C+F(t))^{-1}+(A+C)^{-1}\leq t $. Then for any $t\in\left( 0,\frac{1}{2(A+C)}\right) $, we have
		\begin{align*}
			F(t)\leq \left( (A+C)^{-1} -t\right)^{-1}- (A+C)\leq  2(A+C)- (A+C)\leq A+C.
		\end{align*}
		Hence, we derive $f(t)\leq A+F(t)\leq 2A+C$.
	\end{proof}
	\begin{proof}[Proof of Lemma~\ref{lem:local_existence}]
		With the above two lemmas at hand, we now proceed the proof of Lemma~\ref{lem:local_existence}. 
		Let $T>0$ and $\alpha \in (0,1)$ be given as in Lemma~\ref{lem:local_existence}. For 
		\begin{align*}
			\tau = \min \left\{ \frac14 \bigg(\|v_0\|_{C^{1+\alpha}_x}+ \|Z\|_{C_{ T}C^{2+\alpha}_x} \bigg)^{-1} , T\right\},
		\end{align*}
		and suppose that $v\in C([0,\tau), C^{\infty}(\T^3,\mR^3))$ is the unique solution to \eqref{euler system}. To obtain the well-posedness of $v$ on $[0,\tau]$, it suffices to verify $T^*>\tau$ by using Lemma~\ref{lemma:e1}. To this end,
		we take $N=1$ in \eqref{estimate: uN} and derive the following for any $t\in [0,\tau)$
		\begin{align*}
			\|v(t)\|_{C^{1+\alpha}_x}
			& \lesssim_{\alpha} 	\|v_0\|_{C^{1+\alpha}_x}+\int_{0}^{t}  \big(\|v(s)\|_{C^{1+\alpha}_x}+\|Z(s)\|_{C^{1+\alpha}_x}\big)^2  \dif s
			\\ &\qquad  \qquad \qquad + \int_{0}^{t}  \|Z(s)\|_{C^{2+\alpha}_x}\big(\|v(s)\|_{C^{\alpha}_x}+\|Z(s)\|_{C^{\alpha}_x}\big)  \dif s
			\\&\lesssim_{\alpha} 	\|v_0\|_{C^{1+\alpha}_x}+ \int_{0}^{t} \big(\|v(s)\|_{C^{1+\alpha}_x}+\|Z\|_{C_{ T}C^{2+\alpha}_x}\big)^2 \dif s .
		\end{align*} 
		Then, we use Lemma~\ref{gronwall lemma} to derive for any $t\in [0,\tau)$
		\begin{align}\label{estimate: u1}
			\|v(t)\|_{C^{1+\alpha}_x}  \lesssim_{\alpha}   \|v_0\|_{C^{1+\alpha}_x}+\|Z\|_{C_{ T }C^{2+\alpha}_x}<\infty.
		\end{align} 
		By Lemma~\ref{lemma:e1}, we deduce $T^*>\tau$ and $v\in C([0,\tau ],C^{\infty}(\T^3,\R^3))$. Thus, Lemma~\ref{lem:local_existence} is proven for $N=1$.
		
		For $N\geq 2$, substituting \eqref{estimate: u1} into \eqref{estimate: uN} and using Gr\"onwall's inequality we obtain for any $t\in [0,\tau ]$
		\begin{equation*}
			\aligned
			\|v(t)\|_{C^{N+\alpha}_x}&\lesssim_{N,\alpha}  	\|v_0\|_{C^{N+\alpha}_x} +  \tau  \|Z\|_{C_{ T}C^{N+1+\alpha}_x}\big(\|v_0\|_{C^{1+\alpha}_x}+\|Z\|_{C_{ T}C^{2+\alpha}_x}\big)\\ &\qquad \qquad \qquad +\big(\|v_0\|_{C^{1+\alpha}_x}+\|Z\|_{C_{ T}C^{2+\alpha}_x}\big) \int_{0}^{t}  \|v(s)\|_{C^{N+\alpha}_x} \dif s
			\\ &\lesssim_{N,\alpha}   \|v_0\|_{C^{N+\alpha}_x} +  \tau  \|Z\|_{C_{ T}C^{N+1+\alpha}_x}\big(\|v_0\|_{C^{1+\alpha}_x}+\|Z\|_{C_{ T }C^{2+\alpha}_x}\big).
			\endaligned
		\end{equation*}
		Hence, Lemma~\ref{lem:local_existence} is proven.
	\end{proof}

	\section*{Acknowledgement} 
The authors would like to thank Martina Hofmanov\'a and Xiangchan Zhu for their helpful comments on an earlier version of this paper. This research has received financial supports from the National Key R\&D Program of China (No. 2022YFA1006300) and the NSFC grant of China (No. 12426205, No. 12271030).  The financial support by the DFG through the CRC 1283 "Taming uncertainty and profiting from randomness and low regularity in analysis, stochastics and their applications" is greatly acknowledged.
		
	\section*{Declarations}
	Data sharing not applicable to this article as no datasets were generated or analyzed during the	current study. The authors have no competing interests to declare that are relevant to the content	of this article.

	\def\cprime{$'$} \def\ocirc#1{\ifmmode\setbox0=\hbox{$#1$}\dimen0=\ht0
		\advance\dimen0 by1pt\rlap{\hbox to\wd0{\hss\raise\dimen0
				\hbox{\hskip.2em$\scriptscriptstyle\circ$}\hss}}#1\else {\accent"17 #1}\fi}


\begin{thebibliography}{BDLSVMMM88}
		
		
		
		
		
		
		\bibitem[Buc15]{Buc15}
		T. Buckmaster. 
		\newblock Onsager’s conjecture almost everywhere in time. 
		{\em Commun. Math. Phys.} $\mathbf{333}$(3), 1175–1198 (2015)
		
		
		\bibitem[BDLIS15]{BDLIS16}
		T. Buckmaster, C. De~Lellis, P. Isett,  L.
		Sz\'{e}kelyhidi, Jr.
		\newblock Anomalous dissipation for {$1/5$}-{H}\"{o}lder {E}uler flows.
		{\em Ann. Math.} (2) $\mathbf{ 182}$(1), 127--172 (2015)
		
		
		
		\bibitem[BDLSV19]{BDLSV19}
		T. Buckmaster, C. De Lellis, L. Sz\'ekelyhidi Jr, V. Vicol.
		\newblock Onsager’s conjecture for admissible weak solutions. 
		{\em Comm. Pure Appl. Math.} $\mathbf{72}$(2), 227–274 (2019)
		
		
			\bibitem[Ber24]{Ber24}
		S.E. Berkemeier.
		\newblock Existence and Non-Uniqueness of Ergodic Leray-Hopf Solutions to the Stochastic Power-Law Flows.
		{\em arXiv:2412.08622}, (2024)
		
		\bibitem[BF99]{BF99}
		\newblock  H. Bessaih, F.Flandoli, 2-D Euler equation perturbed by noise.
		{\em NoDEA, Nonlinear differ. equ. appl.} $\mathbf{6}$, 35–54 (1999)
		
		
		\bibitem[BFH20]{BFH20} D. Breit, E. Feireisl, M. Hofmanov\'{a}. \newblock{On solvability and ill-posedness of the compressible Euler system subject to stochastic forces.} \newblock{\em Anal. PDE.} $\mathbf{13}$ (2), 371–402 (2020)
		
		\bibitem[BFM16]{BFM16}		
		\newblock
		Z. Brzez\'niak, F. Flandoli, M. Maurelli. Existence and uniqueness for stochastic 2D Euler flows with bounded vorticity. 
		\newblock{\em Arch Ration Mech Anal.} $\mathbf{221}$(1), 107–142 (2016)
		
		
		\bibitem[BHP23]{BHP23}
		A. Bulut, M. Khang Huynh,  S. Palasek.
		Convex integration above the onsager exponent for the forced Euler equations.
		{\em arXiv: 2301.00804v1} (2023)
		
		\bibitem[BLW24]{BLW24}
		F. Bechtold, T. Lange, J. Wichmann.
		\newblock On convex integration solutions to the surface quasi-geostrophic equation driven by generic additive noise.
		\newblock{\em Electron. J. Probab.} $\mathbf{29}$, 1--38 (2024)
		
		
		\bibitem[BM02]{BM02}
		A. L. Bertozzi, A. J. Majda, 
		\newblock Vorticity and Incompressible Flow. 
		Cambridge Texts in Applied Mathematics 27. Cambridge Univ. Press, (2002) 
		
		
		\bibitem[BP01]{BP01}
		Z. Brzez\'niak, S. Peszat.
		\newblock Stochastic two dimensional Euler equations.
		\newblock{\em Ann. Probab.} $\mathbf{29}$(4), 1796--1832 (2001)
		
		\bibitem[BV19]{BV19}
		T. Buckmaster, V. Vicol.
		\newblock Convex integration and phenomenologies in turbulence.
		\newblock {\em EMS Surv. Math. Sci.} $\mathbf{6}$(1-2), 173--263 (2019)
		
		\bibitem[CC99]{CC99}
		M. Capi\'nski, N. J. Cutland.
		\newblock Stochastic Euler equations on the torus.
		\newblock{\em Ann. Appl. Probab.} $\mathbf{9}$(3), 688--705 (1999)
		
		
		\bibitem[CDS12]{CDS12}
		S. Conti, C.~De~Lellis, L.~Sz{\'e}kelyhidi, Jr.
		\newblock h-principle and rigidity for $C^{1,\alpha}$ isometric embeddings. {\em Nonlinear partial differential equations.} Abel Symposia, $\mathbf{7}$. Springer, Heidelberg, (2012)
		
		
		\bibitem[CDZ24]{CDZ22}
		W. Chen, Z. Dong, X. Zhu.
		\newblock Sharp Nonuniqueness of Solutions to Stochastic Navier–Stokes Equations.
		\newblock{\em SIAM J. Math. Anal.}  $\mathbf{56}$ (2024)
		
		\bibitem[CET94]{CET94}
		P. Constantin, W. E, and E.S. Titi. 
		\newblock Onsager’s conjecture on the energy conservation for solutions of Euler’s equation. 
		\newblock{\em Comm. Math. Phys.} $\mathbf{165}$(1), 207--209 (1994)
		
		\bibitem[CFF21]{CFF19} 
		E. Chiodaroli, E. Feireisl, F. Flandoli.
		\newblock Ill posedness for the full euler system driven by multiplicative white noise.
		\newblock{\em Indiana Univ. Math. J.} $\mathbf{70}$, 1267--1282 (2021)
		
		
		\bibitem[CFH19]{CFH19}
		D. Crisan, F. Flandoli, D. D. Holm.
		\newblock Solution properties of a 3D stochastic Euler fluid equation. 
		\newblock{\em J. Nonlinear Sci.} $\mathbf{29 }$(3), 813--870 (2019)
		
			\bibitem[CLZ24]{CLZ24}
		W. Cao, Y. Li, D. Zhang.
		\newblock  Existence and non-uniqueness of probabilistically strong solutions to 3D stochastic magnetohydrodynamic equations.
		\newblock{\em  arXiv: 2408.05450.} (2024)
		
		\bibitem[CZZ24]{CZZ24}
		A. Cheskidov, Z. Zeng, D. Zhang.  
		\newblock Existence and non-uniqueness of weak solutions with continuous energy to the 3D deterministic and stochastic Navier-Stokes equations.
		{\em arXiv: 2407.17463.} (2024)
		
		
		\bibitem[DLS13]{DelSze13}
		C.~De~Lellis and L.~Sz{\'e}kelyhidi, Jr.
		\newblock Dissipative continuous {E}uler flows.
		{\em Invent. Math.} $\mathbf{193}$, 377--407 (2013)
		
		\bibitem[DLS14]{DLS14}
		C.~De~Lellis and L.~Sz{\'e}kelyhidi, Jr.
		\newblock Dissipative Euler flows and Onsager’s conjecture.
		{\em J. Eur. Math. Soc.} $\mathbf{16}$ (7): 1467--1505, (2014)
		
		\bibitem[DS17]{DS17}
		S. Daneri, L. Sz\'ekelyhidi, Jr.
		\newblock Non-uniqueness and h-Principle for H\"older-Continuous Weak Solutions of the Euler Equations.
		{\em Arch Rational Mech Anal.} $\mathbf{224}$, 471--514 (2017)
		
		\bibitem[DPZ92]{DPZ92} G. Da Prato, J. Zabczyk. Stochastic Equations in Infinite Dimensions, Encyclopedia of Mathematics and its Applications, vol. 44, Cambridge University Press, Cambridge, (1992)
		
		\bibitem[GHV14]{GV!4}
		N.E. Glatt-Holtz, V. Vicol.
		\newblock{Local and global existence of smooth solutions for the stochastic Euler equations with multiplicative noise. }
		\newblock{\em Ann. Probab.} $\mathbf{42}$(1), 80--145 (2014)
		
		
		
		\bibitem[GKN23]{GKN23}
		V. Giri, H. Kwon, M. Novack. 
		\newblock The $L^3$-based strong Onsager theorem.
		{\em arXiv:2305.18509}, (2023)
		
		
		\bibitem[GR24]{GR23}
		V. Giri, R.-O. Radu. 
		\newblock The 2D Onsager conjecture: a Newton-Nash iteration. 
		{\em Invent. math.} $\mathbf{238}$, 691--768 (2024)
		
		\bibitem[HLP24]{HLP22}
		M. Hofmanov{\'a}, T. Lange, and U. Pappalettera. 
		\newblock Global existence and non-uniqueness of 3D Euler equations
		perturbed by transport noise.
		{\em Probab. Theory Relat. Fields.} $\mathbf{188}$, 1183--1255 (2024)
		
		
		
		\bibitem[HLZZ24]{HLZZ24}
		M. Hofmanov{\'a}, X. Luo, R. Zhu, X. Zhu.
		\newblock Surface quasi-geostrophic equation perturbed by derivatives of space-time white noise.
		{\em Math. Ann.} $\mathbf{390}$, 5111–5152 (2024)
		
		\bibitem[HPZZ25]{HPZZ25}
		M. Hofmanov{\'a}, U. Pappalettera, R.Zhu, X.Zhu.
		\newblock Kolmogorov 4/5 law for the forced 3D Navier–Stokes equations.
		{\em Stoch PDE: Anal Comp.} $\mathbf{13}$, 2085--2108 (2025).
		
		
		\bibitem[HZZ22]{HZZ22a}
		M. Hofmanov{\'a}, R. Zhu,  X. Zhu.
		\newblock On ill- and well-posedness of dissipative martingale solutions to
		stochastic 3{D} {E}uler equations.
		{\em Comm. Pure Appl Math.} $\mathbf{75}$(11), 2446--2510 (2022)
		
		\bibitem[HZZ23a]{HZZ21markov}
		M. Hofmanov{\'a}, R. Zhu, X. Zhu.
		\newblock Global-in-time probabilistically strong and Markov solutions to stochastic 3D Navier--Stokes equations: existence and non-uniqueness.
		{\em Ann. Probab.}  $\mathbf{51}$(2), 524--579 (2023)
		
		
		\bibitem[HZZ23b]{HZZ23b}
		M. Hofmanov{\'a}, R. Zhu, X. Zhu.
		\newblock Global existence and non-uniqueness for 3D Navier–Stokes equations with space-time white noise.
		{\em Arch Rational Mech Anal.} $\mathbf{247}$ (46), (2023)
		
		\bibitem[HZZ23c]{HZZ23c}
		M. Hofmanov{\'a}, R. Zhu, X. Zhu.
		\newblock A class of supercritical/critical singular stochastic PDEs: exis- tence, non-uniqueness, non-Gaussianity, non-unique ergodicity. 
		\newblock {\em J. Funct. Anal.} $\mathbf{285}$ (5), (2023)
		
		\bibitem[HZZ24]{HZZ19}
		M. Hofmanov{\'a}, R. Zhu, X. Zhu.
		\newblock Non-uniqueness in law of stochastic 3D Navier--Stokes
		equations.
		{\em J. Eur. Math. Soc.} $\mathbf{26}$, 163--260 (2024)
	
		
		
		\bibitem[HZZ25]{HZZ22b}
		M. Hofmanov{\'a}, R. Zhu, X. Zhu.
		\newblock Non-unique ergodicity for deterministic and stochastic 3D Navier–Stokes equations and euler equations. 
		\newblock {\em Arch Rational Mech Anal.} $\mathbf{249}$ (33), (2025)
		
		
		\bibitem[Ise18]{Ise18}
		P. Isett. \newblock A proof of Onsager’s conjecture. {\em Ann.Math.} $\mathbf{188}$(3), 871–963 (2018)
		
		\bibitem[LR15]{LR15} W. Liu, M.~R\"ockner. Stochastic Partial Differential Equations: An Introduction. Springer, Berlin, (2015)
		
		
			\bibitem[LRS24]{LRS24} 
		T. Lange, M. Rehmeier, A. Schenke.
		\newblock Non-uniqueness of Leray--Hopf solutions for the 3D fractional Navier--Stokes equations perturbed by transport noise.
		{\em arXiv:2412.16532}, (2024)
		
			\bibitem[L\"u25]{Lu24} L. L\"u.
		H\"{o}lder continuous solutions to stochastic 3D Euler equations via stochastic convex integration. 
		{\em   J. Evol. Equ.} $\mathbf{25}$, 41 (2025)
		
		
		\bibitem[LZ23]{LZ23} H. L\"{u}, X. Zhu.
		\newblock Global-in-time probabilistically strong solutions to stochastic power-law equations: Existence and non-uniqueness. 
		{\em Stochastic Process. Appl.} $\mathbf{164}$, 62-98 (2023)
		
		
		
		
		\bibitem[LZ24]{LZ24} L. L\"{u}, R. Zhu.
		\newblock Stationary solutions to stochastic 3D Euler equations in H\"{o}lder space. \newblock{\em Stochastic Process. Appl.}  $\mathbf{177}$ (2024)
		
	
		
		
		\bibitem[LZ25a]{LZ24a} H. L\"{u}, X. Zhu.
		\newblock Non-unique ergodicity for the 2D stochastic Navier-Stokes equations with derivative of space-time white noise.
		{\em J. Differ. Equ.}  $\mathbf{164}$, 383-433 (2025)
		
			\bibitem[LZ25b]{LZ23b} H. L\"{u}, X. Zhu.
		\newblock Sharp non-uniqueness of solutions to 2D Navier-Stokes equations with space-time white noise.
		{\em Ann. Appl. Probab.}  $\mathbf{35}$ (3), 1980-2030 (2025)
		
		
		\bibitem[Kim09]{Kim09}
		J. U. Kim.
		\newblock  Existence of a local smooth solution in probability to the stochastic Euler equations in $\R^d$. 
		{\em J. Funct. Anal.} $\mathbf{256}$(11), 3660--3687 (2009)
		
		
		\bibitem[KK24]{KK24}
		K. Kinra, U. Koley.
		\newblock Non-uniqueness of Hölder continuous solutions for stochastic Euler and Hypodissipative Navier-Stokes equations.
		{\em arXiv:2407.20270,} (2024)
		
		
		\bibitem[KMY22]{KMY22} C. Khor, C. Miao, W.Ye.
		\newblock Infinitely many non-conservative solutions for the three-dimensional Euler equations with arbitrary initial data in $C^{1/3-\varepsilon}$.
		\newblock{\em arXiv:2204.03344v1,} (2022)
		
		\bibitem[MV00]{MV00}
		R. Mikulevicius, G. Valiukevicius.
		\newblock On stochastic Euler equation in $\R^d$. 
		\newblock{\em Electron. J. Probab.} $\mathbf{5}$, 1-20 (2000)
		
		\bibitem[NV23]{NV23}
		M. Novack, V. Vicol. 
		\newblock An intermittent Onsager theorem.
		{\em Invent. Math.} $\mathbf{233}$(1), 223--323 (2023)
		
		\bibitem[Ons49]{Ons49} L. Onsager.
		\newblock Statistical hydrodynamics.
		{\em Nuovo Cim. Suppl.} $\mathbf{6}$, 279--287 (1949)
		
		
		\bibitem[Pap24]{Umb23} U. Pappalettera.
		\newblock Global existence and non-uniqueness for the Cauchy problem associated to 3D Navier–Stokes equations perturbed by transport noise. 
		\newblock{\em Stoch PDE: Anal Comp.} $\mathbf{12}$, 1769--1804 (2024)
		
		
		\bibitem[RS23]{RS23} M. Rehmeier, A. Schenke.
		\newblock  Nonuniqueness in law for stochastic hypodissipative Navier–Stokes equations.
		{\em Nonlinear Anal.} $\mathbf{227}$(4), 113179 (2023)
		
		
		
		\bibitem[Sch93]{Sch93} V. Scheffer.
		\newblock An inviscid flow with compact support in space-time.
		{\em J. Geom. Anal.} $\mathbf{3}$(4), 343--401 (1993)
		
		
		\bibitem[Shn97]{Shn97} A. Shnirelman.
		\newblock On the nonuniqueness of weak solution of the Euler equation. 
		{\em 	Comm. Pure Appl. Math.} $\mathbf{50}$(12), 1261--1286 (1997)
		
		\bibitem[Shn00]{Shn00} A. Shnirelman.
		\newblock Weak solutions with decreasing energy of incompressible Euler equations.
		{\em Comm. Math. Phys.} $\mathbf{3}$, 541--603 (2000)
		
		
		\bibitem[Tem76]{Tem76} R. Temam.
		Local existence of $C^{\infty}$ solutions of the Euler equations of incompressible perfect fluids, in Turbulence and Navier Stokes Equations. Lecture Notes in Mathematics. $\mathbf{565}$, 184–194. Springer Berlin Heidelberg, (1976)
		
		
		\bibitem[WY24]{WY24} E. Walker, K. Yamazaki.
		\newblock  	Surface quasi-geostrophic equations forced by random noise: prescribed energy and non-unique Markov selections.
		{\em arXiv:2407.00920.} (2024) 
		
		
		\bibitem[Yam22a]{Yam22a} K. Yamazaki.
		\newblock  	Non-uniqueness in law for the Boussinesq system forced by random noise.
		{\em Calc. Var. Partial Differential Equations.} $\mathbf{61}$(177), (2022)
		
		
		\bibitem[Yam22b]{Yam22b} K. Yamazaki.
		\newblock  	Nonuniqueness in law for two-dimensional Navier–Stokes equations with diffusion weaker than a full Laplacian.
		{\em SIAM J. Math. Anal.} $\mathbf{54}$, 3997–4042 (2022)
		
		\bibitem[Yam22c]{Yam22c} K. Yamazaki.
		\newblock  	Remarks on the non-uniqueness in law of the Navier–Stokes equations up to the J.-L. Lions’ exponent.
		{\em Stochastic Process. Appl.} $\mathbf{147}$, 226–269 (2022) 
	
		
		
		\bibitem[Yam23]{Yam23} K. Yamazaki.
		\newblock  	Non-uniqueness in law of the surface quasi-geostrophic equations: the case of linear multiplicative noise.
		{\em arXiv:2312.15558v2,} (2023)
		
		\bibitem[Yam24a]{Yam24a} K. Yamazaki.
		\newblock  	Non-uniqueness in law of three-dimensional Navier–Stokes equations diffused via a fractional Laplacian with power less than one half.  
		{\em Stoch PDE: Anal Comp.} $\mathbf{12}$, 794–855 (2024) 
		
		\bibitem[Yam24b]{Yam24b} K. Yamazaki.
		\newblock  	Non-uniqueness in law of three-dimensional magnetohydrodynamics system forced by random noise. 
		{\em Potential Anal.} (2024) 
		
			\bibitem[Yam25]{Yam22d} K. Yamazaki.
		\newblock  Non-uniqueness in law of the two-dimensional surface quasi-geostrophic equations forced by random noise.
		{\em  Ann. Inst. H. Poincar\'e Probab. Statist.} $\mathbf{61}$(1), 457-509 (2025) 
		
		\bibitem[Yud63]{Yud63} V. I. Yudovich.
		\newblock  Non-stationary flow of an ideal incompressible liquid. 
		{\em USSR Computational Mathematics and
			Mathematical Physics.} $\mathbf{3}$(6), 1407–1456 (1963)
		
		
		
	\end{thebibliography}
\end{document}